%% file: main.tex
\documentclass[aos,preprint]{imsart}

\RequirePackage{amsthm,amsmath,amsfonts,amssymb}
\RequirePackage{dsfont}
\RequirePackage[authoryear]{natbib}
\RequirePackage{algorithm}
\RequirePackage[colorlinks,citecolor=blue,urlcolor=blue,linkcolor=blue]{hyperref}

\startlocaldefs
\makeatletter\def\journal@name{}\makeatother
\newtheorem{theorem}{Theorem}[section]
\newtheorem{lemma}{Lemma}[section]
\newtheorem{corollary}{Corollary}[section]
\newtheorem{prop}{Proposition}[section]
\newtheorem{assumption}{Assumption}[section]

\theoremstyle{definition}

\newtheorem{remark}{Remark}[section]

\DeclareMathOperator{\E}{\mathbf{E}}
\DeclareMathOperator{\p}{\mathbf{P}}
\endlocaldefs

\begin{document}
\begin{frontmatter}
\title{Adaptive Inference in Multivariate Nonparametric Regression Models Under Monotonicity}
\runtitle{Adaptive inference under monotonicity}
\begin{aug}
\author[A]{\fnms{Koohyun}~\snm{Kwon}\thanksref{t1}\ead[label=e1]{koohyun.kwon@capitalone.com}}
\author[B]{\fnms{Soonwoo}~\snm{Kwon}\ead[label=e2]{soonwoo\_kwon@brown.edu}}
\address[A]{Capital One, McLean, Virginia, USA\printead[presep={,\ }]{e1}}
\address[B]{Brown University, Providence, Rhode Island, USA\printead[presep={,\ }]{e2}}
\thankstext{t1}{The views expressed are those of the author and do not
necessarily reflect those of Capital One.}
\runauthor{K. Kwon and S. Kwon}
\end{aug}

\begin{abstract}
  We consider the problem of adaptive inference on a regression function at a
  point under a multivariate nonparametric regression setting. The regression
  function belongs to a H\"older class and is assumed to be monotone with
  respect to some or all of the arguments. We derive the minimax rate of
  convergence for confidence intervals (CIs) that adapt to the underlying
  smoothness, and provide an adaptive inference procedure that obtains this
  minimax rate. When the regression function is monotone in $k_{+}$ of its $k$
  arguments, the minimax expected length over a H\"older class with exponent
  $\gamma'$ of CIs that maintain coverage over a rougher class with exponent
  $\gamma$ is of order $n^{-1/(2+k_{+}/\gamma'+(k-k_{+})/\gamma)}$. The rate is
  unchanged if the rougher class restricts the smoothness only at the point,
  along the nonmonotone arguments.
  Under monotonicity along a closed convex cone that contains no line, only the
  dimension of its linear span enters the rate. The procedure differs from that of
  \cite{cai2004adaptation}, intended to yield shorter CIs under practically
  relevant specifications. It is valid in finite samples under Gaussian errors
  with a known variance. The proposed method applies to general linear
  functionals of the regression function.
\end{abstract}



\end{frontmatter}

\section{Introduction}

We consider the problem of inference on a regression function at a point under
the nonparametric regression model
\begin{equation*}
  y_{i} = f(x_{i}) + u_{i}, \quad u_{i} \sim N(0, \sigma^{2}(x_{i}))\text{ independently},
\end{equation*}
where $f$ is assumed to lie in a H\"older class with exponent
$\gamma \in (0,1]$. Procedures based on $\gamma$ are conservative (or suboptimal) when the
true regression function in fact lies in a smoother H\"older class with
$\gamma' > \gamma$. Adaptive procedures try to overcome this issue by
automatically adjusting to the (unknown) underlying smoothness class. However,
unlike in the case of estimation, where adaptation to the unknown smoothness
class is in general possible with an additional logarithmic term
(\citealp{lepskii1991ProblemAdaptiveEstimationa}), adaptation is impossible in the case of inference without further
restrictions on the function class (\citealp{low1997nonparametric}).

Two shape restrictions that can be used to overcome this impossibility have been
discussed in the literature, convexity and monotonicity. In this paper, we
impose monotonicity on the regression function to construct a confidence interval (CI) that adapts to
the underlying smoothness of the regression function. The main difference from other papers that consider
adaptation under a monotonicity condition (e.g.,
\citealp{cai2013AdaptiveConfidenceIntervals}) is our general treatment of the
dimension of $x_{i}$. We allow the regression function to be monotone in
only some of the variables. In multivariate regression,
it can be reasonable to assume that the response is nondecreasing in a
treatment dose or a production input, while leaving its dependence on the
other covariates unrestricted \citep[see][for a review of shape restrictions
in econometrics]{chetverikov2018econometrics}. To our knowledge, this is the
first paper to study CIs that adapt to the unknown H\"older exponent when the
regression function is monotone in only some of the variables. In a regression
discontinuity setting, \cite{kwon2020InferenceRegressionDiscontinuity} adapt
to the Lipschitz constant under this type of monotonicity.

We consider coordinate-wise monotonicity with respect to all or some of the
coordinates. Write $x = (x_1,\dots,x_k)$ and $z = (z_1,\dots,z_k)$ so that $x, z \in \mathbb{R}^k$. Then, a function $f$ is \textit{coordinate-wise monotone} with respect to
$\mathcal{V} \subseteq \{1, \dots, k\}$ if $x_{v}\geq z_{v}$ for all
$v \in\mathcal{V}$ and $x_{v}= z_{v}$ for all $v \notin \mathcal{V}$ imply
$f(x) \geq f(z)$. The minimax expected length of a CI over the H\"older class
with exponent $\gamma$ converges to $0$ at the well-known rate of
$n^{-1/(2 + k/\gamma)}$. Suppose that the regression function is monotone
in all variables, i.e., $\mathcal{V} = \{1, \dots, k\}$. Then, just as in the
univariate case, we can construct a CI that achieves this minimax rate for
each of finitely many exponents in $(0,1]$ while maintaining coverage over the
roughest class. Also, again as in the univariate case, if the
regression function is not monotone in any of the variables so that
$\mathcal{V} = \emptyset$, there is no scope for adaptation.

An interesting case is when the function is monotone with respect to only some
of the variables (partial monotonicity), so that
$k_{+} := \lvert \mathcal{V} \rvert < k$, which can arise due to the multivariate nature of the problem. In this case, we
show that for a CI that maintains coverage over the H\"older class with exponent
$\gamma$, the minimax expected length over a smoother class with exponent $\gamma' > \gamma$
converges to $0$ at the rate $n^{-1/(2 + k_{+}/\gamma' +
  (k-k_{+})/\gamma)}$. This holds for almost every i.i.d.\ design whose density is
continuous and positive at an interior point of interest. The denominator of the exponent can be written as
$2 + k/\gamma - k_{+}( 1/\gamma- 1/\gamma')$. This is the sum of a term that comes from the minimax rate over $\gamma$, $2 + k/\gamma$, and $- k_{+}( 1/\gamma- 1/\gamma')$. In this sense,
$k_{+}( 1/\gamma- 1/\gamma')$ exactly quantifies the possible gain from
monotonicity, indicating larger gains if the regression function is monotone in
more variables and/or smoother. Hence, with partial monotonicity, the rate of
adaptation lies strictly between the minimax rates over the two classes.

We also give the rates at boundary points of the support, where the lower
and upper ends of a CI can adapt at different rates. The rate result extends to
smoothness that differs across coordinates and to a finite family of candidate
sets of monotone coordinates. Under monotonicity along a closed convex cone of
directions that contains no line, the rate depends on the cone only through
the dimension of its linear span.

Partial monotonicity also changes the role of the bound on the smoothness of
the largest parameter space. When the regression function is monotone in all
variables, monotonicity alone is enough for honest inference. In
Section~\ref{sec:weaker}, we show that a CI can then maintain coverage over all
monotone functions and still attain the minimax rate over a given H\"older
class. \cite{datta2026optimal} construct honest confidence bands with a
similar property. On the other hand, when the regression
function is monotone in only some of the variables, every CI that maintains
coverage over all functions that are monotone in those variables has infinite
expected length, for almost every design. Some bound
on the smoothness of the largest parameter space is then necessary. However,
we show that it suffices to bound the smoothness at $x_{0}$ along the
variables in which the regression function is not monotone. This weaker bound
gives the same rate of adaptation, which our procedure attains.

We propose a CI that obtains this minimax rate (of adaptation) for a sequence of
H\"older exponents $\{\gamma_{j}\}_{j=1}^{J} \subset (0, 1]$. While the method
provided by \cite{cai2004adaptation} can be used to construct such a CI, we
provide an alternative method that builds upon the one-sided CI proposed by
\cite{armstrong2018optimal}. Their one-sided CI ``directs power'' to a smoother class while maintaining
coverage over a larger class of functions. Our CI is constructed by combining
the lower and upper versions of their one-sided CI to create a two-sided CI, and
then taking the intersection of a sequence of such two-sided CIs that direct
power to each $\gamma_{j}$. An appropriate Bonferroni correction is used to
obtain correct coverage. We then show how to improve upon the Bonferroni
correction using two features of the problem: the one-sided CIs are
correlated, and a lower and an upper one-sided CI cannot both be least
favorable at the same regression function. For a single parameter space, the
resulting CI is a fixed-length CI of the type proposed by \cite{Donoho1994}.

The CI has coverage at least $1-\alpha$ in finite samples under Gaussian
errors with a known variance, and its calibration requires only univariate and
bivariate normal probabilities and thus involves no simulation error. When
the errors are homoskedastic, the CI remains asymptotically valid, uniformly
over the largest parameter space and over the distribution of the errors, even
if the noise level is estimated and the errors are not Gaussian. This extends
to heteroskedastic errors with a local estimate of the noise level. Moreover,
with Gaussian errors, a known variance and a grid of about $\log n$
H\"older exponents, the CI adapts to every exponent in an interval at the
cost of a factor of at most $(\log\log n)^{1/(2+k)}$. This CI can be used
in more general nonparametric regression settings, as long as the parameter of
interest is a linear functional of the regression function and the regression
function lies in a convex function class.

While the proposed CI obtains the minimax length over $\gamma_{j}$ for each $j$
up to a constant factor that does not depend on the sample size, this constant
does depend on the number of parameter spaces $J$ the CI adapts to. This is in
contrast with the CI of \cite{cai2004adaptation} for nested parameter spaces,
whose bound on the constant does not depend on $J$. However, our bound grows
slowly with $J$, at a $(\log J)^{1/2}$ rate, and it is smaller than theirs for
any reasonable $J$.\footnote{Even if one wishes to adapt to $J=10^{3}$
  parameter spaces, our bound on the constant is $4.14$ for $\alpha=0.05$,
  whereas theirs is $16$.} For the parameter spaces we consider, it grows only
like $(\log J)^{r}$ with $r\leq1/(2+k)$ for large $n$. Their construction is
designed to establish rates with a constant that does not depend on $J$ or
$n$. In the design of Section~\ref{sec:sim-one}, their CI does not adapt at
the sample size we consider. Nonetheless,
their uniform constant is theoretically attractive. Combined with our rate
result, it gives a CI that adapts to every H\"older exponent in an interval
without the factor above. This extends their result for univariate monotone
H\"older classes \citep[Section~3.2]{cai2004adaptation}.

Simulation results given in Section~\ref{sec:Simulation} confirm the role of
monotonicity that the rate result predicts. Consider a design with three
covariates, $n=20{,}000$ and a constant regression function. When monotonicity
is imposed on two of the covariates, the proposed CI is $14\%$ shorter than the
minimax CI for the largest parameter space, and it is $52\%$ shorter when
monotonicity is imposed on all three. When it is imposed on only one, the gain
is small at the sample sizes we consider. In a design with two covariates, the
CI of \citet{deng2021confidence}, designed for regression functions that are
smooth at the point of interest, is shorter than our CI that maintains
coverage over all monotone functions. However, its coverage probability falls
to about $0.7$ at functions with a cusp at the point of interest, which their
theory does not cover, while ours stays above $0.97$.

\medskip
\noindent \textbf{Related literature. }An adaptation theory for CIs in a nonparametric regression setting was developed
by \cite{cai2004adaptation}. \cite{cai2013AdaptiveConfidenceIntervals} provide a
procedure for constructing adaptive CIs that adapt to each individual function
under monotonicity or under convexity. \cite{armstrong2015AdaptiveTestingRegression}
provides an inference method for the regression function at a point, possibly
on the boundary of the support, that adapts to the underlying H\"older classes
under a monotonicity assumption.
\cite{duembgen2003optimal} constructs adaptive confidence bands for monotone
and convex regression functions. These papers consider the univariate case. As
noted earlier, the main difference of our paper is that we consider a
multivariate regression setting where there is no restriction on the dimension
of the independent variable as long as it is fixed and finite. Moreover, we
allow the regression function to be monotone in only some of the variables. The
adaptation theory for CIs builds upon the more classical minimax theory for
CIs, which has been developed in \cite{Donoho1994} and
\cite{low1997nonparametric}. \cite{cai2012MinimaxAdaptiveInference} provides an
excellent review of the theory of minimax and adaptive CIs, along with the
minimax and adaptive estimation problems. We construct our CI using the
minimax one-sided CIs of \cite{armstrong2018optimal}. See also
\cite{armstrong2020simple} for simple CIs that are valid under explicit
smoothness restrictions.

In a companion paper,
\cite{kwon2020InferenceRegressionDiscontinuity} consider a regression
discontinuity setting with a Lipschitz bound, possibly several running
variables, and monotonicity in some or all of them. That paper constructs minimax
two-sided CIs and one-sided CIs that adapt to the Lipschitz constant. When the
regression function is monotone in all the running variables, such one-sided
CIs can maintain coverage over all monotone functions and still have finite
excess length. The companion paper uses, with $\gamma=1$, the solution to the
modulus problem in Section~\ref{sec:solut-modul-probl}, which first appeared in
the 2020 version of this paper \citep{kwon2020adaptive}. It also uses a version
of the calibration in Section~\ref{sec:constr-adapt-cis} for intersections of
lower one-sided CIs, based only on their correlation. The rate theory, the parameter spaces of
Section~\ref{sec:weaker}, the two-sided adaptive CIs and their calibration,
and the results on adapting to an interval of exponents and on an estimated
noise level are specific to the present paper.

While the focus of this paper is on adaptive CIs, there are other forms of
confidence sets that are of interest in the context of nonparametric
regression. Adaptive confidence balls have been considered in
\cite{genovese2005ConfidenceSetsNonparametric},
\cite{cai2006AdaptiveConfidenceBalls} and
\cite{robins2006AdaptiveNonparametricConfidencea}. An adaptation theory for
confidence bands has been considered in, for example,
\cite{dumbgen1998NewGoodnessoffitTests},
\cite{genovese2008AdaptiveConfidenceBands}, and
\cite{cai2014AdaptiveConfidenceBands}. In the context of density estimation,
adaptive confidence bands have also been considered in
\cite{hengartner1995FiniteSampleConfidenceEnvelopes},
\cite{gine2010ConfidenceBandsDensitya}, and
\cite{hoffmann2011AdaptiveInferenceConfidencea}.

See \cite{groeneboom2014shape} for a monograph and \cite{guntuboyina2018shape}
for a review of estimation under shape restrictions, and
\cite{dumbgen2001multiscale} for multiscale tests of qualitative hypotheses.
Recently, there has been interest in isotonic regression in general
dimensions. The monotonicity condition imposed in such models is the same as the
one we impose here with $\mathcal{V} = \{1, \dots, k\}$.
\cite{han2019IsotonicRegressionGeneral} derive minimax rates for the least
squares estimation problem, and \cite{han2020limit} derive the limiting
distribution of block estimators at a point. \cite{deng2021confidence} provide
a method for constructing CIs at a point based on block max-min and min-max
estimators. Their CIs adapt to the unknown local smoothness of the regression
function, and their coverage is asymptotically correct at each given
regression function, rather than uniformly over a parameter space.
\cite{deng2020isotonic} and \cite{fokianos2020integrated} study the risk of
block estimators, \cite{chatterjee2018matrix} study bivariate isotonic
regression, and \cite{wang2023coverage} study the coverage of Bayesian
credible intervals in multivariate isotonic regression.
\cite{datta2026optimal} construct confidence bands in a multidimensional white
noise model that are honest over all functions that are monotone in every
coordinate, without any smoothness restriction. Their bands adapt, up to a
logarithmic factor, to the unknown H\"older exponent and constant and to the
intrinsic dimension. These papers impose monotonicity in all
the variables.

\cite{chen2016shape} estimate generalized additive and index
models in which each component satisfies its own shape restriction, such as
linearity, monotonicity or convexity, so that monotonicity can be imposed on
only some of the components. In Section~\ref{sec:Simulation}, we compare our CI with that of
\cite{deng2021confidence} in a design where both apply.

In the univariate case, inference at a point
under monotonicity has also been studied using the limiting distribution of
specific estimators \citep[e.g.,][]{groeneboom2015nonparametric,cattaneo2024bootstrap}
and likelihood ratio statistics \citep{banerjee2001likelihood}.
\cite{chernozhukov2009improving} construct interval estimators of monotone
functions by rearrangement, and \cite{westling2020causal} consider a dose--response
function that is monotone in the treatment, after adjusting for other
covariates.

\medskip
\noindent \textbf{Outline.} Section~\ref{sec:Setup} describes the nonparametric
regression model and the function classes we consider. Section~\ref{sec:adapt-conf-interv} introduces the notion of adaptivity in more detail
and describes our procedure for constructing adaptive CIs. Section~\ref{sec:adapt-infer-f0} presents the main result of the paper, the
minimax rate of adaptation, by solving the corresponding modulus problem,
together with the extensions described above. Section~\ref{sec:adaptive-ci}
provides an adaptive CI that obtains this rate.
Section~\ref{sec:Simulation} provides a simulation study, and
Section~\ref{sec:empir-illustr} illustrates our method in the
context of production function estimation.

Any proof omitted from the main text can be found in the supplement. The
supplement also contains results on the
parameter spaces of Section~\ref{sec:weaker}, on the limit of the calibration,
on adapting to an interval of exponents, on estimating the noise level, on
heteroskedasticity and on adaptation under only monotonicity. It also gives
the details of the extensions in Section~\ref{sec:extensions} and further
remarks. Numbers with the prefix S, such as Section~S2
and Lemma~S2.1, refer to the supplement.

\section{Nonparametric Regression Under Monotonicity\label{sec:Setup}}

We observe $\{ (y_{i},x_{i})\} _{i=1}^{n}$ and consider a
nonparametric regression model,
\begin{equation}
  y_{i}=f(x_{i})+u_{i},\label{eq:npr_model}
\end{equation}
where $x_{i}\in\mathcal{X}\subseteq\mathcal{D}\subset\mathbb{R}^{k}$ is a
(fixed) regressor and $f:\mathcal{D}\to\mathbb{R}$ is the unknown regression
function that lies in some function class $\mathcal{F}$. Here, $\mathcal{D}$
is a bounded set that contains the support $\mathcal{X}$ and on which the
regression function is defined. The $u_{i}$'s are independent with
$u_{i}\sim N(0,\sigma^{2}(x_{i}))$, where $\sigma^{2}(\cdot)$ is known and
$\sigma(x_{i})>0$. The parameter of interest is $f(x_{0})$ for a given
$x_{0}\in\mathcal{D}$. For the rate
results provided in Section~\ref{sec:Asymptotic-rate-results}, we require that
$x_{0} \in \mathrm{Int}\, \mathcal{X}$. In Section~\ref{sec:extensions}, we give
the rates when $x_{0}$ is on the boundary. However, we note that the solution to
the modulus problem given in Section~\ref{sec:solut-modul-probl} does not
depend on whether $x_0$ is on the boundary or not, as long as the domain
condition \eqref{eq:domain} below holds. Without loss of generality, we
normalize $x_{0}$ to be $0$.

We take $\mathcal{F}$ to be the class of functions that are H\"older
continuous and nondecreasing in all or some of the variables. Let
$\Lambda(\gamma,C)$ denote the set of functions from $\mathcal{D}$ to
$\mathbb{R}$ that are H\"older continuous with H\"older constants $(\gamma,C)$,
\begin{equation*}
  \Lambda(\gamma,C):=\left\{ f:\mathcal{D}\to\mathbb{R}:\left|f(x)-f(z)\right\rvert \leq C\left\Vert x-z \right\Vert
    ^{\gamma}\text{ for all }x,z\in\mathcal{D}\right\} ,
\end{equation*}
where $\gamma\in(0,1]$, $C > 0$ and $\left\Vert \cdot\right\Vert $ is a norm on
$\mathbb{R}^{k}$. For $\gamma>1$, monotonicity yields no gain beyond the
Lipschitz rate, as we show in Remark~S8.6. For notational
simplicity, we omit the dependence of the
function class on the choice of the norm $\lVert \cdot \rVert$. We impose the
following restriction that $\lVert \cdot \rVert$ is monotone in the magnitude
of each element, which is satisfied by most norms used in practice, such as
the $\ell_{p}$ norm or a weighted version of it. Remark~S8.3
discusses the relationship between this assumption and the monotonicity of the
regression function, and extends the results to monotonicity along the
coordinates of another orthonormal basis, with a norm that is monotone in those
coordinates, such as the Mahalanobis distance.
\begin{assumption}
  \label{assu:norm}$\left\Vert \cdot\right\Vert $ is a norm on $\mathbb{R}^{k}$
  such that $\left\Vert z \right\Vert $ is nondecreasing in
  $\lvert z_{j}\rvert $ for each $j\in\{1,\dots,k\}$, holding the
  magnitudes of the other elements fixed.
\end{assumption}
Norms with this property are called monotonic and are exactly the
absolute norms, for which $\lVert z\rVert$ depends on $z$ only through
$(\lvert z_{1}\rvert,\dots,\lvert z_{k}\rvert)$ \citep{bauer1961absolute}.

We now define the (coordinate-wise) monotone H\"older class. For a subset of the
covariate indices $\mathcal{V}\subseteq\{ 1,\dots,k\} $, write
\begin{equation*}
  \Lambda_{+,\mathcal{V}}(\gamma,C):=\left\{ f\in\Lambda(\gamma,C):f(x)\geq
    f(z)\,\,\text{if }x_{j}\geq z_{j}\text{ }\forall
    j \in\mathcal{V} \text{ and }x_{j}= z_{j}\,\,\forall
    j\notin\mathcal{V}\right\},
\end{equation*}
where the monotonicity condition is imposed for all $x,z\in\mathcal{D}$. This
is the set of H\"older continuous functions that are nondecreasing,
coordinate-wise, with respect to the $j$th element for $j \in
\mathcal{V}$. No shape restriction is imposed with respect to the elements
not in $\mathcal{V}$. Define $k_{+}:=\lvert \mathcal{V}\rvert .$ By
a relabeling argument, which also permutes the coordinates in the norm, it is
without loss of generality to write $\mathcal{V}:=\{ 1,\dots,k_{+}\} $.  If $k_{+}=k$, then
$\Lambda_{+,\mathcal{V}}(\gamma,C)$ is the set of nondecreasing and H\"older
continuous functions where the monotonicity is with respect to the
coordinate-wise partial ordering on $\mathbb{R}^{k}$.

Finally, for $z=(z_{1},\dots,z_{k})\in\mathbb{R}^{k}$, define the vector
$(z)_{\mathcal{V}+}$ by
\[
  \left[(z)_{\mathcal{V}+}\right]_{v}=\begin{cases}
    \max\{ z_{v},0\}  & v\in\mathcal{V}\\
    z_{v} & v\notin\mathcal{V},
  \end{cases}
\]
and let $(z)_{\mathcal{V}-}=(-z)_{\mathcal{V}+}$. That
is, $(z)_{\mathcal{V}+}$ replaces the negative elements of $z$ in
the monotone coordinates by $0$, and $-(z)_{\mathcal{V}-}$ replaces
the positive ones by $0$. We assume that the domain $\mathcal{D}$ satisfies
\begin{equation}
  x\in\mathcal{D}\quad\Longrightarrow\quad(x)_{\mathcal{V}+}\in\mathcal{D}\text{ and }{-}(x)_{\mathcal{V}-}\in\mathcal{D}.\label{eq:domain}
\end{equation}
Recall that $0\in\mathcal{D}$ after the normalization $x_{0}=0$. This
condition is satisfied, for example, if $\mathcal{D}$ is a rectangle that
contains $0$, or a ball with respect to $\lVert \cdot \rVert$ that is centered
at $0$. It ensures that the functions in
$\Lambda_{+,\mathcal{V}}(\gamma,C)$ are evaluated only on $\mathcal{D}$ when
we solve the modulus problem in Section~\ref{sec:solut-modul-probl}.

\section{Adaptive Confidence Intervals}
\label{sec:adapt-conf-interv}
\subsection{Notion of Adaptivity}

In this section, we discuss the problem of inference for a general linear
functional of the regression function, $Lf$. Consider a sequence of nonempty
convex parameter spaces $\mathcal{F}_{1}$, \dots, $\mathcal{F}_{J}$, with the
requirement that $\mathcal{F}_{j} \subseteq \mathcal{F}_{J}$ for all $j \leq
J$. Note that the parameter spaces are not necessarily nested, but there is a
largest convex parameter space that nests all the other parameter spaces. Here,
$\mathcal{F}_{J}$ reflects a conservative choice of the parameter space in which
the researcher believes the true regression function lies. Hence, the CI we
construct will be required to maintain correct coverage over this space. An
adaptive CI maintains this correct coverage over the largest parameter space
$\mathcal{F}_{J}$ while having good performance (e.g., shorter expected length)
when the true function happens to lie in the smaller parameter space
$\mathcal{F}_{j}$, simultaneously for all $j \leq J$. Throughout, we take
$0<\alpha<1/2$.

Then, a natural question is how well a CI that maintains coverage over
$\mathcal{F}_{J}$ can perform over $\mathcal{F}_{j}$, which is one of the main
questions that \cite{cai2004adaptation} raise and address in detail in the
context of two-sided CIs. The case of one-sided CIs has been considered by
\cite{armstrong2018optimal}, along with other questions.

\subsubsection{Two-Sided Adaptive CIs}
\label{sec:two-sided-adaptive}

Let $\mathcal{I}_{\alpha,2}^{J}$ denote the set of all two-sided CIs that have
coverage at least $1-\alpha$ over $\mathcal{F}_{J}$.\footnote{Formally, a CI is
  a Borel measurable map from the data $\{y_{i}\}_{i=1}^{n}$ to the set of
  closed, possibly unbounded, intervals in $\mathbb{R}$.} Following
\citet{cai2004adaptation}, the performance criterion we consider for two-sided
CIs is the worst-case expected length. That is, the performance of a CI, $CI$,
over the parameter space $\mathcal{F}_{j}$ is measured by
$\sup_{f\in\mathcal{F}_{j}}\E_{f}\mu(CI)$ with smaller values of this quantity
meaning better performance. Here, $\E_{f}$ denotes the expectation when the true
regression function is $f$ and $\mu$ is the Lebesgue measure on the real
line. Then, the shortest possible worst-case expected length a CI can
achieve over $\mathcal{F}_{j}$ (while maintaining correct coverage over
$\mathcal{F}_{J}$) is characterized by the quantity
\begin{equation*}
  L_{j,J}^{\ast}:=\inf_{CI\in\mathcal{I}_{\alpha,2}^{J}}\sup_{f\in\mathcal{F}_{j}}\E_{f}\mu(CI).
\end{equation*}

Following \citet{cai2004adaptation}, we say a CI is \textit{adaptive} if it
achieves $L_{j,J}^{\ast}$ for all $j \leq J$ up to a multiplicative constant
that does not depend on the sample size. Let $z_{q}$ denote the $q$-quantile of
the standard normal distribution, and define $\omega_{+}(\delta,\mathcal{F}_{j},\mathcal{F}_{J})$, the \textit{between
  class modulus of continuity}, as
\begin{align*}
  &\omega_{+}(\delta,\mathcal{F}_{j},\mathcal{F}_{J}) \\
  &:=\sup\left\{
      \lvert Lf_{J}-Lf_{j}\rvert :\textstyle\sum_{i=1}^{n}\left((f_{J}(x_{i})-f_{j}(x_{i}))/\sigma(x_{i})\right)^{2}\leq\delta^{2},f_{j}\in\mathcal{F}_{j},f_{J}\in\mathcal{F}_{J}\right\},
\end{align*}
for $\delta \geq 0.$\footnote{Unlike in \cite{cai2004adaptation}, the
  summands are divided by $\sigma^{2}(x_{i})$, which normalizes
  \eqref{eq:npr_model} to a Gaussian model with unit noise variance.} Then,
the following lemma, which restates results of \cite{cai2004adaptation} for
the model \eqref{eq:npr_model}, shows that $L_{j,J}^{\ast}$ is characterized by
the between class modulus, up to multiplicative constants.
\begin{lemma}[\citealp{cai2004adaptation}]
  \label{lem:CL-bounds}For any $j\leq J$,
  \begin{equation*}
    \left(\frac{1}{2}-\alpha\right)\omega_{+}(z_{1-\alpha},\mathcal{F}_{j},\mathcal{F}_{J})\leq L_{j,J}^{\ast}\leq\left(4+\frac{9}{z_{1-\alpha/2}}\right)\omega_{+}(z_{1-\alpha/2},\mathcal{F}_{j},\mathcal{F}_{J}).
  \end{equation*}
\end{lemma}
The proof verifies that the arguments of \cite{cai2004adaptation} apply to
our setting. Since
$\mathcal{F}_{j}\subseteq\mathcal{F}_{J}$ and both classes are convex,
$\omega_{+}(\delta',\mathcal{F}_{j},\mathcal{F}_{J})\leq(\delta'/\delta)\,\omega_{+}(\delta,\mathcal{F}_{j},\mathcal{F}_{J})$
for $0<\delta\leq\delta'$. Hence, the lemma implies
$L_{j,J}^{\ast} \asymp\omega_{+}(z_{1-\alpha},\mathcal{F}_{j},\mathcal{F}_{J})$
as $n\to\infty$, where $a_{n}\asymp b_{n}$ means that $a_{n}/b_{n}$ is
bounded away from zero and infinity. In general,
$\omega_{+}(z_{1-\alpha},\mathcal{F}_{j},\mathcal{F}_{J})$ is more tractable than
$L_{j,J}^{\ast}$, and thus the strategy is to construct a CI that has worst-case
length over $\mathcal{F}_{j}$ bounded by
$\omega_{+}(z_{1-\alpha/2},\mathcal{F}_{j},\mathcal{F}_{J})$, up to a multiplicative
constant. We refer to the rate at which
$\omega_{+}(z_{1-\alpha}, \mathcal{F}_{j}, \mathcal{F}_{J}) $ converges to $0$ as the
\textit{minimax rate of adaptation} (of $\mathcal{F}_{j}$ over
$\mathcal{F}_{J}$). If $\mathcal{F}_{j}=\mathcal{F}_{J}$, this is the minimax
rate over $\mathcal{F}_{J}$, which is the fastest rate at which the worst-case
expected length over $\mathcal{F}_{J}$ of a CI that maintains correct coverage
over the same space $\mathcal{F}_{J}$ can converge to $0$.

\subsubsection{One-Sided CIs}
\label{sec:one-sided-cis}
The construction of our adaptive CI relies heavily on the one-sided CI
proposed by \cite{armstrong2018optimal}, who measure the performance of a
one-sided lower CI, $[\hat{c},\infty)$, by the $\beta$th quantile of its excess
length at $f$, $q_{\beta,f}(Lf-\hat{c})$, for some $\beta\in(\alpha,1)$. Under
this criterion, the best possible performance over $\mathcal{F}_{j}$ is
quantified by
\begin{equation*}
  \ell_{j,J}^{\ast}:=\inf_{\hat{c}:[\hat{c},\infty)\in\mathcal{I}_{\alpha,\ell}^{J}}\sup_{f\in\mathcal{F}_{j}}q_{\beta,f}(Lf-\hat{c}),
\end{equation*}
where $\mathcal{I}_{\alpha, \ell}^{J}$ denotes the set of all one-sided lower
CIs that have coverage at least $1-\alpha$ over $\mathcal{F}_{J}.$
\citet{armstrong2018optimal} showed that, under regularity conditions given in
Lemma~\ref{thm:ArmKol} below,
$\ell_{j,J}^{\ast}=\omega(z_{1-\alpha}+z_{\beta},\mathcal{\mathcal{F}}_{J},\mathcal{\mathcal{F}}_{j})$,
where
$\omega(z_{1-\alpha}+z_{\beta},\mathcal{\mathcal{F}}_{J},\mathcal{\mathcal{F}}_{j})$
is the \textit{ordered class modulus of continuity} defined as
\begin{align}
\label{eq: omega-def}
\begin{split}
      &\omega(\delta,\mathcal{F}_{j},\mathcal{F}_{k}) \\
  &:=\sup\left\{
      Lf_{k}-Lf_{j}:\textstyle\sum_{i=1}^{n}\left((f_{k}(x_{i})-f_{j}(x_{i}))/\sigma(x_{i})\right)^{2}
      \leq \delta^{2},f_j \in\mathcal{F}_{j},f_k \in\mathcal{F}_k\right\},
\end{split}
\end{align}
for any $\delta \geq 0$ and $j,k \leq J$. We refer to the optimization problem
in the definition as the ordered modulus problem. Naturally, an analogous result
holds for upper one-sided CIs:
$u_{j,J}^{\ast}:=\inf_{\hat{c}:(-\infty,\hat{c}]\in\mathcal{I}_{\alpha,u}^{J}}\sup_{f\in\mathcal{F}_{j}}q_{\beta,f}(\hat{c}-Lf)=\omega(z_{1-\alpha}+z_{\beta},\mathcal{F}_{j},\mathcal{F}_{J})$,
with $\mathcal{I}_{\alpha, u}^{J}$ denoting the set of all one-sided upper CIs
that have coverage at least $1-\alpha$ over $\mathcal{F}_{J}$. We say a
one-sided lower CI, $[\hat{c}^{\ast},\infty)$, is adaptive if
$\sup_{f\in\mathcal{F}_{j}}q_{\beta,f}(Lf-\hat{c}^{\ast})\leq
c\,\omega(z_{1-\alpha}+z_{\beta},\mathcal{F}_{J},\mathcal{F}_{j})$ for all
$j\leq J$, for some $c>0$ that does not depend on $n$, and similarly for
one-sided upper CIs.

\subsubsection{Modes of Adaptation}
\label{sec:scope-adaptation}

Since $\mathcal{F}_{j}\subseteq\mathcal{F}_{J}$,
$\omega_{+}(z_{1-\alpha},\mathcal{F}_{j},\mathcal{F}_{j})\leq\omega_{+}(z_{1-\alpha},\mathcal{F}_{j},\mathcal{F}_{J})\leq\omega_{+}(z_{1-\alpha},\mathcal{F}_{J},\mathcal{F}_{J})$,
and similarly for the ordered moduli. Suppose that
$\omega_{+}(z_{1-\alpha},\mathcal{F}_{j},\mathcal{F}_{j})\asymp\omega_{+}(z_{1-\alpha},\mathcal{F}_{j},\mathcal{F}_{J})$
for all $j\leq J$. Then, an adaptive CI attains $L_{j,j}^{\ast}$ up to a
multiplicative constant, which is the minimax length we could have achieved
if we ``knew'' that our true regression function lies in the smaller class
$\mathcal{F}_{j}$. \citet[Section~2.2]{cai2004adaptation} call such a CI
\textit{strongly adaptive}. Whether strong adaptation is possible is not a
property of a given procedure, but of the given statistical model. The least
desirable case is when
$\omega_{+}(z_{1-\alpha},\mathcal{F}_{j},\mathcal{F}_{J})\asymp\omega_{+}(z_{1-\alpha},\mathcal{F}_{J},\mathcal{F}_{J})$,
because this leaves no scope for adaptation. An intermediate case is when
$\omega_{+}(z_{1-\alpha},\mathcal{F}_{j},\mathcal{F}_{j})\prec\omega_{+}(z_{1-\alpha},\mathcal{F}_{j},\mathcal{F}_{J})\prec\omega_{+}(z_{1-\alpha},\mathcal{F}_{J},\mathcal{F}_{J})$,
with $a_{n}\prec b_{n}$ meaning that $b_{n}/a_{n}\to\infty$. Then, the minimax
rate of adaptation is better than the
worst-case minimax rate over $\mathcal{F}_{J}$ but not as good as the minimax
rate over $\mathcal{F}_{j}$. That is, one can do better than simply taking the
most conservative parameter space as the true space but not quite as good as
knowing that the true function actually lies in the smaller parameter space.
In Section~\ref{sec:Asymptotic-rate-results}, we derive the minimax rates of
adaptation under the model given in Section~\ref{sec:Setup}.

\subsection{Construction of Adaptive CIs}
\label{sec:constr-adapt-cis}
\citet{cai2004adaptation} provide a general method of constructing adaptive CIs
of $Lf$ under the general model \eqref{eq:npr_model}.  Here, we provide an
alternative method that is intuitive and gives smaller constants in the case of
nonnested parameter spaces.\footnote{For a given adaptive CI, $CI^{\ast}$, we
  refer to the positive number $c$ (that does not depend on $n$) such that
  $ \sup_{f\in\mathcal{F}_{j}}\E_{f}\mu(CI^{\ast})\leq
  c\,\omega_{+}(z_{1-\alpha/2},\mathcal{F}_{j},\mathcal{F}_{J})$ for
  all $j\leq J$, as the ``constant'' of $CI^{\ast}$.} In the nested case, too,
our bound on the constant is smaller than that of \cite{cai2004adaptation} for
any reasonable number of parameter spaces.

The main building block for our adaptive CI is the minimax one-sided CI proposed
by \citet{armstrong2018optimal}, which relies on the ordered modulus. We say
that $(f_{j}, f_{k}) \in \mathcal{F}_{j} \times \mathcal{F}_{k} $ is a solution
to $\omega(\delta, \mathcal{F}_{j}, \mathcal{F}_{k})$ if $(f_{j}, f_{k})$ solves
the optimization problem corresponding to
$\omega(\delta, \mathcal{F}_{j}, \mathcal{F}_{k})$ defined in (\ref{eq: omega-def}).  Let
$(f_{J,{\delta}}^{*,Jj},
g_{j,{\delta}}^{*,Jj})\in\mathcal{F}_{J}\times\mathcal{F}_{j}$ be a solution to
the ordered modulus
$\omega({\delta},\mathcal{F}_{J},\mathcal{F}_{j}),$ and define the
estimator
\begin{align}
  \begin{aligned}
    \hat{L}_{\delta}^{\ell,j} &= \frac{1}{2}\, L({g_{j,\delta}^{*,Jj}
      +f_{J,\delta}^{*,Jj}}) \\
    &\quad+
    \frac{\omega'(\delta,\mathcal{F}_{J},\mathcal{F}_{j})}{\delta}\sum_{i=1}^{n}\frac{g_{j,\delta}^{*,Jj}(x_{i})-f_{J,\delta}^{*,Jj}(x_{i})}{\sigma^{2}(x_{i})}\Bigg(y_{i}-\frac{g_{j,\delta}^{*,Jj}(x_{i})+f_{J,\delta}^{*,Jj}(x_{i})}{2}\Bigg),
  \end{aligned}
\label{eq:lhat_lower}
\end{align}
where $\omega'(\cdot, \mathcal{F}_{J}, \mathcal{F}_{j})$ is the derivative of
$\omega(\cdot, \mathcal{F}_{J}, \mathcal{F}_{j})$. If the solution to the
ordered modulus is not unique, we fix one solution for each $\delta$ before
observing the data. Based on this estimator, define a lower one-sided CI by
subtracting the maximum bias and an appropriately scaled normal quantile:
\begin{equation}\label{eq:opt lower CI}
  \hat{c}_{\alpha, \delta}^{\ell,j}:=  \hat{L}_{\delta}^{\ell,j}-\frac{1}{2}\omega(\delta,\mathcal{F}_{J},\mathcal{F}_{j})+\frac{1}{2}\delta\omega'(\delta,\mathcal{F}_{J},\mathcal{F}_{j})-z_{1-\alpha}\omega'(\delta,\mathcal{F}_{J},\mathcal{F}_{j}).
\end{equation}
The following lemma, which restates Theorem~3.1 of
\citet{armstrong2018optimal}, shows that for a specific choice of $\delta$,
this CI is optimal in the sense that it achieves
$\ell^{\ast}_{j,J}.$
\begin{lemma}[{\citealp[Theorem~3.1]{armstrong2018optimal}}]
  \label{thm:ArmKol}Let $\underline{\delta}=z_{\beta}+z_{1-\alpha}$, and suppose
  that the ordered modulus $\omega(\cdot,\mathcal{F}_{J},\mathcal{F}_{j})$ is
  finite and differentiable at $\underline{\delta}$, and that it is attained
  at $\underline{\delta}$ by a solution for which the constraint binds. Then,
  $[\hat{c}_{\alpha,\underline{\delta}}^{\ell,j},\infty)$ has coverage at
  least $1-\alpha$ over $\mathcal{F}_{J}$, and
  \begin{equation*}
    \underset{f\in\mathcal{F}_{j}}{\sup}q_{\beta,f}(Lf-\hat{c}_{\alpha,
      \underline{\delta}}^{\ell,j})= \ell^{\ast}_{j,J}= \omega(\underline{\delta},\mathcal{F}_{J},\mathcal{F}_{j}).
  \end{equation*}
\end{lemma}
If the ordered modulus is attained and
$\omega'(\underline{\delta},\mathcal{F}_{J},\mathcal{F}_{j})>0$, the
constraint necessarily binds. Indeed, if it did not bind, the solution would
remain feasible for slightly smaller values of $\delta$, and the modulus would
be constant to the left of $\underline{\delta}$. For the parameter spaces we consider in
Section~\ref{sec:adapt-infer-f0}, the ordered moduli are differentiable with
positive derivatives, which can be seen from their explicit form.

The excess length of this CI follows a Gaussian distribution because it is an affine transformation of the data, which
follows a Gaussian distribution by assumption. Hence, the median and mean of the
excess length are the same. Taking $\beta=1/2$, we can replace $q_{\beta,f}$
with the expectation under $f$, which gives
\begin{equation}
  \underset{f\in\mathcal{F}_{j}}{\sup}\E_{f}(Lf-\hat{c}_{\alpha}^{\ell,j})
  = \omega(z_{1-\alpha},\mathcal{F}_{J},\mathcal{F}_{j}),\label{eq:exp_ex_len}
\end{equation}
where we define
$\hat{c}_{\alpha}^{\ell,j}:= \hat{c}_{\alpha, z_{1-\alpha}}^{\ell,j}.$ Likewise,
let $(f_{j,\delta}^{*,jJ},g_{J,\delta}^{*,jJ})\in\mathcal{F}_{j}\times\mathcal{F}_{J}$
be a solution to $\omega(\delta,\mathcal{F}_{j},\mathcal{F}_{J})$, and let
$\hat{L}_{\delta}^{u,j}$ and $\hat{c}_{\alpha, \delta}^{u,j}$ be given by
\eqref{eq:lhat_lower} and \eqref{eq:opt lower CI}, with
$\omega(\delta,\mathcal{F}_{j},\mathcal{F}_{J})$, its derivative,
$f_{j,\delta}^{*,jJ}$ and $g_{J,\delta}^{*,jJ}$ in place of
$\omega(\delta,\mathcal{F}_{J},\mathcal{F}_{j})$, its derivative,
$f_{J,\delta}^{*,Jj}$ and $g_{j,\delta}^{*,Jj}$, and with the signs of the
last three terms of \eqref{eq:opt lower CI} reversed. Under the conditions of
Lemma~\ref{thm:ArmKol} with $\mathcal{F}_{J}$ and $\mathcal{F}_{j}$
interchanged in the ordered modulus, the upper one-sided CI
$(-\infty, \hat{c}_{\alpha, \underline{\delta}}^{u,j}]$ has coverage at least
$1-\alpha$ over $\mathcal{F}_{J}$, and
\begin{equation}
  \sup_{f\in\mathcal{F}_{j}} q_{\beta,f}(\hat{c}_{\alpha, \underline{\delta}}^{u,j}-Lf) =
  u^{\ast}_{j,J} = \omega(\underline{\delta},\mathcal{F}_{j},\mathcal{F}_{J}).\label{eq:one-sd-optim-up}
\end{equation}
Similarly, let
$\hat{c}_{\alpha}^{u,j}:=\hat{c}_{\alpha, z_{1-\alpha}}^{u,j}$.

Using the optimal one-sided CIs, we first show how a naive Bonferroni procedure
leads to a two-sided adaptive CI. We then provide a method that improves upon
this naive Bonferroni CI by taking into account the correlation among the
CIs. Let $A_{\alpha}=\max_{j\leq J}\hat{c}_{\alpha/(2J)}^{\ell,j}$ and
$B_{\alpha}=\min_{j\leq J}\hat{c}_{\alpha/(2J)}^{u,j}$. The naive Bonferroni CI
is defined as
\begin{equation}\label{eq:naive_Bon_CI}
  CI_{\alpha}^{Bon,J}:=\begin{cases}
    [A_{\alpha},B_{\alpha}] & \text{if }A_{\alpha}\leq B_{\alpha}\\
    \{(A_{\alpha}+B_{\alpha})/2\} & \text{otherwise}.
  \end{cases}
\end{equation}
That is, $CI_{\alpha}^{Bon,J}$ is the intersection
$\cap_{j=1}^{J}[\hat{c}_{\alpha/(2J)}^{\ell,j},\hat{c}_{\alpha/(2J)}^{u,j}]$,
except that an empty intersection is replaced by the singleton
$\{(A_{\alpha}+B_{\alpha})/2\}$, which has length zero. The intersection can be
empty only if one of the one-sided CIs fails to cover $Lf$, which has
probability at most $\alpha$ under the conditions below. Hence, an empty
intersection rejects, at level $\alpha$, the hypothesis that
$f\in\mathcal{F}_{J}$. This can happen, for example, if $\mathcal{F}_{J}$
imposes too small a H\"older constant. Under the conditions
of Theorem~\ref{thm:BonfCI} below, this CI has coverage at least $1-\alpha$
over $\mathcal{F}_{J}$ because each $[\hat{c}_{\alpha/(2J)}^{\ell,j},\hat{c}_{\alpha/(2J)}^{u,j}]$ has
coverage at least $1 - \alpha/J$ over $\mathcal{F}_{J}$, and $Lf\in
[A_{\alpha},B_{\alpha}]$ if $Lf$ is contained in all such CIs. The
following theorem shows that this CI is indeed adaptive.
\begin{theorem}
  \label{thm:BonfCI} Suppose that, for each $j\leq J$, the ordered moduli
  $\omega(\cdot,\mathcal{F}_{J},\mathcal{F}_{j})$ and
  $\omega(\cdot,\mathcal{F}_{j},\mathcal{F}_{J})$ are attained at
  $\delta_{J}:=z_{1-\alpha/(2J)}$, and that they are differentiable there with
  positive derivatives.\footnote{The positive derivative
    condition is imposed only for simplicity. If a derivative is zero, the
    corresponding one-sided CI is nonrandom, and the theorem continues to
    hold.} Then, $CI_{\alpha}^{Bon,J}$ has coverage at least
  $1-\alpha$ over $\mathcal{F}_{J}$, and for any $j=1,\dots,J$, we have
  \begin{equation}
    \sup_{f\in\mathcal{F}_{j}}\E_{f} \mu(CI_{\alpha}^{Bon,J})\leq\frac{2\{\delta_{J}\Phi(\delta_{J})+\phi(\delta_{J})\}}{z_{1-\alpha/2}}\,\omega_{+}(z_{1-\alpha/2},\mathcal{F}_{j},\mathcal{F}_{J}),\label{eq:BonBound}
  \end{equation}
  where $\Phi$ and $\phi$ denote the standard normal cdf and pdf, respectively.
\end{theorem}
Since $\Phi(\delta_{J})\approx1$ and $\phi(\delta_{J})\approx0$, the
constant in \eqref{eq:BonBound} is approximately
$2z_{1-\alpha/(2J)}/z_{1-\alpha/2}$. The small difference
accounts for the possibility that the intersection is empty. The constant
increases with the number of parameter spaces $J$.\footnote{The constant
  grows with $J$ at the rate $(\log J)^{1/2}$. This is the same rate that
  \cite{cai2004adaptation} find for nonnested parameter spaces in their
  Theorem~3 and Remark~7, and their Section~5.2 shows that this rate cannot be
  improved in general. See Remark~S8.1 for a
  comparison of the constants. All such comparisons are of upper bounds, which are not necessarily
  sharp. The lower bound in Lemma~\ref{lem:CL-bounds} has the constant
  $1/2-\alpha$.} For the parameter spaces we consider in
Section~\ref{sec:adapt-infer-f0}, the constant in \eqref{eq:BonBound} can be
sharpened considerably. For large $n$, it grows only like $(\log J)^{r}$ for
some $r\leq1/(2+k)$, rather than like $(\log J)^{1/2}$. For $\alpha=0.05$
and $r=1/5$, it is about $2.2$ for $J=4$ and $2.7$ for $J=10^{3}$. See
Section~S4.1. On the other hand, when the
parameter spaces are nested, the bound established in
\citet[Theorem~2]{cai2004adaptation} is 16 (for $\alpha\leq0.2$), and thus
does not depend on the number of parameter spaces.\footnote{However, we note
  that our bound is not too large, in fact smaller than $16$, for reasonable
  specifications of $J$. For example, when $\alpha=0.05$ and $J=50$, the
  constant in \eqref{eq:BonBound} is approximately $3.36$, which is
  considerably smaller than the bound given in \citet{cai2004adaptation}. Even
  for unrealistically large $J$ such as $J=10^{10}$, the constant is
  approximately $7.05$, which is still less than half of the bound given by
  \citet{cai2004adaptation}.} Their construction is designed to establish
rates with a constant that does not depend on $J$ or $n$, rather than to
optimize the constant. In the design of Section~\ref{sec:sim-one}, its
thinning step keeps only the largest parameter space, so that their CI does
not adapt there at the sample size we consider. Simulation results
given in Section~\ref{sec:Simulation} show that, in the designs we consider,
not only the upper bound, but also the actual length itself is much shorter
for our CI.

An adaptive one-sided CI can be obtained in the same way, by taking the
maximum of the $J$ lower ends with level $\alpha/J$ (Remark~S8.2).

The naive CI given in \eqref{eq:naive_Bon_CI} does not take into account the
possible correlation among the CIs that we take the intersection of. However, if
parameter spaces are ``close'' to each other, the corresponding CIs will be
correlated, implying that there is room for improvement over the Bonferroni
procedure. Consider the CIs of the form $CI_{\tau}^{\mathcal{J}}$, defined as
in \eqref{eq:naive_Bon_CI} with $\alpha/(2J)$ replaced by $\tau$,
where $\mathcal{J}=\{1,\dots,J\}$. If
we take $\tau = \alpha/(2J)$, this is precisely the CI given in
\eqref{eq:naive_Bon_CI}. Consider the largest possible $\tau$ such that
$CI_{\tau}^{\mathcal{J}}$ has correct coverage over $\mathcal{F}_{J}$:
\begin{equation*}
  \tau^{\dagger}:= \sup_{\tau} \tau \,\, \text{ s.t. } \inf_{f \in
    \mathcal{F}_{J}}\p _{f}(Lf \in CI_{\tau}^{\mathcal{J}})  \geq 1- \alpha.
\end{equation*}
We know that $\tau = \alpha/(2J)$ satisfies the constraint, and also that any
$ \tau > \alpha$ does not because then $[\hat{c}_{\tau}^{\ell,J}, \infty)$ will
have coverage probability $1-\tau < 1 - \alpha$ at the least favorable
function. Hence, we can restrict $\tau$ to lie in $[\alpha/(2J), \alpha]$.

However, the coverage probability
$ \inf_{f \in \mathcal{F}_{J}}\p _{f}(Lf \in CI_{\tau}^{\mathcal{J}})$ is unknown in
general, rendering $\tau^{\dagger}$ infeasible. Instead, we
replace this coverage probability with a lower bound that we can calculate
numerically. Then, we take $\tau^{\ast}$ as the
largest value that makes this lower bound at least $1-\alpha$. Using
$\tau^{\ast}$ rather than $\alpha/(2J)$ leads to a weakly smaller critical
value, $z_{1-\tau^{\ast}}\leq z_{1-\alpha/(2J)}$. In the designs of
Sections~\ref{sec:sim-one} and \ref{sec:sim-both} with four H\"older exponents,
this reduces the average length of the CI by about $9\%$ to $14\%$ at the
regression functions in Tables~\ref{tab:sim-one} and \ref{tab:fm-lattice}.
However, the calibration does not
necessarily shorten the CI in every sample because
the estimators $\hat{L}_{\delta}^{\ell,j}$ and
$\hat{L}_{\delta}^{u,j}$ also change with $\tau$ through $\delta=z_{1-\tau}$.

The lower bound takes two features of the problem into account. First, the
one-sided CIs are correlated. Second, a lower end and an upper end cannot both
be least favorable at the same regression function. The lower ends are most
likely to fail when $f$ is large at the design points relative to $Lf$, and
the upper ends when it is small. To describe the bound, write $\delta_{\tau}:=z_{1-\tau}$ and
$\varepsilon_{i}:=u_{i}/\sigma(x_{i})$. Define the
vectors $d_{j}^{-}(\tau)$ and $d_{j}^{+}(\tau)$ in $\mathbb{R}^{n}$ by
\begin{equation*}
  d_{j,i}^{-}(\tau)=\frac{g_{j,\delta_{\tau}}^{*,Jj}(x_{i})-f_{J,\delta_{\tau}}^{*,Jj}(x_{i})}{\sigma(x_{i})}\quad\text{and}\quad d_{j,i}^{+}(\tau)=\frac{g_{J,\delta_{\tau}}^{*,jJ}(x_{i})-f_{j,\delta_{\tau}}^{*,jJ}(x_{i})}{\sigma(x_{i})},
\end{equation*}
and let $V(\tau)=(V_{1}(\tau),\dots,V_{J}(\tau))'$ and
$W(\tau)=(W_{1}(\tau),\dots,W_{J}(\tau))'$, where
\begin{equation*}
  V_{j}(\tau)=\frac{d_{j}^{-}(\tau)'\varepsilon}{\delta_{\tau}}\quad\text{and}\quad W_{j}(\tau)=-\frac{d_{j}^{+}(\tau)'\varepsilon}{\delta_{\tau}},
\end{equation*}
with $\varepsilon=(\varepsilon_{1},\dots,\varepsilon_{n})'$. As we show in the
proof of Lemma~\ref{lem:Bonf-improv}, $V_{j}(\tau)$ and $W_{j}(\tau)$ are the
random parts of the standardized errors of the lower and upper ends of
$CI_{\tau}^{\mathcal{J}}$. The vector $(V(\tau)', W(\tau)')'$ is centered
Gaussian, and each element has unit variance when the constraints in the
ordered moduli bind. The covariance terms are given by
\begin{equation}
  \begin{aligned}
    \mathrm{Cov}(V_{j}(\tau),V_{\ell}(\tau)) & =\frac{d_{j}^{-}(\tau)'d_{\ell}^{-}(\tau)}{\delta_{\tau}^{2}},\\
    \mathrm{Cov}(W_{j}(\tau),W_{\ell}(\tau)) & =\frac{d_{j}^{+}(\tau)'d_{\ell}^{+}(\tau)}{\delta_{\tau}^{2}},\text{ and}\\
    \mathrm{Cov}(V_{j}(\tau),W_{\ell}(\tau)) & =-\frac{d_{j}^{-}(\tau)'d_{\ell}^{+}(\tau)}{\delta_{\tau}^{2}}.
  \end{aligned}\label{eq:VWcov}
\end{equation}
Note the negative sign in the covariance between $V(\tau)$ and $W(\tau)$. In
the setting of Section~\ref{sec:adapt-infer-f0}, the elements of
$d_{j}^{\pm}(\tau)$ are all nonnegative. Hence, the elements of $V(\tau)$ are
nonnegatively correlated with each other, and likewise for $W(\tau)$, while
$V_{j}(\tau)$ and $W_{\ell}(\tau)$ are nonpositively correlated.

Whether an end of the CI fails to cover $Lf$ also depends on its bias at $f$,
which is at most its maximum bias over $\mathcal{F}_{J}$. Write
$s_{j}^{-}=\omega'(\delta_{\tau},\mathcal{F}_{J},\mathcal{F}_{j})$ and
$s_{j}^{+}=\omega'(\delta_{\tau},\mathcal{F}_{j},\mathcal{F}_{J})$ for the
standard deviations of the lower and upper ends that direct power to
$\mathcal{F}_{j}$, and, for $f\in\mathcal{F}_{J}$, let
\begin{align*}
  a_{j}(f) & =\frac{1}{s_{j}^{-}}\Big[\frac{1}{2}\{\omega(\delta_{\tau},\mathcal{F}_{J},\mathcal{F}_{j})-\delta_{\tau}s_{j}^{-}\}-\big(\E_{f}\hat{L}_{\delta_{\tau}}^{\ell,j}-Lf\big)\Big],\\
  c_{j}(f) & =\frac{1}{s_{j}^{+}}\Big[\frac{1}{2}\{\omega(\delta_{\tau},\mathcal{F}_{j},\mathcal{F}_{J})-\delta_{\tau}s_{j}^{+}\}+\big(\E_{f}\hat{L}_{\delta_{\tau}}^{u,j}-Lf\big)\Big].
\end{align*}
The first term in each bracket is the maximum bias of the corresponding
estimator over $\mathcal{F}_{J}$ (in the direction that makes the end fail
to cover), by Lemma~A.1 of \citet{armstrong2018optimal}. Hence, $a_{j}(f)$
and $c_{j}(f)$ are nonnegative and measure, in units of the standard
deviation, how far the bias of each end at $f$ is from its worst case. We refer
to them as the bias slacks of the ends. As we show in the proof of Lemma~\ref{lem:Bonf-improv}, the lower end that directs
power to $\mathcal{F}_{j}$ fails to cover $Lf$ if and only if
$V_{j}(\tau)>z_{1-\tau}+a_{j}(f)$. Likewise, the upper end that directs power to
$\mathcal{F}_{\ell}$ fails if and only if $W_{\ell}(\tau)>z_{1-\tau}+c_{\ell}(f)$.

Let $K_{\tau}\geq0$ be a lower bound on the sum of the bias slacks of a lower
end and an upper end,
\begin{equation}
  a_{j}(f)+c_{\ell}(f)\geq K_{\tau}\quad\text{for all }f\in\mathcal{F}_{J}\text{ and }j,\ell\leq J,\label{eq:slack}
\end{equation}
which always holds with $K_{\tau}=0$. Since $a_{j}(f)+c_{\ell}(f)$ is affine
in $f$, the largest such $K_{\tau}$ is the value of a convex program. For the
parameter spaces of Section~\ref{sec:adapt-infer-f0}, we give a valid
$K_{\tau}$ in closed form in Lemma~\ref{lem:slack}. For $K\geq0$, define
\begin{equation}
  p_{\tau}(K)=\sup_{0\leq t\leq K}\p\Big(\max_{j\leq J}V_{j}(\tau)>z_{1-\tau}+t\ \text{ or }\ \max_{\ell\leq J}W_{\ell}(\tau)>z_{1-\tau}+K-t\Big).\label{eq:pK}
\end{equation}
Then, $p_{\tau}(K)$ is nonincreasing in $K$, and
$p_{\tau}(0)=\p(\max\{V(\tau)',W(\tau)'\}>z_{1-\tau})$. The following lemma
shows how this Gaussian random vector can be used to tune the critical value.

\begin{lemma}
  \label{lem:Bonf-improv}
  Let $\tau\in[\alpha/(2J),\alpha]$, and suppose that, for
  each $j\leq J$, the ordered moduli $\omega(\cdot,\mathcal{F}_{J},\mathcal{F}_{j})$
  and $\omega(\cdot,\mathcal{F}_{j},\mathcal{F}_{J})$ are attained, and
  differentiable with positive derivatives, at $\delta_{\tau}$. Let
  $K_{\tau}\geq0$ satisfy \eqref{eq:slack}. If
  \begin{equation}
     p_{\tau}(K_{\tau}) \leq \alpha,\label{eq:VWmaxCond}
  \end{equation}
  then we have
  $\sup_{f\in\mathcal{F}_{J}}\p_{f}(Lf\notin CI_{\tau}^{\mathcal{J}})\leq\alpha$.
\end{lemma}
The inequality \eqref{eq:VWmaxCond} always holds with $\tau = \alpha/(2J)$
because, by the union bound, $p_{\tau}(K_{\tau})\leq p_{\tau}(0)\leq2J\{1-\Phi(z_{1-\tau})\}$. Let $\tau^{*}\in[\alpha/(2J),\alpha]$
be the largest value of $\tau$ such that \eqref{eq:VWmaxCond} holds. Such a
$\tau^{*}$ exists if the left-hand side of \eqref{eq:VWmaxCond} is
continuous in $\tau$. This is the case in the setting of
Section~\ref{sec:adapt-infer-f0}. There, the vectors $d_{j}^{-}(\tau)$ and
$d_{j}^{+}(\tau)$ are given explicitly in Section~\ref{sec:solut-modul-probl},
do not depend on which solution of the ordered moduli is used, and are
continuous in $\tau$, and $K_{\tau}$ is given by Lemma~\ref{lem:slack}.\footnote{In general, one can instead take
  $\tau^{*}$ to be the largest element of a finite grid on
  $[\alpha/(2J),\alpha]$ that contains $\alpha/(2J)$ at which
  \eqref{eq:VWmaxCond} holds. Such an element always exists because
  \eqref{eq:VWmaxCond} holds at $\alpha/(2J)$.
  Lemma~\ref{lem:Bonf-improv} and Theorem~\ref{thm:calibrated-length} apply
  to any such choice of $\tau^{*}$ that does not depend on
  $\{y_{i}\}_{i=1}^{n}$.} Note that $\tau^{*}$ depends only on the design
points and the parameter spaces, and not on $\{y_{i}\}_{i=1}^{n}$, so that
$CI_{\tau^{*}}^{\mathcal{J}}$ has coverage at least $1-\alpha$ over
$\mathcal{F}_{J}$ by Lemma~\ref{lem:Bonf-improv}. Moreover,
$p_{\tau}(K_{\tau})$ in \eqref{eq:VWmaxCond} can be replaced by any upper
bound $\bar{p}_{\tau}$ that does not depend on $\{y_{i}\}_{i=1}^{n}$ and
satisfies $\bar{p}_{\alpha/(2J)}\leq\alpha$. This gives a weakly smaller
$\tau^{\ast}$, and Lemma~\ref{lem:Bonf-improv} and
Theorem~\ref{thm:calibrated-length} below continue to hold. In
Sections~\ref{sec:Simulation} and \ref{sec:empir-illustr}, we use such a
bound from Lemma~S2.1, which requires only
univariate and bivariate normal probabilities. Hence, with Gaussian errors and
a known variance, the CIs we report have coverage at least $1-\alpha$ in finite
samples, without simulation error.

Taking $K_{\tau}=0$ treats the $2J$ ends as if all of them could be least
favorable at the same $f$, and a positive $K_{\tau}$ can lower the critical
value considerably.
For example, suppose that $J=1$ and the same solution of
$\omega(\delta,\mathcal{F}_{1},\mathcal{F}_{1})$ is used for both ends. Then,
the bias at $f$ cancels in $a_{1}(f)+c_{1}(f)$. Hence, when $\tau^{\ast}$
solves
$p_{\tau}(K_{\tau})=\alpha$, $CI_{\tau^{\ast}}^{\{1\}}$ is the
fixed-length CI of \citet{Donoho1994} based on
$\hat{L}_{\delta}:=\hat{L}_{\delta}^{\ell,1}=\hat{L}_{\delta}^{u,1}$ with
$\delta=\delta_{\tau^{\ast}}$, whereas $K_{\tau}=0$ gives $\tau^{\ast}=\alpha/2$.
See the end of the proof of Lemma~\ref{lem:Bonf-improv}.
Choosing $\delta$ to minimize the length gives the minimax affine fixed-length
CI (FLCI) of \citet{Donoho1994} and \citet{armstrong2018optimal}, which we use
as a nonadaptive benchmark in Section~\ref{sec:Simulation}.

The following theorem shows that $CI_{\tau^{*}}^{\mathcal{J}}$
is adaptive as well.
\begin{theorem}
  \label{thm:calibrated-length}
  Suppose that, for each $j\leq J$ and each
  $\tau\in[\alpha/(2J),\alpha]$, the ordered moduli
  $\omega(\cdot,\mathcal{F}_{J},\mathcal{F}_{j})$ and
  $\omega(\cdot,\mathcal{F}_{j},\mathcal{F}_{J})$ are attained, and
  differentiable with positive derivatives, at $\delta_{\tau}$. Let
  $\tau^{\ast}$ be as defined above, and let
  $\delta^{\ast}:=z_{1-\tau^{\ast}}$. Then, $CI_{\tau^{*}}^{\mathcal{J}}$ has
  coverage at least $1-\alpha$ over $\mathcal{F}_{J}$, and for any
  $j=1,\dots,J$, we have
  \begin{align}
    \begin{aligned}
      \sup_{f\in\mathcal{F}_{j}}\E_{f} \mu(CI_{\tau^{*}}^{\mathcal{J}}) & \leq2\left\{ \Phi(\delta^{\ast})+\frac{\phi(\delta^{\ast})}{\delta^{\ast}}\right\} \omega_{+}(\delta^{\ast},\mathcal{F}_{j},\mathcal{F}_{J})\\
      & \leq\frac{2z_{1-\alpha/(2J)}}{z_{1-\alpha/2}}\left\{ \Phi(z_{1-\alpha})+\frac{\phi(z_{1-\alpha})}{z_{1-\alpha}}\right\} \omega_{+}(z_{1-\alpha/2},\mathcal{F}_{j},\mathcal{F}_{J}).
    \end{aligned}\label{eq:calibrated-length}
  \end{align}
\end{theorem}
The constant in the second bound has the same form as the constant in
\eqref{eq:BonBound}, and differs from it by less than $1.3\%$ when
$\alpha=0.05$. Hence, the CI that takes the correlation into account is
adaptive, with almost the same bound on its constant as the naive
Bonferroni CI. The second bound does not reflect the reduction in length from
using $\tau^{\ast}$, which the first bound does through $\delta^{\ast}$.
In Section~\ref{sec:constr-adapt-ci}, we discuss how $\tau^{\ast}$ behaves as
$n\to\infty$ under our setting of $Lf = f(0)$ with $f$ belonging to a H\"older
class.

\section{Adaptive Inference for \texorpdfstring{$f(0)$}{f(0)}}
\label{sec:adapt-infer-f0}
In this section and the next, we provide an adaptive inference procedure for
$f(0)$. To construct the adaptive CI introduced in
Section~\ref{sec:adapt-conf-interv}, we first solve the corresponding modulus
problem. By using this solution to the modulus problem, we derive the minimax
rate of adaptation. We then show that the rate does not change when the
largest parameter space restricts the smoothness only at $0$ along the
nonmonotone coordinates, and we extend the results to boundary points. In
Section~\ref{sec:adaptive-ci}, we provide a CI that obtains this rate, using
the method described in Section~\ref{sec:adapt-conf-interv}.

\subsection{Solution to the Modulus Problem}
\label{sec:solut-modul-probl}

Let $\gamma_1,\dots,\gamma_J \in (0, 1]$ and $C_1,\dots,C_J > 0$ be such that
$\Lambda_{+,\mathcal{V}}(\gamma_{j},C_{j})\subseteq\Lambda_{+,\mathcal{V}}(\gamma_{J},C_{J})$
for all $j = 1,\dots,J$. A simple sufficient condition is that
$\gamma_{j}\geq\gamma_{J}$ and $C_{j}\leq C_{J}$ for all $j$, and that the
diameter of $\mathcal{D}$ with respect to $\lVert\cdot\rVert$ is at most
$1$. Without such a normalization, the inequalities do not imply the
inclusion. More generally, if the diameter is $\Delta$, then
$\gamma_{j}\geq\gamma_{J}$ and $C_{j}\Delta^{\gamma_{j}-\gamma_{J}}\leq C_{J}$
suffice because $\lVert x-z\rVert^{\gamma_{j}}\leq\Delta^{\gamma_{j}-\gamma_{J}}\lVert x-z\rVert^{\gamma_{J}}$
for $x,z\in\mathcal{D}$.
To construct an adaptive CI, we first calculate the ordered moduli,
$\omega(\delta,\Lambda_{+,\mathcal{V}}(\gamma_{j},C_{j}),\Lambda_{+,\mathcal{V}}(\gamma_{J},C_{J}))$
and
$\omega(\delta,\Lambda_{+,\mathcal{V}}(\gamma_{J},C_{J}),\Lambda_{+,\mathcal{V}}(\gamma_{j},C_{j})),$
for each $j=1,\dots,J.$ For notational simplicity, we consider the case with
$J=2$ and solve
$\omega_{+}(\delta,\Lambda_{+,\mathcal{V}}(\gamma_{1},C_{1}),\Lambda_{+,\mathcal{V}}(\gamma_{2},C_{2}))$, from which the general solution follows immediately.

Recall the definition of the ordered modulus of continuity
\begin{align*}
  &\sup
  f_{2}(0)-f_{1}(0) \\
  \text{s.t. }
  &\sum_{i=1}^{n}\left((f_{2}(x_{i})-f_{1}(x_{i}))/\sigma(x_{i})\right)^{2}\leq\delta^{2}, \,\,
  f_{j}\in\Lambda_{+,\mathcal{V}}(\gamma_{j},C_{j}) \text{ for }j = 1,2,
\end{align*}
with the maximized value denoted by
$\omega(\delta,\Lambda_{+,\mathcal{V}}(\gamma_{1},C_{1}),\Lambda_{+,\mathcal{V}}(\gamma_{2},C_{2})).$
It is convenient to solve the inverse modulus problem instead, which is defined
as
\begin{align}
  \begin{aligned}
    \inf \,
    &\sum_{i=1}^{n}\left((f_{2}(x_{i})-f_{1}(x_{i}))/\sigma(x_{i})\right)^{2} \\
    \text{s.t. }& f_{2}(0)-f_{1}(0)=b,\,\,
    f_{j}\in\Lambda_{+,\mathcal{V}}(\gamma_{j},C_{j}) \text{ for }j =
    1,2,
  \end{aligned}\label{eq:InvMod}
\end{align}
for $b\geq0$, with the square root of the minimized value denoted by the inverse
(ordered) modulus
$
\omega^{-1}(b,\Lambda_{+,\mathcal{V}}(\gamma_{1},C_{1}),\Lambda_{+,\mathcal{V}}(\gamma_{2},C_{2})).$
We provide a closed-form solution for this inverse problem, from which we can
recover the solution to the original problem. For $\delta>0$, the ordered
modulus is the unique $b$ such that
\[
  \omega^{-1}(b,\Lambda_{+,\mathcal{V}}(\gamma_{1},C_{1}),\Lambda_{+,\mathcal{V}}(\gamma_{2},C_{2}))=\delta,
\]
and for $\delta=0$, it is the largest $b$ such that the inverse modulus is
zero. This follows from the closed form below, which shows that the inverse
modulus is continuous and nondecreasing in $b$, and strictly increasing where
it is positive, and from the invariance of the parameter spaces to adding
constants. Since the squared inverse modulus is piecewise quadratic in $b$, with
breakpoints at the thresholds that appear in \eqref{eq:InvModFormula},
such $b$ can be computed in closed form after sorting the thresholds.

To characterize the solution to \eqref{eq:InvMod}, we show two simple lemmas
about the properties of the class
$\Lambda_{+,\mathcal{V}}(\gamma,C)$, using the vectors
$(z)_{\mathcal{V}+}$ and $(z)_{\mathcal{V}-}$ defined in
Section~\ref{sec:Setup}.
\begin{lemma}
  \label{lem:h_plus} Suppose Assumption~\ref{assu:norm} holds, and let
  $\gamma\in(0,1]$ and $C>0$. Define, for $x\in\mathcal{D}$,
  \begin{equation*}
    h_{+}(x)  =  C\lVert  (x)_{\mathcal{V}+}\rVert^{\gamma}\quad\text{and}\quad
    h_{-}(x)  =  -C\lVert (x)_{\mathcal{V}-}\rVert^{\gamma}.
  \end{equation*}
  Then, $h_{+},\ h_{-}\in\Lambda_{+,\mathcal{V}}(\gamma,C)$.
\end{lemma}
The following lemma asserts that the class of functions we consider is closed
under the maximum and minimum operators.
\begin{lemma}
  \label{lem:holder_max_holder}Suppose
  $h_{1},h_{2}\in\Lambda_{+,\mathcal{V}}(\gamma,C)$.  Then,
  $\max \{ h_{1},h_{2}\}
  \in\Lambda_{+,\mathcal{V}}(\gamma,C)$ and $\min \{ h_{1},h_{2}\}
  \in\Lambda_{+,\mathcal{V}}(\gamma,C)$.
\end{lemma}
The next lemma can be used to establish the solutions to the problem
(\ref{eq:InvMod}). It gives the largest and the smallest values that a function
in $\Lambda_{+,\mathcal{V}}(\gamma,C)$ can take at a given point,
once its value at $0$ is fixed. Proposition~4.1 of
\citet{beliakov2005monotonicity} gives related envelope constructions for
Lipschitz functions that are monotone in a subset of the coordinates, which
corresponds to the case $\gamma=1$.
\begin{lemma}
  \label{lem:InvModMax}Suppose Assumption~\ref{assu:norm} and
  \eqref{eq:domain} hold. Given $f_{0}\in\mathbb{R}$, $0<\gamma\leq1$ and
  $C>0$, define
  \[
    \Lambda_{+,\mathcal{V}}^{f_{0}}(\gamma,C)=\left\{
      f\in\Lambda_{+,\mathcal{V}}(\gamma,C):\ f(0)=f_{0}\right\} .
  \]
  Then, for any $x\in\mathcal{D}$, we have
  \begin{align*}
    \underset{f\in\Lambda_{+,\mathcal{V}}^{f_{0}}(\gamma,C)}{\max }f(x) &= f_{0}+C\lVert (x)_{\mathcal{V}+}\rVert^{\gamma}\\
    \underset{f\in\Lambda_{+,\mathcal{V}}^{f_{0}}(\gamma,C)}{\min }f(x) &= f_{0}-C\lVert (x)_{\mathcal{V}-}\rVert^{\gamma}.
  \end{align*}
\end{lemma}
We are now ready to characterize the solution to the inverse modulus problem
(\ref{eq:InvMod}). For $r\in\mathbb{R}$, define
$(r)_{+}:= \max\{ r,0\} $.
\begin{prop}
  \label{prop:InvModRes} Suppose Assumption~\ref{assu:norm} and
  \eqref{eq:domain} hold, and that
  $\Lambda_{+,\mathcal{V}}(\gamma_{1},C_{1})\subseteq\Lambda_{+,\mathcal{V}}(\gamma_{2},C_{2})$.
  For $b\geq0$, define
  \begin{align*}
    f_{1}^{*}(x) &= C_{1}\lVert (x)_{\mathcal{V}+}\rVert^{\gamma_{1}}\\
    f_{2}^{*}(x) &= \max \left\{ b-C_{2}\lVert (x)_{\mathcal{V}-}\rVert^{\gamma_{2}},\ C_{1}\lVert (x)_{\mathcal{V}+}\rVert^{\gamma_{1}}\right\} .
  \end{align*}
  Then, $(f_{1}^{*},f_{2}^{*})$ solves the inverse modulus problem
  (\ref{eq:InvMod}), and the inverse modulus is given by
  \begin{align}
    \begin{aligned}
      &\omega^{-1}(b,\Lambda_{+,\mathcal{V}}(\gamma_{1},C_{1}),\Lambda_{+,\mathcal{V}}(\gamma_{2},C_{2})) \\
       &= \Big( \sum\nolimits_{i=1}^{n}\left(\left(b-C_{1}\left\lVert
                 (x_{i})_{\mathcal{V}+}\right\rVert
             ^{\gamma_{1}}-C_{2}\left\lVert
                 (x_{i})_{\mathcal{V}-}\right\rVert
             ^{\gamma_{2}}\right)/\sigma(x_{i})\right)_{+}^{2} \Big)^{1/2}.
    \end{aligned}\label{eq:InvModFormula}
  \end{align}
\end{prop}
\begin{proof}
  To solve (\ref{eq:InvMod}), note that it is without loss of generality to
  restrict attention to the functions with $f_{1}(0)=0$ and $f_{2}(0)=b$, since
  adding the same constant to $f_{1}$ and $f_{2}$ changes neither the objective
  nor the constraints, and both classes are invariant to adding constants. For
  any such $f_{1}$ and $f_{2}$, Lemma~\ref{lem:InvModMax} gives
  $f_{1}(x_{i})\leq C_{1}\lVert (x_{i})_{\mathcal{V}+}\rVert^{\gamma_{1}}$
  and $f_{2}(x_{i})\geq b-C_{2}\lVert (x_{i})_{\mathcal{V}-}\rVert^{\gamma_{2}}$,
  so that
  \[
    (f_{2}(x_{i})-f_{1}(x_{i}))^{2}\geq(b-C_{1}\lVert (x_{i})_{\mathcal{V}+}\rVert^{\gamma_{1}}-C_{2}\lVert (x_{i})_{\mathcal{V}-}\rVert^{\gamma_{2}})_{+}^{2}
  \]
  for each $i$. Hence, the objective in (\ref{eq:InvMod}) is bounded from
  below by the square of the right-hand side of \eqref{eq:InvModFormula}. Note
  that $f_{1}^{*}$ and $f_{2}^{*}$ are designed exactly to achieve this lower
  bound, since $f_{1}^{*}(0)=0$, $f_{2}^{*}(0)=b$, and
  \[
    f_{2}^{*}(x)-f_{1}^{*}(x)=(b-C_{1}\lVert (x)_{\mathcal{V}+}\rVert^{\gamma_{1}}-C_{2}\lVert (x)_{\mathcal{V}-}\rVert^{\gamma_{2}})_{+}.
  \]

  It remains to check whether
  $f_{1}^{*}\in\Lambda_{+,\mathcal{V}}(\gamma_{1},C_{1})$ and
  $f_{2}^{*}\in\Lambda_{+,\mathcal{V}}(\gamma_{2},C_{2})$.  The former follows
  from Lemma~\ref{lem:h_plus}. For the latter case, note that
  $f_{1}^{\ast}\in\Lambda_{+,\mathcal{V}}(\gamma_{1},C_{1})\subseteq\Lambda_{+,\mathcal{V}}(\gamma_{2},C_{2})$,
  and that
  $b-C_{2}\lVert (\cdot)_{\mathcal{V}-}\rVert^{\gamma_{2}}\in\Lambda_{+,\mathcal{V}}(\gamma_{2},C_{2})$
  by Lemma~\ref{lem:h_plus}. Now, by Lemma~\ref{lem:holder_max_holder}, we have
  $f_{2}^{*}\in\Lambda_{+,\mathcal{V}}(\gamma_{2},C_{2}).$
\end{proof}
Since the lower bound in the proof holds term by term and is attained, every
solution $(f_{1},f_{2})$ of \eqref{eq:InvMod} satisfies
$f_{2}(x_{i})-f_{1}(x_{i})=(b-C_{1}\lVert (x_{i})_{\mathcal{V}+}\rVert^{\gamma_{1}}-C_{2}\lVert (x_{i})_{\mathcal{V}-}\rVert^{\gamma_{2}})_{+}$
for all $i$. However, the solution functions themselves are not necessarily
unique.
For $\delta>0$, the equality also applies to any solution of the ordered modulus
problem because such a solution solves \eqref{eq:InvMod} with $b$ equal to the
ordered modulus.

The following corollary states an analogous result regarding the inverse modulus
$\omega^{-1}(b,\Lambda_{+,\mathcal{V}}(\gamma_{2},C_{2}),\Lambda_{+,\mathcal{V}}(\gamma_{1},C_{1}))$,
in which the order of the two parameter spaces is reversed. The proof is
analogous.
\begin{corollary}
  \label{cor:InvModRes(21)} Suppose the conditions of Proposition~\ref{prop:InvModRes} hold. For $b\geq0$, define
  \begin{align*}
    g_{1}^{*}(x) &= b-C_{1}\lVert (x)_{\mathcal{V}-}\rVert^{\gamma_{1}}\\
    g_{2}^{*}(x) &= \min \left\{ b-C_{1}\lVert (x)_{\mathcal{V}-}\rVert^{\gamma_{1}},\ C_{2}\lVert (x)_{\mathcal{V}+}\rVert^{\gamma_{2}}\right\} .
  \end{align*}
  Then, $(g_{2}^{*},g_{1}^{*})\in\Lambda_{+,\mathcal{V}}(\gamma_{2},C_{2})\times\Lambda_{+,\mathcal{V}}(\gamma_{1},C_{1})$
  solves the inverse modulus problem corresponding to
  $\omega^{-1}(b,\Lambda_{+,\mathcal{V}}(\gamma_{2},C_{2}),\Lambda_{+,\mathcal{V}}(\gamma_{1},C_{1}))$,
  and
  \begin{align*}
    &\omega^{-1}(b,\Lambda_{+,\mathcal{V}}(\gamma_{2},C_{2}),\Lambda_{+,\mathcal{V}}(\gamma_{1},C_{1})) \\
     &= \Big( \sum\nolimits_{i=1}^{n}\left(\left(b-C_{2}\left\lVert
               (x_{i})_{\mathcal{V}+}\right\rVert
           ^{\gamma_{2}}-C_{1}\left\lVert
               (x_{i})_{\mathcal{V}-}\right\rVert
           ^{\gamma_{1}}\right)/\sigma(x_{i})\right)_{+}^{2} \Big)^{1/2}.
  \end{align*}
\end{corollary}

Section~\ref{sec:Asymptotic-rate-results} gives the rate for general norms and
cones of monotone directions.

\subsection{Minimax Rate of Adaptation}
\label{sec:Asymptotic-rate-results}

Using this solution to the inverse modulus, we derive the rate of convergence of
the between class modulus, which characterizes how fast the worst-case
expected length of the adaptive CIs can go to 0 as $n\rightarrow\infty$. We
derive the rates under the assumption that the sequence of design points
$\{x_{i}\}_{i=1}^{\infty}$ is a realization of a sequence of independent and
identically distributed random vectors $\{ X_{i} \}_{i=1}^{\infty}$ drawn from a
distribution that satisfies some mild regularity conditions. This gives an
intuitive restriction on the design points, and also shows that the result
applies under random design points as well.\footnote{Consider the model
  $y_{i}=f(X_{i})+u_{i}$, for $i=1,\dots,n$, with the
  $X_{i}\overset{\mathrm{i.i.d.}}{\sim}p_{X}$ and
  $u_{i}|X_{i}\sim N(0,\sigma^{2}(X_{i}))$.  Then, conditional on
  $\{ X_{i} \}_{i=1}^{n}=\{x_{i}\}_{i=1}^{n}$, this model is equivalent to our
  model.}

We state the result for an arbitrary norm $\lVert\cdot\rVert$ and a
closed convex cone $K\subseteq\mathbb{R}^{k}$ of directions in which the
regression function is nondecreasing. Let $\Lambda_{K}(\gamma,C)$ be the set
of $f\in\Lambda(\gamma,C)$ such that $f(x+h)\geq f(x)$ whenever $h\in K$ and
$x,x+h\in\mathcal{D}$ \citep[see][]{borwein2004cone}, so that
$\Lambda_{K}(\gamma,C)=\Lambda_{+,\mathcal{V}}(\gamma,C)$ for
$K=K_{\mathcal{V}}:=\{h:h_{v}\geq0\ \forall v\in\mathcal{V},\ h_{v}=0\ \forall v\notin\mathcal{V}\}$,
and write $d(x,A)=\inf_{a\in A}\lVert x-a\rVert$. In place of
\eqref{eq:domain}, suppose that, for every $x\in\mathcal{D}$, there exist
$h,h'\in K$ such that
\begin{equation}
  x+h,\,x-h'\in\mathcal{D},\qquad\lVert x+h\rVert=d(x,-K)\quad\text{and}\quad\lVert x-h'\rVert=d(x,K),\label{eq:cone-domain}
\end{equation}
as when $\mathcal{D}$ is a ball with respect to $\lVert\cdot\rVert$ centered
at $0$. Then, if
$\Lambda_{K}(\gamma_{1},C_{1})\subseteq\Lambda_{K}(\gamma_{2},C_{2})$, the
inverse moduli are those of Proposition~\ref{prop:InvModRes} and
Corollary~\ref{cor:InvModRes(21)}, with $\lVert(x_{i})_{\mathcal{V}+}\rVert$
and $\lVert(x_{i})_{\mathcal{V}-}\rVert$ replaced by $d(x_{i},-K)$ and
$d(x_{i},K)$ (Proposition~S7.2). Let $S_{K}$ be the
linear span of $K$, $m$ its dimension and $S_{K}^{\perp}$ its Euclidean
orthogonal complement, and let
$r_{K}(\gamma_{1},\gamma_{2})=(2+m/\gamma_{1}+(k-m)/\gamma_{2})^{-1}$. The following theorem
characterizes the minimax rate of adaptation for a general norm and cone.
\begin{theorem}
  \label{thm:cone-rate}Suppose that $K\cap(-K)=\{0\}$ and that the domain
  condition \eqref{eq:cone-domain} holds, and let $0<\gamma_{2}\leq\gamma_{1}\leq1$ and $C_{1},C_{2}>0$
  be such that
  $\Lambda_{K}(\gamma_{1},C_{1})\subseteq\Lambda_{K}(\gamma_{2},C_{2})$. Let
  $\{ X_{i}\} _{i=1}^{\infty}$ be an i.i.d. sequence of random vectors with
  support $\mathcal{X}\subseteq\mathcal{D}$ such that
  $0\in\mathrm{Int}\,\mathcal{X}$. Suppose $X_{i}$ admits a probability
  density function $p_{X}(\cdot)$ that is continuous at $0$ with $p_{X}(0)>0,$
  and assume $\sigma(\cdot) = 1$. Then, for almost all realizations
  $\{ x_{i}\} _{i=1}^{\infty}$ of $\{ X_{i}\} _{i=1}^{\infty}$, we have, for
  all $\delta > 0$,
  \begin{align*}
    \lim_{n\to\infty}n^{r_{K}}\omega(\delta,\Lambda_{K}(\gamma_{1},C_{1}),\Lambda_{K}(\gamma_{2},C_{2}))
    & =\{\delta^{2}/(p_{X}(0)c_{K})\}^{r_{K}},\text{ and}\\
    \lim_{n\to\infty}n^{r_{K}}\omega(\delta,\Lambda_{K}(\gamma_{2},C_{2}),\Lambda_{K}(\gamma_{1},C_{1}))
    & =\{\delta^{2}/(p_{X}(0)c_{K})\}^{r_{K}},
  \end{align*}
  where $r_{K}=r_{K}(\gamma_{1},\gamma_{2})$,
  $c_{K}=\int_{K}\int_{S_{K}^{\perp}}(1-C_{1}d(u,-K)^{\gamma_{1}}-C_{2}d(v,S_{K})^{\gamma_{2}})_{+}^{2}\,dv\,du$
  if $\gamma_{1}>\gamma_{2}$, and
  $c_{K}=\int_{\mathbb{R}^{k}}(1-C_{1}d(x,-K)^{\gamma_{1}}-C_{2}d(x,K)^{\gamma_{1}})_{+}^{2}\,dx$
  if $\gamma_{1}=\gamma_{2}$, with the integrals taken with respect to the
  Lebesgue measures on $S_{K}$, $S_{K}^{\perp}$ and $\mathbb{R}^{k}$.
\end{theorem}
The rate
depends on the monotonicity restriction only through $m$, and not on the
norm, as follows from the equivalence of norms. The condition
$K\cap(-K)=\{0\}$ excludes cones that contain a line, along which every
function in the parameter spaces is constant. For $K=K_{\mathcal{V}}$,
$m=k_{+}$, and we write $r(\gamma_{1}, \gamma_{2}) =
({2+k_{+}/\gamma_{1}+(k-k_{+})/\gamma_{2}})^{-1}$. Under
Assumption~\ref{assu:norm}, $d(x,-K_{\mathcal{V}})=\lVert(x)_{\mathcal{V}+}\rVert$
and $d(x,K_{\mathcal{V}})=\lVert(x)_{\mathcal{V}-}\rVert$. This is the only
role of the assumption in Section~\ref{sec:solut-modul-probl} and in the
following theorem.
The following theorem, the case $K=K_{\mathcal{V}}$ of
Theorem~\ref{thm:cone-rate}, characterizes the minimax rate of adaptation for
the parameter spaces of Section~\ref{sec:Setup}.

\begin{theorem}\label{thm:rate_result}
  Suppose Assumption~\ref{assu:norm} and \eqref{eq:domain} hold, and let
  $0<\gamma_{2}\leq\gamma_{1}\leq1$ and $C_{1},C_{2}>0$ be such that
  $\Lambda_{+,\mathcal{V}}(\gamma_{1},C_{1})\subseteq\Lambda_{+,\mathcal{V}}(\gamma_{2},C_{2})$.
  Let $\{ X_{i}\} _{i=1}^{\infty}$ be an i.i.d. sequence of random
  vectors with support $\mathcal{X}\subseteq\mathcal{D}$ such that
  $0\in\mathrm{Int}\,\mathcal{X}$.  Suppose $X_{i}$ admits a probability
  density function $p_{X}(\cdot)$ that is continuous at $0$ with $p_{X}(0)>0,$
  and assume $\sigma(\cdot) = 1$.  Then, for almost all realizations
  $\{ x_{i}\} _{i=1}^{\infty}$ of
  $\{ X_{i}\} _{i=1}^{\infty}$, we have, for all $\delta > 0$,
  \begin{align*}
    \lim_{n\to\infty}n^{r(\gamma_{1},\gamma_{2})}\omega(\delta,\Lambda_{+,\mathcal{V}}(\gamma_{1},C_{1}),\Lambda_{+,\mathcal{V}}(\gamma_{2},C_{2}))
    & =(\delta^{2}/c_{1,2}^{\ast})^{r(\gamma_{1}, \gamma_{2})},\text{ and}\\
    \lim_{n\to\infty}n^{r(\gamma_{1},\gamma_{2})}\omega(\delta,\Lambda_{+,\mathcal{V}}(\gamma_{2},C_{2}),\Lambda_{+,\mathcal{V}}(\gamma_{1},C_{1})) & =(\delta^{2}/c_{2,1}^{\ast})^{r(\gamma_{1}, \gamma_{2})},
  \end{align*}
  where $c_{1,2}^{\ast}$ and $c_{2,1}^{\ast}$ are positive and finite constants
  that depend only on $p_{X}(0)$, $\lVert\cdot\rVert$, $\mathcal{V}$, and
  $(\gamma_{1},C_{1},\gamma_{2},C_{2})$, and that are given in the proof.
\end{theorem}

\begin{remark}
  The result immediately implies the rate of convergence for the between class
  modulus
  \begin{align*}
    & \lim_{n\to\infty}n^{r(\gamma_{1},\gamma_{2})}\omega_{+}(\delta,\Lambda_{+,\mathcal{V}}(\gamma_{1},C_{1}),\Lambda_{+,\mathcal{V}}(\gamma_{2},C_{2}))\\
    &= \max\{ \delta^{2}/c_{1,2}^{\ast},\delta^{2}/c_{2,1}^{\ast}\} ^{r(\gamma_{1},\gamma_{2})}.
  \end{align*}
  In fact, the proof shows that $c_{1,2}^{\ast}=c_{2,1}^{\ast}$, which follows
  from a change of variables $x\mapsto-x$, and both equal $p_{X}(0)c_{K}$ with
  $K=K_{\mathcal{V}}$ in Theorem~\ref{thm:cone-rate}.
\end{remark}
Theorem~\ref{thm:rate_result} shows how the monotonicity restriction plays a
role in determining the minimax rates of adaptation to H\"older exponents
under the multivariate nonparametric regression setting. Applied with
$(\gamma_{1},C_{1})=(\gamma_{2},C_{2})$, it gives the minimax rate
$n^{-1/(2+k/\gamma)}$ over $\Lambda_{+,\mathcal{V}}(\gamma,C)$, regardless of
$\mathcal{V}$. When $k_{+}=k$, the minimax rate of adaptation is
$n^{-1/(2+k/\gamma_{1})}$, the minimax rate over the smaller parameter space.
This shows that strong adaptation is possible, just like in the univariate case. When
$k_{+}=0$, the rate becomes $n^{-1/(2+k/\gamma_{2})}$, the minimax rate over
the larger parameter space, consistent with the previous findings that there
is no scope for adaptation for general H\"older classes without any shape
constraint. When $0<k_{+}<k$, Theorem~\ref{thm:rate_result} gives an
intermediate rate between the two extremes if $\gamma_{1}>\gamma_{2}$.
\citet[Section~2.3, Example~2]{cai2004adaptation} give a univariate example of
such an intermediate case, in which the smoothness of the regression function
differs on the two sides of the point of interest. In the multivariate
setting, intermediate cases arise naturally when the regression function is
monotone in only some of the variables. For example, when $k=2$, $\gamma_{1}=1$ and $\gamma_{2}=1/2$, the minimax rate
of adaptation is $n^{-1/6}$, $n^{-1/5}$ and $n^{-1/4}$ when $k_{+}=0$, $1$ and
$2$, respectively.

The constant $c_{1,2}^{\ast}$ in Theorem~\ref{thm:rate_result} is $p_{X}(0)$ times the integral of a simple
function. For example, $c_{1,2}^{\ast}=p_{X}(0)/15$ if $k=2$, $\mathcal{V}=\{1\}$,
$\lVert\cdot\rVert$ is the Euclidean norm, $(\gamma_{1},\gamma_{2})=(1,1/2)$
and $C_{1}=C_{2}=1$.
The design enters only through $p_{X}(0)$ because the moduli depend on the design
points only through those near $0$.

The proof of the theorem gives a geometric picture.
For $\omega(\delta,\Lambda_{+,\mathcal{V}}(\gamma_{1},C_{1}),\Lambda_{+,\mathcal{V}}(\gamma_{2},C_{2}))$,
the two functions that solve the ordered modulus problem differ on a set that
extends to order $b^{1/\gamma_{1}}$ in the positive direction of each monotone
coordinate and to order $b^{1/\gamma_{2}}$ in all other directions. Here, $b$ is
of order $n^{-r}$, and counting the design points in this set gives the rate.
Hence, when $\gamma_{1}>\gamma_{2}$, only the orthants on which all the
monotone coordinates are positive contribute to $c_{1,2}^{\ast}$, and only
those on which all are negative contribute to $c_{2,1}^{\ast}$. Moreover,
$n^{-r(\gamma_{1},\gamma_{2})}$ is the minimax rate, with or without
monotonicity, over the H\"older class with exponent $\gamma_{1}$ in the
coordinates in $\mathcal{V}$ and $\gamma_{2}$ in the others
(Theorem~S7.1 with $\mathcal{F}_{1}=\mathcal{F}_{2}$).
Combining Theorem~\ref{thm:rate_result} with Lemma~\ref{lem:CL-bounds}, we
obtain the minimax rate of the expected length of CIs.
\begin{corollary}
  \label{cor:minimax-all-CI}Suppose the conditions of Theorem~\ref{thm:rate_result} hold, and let $L_{n}^{\ast}$ denote $L_{1,2}^{\ast}$
  defined in Section~\ref{sec:two-sided-adaptive} with
  $\mathcal{F}_{1}=\Lambda_{+,\mathcal{V}}(\gamma_{1},C_{1})$ and
  $\mathcal{F}_{2}=\Lambda_{+,\mathcal{V}}(\gamma_{2},C_{2})$. Write
  $r=r(\gamma_{1},\gamma_{2})$ and $c^{\ast}=c_{1,2}^{\ast}=c_{2,1}^{\ast}$.
  Then, for almost all realizations $\{ x_{i}\} _{i=1}^{\infty}$
  of $\{ X_{i}\} _{i=1}^{\infty}$,
  \begin{equation*}
    \left(\frac{1}{2}-\alpha\right)\left(\frac{z_{1-\alpha}^{2}}{c^{\ast}}\right)^{r}\leq\liminf_{n\to\infty}n^{r}L_{n}^{\ast}\leq\limsup_{n\to\infty}n^{r}L_{n}^{\ast}\leq\left(4+\frac{9}{z_{1-\alpha/2}}\right)\left(\frac{z_{1-\alpha/2}^{2}}{c^{\ast}}\right)^{r}.
  \end{equation*}
\end{corollary}
Hence, if a CI maintains coverage over
$\Lambda_{+,\mathcal{V}}(\gamma_{2},C_{2})$, the best possible
worst-case expected length of this CI over
$\Lambda_{+,\mathcal{V}}(\gamma_{1},C_{1})$ goes to $0$ at the
rate $n^{-r(\gamma_{1},\gamma_{2})}$.

For simplicity, the theorem imposes a homo\-skedas\-ticity condition (i.e.,
$\sigma(\cdot) = 1$). For a general known $\sigma(\cdot)$ that is continuous
at $0$ with $\sigma(0)>0$, it holds with $p_{X}(0)$ replaced by
$p_{X}(0)/\sigma^{2}(0)$ in the constants. In the proof, we use the randomness
of the design points only through local averages near $0$, so that the
theorem also holds for deterministic designs that satisfy a corresponding condition.
See Section~S5 and
Remark~S8.4 for details.

\begin{remark}[Adaptation under only monotonicity]\label{rem:onlymon}
  As $\gamma_{2}\to0$, $r(\gamma_{1},\gamma_{2})$ converges to
  $(2+k/\gamma_{1})^{-1}$ if $k_{+}=k$, and to $0$ if $k_{+}<k$. In
  Section~S6, we show that this is also what
  happens when the larger parameter space imposes only monotonicity. If
  $k_{+}=k$, strong adaptation to $\Lambda_{+,\mathcal{V}}(\gamma_{1},C_{1})$
  is possible even if we require correct coverage over all functions that are
  monotone in every coordinate. If $k_{+}<k$, every CI that has correct coverage over the functions that are
  monotone in the coordinates in $\mathcal{V}$ has infinite worst-case expected
  length over $\Lambda_{+,\mathcal{V}}(\gamma_{1},C_{1})$, for almost all
  realizations of the design points. This is because the set
  $A_{\mathcal{V}}=\{x\in\mathcal{D}:x_{v}\geq0\ \forall v\in\mathcal{V}\text{ and }x_{v}=0\ \forall v\notin\mathcal{V}\}$
  has Lebesgue measure zero. Hence, almost surely, $f$ and
  $f+b\mathds{1}\{x\in A_{\mathcal{V}}\}$ coincide at every design point,
  while both are monotone in the coordinates in $\mathcal{V}$ and their
  values at $0$ differ by $b$. For the same reason, almost surely no two
  design points are comparable in the partial order that defines
  monotonicity, so that the isotonic least squares estimator for this order
  \citep[see][]{deng2020isotonic} interpolates the data. Hence, some a priori
  bound on the smoothness of the largest parameter space is necessary.
\end{remark}

\subsection{Smoothness Only Along the Nonmonotone Coordinates}
\label{sec:weaker}

Remark~\ref{rem:onlymon} shows that, when $k_{+}<k$, the largest parameter
space must restrict the smoothness of the regression function. Here, we show
that it suffices to restrict the smoothness at $0$ along the coordinates that
are not in $\mathcal{V}$. Note that the parameter spaces enter the solution to
the modulus problem only through the bounds in Lemma~\ref{lem:InvModMax},
which restrict how far $f(x)$ can be from $f(0)$, and through the functions
that attain them. Let
$\mathcal{M}_{\mathcal{V}}$ denote the set of functions from $\mathcal{D}$ to
$\mathbb{R}$ that satisfy the monotonicity condition in the definition of
$\Lambda_{+,\mathcal{V}}(\gamma,C)$, and, for $\gamma\in(0,1]$ and $C>0$,
define
\begin{align*}
  \mathcal{P}_{\mathcal{V}}(\gamma,C) & :=\left\{ f\in\mathcal{M}_{\mathcal{V}}:\lvert f(x)-f(0)\rvert\leq C\lVert x\rVert^{\gamma}\text{ for all }x\in\mathcal{D}\right\} \text{ and}\\
  \mathcal{G}_{\mathcal{V}}(\gamma,C) & :=\left\{ f\in\mathcal{M}_{\mathcal{V}}:\lvert f(x)-f(0)\rvert\leq C\lVert x\rVert^{\gamma}\text{ for all }x\in\mathcal{D}\text{ with }x_{v}=0\ \forall v\in\mathcal{V}\right\} .
\end{align*}

Both parameter spaces are convex, and
$\Lambda_{+,\mathcal{V}}(\gamma,C)\subseteq\mathcal{P}_{\mathcal{V}}(\gamma,C)\subseteq\mathcal{G}_{\mathcal{V}}(\gamma,C)$.
The space $\mathcal{P}_{\mathcal{V}}(\gamma,C)$ restricts the smoothness of
$f$ only at $0$, similarly to the local conditions used by
\citet{han2020limit} and \citet{deng2021confidence}, which also fix the
exact order of the local behavior of $f$ at the point of interest. The space
$\mathcal{G}_{\mathcal{V}}(\gamma,C)$ restricts it only along the slice
$\{x\in\mathcal{D}:x_{v}=0\ \forall v\in\mathcal{V}\}$, which passes through
$0$, and only relative to $f(0)$. For example, if $1\in\mathcal{V}$ and
$0\in\mathrm{Int}\,\mathcal{D}$, the function $\mathds{1}\{x_{1}>0\}$ is in $\mathcal{G}_{\mathcal{V}}(\gamma,C)$
but not in $\mathcal{P}_{\mathcal{V}}(\gamma,C)$. If
$\mathcal{V}=\{1,\dots,k\}$, the slice is $\{0\}$, so that
$\mathcal{G}_{\mathcal{V}}(\gamma,C)=\mathcal{M}_{\mathcal{V}}$. Let
$\mathcal{D}_{\mathcal{V}}^{+}:=\{x\in\mathcal{D}:x_{v}\geq0\ \forall v\in\mathcal{V}\}$
and
$\mathcal{D}_{\mathcal{V}}^{-}:=\{x\in\mathcal{D}:x_{v}\leq0\ \forall v\in\mathcal{V}\}$.

Over $\mathcal{P}_{\mathcal{V}}(\gamma,C)$, the maximum and the minimum of
$f(x)$ given $f(0)$ are those in Lemma~\ref{lem:InvModMax}, and, by
Lemma~S2.2, both are attained by functions
in $\Lambda_{+,\mathcal{V}}(\gamma,C)$. Hence, the ordered moduli
between $\Lambda_{+,\mathcal{V}}(\gamma_{1},C_{1})$ and
$\mathcal{P}_{\mathcal{V}}(\gamma_{2},C_{2})$ equal those between
$\Lambda_{+,\mathcal{V}}(\gamma_{1},C_{1})$ and
$\Lambda_{+,\mathcal{V}}(\gamma_{2},C_{2})$, so that
Theorem~\ref{thm:rate_result} and Corollary~\ref{cor:minimax-all-CI} continue
to hold if we replace $\Lambda_{+,\mathcal{V}}(\gamma_{2},C_{2})$ by
$\mathcal{P}_{\mathcal{V}}(\gamma_{2},C_{2})$. The same is true for the CIs we
construct in Section~\ref{sec:constr-adapt-ci}. Their coverage depends on the
largest parameter space only through the maximum biases of their ends and
$K_{\tau}$ in \eqref{eq:slack}, which are determined by the same bounds. Hence,
with Gaussian errors and a known variance function, the CIs maintain coverage
over $\mathcal{P}_{\mathcal{V}}(\gamma_{J},C_{J})$ without change.

A function in $\mathcal{G}_{\mathcal{V}}(\gamma,C)$, on the other hand, is not
restricted from above outside $\mathcal{D}_{\mathcal{V}}^{-}$, or from below
outside $\mathcal{D}_{\mathcal{V}}^{+}$. Hence, by Proposition~S2.1, the inverse
moduli between $\Lambda_{+,\mathcal{V}}(\gamma_{1},C_{1})$ and
$\mathcal{G}_{\mathcal{V}}(\gamma_{2},C_{2})$ are those in
Proposition~\ref{prop:InvModRes} and Corollary~\ref{cor:InvModRes(21)}, with
the sums restricted to the design points in $\mathcal{D}_{\mathcal{V}}^{+}$
and $\mathcal{D}_{\mathcal{V}}^{-}$, respectively. Since the sums are
restricted, the ordered moduli are at least as large as those with
$\Lambda_{+,\mathcal{V}}(\gamma_{2},C_{2})$. However, when $\gamma_{1}>\gamma_{2}$, the design points that determine the
rate of
$\omega(\delta,\Lambda_{+,\mathcal{V}}(\gamma_{1},C_{1}),\Lambda_{+,\mathcal{V}}(\gamma_{2},C_{2}))$
in Theorem~\ref{thm:rate_result} are those in the orthants on which all the
monotone coordinates are positive. Likewise, those for the reverse order are in the
orthants on which all are negative, as discussed after the theorem. Such
orthants lie in $\mathcal{D}_{\mathcal{V}}^{+}$ and
$\mathcal{D}_{\mathcal{V}}^{-}$, respectively. This gives the following theorem.
\begin{theorem}
  \label{thm:rate-G}Suppose the conditions of Theorem~\ref{thm:rate_result}
  hold, and write $r=r(\gamma_{1},\gamma_{2})$,
  $\mathcal{F}_{1}=\Lambda_{+,\mathcal{V}}(\gamma_{1},C_{1})$ and
  $\mathcal{G}=\mathcal{G}_{\mathcal{V}}(\gamma_{2},C_{2})$. Then, for almost
  all realizations $\{x_{i}\}_{i=1}^{\infty}$ of $\{X_{i}\}_{i=1}^{\infty}$,
  we have, for all $\delta>0$,
  \[
    \lim_{n\to\infty}n^{r}\omega(\delta,\mathcal{F}_{1},\mathcal{G})=\lim_{n\to\infty}n^{r}\omega(\delta,\mathcal{G},\mathcal{F}_{1})=(\delta^{2}/c_{\mathcal{G}}^{\ast})^{r},
  \]
  where $c_{\mathcal{G}}^{\ast}$ is a positive and finite constant that
  depends only on $p_{X}(0)$, $\lVert\cdot\rVert$, $\mathcal{V}$, and
  $(\gamma_{1},C_{1},\gamma_{2},C_{2})$, and $c_{\mathcal{G}}^{\ast}=c^{\ast}$,
  the constant in Corollary~\ref{cor:minimax-all-CI}, if
  $\gamma_{1}>\gamma_{2}$. Moreover, the conclusion of
  Corollary~\ref{cor:minimax-all-CI} holds with
  $\mathcal{F}_{2}=\mathcal{G}$ and $c^{\ast}=c_{\mathcal{G}}^{\ast}$.
\end{theorem}
Hence, requiring coverage over $\mathcal{G}_{\mathcal{V}}(\gamma_{2},C_{2})$,
which restricts the smoothness of the regression function only at $0$ along
the coordinates that are not in $\mathcal{V}$, gives the same minimax rate of
adaptation as requiring coverage over
$\Lambda_{+,\mathcal{V}}(\gamma_{2},C_{2})$. When $\gamma_{1}>\gamma_{2}$, it
even gives the same constant $c^{\ast}$ in Theorem~\ref{thm:rate_result} and
Corollary~\ref{cor:minimax-all-CI}. If $\mathcal{V}=\{1,\dots,k\}$, no
smoothness bound is needed, and the theorem gives the rate
$n^{-1/(2+k/\gamma_{1})}$ under monotonicity alone, as noted in
Remark~\ref{rem:onlymon}. However, if $\gamma_{1}=\gamma_{2}$,
only the orthants on which all the monotone coordinates are positive
contribute to $c_{\mathcal{G}}^{\ast}$. For the reverse ordering, only those on
which all are negative contribute. Hence, in this case, $c_{\mathcal{G}}^{\ast}<c^{\ast}$
if $k_{+}\geq1$. For example, $c_{\mathcal{G}}^{\ast}=c^{\ast}/2$ if $k=2$,
$\mathcal{V}=\{1\}$, $\gamma_{1}=\gamma_{2}=1$ and $C_{1}=C_{2}$, so that the
limits are larger by the factor $2^{1/4}\approx1.19$. The following corollary
shows that the rate depends on the parameter space over which coverage is
required only through its smoothness along the slice, and that the pointwise
parameter spaces can serve as the smaller ones.
\begin{corollary}
  \label{cor:sandwich}Suppose the conditions of
  Theorem~\ref{thm:rate_result} hold, and write $r=r(\gamma_{1},\gamma_{2})$
  and $\mathcal{F}_{1}=\Lambda_{+,\mathcal{V}}(\gamma_{1},C_{1})$. For a
  parameter space $\mathcal{C}$ that contains $\mathcal{F}_{1}$, let
  $L_{n}^{\ast}(\mathcal{C})$ denote $L_{1,2}^{\ast}$ defined in
  Section~\ref{sec:two-sided-adaptive} with $\mathcal{F}_{2}=\mathcal{C}$.
  Then, for almost all realizations $\{x_{i}\}_{i=1}^{\infty}$ of
  $\{X_{i}\}_{i=1}^{\infty}$, the following hold. (i) If $\mathcal{C}$ is
  convex and
  $\Lambda_{+,\mathcal{V}}(\gamma_{2},C_{2})\subseteq\mathcal{C}\subseteq\mathcal{G}_{\mathcal{V}}(\gamma_{2},C_{2}')$
  for some $C_{2}'\geq C_{2}$, then $n^{r}L_{n}^{\ast}(\mathcal{C})$ is
  bounded away from zero and infinity. (ii) If $k_{+}<k$,
  $0<\gamma_{2}'<\gamma_{2}$, $C'>0$,
  $\mathcal{F}_{1}\subseteq\Lambda_{+,\mathcal{V}}(\gamma_{2}',C')$, and
  $\mathcal{C}\supseteq\mathcal{G}_{\mathcal{V}}(\gamma_{2}',C')$ is convex,
  then $\liminf_{n\to\infty}n^{r'}L_{n}^{\ast}(\mathcal{C})>0$, where
  $r'=r(\gamma_{1},\gamma_{2}')<r$. (iii)
  $\mathcal{P}_{\mathcal{V}}(\gamma_{1},C_{1})\subseteq\mathcal{P}_{\mathcal{V}}(\gamma_{2},C_{2})$,
  and the conclusions of Theorem~\ref{thm:rate_result} and
  Corollary~\ref{cor:minimax-all-CI} hold with $\mathcal{F}_{1}$ and
  $\Lambda_{+,\mathcal{V}}(\gamma_{2},C_{2})$ replaced by
  $\mathcal{P}_{\mathcal{V}}(\gamma_{1},C_{1})$ and
  $\mathcal{P}_{\mathcal{V}}(\gamma_{2},C_{2})$, and those of
  Theorem~\ref{thm:rate-G} hold with $\mathcal{F}_{1}$ replaced by
  $\mathcal{P}_{\mathcal{V}}(\gamma_{1},C_{1})$.
\end{corollary}
By part (iii), the CIs of Section~\ref{sec:constr-adapt-ci} adapt to the
pointwise parameter spaces at no cost. Their ends and the bounds
\eqref{eq:BonBound} and \eqref{eq:calibrated-length} are unchanged if we take
$\mathcal{F}_{j}=\mathcal{P}_{\mathcal{V}}(\gamma_{j},C_{j})$ for all $j$,
which are nested by the argument in the proof. The
parameter space $\mathcal{G}_{\mathcal{V}}(\gamma,C)$ is invariant to
strictly increasing transformations of the coordinates in $\mathcal{V}$ that
fix $0$, and does not depend on the norm if $\mathcal{V}=\{1,\dots,k\}$.
Such a transformation, and, in that case, the choice of the norm, can thus
change the length of a CI with coverage over
$\mathcal{G}_{\mathcal{V}}(\gamma,C)$ but not its coverage.

The parameter space $\mathcal{G}=\mathcal{G}_{\mathcal{V}}(\gamma_{2},C_{2})$ of Theorem~\ref{thm:rate-G} is large.
Suppose $k_{+}\geq1$. Then, for almost all realizations of the design points,
no design point lies on the slice
$\mathcal{D}_{\mathcal{V}}^{+}\cap\mathcal{D}_{\mathcal{V}}^{-}$. Hence,
$\omega(\delta,\mathcal{G},\mathcal{G})=\infty$ for all $\delta\geq0$
(Section~S2.13). By
Lemma~\ref{lem:CL-bounds}, every CI that has coverage at least $1-\alpha$ over
$\mathcal{G}$ thus has infinite worst-case
expected length over this space. Adaptation is then necessary if we
require coverage over $\mathcal{G}$. In
Section~\ref{sec:constr-adapt-ci}, we
construct a CI that maintains coverage over
$\mathcal{G}_{\mathcal{V}}(\gamma_{J},C_{J})$ and attains the rate in
Theorem~\ref{thm:rate-G}.

\subsection{Extensions}
\label{sec:extensions}

We now extend Theorem~\ref{thm:rate_result} to points on the boundary of the
support, to smoothness that differs across coordinates, and to uncertainty
about the set $\mathcal{V}$. We also discuss our notion of adaptation.

We first consider the case where $0$ lies on the boundary of the support. This
case arises, for example, when $x_{0}$ corresponds to the zero dose of a
treatment, or in regression discontinuity designs, where the regression
function is evaluated at the boundary of the support on each side of the
cutoff. As noted in Section~\ref{sec:Setup}, the solution to the modulus
problem does not depend on whether $0$ is on the boundary. The CIs in
Section~\ref{sec:constr-adapt-ci} thus remain valid, but the rates change. Suppose
that, near $0$, the support consists of the points whose coordinates in
$\mathcal{B}_{+}$ are nonnegative and whose coordinates in $\mathcal{B}_{-}$
are nonpositive, where $\mathcal{B}_{+}$ and $\mathcal{B}_{-}$ are disjoint
subsets of $\{1,\dots,k\}$. For example, $\mathcal{B}_{+}=\{1\}$ and
$\mathcal{B}_{-}=\emptyset$ if the support is a rectangle and $0$ lies in the
relative interior of the face on which $x_{1}$ attains its minimum. Let
$\ell_{+}=\lvert\mathcal{V}\cap\mathcal{B}_{+}\rvert$ and
$\ell_{-}=\lvert\mathcal{V}\cap\mathcal{B}_{-}\rvert$ be the numbers of
monotone coordinates in which $0$ is at the lower and at the upper end of the
support, respectively, and, for $\ell=0,\dots,k_{+}$, define
\[
  r_{\ell}(\gamma_{1},\gamma_{2})=\left(2+(k_{+}-\ell)/\gamma_{1}+(k-k_{+}+\ell)/\gamma_{2}\right)^{-1},
\]
so that $r_{0}(\gamma_{1},\gamma_{2})=r(\gamma_{1},\gamma_{2})$.
\begin{theorem}
  \label{thm:boundary}Suppose the conditions of
  Theorem~\ref{thm:rate_result} hold, except that, instead of
  $0\in\mathrm{Int}\,\mathcal{X}$ and the continuity of $p_{X}$ at $0$, there
  exists $\rho>0$ such that
  $\mathcal{X}\cap B_{\rho}=\{x\in B_{\rho}:x_{v}\geq0\ \forall v\in\mathcal{B}_{+},\ x_{v}\leq0\ \forall v\in\mathcal{B}_{-}\}$,
  where $B_{\rho}=\{x\in\mathbb{R}^{k}:\lVert x\rVert\leq\rho\}$, and
  $p_{X}(x)\to p_{X}(0)>0$ as $x\to0$ in $\mathcal{X}$. Then, for almost all
  realizations $\{x_{i}\}_{i=1}^{\infty}$ of $\{X_{i}\}_{i=1}^{\infty}$, we
  have, for all $\delta>0$,
  \begin{align*}
    \lim_{n\to\infty}n^{r_{\ell_{-}}}\omega(\delta,\Lambda_{+,\mathcal{V}}(\gamma_{1},C_{1}),\Lambda_{+,\mathcal{V}}(\gamma_{2},C_{2})) & =(\delta^{2}/\bar{c}_{1,2})^{r_{\ell_{-}}},\text{ and}\\
    \lim_{n\to\infty}n^{r_{\ell_{+}}}\omega(\delta,\Lambda_{+,\mathcal{V}}(\gamma_{2},C_{2}),\Lambda_{+,\mathcal{V}}(\gamma_{1},C_{1})) & =(\delta^{2}/\bar{c}_{2,1})^{r_{\ell_{+}}},
  \end{align*}
  where $r_{\ell}=r_{\ell}(\gamma_{1},\gamma_{2})$, and $\bar{c}_{1,2}$ and
  $\bar{c}_{2,1}$ are positive and finite constants that depend only on
  $p_{X}(0)$, $\lVert\cdot\rVert$, $\mathcal{V}$, $\mathcal{B}_{+}$,
  $\mathcal{B}_{-}$, and $(\gamma_{1},C_{1},\gamma_{2},C_{2})$.
\end{theorem}
The ordered modulus $\omega(\delta,\mathcal{F}_{j},\mathcal{F}_{J})$ governs
the upper ends of the CIs, and $\omega(\delta,\mathcal{F}_{J},\mathcal{F}_{j})$
the lower ends (Section~\ref{sec:one-sided-cis}). Hence, when
$\gamma_{1}>\gamma_{2}$, each monotone coordinate in which $0$ is at the upper
end of the support slows down the upper ends. Likewise, each monotone
coordinate in which $0$ is at the lower end slows down the lower ends. In both
cases, this coordinate moves from the smoother exponent to the rougher one in
the rate. A two-sided CI adapts at the rate $n^{-r_{\max\{\ell_{+},\ell_{-}\}}}$.

For example, at a zero dose, with $\mathcal{B}_{+}=\{1\}$,
$\mathcal{B}_{-}=\emptyset$ and $1\in\mathcal{V}$, a one-sided upper CI adapts
at the rate $n^{-r(\gamma_{1},\gamma_{2})}$ because the values of the
regression function at positive doses bound its value at the zero dose from
above. A one-sided lower CI, and thus a two-sided CI, adapts only at the rate
$n^{-r_{1}(\gamma_{1},\gamma_{2})}$, which leaves no scope for adaptation if
$k_{+}=1$. This is the structure behind the boundary results of
\citet{armstrong2015AdaptiveTestingRegression} in the univariate case, where
the regression function is nonincreasing away from the boundary point and thus
the one-sided CIs that adapt are lower CIs. It is also the structure behind the
one-sided CIs of \citet{kwon2020InferenceRegressionDiscontinuity} in regression
discontinuity designs, whose direction depends on which side of the cutoff is
treated and on the direction of monotonicity. A boundary in a coordinate that is not in
$\mathcal{V}$, on the other hand, changes only the constants.

The proof of the theorem uses only that, near $0$, the support is a union of orthants. The theorem thus extends to
any such support, with $\ell_{-}$ and $\ell_{+}$ replaced by $k_{+}-m_{+}$ and
$k_{+}-m_{-}$, where $m_{+}$ and $m_{-}$ are the largest numbers of positive
and of negative monotone coordinates on an orthant contained in the support.
For example, consider a corner of a regression discontinuity design in which
the treatment requires $x_{1}\geq0$ and $x_{2}\geq0$, with
$\mathcal{V}=\{1,2\}$. Then, $(m_{+},m_{-})$ is $(2,0)$ on the treated side and
$(1,2)$ on the control side, which is the union of the other three quadrants.
Hence, an upper CI for the treatment effect, which combines the upper end on
the treated side with the lower end on the control side, adapts at the rate
$n^{-r(\gamma_{1},\gamma_{2})}$, while a lower CI does not adapt.

The last two extensions are given in Sections~S7.1 and
S7.2. When the smoothness differs across
coordinates, the minimax rate of adaptation uses the smoother exponent for
each monotone coordinate and the rougher exponent for the other
coordinates. Suppose the researcher specifies, before observing the data, a
finite family of candidate monotonicity sets that contains the intersection
$\mathcal{V}_{0}$ of all its members, and requires correct coverage under
every candidate set. Then, the minimax rate of adaptation is
$n^{-r_{\mathcal{V}_{0}}}$, where $r_{\mathcal{V}_{0}}$ is
$r(\gamma_{1},\gamma_{2})$ with $k_{+}$ replaced by
$\lvert\mathcal{V}_{0}\rvert$. Hence, only the coordinates that are monotone
under all candidate sets help.

Finally, the notion of adaptation in this paper is adaptation to parameter
spaces. A CI is adaptive if its worst-case expected length over each
$\mathcal{F}_{j}$ is of the order of the smallest possible one. Under
monotonicity in all the variables, a stronger notion is adaptation to each
function, in which the benchmark is the smallest expected length at the given
regression function among all CIs that maintain coverage over the whole
parameter space. \citet{cai2013AdaptiveConfidenceIntervals} consider this
notion in the univariate case and bound this benchmark from below by a local
modulus of continuity at the given function. In the multivariate
case, the CIs of \citet{deng2021confidence}, discussed in the introduction,
have length of the order of an oracle length at each function that satisfies
their local conditions, with coverage that holds asymptotically at each
function. The confidence bands of \citet{datta2026optimal}, in a white noise
model, adapt up to a logarithmic factor to the local smoothness and the
intrinsic dimension of each function while maintaining coverage over all
monotone functions.

For example, if
$k_{+}=k$, $0\in\mathrm{Int}\,\mathcal{X}$ and $f$ is constant, every
monotone function whose value at $0$ differs from that of $f$ by $b$ differs
from $f$ by at least $b$ on the nonnegative or on the nonpositive orthant.
Hence, the local modulus is of order $n^{-1/2}$, while our CIs shrink at the rate
$n^{-1/(2+k/\gamma_{1})}$ at this function. We leave open
whether adaptation to each function is possible in the
monotone coordinates when $0<k_{+}<k$.

\section{Construction of the Adaptive CI}
\label{sec:adaptive-ci}
In this section, we provide a CI that obtains the minimax rate of adaptation
derived in Section~\ref{sec:adapt-infer-f0}. We first construct it under
Gaussian errors with a known variance function. We then show that it remains
asymptotically valid when the errors are not Gaussian and the noise level is
estimated.

\subsection{Gaussian Errors with a Known Variance}
\label{sec:constr-adapt-ci}
Here, we give the explicit formula of the CIs for our parameter of interest,
$Lf=f(0)$, now that we have derived the form of the moduli of continuity and
the solutions to the modulus problems in Section~\ref{sec:solut-modul-probl}. Let
$\mathcal{F}_{j}=\Lambda_{+,\mathcal{V}}(\gamma_{j},C_{j})$ for
$j=1,\dots,J$, where the parameter spaces satisfy the conditions given at the
beginning of Section~\ref{sec:solut-modul-probl}, and suppose Assumption~\ref{assu:norm} and \eqref{eq:domain} hold. Before stating the result, it is
convenient to define the following functions
\begin{align*}
  D_{Jj,\delta}(x_{i}) &:=
                            (\omega(\delta,\mathcal{F}_{J},\mathcal{F}_{j})-C_{j}\lVert
                            (x_{i})_{\mathcal{V}-}\rVert
                            ^{\gamma_{j}}-C_{J}\lVert
                            (x_{i})_{\mathcal{V}+}\rVert^{\gamma_{J}})_{+},
                            \text{ and}\\
  D_{jJ, \delta}(x_{i}) &:=( \omega(\delta,\mathcal{F}_{j},\mathcal{F}_{J})-C_{J}\lVert (x_{i})_{\mathcal{V}-}\rVert^{\gamma_{J}}-C_{j}\lVert (x_{i})_{\mathcal{V}+}\rVert^{\gamma_{j}})_{+}.
\end{align*}
By Corollary~\ref{cor:InvModRes(21)} and Proposition~\ref{prop:InvModRes},
both are exactly the differences between the two functions that solve the
ordered moduli $\omega(\delta,\mathcal{F}_{J},\mathcal{F}_{j})$ and
$\omega(\delta,\mathcal{F}_{j},\mathcal{F}_{J})$, respectively, evaluated at
$x_{i}$. Moreover, the explicit form of the inverse moduli shows that the
ordered moduli are differentiable for $\delta>0$, with
\[
  \omega'(\delta,\mathcal{F}_{J},\mathcal{F}_{j})=\frac{\delta}{\sum_{i=1}^{n}D_{Jj,\delta}(x_{i})/\sigma^{2}(x_{i})}>0,
\]
and likewise for $\omega'(\delta,\mathcal{F}_{j},\mathcal{F}_{J})$. Hence, the
conditions of Lemma~\ref{thm:ArmKol}, Theorems~\ref{thm:BonfCI} and
\ref{thm:calibrated-length}, and Lemma~\ref{lem:Bonf-improv} are satisfied. In
the following corollary, we write
$\hat{c}_{[\delta]}^{\ell,j}:=\hat{c}_{1-\Phi(\delta),\delta}^{\ell,j}$ for
the lower end of the one-sided CI with level $\Phi(\delta)$ and $\beta=1/2$,
so that $\hat{c}_{\alpha}^{\ell,j}=\hat{c}_{[z_{1-\alpha}]}^{\ell,j}$, and
similarly for the upper end. We also write
\[
  w_{j,i}^{-}=\frac{D_{Jj,\delta}(x_{i})}{\sum_{i'=1}^{n}D_{Jj,\delta}(x_{i'})}\quad\text{and}\quad w_{j,i}^{+}=\frac{D_{jJ,\delta}(x_{i})}{\sum_{i'=1}^{n}D_{jJ,\delta}(x_{i'})}
\]
for the weights, which are nonnegative and sum to one, and
$s_{j}^{\pm}=(\sum_{i=1}^{n}(w_{j,i}^{\pm})^{2})^{1/2}$, which equal
$\omega'(\delta,\mathcal{F}_{J},\mathcal{F}_{j})$ and
$\omega'(\delta,\mathcal{F}_{j},\mathcal{F}_{J})$ when $\sigma(\cdot)=1$.

\begin{corollary}
  \label{cor:chat formula}Suppose $\sigma(\cdot)=1$ and $\beta=1/2$. Then, for
  any $\delta>0$, the ends of the one-sided CIs defined in
  Section~\ref{sec:constr-adapt-cis} are given by
  \begin{align*}
    \hat{c}_{[\delta]}^{\ell,j} & =\sum_{i=1}^{n}w_{j,i}^{-}\big(y_{i}-C_{J}\lVert (x_{i})_{\mathcal{V}+}\rVert^{\gamma_{J}}\big)-\delta s_{j}^{-},\quad\text{and}\\
    \hat{c}_{[\delta]}^{u,j} & =\sum_{i=1}^{n}w_{j,i}^{+}\big(y_{i}+C_{J}\lVert (x_{i})_{\mathcal{V}-}\rVert^{\gamma_{J}}\big)+\delta s_{j}^{+}.
  \end{align*}
\end{corollary}
That is, each end is a weighted average of the $y_{i}$, corrected by the
maximum bias of this weighted average over $\mathcal{F}_{J}$, and moved by
$\delta$ times its standard deviation. By Lemma~\ref{lem:InvModMax}, the bias
of $\sum_{i}w_{j,i}^{-}y_{i}$ at $f\in\mathcal{F}_{J}$ is
$\sum_{i}w_{j,i}^{-}\{f(x_{i})-f(0)\}\leq\sum_{i}w_{j,i}^{-}C_{J}\lVert (x_{i})_{\mathcal{V}+}\rVert^{\gamma_{J}}$,
with equality for $f=C_{J}\lVert (\cdot)_{\mathcal{V}+}\rVert^{\gamma_{J}}$.
The parameter space $\mathcal{F}_{j}$ to which the end directs power enters
only through the weights. For general $\sigma(\cdot)$, the same formulas hold
with $D_{Jj,\delta}(x_{i})$ and $D_{jJ,\delta}(x_{i})$ replaced by
$D_{Jj,\delta}(x_{i})/\sigma^{2}(x_{i})$ and $D_{jJ,\delta}(x_{i})/\sigma^{2}(x_{i})$
in the weights, and with $s_{j}^{\pm}$ replaced by the standard deviation
$(\sum_{i}(w_{j,i}^{\pm})^{2}\sigma^{2}(x_{i}))^{1/2}$. If
$\mathcal{V}=\{1,\dots,k\}$, the weights can be viewed as those of a
Nadaraya--Watson type estimator with a bandwidth that depends on the design
points and on the direction.

Corollary~\ref{cor:chat formula} also gives the bias slacks in
\eqref{eq:slack}. Write $h_{+}$ and $h_{-}$ for the functions in
Lemma~\ref{lem:h_plus} with $(\gamma,C)=(\gamma_{J},C_{J})$, which are the
envelopes of $\mathcal{F}_{J}$ in Lemma~\ref{lem:InvModMax} with $f_{0}=0$.
\begin{lemma}
  \label{lem:slack}Suppose the conditions of Corollary~\ref{cor:chat formula}
  hold, and let $\tau\in(0,1/2)$ and $\delta=\delta_{\tau}$. Then, for any
  $f\in\mathcal{F}_{J}$ and $j,\ell\leq J$,
  \begin{align*}
    a_{j}(f) & =\frac{1}{\delta}\sum_{i=1}^{n}D_{Jj,\delta}(x_{i})\{h_{+}(x_{i})-f(x_{i})+f(0)\},\quad\text{and}\\
    c_{\ell}(f) & =\frac{1}{\delta}\sum_{i=1}^{n}D_{\ell J,\delta}(x_{i})\{f(x_{i})-f(0)-h_{-}(x_{i})\},
  \end{align*}
  and \eqref{eq:slack} holds with
  \[
    K_{\tau}=\frac{1}{\delta}\min_{j,\ell\leq J}\sum_{i=1}^{n}\min\{D_{Jj,\delta}(x_{i}),D_{\ell J,\delta}(x_{i})\}\{h_{+}(x_{i})-h_{-}(x_{i})\}.
  \]
\end{lemma}
The lemma uses only the envelopes of $\mathcal{F}_{J}$, so that $K_{\tau}$ can be
conservative. If $\mathcal{D}$ is a rectangle, the largest valid
$K_{\tau}$ is the smallest, over $j$ and $\ell$, of the values of the linear
programs that minimize $a_{j}(f)+c_{\ell}(f)$ over
$f(0),f(x_{1}),\dots,f(x_{n})$ subject to
$f(p)-f(q)\leq C_{J}\lVert(p-q)_{\mathcal{V}+}\rVert^{\gamma_{J}}$ for
$p,q\in\{0,x_{1},\dots,x_{n}\}$. This is because the constraints then
characterize the values of the functions in $\mathcal{F}_{J}$ at those points.
The lemma holds for a general known $\sigma(\cdot)$, with
$D_{Jj,\delta}(x_{i})$ and $D_{\ell J,\delta}(x_{i})$ replaced by
$D_{Jj,\delta}(x_{i})/\sigma^{2}(x_{i})$ and
$D_{\ell J,\delta}(x_{i})/\sigma^{2}(x_{i})$, and for
$f\in\mathcal{P}_{\mathcal{V}}(\gamma_{J},C_{J})$, which has the same
envelopes (Section~\ref{sec:weaker}).

The proposed CI is $CI_{\tau^{\ast}}^{\mathcal{J}}$, with $\tau^{\ast}$ defined
after Lemma~\ref{lem:Bonf-improv} and $K_{\tau}$ given in
Lemma~\ref{lem:slack}. We refer to it as AdaptCalib, and to the naive
Bonferroni CI in \eqref{eq:naive_Bon_CI} as AdaptBonf. Algorithm~\ref{alg:adaptcalib}
summarizes the computation of AdaptCalib, with $\tau^{\ast}$ computed on a
grid from an upper bound on $p_{\tau}(K_{\tau})$, as allowed after
Lemma~\ref{lem:Bonf-improv}. This gives a weakly smaller $\tau^{\ast}$, to
which Lemma~\ref{lem:Bonf-improv} and Theorem~\ref{thm:calibrated-length}
still apply. When $k_{+}\geq1$ and the exponents of the parameter spaces are
pairwise different, the reductions in length from using $\tau^{\ast}$ in
Sections~\ref{sec:Simulation} and \ref{sec:empir-illustr} are finite-sample
improvements, which fade slowly as $n\to\infty$. When the parameter spaces
differ only in $C$, on the other hand, the improvement from $\tau^{\ast}$
computed from $p_{\tau}(K_{\tau})$ itself persists as $n\to\infty$. If, in
addition, the constants $C_{j}$ lie in a fixed interval
$[\underline{C},C_{J}]$ with $\underline{C}>0$, this $\tau^{\ast}$ stays
bounded away from zero as $J$ grows. Hence, the constant of the CI does not
grow to infinity. See Sections~S2.17 and
S2.18.

\begin{algorithm}[t]
  \caption{Computing $CI_{\tau^{\ast}}^{\mathcal{J}}$ (AdaptCalib) and AdaptMono}
  \label{alg:adaptcalib}
  \begin{enumerate}
  \item With $\lVert\cdot\rVert$ satisfying Assumption~\ref{assu:norm} and
    $\mathcal{D}$ satisfying \eqref{eq:domain}, choose
    $\mathcal{F}_{j}=\Lambda_{+,\mathcal{V}}(\gamma_{j},C_{j})$,
    $j=1,\dots,J$, with $\gamma_{j}\geq\gamma_{J}$ and
    $C_{j}\Delta^{\gamma_{j}-\gamma_{J}}\leq C_{J}$, where $\Delta$ is the
    diameter of $\mathcal{D}$, and a grid on $[\alpha/(2J),\alpha]$ that
    contains $\alpha/(2J)$.
  \item For each $\tau$ on the grid, let $\delta=z_{1-\tau}$, and, for each
    $j\leq J$, compute $\omega(\delta,\mathcal{F}_{J},\mathcal{F}_{j})$ and
    $\omega(\delta,\mathcal{F}_{j},\mathcal{F}_{J})$ as the values of $b$ at
    which the inverse moduli in Corollary~\ref{cor:InvModRes(21)} and
    Proposition~\ref{prop:InvModRes}, respectively, which are nondecreasing in
    $b$, equal $\delta$.
  \item Compute $D_{Jj,\delta}(x_{i})$ and $D_{jJ,\delta}(x_{i})$, the
    covariances in \eqref{eq:VWcov}, and $K_{\tau}$ in Lemma~\ref{lem:slack},
    for a general $\sigma(\cdot)$ as described after the lemma.
  \item Compute an upper bound $\bar{p}_{\tau}$ on $p_{\tau}(K_{\tau})$ in
    \eqref{eq:pK} that does not depend on $\{y_{i}\}_{i=1}^{n}$ and satisfies
    $\bar{p}_{\alpha/(2J)}\leq\alpha$, such as that of
    Lemma~S2.1 (we use a maximum-weight
    spanning tree and $24$ subintervals of $[0,K_{\tau}]$), and let
    $\tau^{\ast}$ be the largest $\tau$ on the grid with
    $\bar{p}_{\tau}\leq\alpha$.
  \item Report the interval
    $[\max_{j\leq J}\hat{c}_{[\delta^{\ast}]}^{\ell,j},\min_{j\leq J}\hat{c}_{[\delta^{\ast}]}^{u,j}]$,
    where $\hat{c}_{[\delta^{\ast}]}^{\ell,j}$ and
    $\hat{c}_{[\delta^{\ast}]}^{u,j}$ are the ends of
    Corollary~\ref{cor:chat formula} at $\delta^{\ast}=z_{1-\tau^{\ast}}$.
    If this interval is empty, report its midpoint, as in
    \eqref{eq:naive_Bon_CI}, together with the fact that the data reject
    $f\in\mathcal{F}_{J}$.
  \item[] For AdaptMono, defined below, compute the ordered moduli in step 2
    with $\mathcal{G}_{\mathcal{V}}(\gamma_{J},C_{J})$ in place of
    $\mathcal{F}_{J}$ (Section~\ref{sec:weaker}),
    which requires a design point in each of
    $\mathcal{D}_{\mathcal{V}}^{-}$ and $\mathcal{D}_{\mathcal{V}}^{+}$,
    multiply $D_{Jj,\delta}(x_{i})$ and $D_{jJ,\delta}(x_{i})$ by
    $\mathds{1}\{x_{i}\in\mathcal{D}_{\mathcal{V}}^{-}\}$ and
    $\mathds{1}\{x_{i}\in\mathcal{D}_{\mathcal{V}}^{+}\}$, respectively, and, if
    $k_{+}\geq1$, take $K_{\tau}=0$. In step 5, an empty interval then
    rejects $f\in\mathcal{G}_{\mathcal{V}}(\gamma_{J},C_{J})$.
  \item[] For both CIs, if $\sigma$ is unknown, compute every step as if
    $\sigma(\cdot)=\hat{\sigma}$ (Section~\ref{sec:unknown-sigma}).
  \end{enumerate}
\end{algorithm}

As noted in Section~\ref{sec:constr-adapt-cis}, for the parameter spaces of
this section and for large $n$, the constant of the CI in
\eqref{eq:naive_Bon_CI} grows only like $(\log J)^{r(\gamma_{j},\gamma_{J})}$,
and the same holds for $CI_{\tau^{\ast}}^{\mathcal{J}}$. Hence, we can adapt to every exponent in
an interval by letting the grid of exponents become finer as $n$ grows.
Corollary~S4.1 and the paragraph after it show that, with a grid of
about $\log n$ exponents in $[\gamma_{\min},1]$, the CI in
\eqref{eq:naive_Bon_CI} and $CI_{\tau^{\ast}}^{\mathcal{J}}$ adapt to every
H\"older exponent in $[\gamma_{\min},1]$, up to a factor
$(\log\log n)^{r(\gamma,\gamma_{\min})}$, which is at most
$(\log\log n)^{1/(2+k)}$.
This factor is not intrinsic to the problem. The CI of
\cite{cai2004adaptation} for nested parameter spaces attains the minimax rate
of adaptation over the whole interval without it, by
Corollary~S4.2. However, at reasonable
sample sizes, their CI retains few of the parameter spaces. At a boundary point, on
the other hand,
\citet[Theorem~2.1 and Corollary~3.1]{armstrong2015AdaptiveTestingRegression}
shows in the univariate case that even one-sided CIs in the favorable direction
can adapt to an interval of exponents only up to a factor of order
$(\log\log n)^{\gamma/(2\gamma+1)}$. This order is
$(\log\log n)^{r(\gamma,\gamma_{\min})}$ when $k=k_{+}=1$.

Finally, the same construction gives a CI that maintains coverage over
$\mathcal{G}_{\mathcal{V}}(\gamma_{J},C_{J})$, instead of $\mathcal{F}_{J}$,
while directing power to $\mathcal{F}_{1},\dots,\mathcal{F}_{J}$. Corollary~\ref{cor:chat formula}
gives its ends, with
$\omega(\delta,\mathcal{G}_{\mathcal{V}}(\gamma_{J},C_{J}),\mathcal{F}_{j})$
and $\omega(\delta,\mathcal{F}_{j},\mathcal{G}_{\mathcal{V}}(\gamma_{J},C_{J}))$,
obtained as in Section~\ref{sec:weaker}, in place of
$\omega(\delta,\mathcal{F}_{J},\mathcal{F}_{j})$ and
$\omega(\delta,\mathcal{F}_{j},\mathcal{F}_{J})$, and with
$D_{Jj,\delta}(x_{i})$ and $D_{jJ,\delta}(x_{i})$ multiplied by
$\mathds{1}\{x_{i}\in\mathcal{D}_{\mathcal{V}}^{-}\}$ and
$\mathds{1}\{x_{i}\in\mathcal{D}_{\mathcal{V}}^{+}\}$, respectively. That is,
the lower ends use only the design points in $\mathcal{D}_{\mathcal{V}}^{-}$,
and the upper ends only those in $\mathcal{D}_{\mathcal{V}}^{+}$.
Lemma~\ref{lem:Bonf-improv} and Theorems~\ref{thm:BonfCI} and
\ref{thm:calibrated-length} apply with
$\mathcal{G}_{\mathcal{V}}(\gamma_{J},C_{J})$ as the coverage class, that is,
as the parameter space over which coverage is required
(Section~S2.14). Suppose $k_{+}\geq1$.
Then, for almost all realizations of the design points, a single function in
$\mathcal{G}_{\mathcal{V}}(\gamma_{J},C_{J})$ attains the maximum bias of
every end at once, by Lemma~S2.3. Hence,
\eqref{eq:slack} holds only with $K_{\tau}=0$. Moreover, the lower and upper
ends are independent because the two use disjoint sets of design points.
We refer to the resulting CI as AdaptMono.

AdaptMono is $CI_{\tau^{\ast}}^{\mathcal{J}}$ with $K_{\tau}=0$ when
$k_{+}\geq1$. By
Theorems~\ref{thm:calibrated-length} and \ref{thm:rate-G}, it attains the
minimax rate of adaptation while maintaining coverage over
$\mathcal{G}_{\mathcal{V}}(\gamma_{J},C_{J})$. If
$\mathcal{V}=\{1,\dots,k\}$, AdaptMono maintains coverage over all functions
that are nondecreasing in every coordinate. In this case, each of its ends
depends only on the parameter space to which it directs power because
$(x)_{\mathcal{V}+}=0$ on $\mathcal{D}_{\mathcal{V}}^{-}$ and
$(x)_{\mathcal{V}-}=0$ on $\mathcal{D}_{\mathcal{V}}^{+}$. The pair
$(\gamma_{J},C_{J})$ thus enters only through the two ends that direct power to
$\mathcal{F}_{J}$, and not through the coverage class. If $k_{+}=0$, on the
other hand,
$\mathcal{G}_{\mathcal{V}}(\gamma_{J},C_{J})=\mathcal{P}_{\mathcal{V}}(\gamma_{J},C_{J})$,
and AdaptMono is the CI of this section, with $K_{\tau}$ given in
Lemma~\ref{lem:slack}.

If $\mathcal{V}=\{1,\dots,k\}$, the lower end of AdaptMono that directs power
to $\mathcal{F}_{j}$ is, by Corollary~\ref{cor:chat formula},
$\sum_{i:x_{i}\in\mathcal{D}_{\mathcal{V}}^{-}}w_{j,i}^{-}y_{i}-\delta s_{j}^{-}$,
with weights proportional to $(b_{j}-C_{j}\lVert x_{i}\rVert^{\gamma_{j}})_{+}$
for some $b_{j}>0$, and the upper ends are analogous. There is no bias
correction because such an average has nonpositive bias at every monotone
function. AdaptMono is thus an intersection of one-sided local averages at $J$
bandwidths. One-sided local averages at many scales are also the building
blocks of the multiscale procedures of \citet{dumbgen2001multiscale},
\citet{duembgen2003optimal} and \citet{cai2013AdaptiveConfidenceIntervals} in
the univariate case and of \citet{datta2026optimal} in several dimensions.
Such procedures combine them over all the scales or, as
\citet{cai2013AdaptiveConfidenceIntervals} do, select one scale from the data.
The CIs of \citet{cai2013AdaptiveConfidenceIntervals} and the bands of
\citet{datta2026optimal} adapt to each function, the latter up to a
logarithmic factor. In this case, we provide the theory at the level of parameter spaces, that
is, the rates and constants of Theorems~\ref{thm:rate_result} and
\ref{thm:rate-G}, and a calibration that is valid in finite samples. A
multiscale version of AdaptMono under partial monotonicity, with blocks in
the coordinates in $\mathcal{V}$ within $\mathcal{D}_{\mathcal{V}}^{\pm}$, is a natural
approach to the question raised at the end of Section~\ref{sec:extensions}.

\subsection{Unknown Noise Level and Non-Gaussian Errors}
\label{sec:unknown-sigma}
The CIs $CI_{\alpha}^{Bon,J}$ and $CI_{\tau^{\ast}}^{\mathcal{J}}$ of
Section~\ref{sec:constr-adapt-ci}, with coverage class $\mathcal{F}_{J}$, have
correct coverage in finite samples when the errors are Gaussian with a known
variance function. In this section, we show that, in the homoskedastic case,
both remain asymptotically valid, uniformly over $\mathcal{F}_{J}$ and over
the distribution of the errors, when the errors are not Gaussian and the noise
level is estimated. Suppose that
\[
  y_{i}=f(x_{i})+\sigma u_{i},
\]
where $\sigma>0$ is unknown, and $\{u_{i}\}_{i=1}^{\infty}$ are i.i.d. and
independent of $\{X_{i}\}_{i=1}^{\infty}$, with $\E u_{i}=0$ and
$\E u_{i}^{2}=1$, and their distribution $P$ is unknown. For $\eta>0$ and
$M<\infty$, let $\mathcal{U}_{\eta,M}$ denote the set of distributions $P$
with $\E u_{i}=0$, $\E u_{i}^{2}=1$ and $\E\lvert u_{i}\rvert^{2+\eta}\leq M$.
Let $\p_{f,P}$ denote the distribution of $\{y_{i}\}_{i=1}^{n}$ given the
design points when the regression function is $f$ and the distribution of
$u_{i}$ is $P$. For $s>0$, let $\tau^{\ast}(s)$, $CI_{\alpha}^{Bon,J}(s)$ and
$CI_{\tau^{\ast}(s)}^{\mathcal{J}}(s)$ denote $\tau^{\ast}$ and the CIs
$CI_{\alpha}^{Bon,J}$ and $CI_{\tau^{\ast}}^{\mathcal{J}}$, with $K_{\tau}$,
computed as if $\sigma(\cdot)=s$. Given
an estimator $\hat{\sigma}$ of $\sigma$, we consider the feasible CIs
$CI_{\alpha}^{Bon,J}(\hat{\sigma})$ and
$CI_{\tau^{\ast}(\hat{\sigma})}^{\mathcal{J}}(\hat{\sigma})$, which we set to
$\mathbb{R}$ if $\hat{\sigma}=0$.
\begin{theorem}
  \label{thm:asy-valid}Suppose that, for each $j\leq J$, the conditions of
  Theorem~\ref{thm:rate_result} hold with
  $(\gamma_{1},C_{1},\gamma_{2},C_{2})=(\gamma_{j},C_{j},\gamma_{J},C_{J})$,
  except that the errors are as described above, and that, for almost
  all realizations of the design points,
  $\sup_{P\in\mathcal{U}_{\eta,M}}\sup_{f\in\mathcal{F}_{J}}\p_{f,P}(\lvert\hat{\sigma}-\sigma\rvert>\epsilon)\to0$
  for every $\epsilon>0$. Let $CI_{n}$ denote either
  $CI_{\alpha}^{Bon,J}(\hat{\sigma})$ or
  $CI_{\tau^{\ast}(\hat{\sigma})}^{\mathcal{J}}(\hat{\sigma})$. Then, for
  almost all realizations $\{x_{i}\}_{i=1}^{\infty}$ of
  $\{X_{i}\}_{i=1}^{\infty}$,
  \[
    \liminf_{n\to\infty}\inf_{P\in\mathcal{U}_{\eta,M}}\inf_{f\in\mathcal{F}_{J}}\p_{f,P}(f(0)\in CI_{n})\geq1-\alpha,
  \]
  and, for each $j\leq J$ and $\epsilon>0$, there exists $K<\infty$ such that
  \[
    \limsup_{n\to\infty}\sup_{P\in\mathcal{U}_{\eta,M}}\sup_{f\in\mathcal{F}_{j}}\p_{f,P}\left(\mu(CI_{n})>Kn^{-r(\gamma_{j},\gamma_{J})}\right)\leq\epsilon.
  \]
\end{theorem}
We require only that $\hat{\sigma}$ be consistent uniformly over
$\mathcal{F}_{J}$ and $\mathcal{U}_{\eta,M}$. The nearest-neighbor estimator
$\hat{\sigma}^{2}=(2n)^{-1}\sum_{i=1}^{n}(y_{i}-y_{i^{\ast}})^{2}$, where
$x_{i^{\ast}}$ is a nearest neighbor of $x_{i}$ among the other design points,
satisfies this condition if the fourth moments are also bounded over
$\mathcal{U}_{\eta,M}$ (Section~S4.2). The
conclusions extend to heteroskedastic errors with a local estimate of
$\sigma(0)$, by Remark~S5.1, and to
AdaptMono with coverage over
$\mathcal{G}_{\mathcal{V}}(\gamma_{J},C_{J})\cap\Lambda(\gamma',C')$ for any
$\gamma'\in(0,1]$ and $C'>0$. The intersection with $\Lambda(\gamma',C')$ is
needed because $\hat{\sigma}$ is not necessarily consistent uniformly over
$\mathcal{G}_{\mathcal{V}}(\gamma_{J},C_{J})$ itself. The theorem is for a
fixed number $J$ of parameter spaces and a fixed $\sigma$. It also covers the
CIs computed by Algorithm~\ref{alg:adaptcalib}, in which $\tau^{\ast}(s)$ is
the largest point of a fixed grid at which the bound that we use in step 4
is at most $\alpha$. For such CIs, we require that the limit of
this bound as $n\to\infty$ differs from $\alpha$ at every point of the grid.

\section{Simulation Results}
\label{sec:Simulation}

In this section, we examine the finite-sample performance of the procedures
proposed in this paper. We compare the adaptive CIs described in
Section~\ref{sec:constr-adapt-cis} with
the adaptive CI of \citet{cai2004adaptation} and with nonadaptive benchmarks.
We first take the regression function to be monotone in one of two variables,
and then in both, where we also consider the CI of
\citet{deng2021confidence}. Finally, we consider three variables. The code
that reproduces all the results in this section is provided in the replication
package.

\input{tables/simulation_tables}
\input{tables/aos_numerics_tables}
\input{tables/phase2_tables}
\suppressfloats[t]

\subsection{Monotonicity in One Variable}
\label{sec:sim-one}
We first compare the performances of the adaptive CIs. We take $k=2$ and
$\mathcal{V}=\{1\}$, so that the regression function is monotone in the first
variable only, and consider adapting to the parameter spaces
$\Lambda_{+,\mathcal{V}}(\gamma_{j},1)$, $j=1,\dots,4$, with the Euclidean
norm and $(\gamma_{1},\gamma_{2},\gamma_{3},\gamma_{4})=(1,4/5,2/3,1/2)$, which gives
$J=4$. We set $n=800$ and $\sigma(\cdot)=1$, and draw the design points from
the uniform distribution over $[-1/(2\sqrt{2}),1/(2\sqrt{2})]^{2}$. We fix the
draw of $\{ x_{i} \}_{i=1}^{n}$ throughout.

We consider seven regression functions. The functions $f_{1}(x)=0$ and
$f_{2}(x)=0.6x_{1}+0.5x_{2}$ belong to $\Lambda_{+,\mathcal{V}}(1,1)$, and
$f_{3}(x)=0.8\lVert (x)_{\mathcal{V}+}\rVert_{2}^{1/2}$ belongs to
$\Lambda_{+,\mathcal{V}}(1/2,1)$. The functions $f_{4}$ and $f_{5}$ in
$\Lambda_{+,\mathcal{V}}(1/2,1)$ solve the ordered moduli corresponding to the
lower and the upper end of the CI that directs power to
$\Lambda_{+,\mathcal{V}}(1,1)$, respectively. Finally,
$h_{+}(x)=\lVert (x)_{\mathcal{V}+}\rVert_{2}^{1/2}$ and
$h_{-}(x)=-\lVert (x)_{\mathcal{V}-}\rVert_{2}^{1/2}$ are the functions of
Lemma~\ref{lem:slack}. The function $f_{4}$ is least favorable for one lower
end of the CI, while $h_{+}$ is least favorable for all of its lower ends at
once, and likewise for $f_{5}$ and $h_{-}$ and the upper ends. For each
regression function, we draw the noise terms from a standard normal
distribution $10{,}000$ times. The nominal coverage probability is $0.95$ for
all CIs.

\SimulationUniformTable
Table~\ref{tab:sim-one} shows the average lengths and the coverage
probabilities of the CIs. Each column corresponds to
1) the naive Bonferroni adaptive CI given in \eqref{eq:naive_Bon_CI}
(AdaptBonf), 2) the adaptive CI $CI_{\tau^{\ast}}^{\mathcal{J}}$ that takes
into account the correlation among the one-sided CIs and the bias slack
(AdaptCalib), 3) the adaptive CI of \citet{cai2004adaptation} (CL,
henceforth), 4) the FLCI for the largest parameter space $\Lambda_{+,\mathcal{V}}(\gamma_{4},1)$, and 5) the FLCI for the
smallest parameter space that contains the true regression function. We refer
to the last two as the conservative CI and the oracle CI, respectively. The
FLCI, described in Section~\ref{sec:constr-adapt-cis}, is the
shortest affine CI that maintains coverage over a given parameter space. The
oracle CI is infeasible
because it requires knowledge of the smallest parameter space that contains
the true regression function.

To compute $\tau^{\ast}$, we follow Algorithm~\ref{alg:adaptcalib} on a grid
with step $0.00025$. This gives $\tau^{\ast}=0.037$, which reduces the critical value
from $z_{1-0.05/8}=2.498$ to $z_{0.963}=1.787$.

In Table~\ref{tab:sim-one}, the CI of CL is about three times as long as
AdaptBonf and more than three times as long as the other CIs, for all the
regression functions we consider. Here, the procedure of CL uses
only the CI that directs power to the largest parameter space, so that its
length does not depend on the regression function. Their thinning step keeps a
parameter space only if its between class modulus is at most half that of the
next parameter space kept. However, the largest and the smallest between class
moduli differ only by a factor of $1.35$.

Using the correlation among the one-sided CIs and the bias slack
reduces the average length of the naive Bonferroni CI by $12\%$ to $14\%$.
AdaptCalib is $3\%$ shorter than the conservative CI when $f=f_{1}$, and at
most $4\%$ longer for the other regression functions.

On the other hand, the oracle CI is about
half as long as the adaptive CIs when the regression function is smooth
($f=f_{1}$ or $f_{2}$). This is consistent with
Theorem~\ref{thm:rate_result}. The regression function is monotone in
only one of the two variables. Hence, the minimax rate of adaptation of
$\Lambda_{+,\mathcal{V}}(1,1)$ over $\Lambda_{+,\mathcal{V}}(1/2,1)$ is
$n^{-1/5}$. This rate is slower than the minimax rate over
$\Lambda_{+,\mathcal{V}}(1,1)$, $n^{-1/4}$, attained by the oracle CI, and
faster than the minimax rate over $\Lambda_{+,\mathcal{V}}(1/2,1)$,
$n^{-1/6}$, attained by the conservative CI. With $n=800$, the gain from the
faster rate relative to the conservative CI, a factor of
$800^{1/30}\approx1.25$, is small, and the adaptive CI is about as long as
the conservative CI.

Table~\ref{tab:sim-one} shows that all the CIs achieve the correct coverage,
up to Monte Carlo error, for all the regression functions we consider. The
coverage probabilities of AdaptCalib are between $0.952$ and $0.955$ at the
least favorable functions $f_{4}$, $f_{5}$ and $h_{\pm}$, so that the
calibration is almost exact there. Those of the naive Bonferroni CI are above
$0.99$, and the CI of CL has coverage probability equal to one in all cases. Note that the coverage probabilities are computed only at the seven
regression functions we consider. Lemma~\ref{lem:Bonf-improv} ensures correct coverage uniformly over
$\Lambda_{+,\mathcal{V}}(\gamma_{4},1)$.

\subsection{Monotonicity in Both Variables}
\label{sec:sim-both}
We now take $\mathcal{V}=\{1,2\}$, in which case
Theorem~\ref{thm:rate_result} shows that strong adaptation is possible. We use
the same parameter spaces and $\sigma(\cdot)=1$, and the $29\times29$ lattice
on the same square, which gives $n=841$ and contains $x_{0}=0$. The regression
functions are defined as before with $\mathcal{V}=\{1,2\}$. We also compute
AdaptMono, the CI of Section~\ref{sec:constr-adapt-ci} that directs power to
the same four parameter spaces and maintains coverage over
$\mathcal{G}_{\mathcal{V}}(\gamma_{J},C_{J})$. When $\mathcal{V}=\{1,2\}$, this
is the class of all functions that are monotone in both variables, without any
smoothness restriction. We compute the CI of \citet{deng2021confidence} (DHZ,
henceforth), which is based on the block max-min and min-max estimators. We
use the known $\sigma$ and their suggested critical value $1.80$ for $k=2$, which
is calibrated for regression functions whose partial derivatives at $x_{0}$
are nonzero. Our implementation reproduces the quantiles of their pivotal
statistic in their Table~2.

\FullyMonotoneLatticeTable
Table~\ref{tab:fm-lattice} shows the results. With monotonicity in both
variables, the gain from
adaptation is larger. AdaptCalib is $34\%$ shorter than the conservative CI
when $f=f_{1}$, and $5.6\%$ shorter when $f=f_{2}$. In the former case, it is
$13\%$ longer than the oracle CI. It is $2\%$ longer than the conservative CI for
$f_{3}$, $f_{4}$ and $f_{5}$ and $7\%$ longer for $h_{\pm}$. The coverage
probability of AdaptCalib is at least $0.958$ for all the regression functions. AdaptMono,
whose coverage does not depend on any smoothness bound, is shorter still when
$f=f_{1}$, only $5\%$ longer than the oracle CI. However, it is $13\%$ to $25\%$
longer than the conservative CI for the rough regression functions. The CI of
CL is again more than three times as long as the conservative CI.

The CI of DHZ is designed to have asymptotically exact coverage at each
regression function that satisfies their local smoothness condition, rather
than to maintain coverage over a parameter space. Like AdaptMono, it requires
no bound on the smoothness. In Table~\ref{tab:fm-lattice}, it is $28\%$ to
$38\%$ shorter than AdaptMono. At $f_{2}$, the only function for which their
critical value is calibrated, its coverage probability is $0.955$. At the
cusps $f_{3}$, $f_{4}$ and $f_{5}$, which lie in our parameter spaces but
violate their condition, its coverage probabilities are between $0.68$ and
$0.70$, while those of AdaptMono are at least $0.975$.

\subsection{Monotonicity in Some of Three Variables}
\label{sec:sim-three}
To see how the gain from adaptation depends on the number of monotone
variables, we take $k=3$ and a uniform design on
$[-1/(2\sqrt{3}),1/(2\sqrt{3})]^{3}$, whose diameter is one. We use the same four
parameter spaces, $\mathcal{V}=\{1\}$, $\{1,2\}$ and $\{1,2,3\}$, and
$n=800$, $5{,}000$ and $20{,}000$. Theorem~\ref{thm:rate_result} predicts that
the minimax rate of adaptation of $\Lambda_{+,\mathcal{V}}(1,1)$ over
$\Lambda_{+,\mathcal{V}}(1/2,1)$ is $n^{-1/7}$, $n^{-1/6}$ and $n^{-1/5}$ for
$k_{+}=1$, $2$ and $3$, compared with $n^{-1/8}$ for the conservative CI and
$n^{-1/5}$ for the oracle CI.

\KThreeLargerNTable
Table~\ref{tab:k3-larger-n} shows the average lengths of AdaptCalib and
AdaptMono relative to those of the FLCIs. When the regression function is
monotone in two of the three variables, AdaptCalib is $4\%$ shorter than the
conservative CI at $f_{1}$ with $n=800$, and $14\%$ shorter with
$n=20{,}000$. However, it is at most about $3\%$ longer at $h_{+}$. When the regression function is monotone
in all three, AdaptCalib is $34\%$ shorter at $f_{1}$ with $n=800$ and $52\%$
shorter with $n=20{,}000$. AdaptMono, whose coverage does not depend on a
smoothness bound, is shorter still. With one monotone variable, the gain is
small at the sample sizes we consider, as in Section~\ref{sec:sim-one}. The coverage
probabilities of AdaptCalib and AdaptMono are at least $0.945$ at all the
regression functions in their largest parameter spaces, which is within
Monte Carlo error of $0.95$.

\section{Empirical Illustration}
\label{sec:empir-illustr}
In this section, we apply our procedure to the production function estimation
problem for the Chinese chemical industry. Specifically, we use the firm-level data of
\citet{jacho2010identification} for the year
2001, which was also used by \citet{horowitz2017nonparametric} to illustrate
their method of constructing the uniform confidence band for the production
function under shape restrictions.

In the dataset, the dependent variable is the logarithm of value-added real
output ($y$), and the explanatory variables are the logarithms of the net value
of real fixed assets ($k$) and the number of employees ($\ell$). After
removing the outliers for $k$ and $\ell$, the remaining sample size is
$n=1{,}663$.\footnote{We use the conventional way of outlier detection,
  removing the observations that are greater than the third quartile plus IQR
  times 1.5, or less than the first quartile minus IQR times 1.5, for each
  input. We do not remove the outliers for $y$, because removing observations
  based on the outcome changes the conditional mean that we are interested
  in.} We are interested in constructing a CI for
$f(k_{0},\ell_{0}) := \E \left[y|k=k_{0},\ell=\ell_{0}\right]$. We
take $(k_{0},\ell_{0})$ to be the medians of $k$ and $\ell$.

\input{tables/phase2_empirical_tables}

The first step is to estimate the variance of the error term. We assume
homoskedastic errors for simplicity. We use the nearest-neighbor estimator
$\hat{\sigma}$ of Section~\ref{sec:unknown-sigma}, with the nearest neighbors
computed after dividing each input by its standard deviation, which gives
$\hat{\sigma}=1.088$.

For the function space, we consider adapting to a sequence of parameter
spaces $\Lambda_{+,\mathcal{V}}(\gamma_{j},C_{j})$, $j=1,\dots,6$, with
$\gamma_j = 1 - (j - 1)/5$ for $j = 1,\dots,5$ and $\gamma_6 = 10^{-3}$. We
first take $\mathcal{V}=\{1,2\}$, assuming that the production function is
nondecreasing in both fixed assets and labor, which is consistent with
economic theory. For the norm, we use the Euclidean norm weighted by the
inverse of the standard deviation of each input,
$\lVert (k,\ell)\rVert
=(({k}/{s_{k}})^{2}+({\ell}/{s_{\ell}})^{2})^{1/2}$, where $s_{k}$ and
$s_{\ell}$ are the standard deviations of $k$ and $\ell$, respectively. We take
the domain $\mathcal{D}$ to be the smallest set that contains
$(k_{0},\ell_{0})$ and the design points and satisfies the domain condition
\eqref{eq:domain}. For each choice of $\mathcal{V}$ in this section, its
diameter with respect to this norm is $\Delta=7.388$, the same as that of the
design points. To make
$\Lambda_{+,\mathcal{V}}(\gamma_{j},C_{j})\subseteq\Lambda_{+,\mathcal{V}}(\gamma_{6},C_{6})$
hold for all $j=1,\dots,5$, we set $C_{j}=C_{6}\Delta^{\gamma_{6}-\gamma_{j}}$, which satisfies the condition in
Section~\ref{sec:solut-modul-probl}, and
\[
  C_{6}=\underset{1\leq i<l\leq n}{\max }\frac{\lvert y_{l}-y_{i}\rvert }{\lVert (k_{l},\ell_{l})-(k_{i},\ell_{i})\rVert^{\gamma_{6}}}=14.612,
\]
so that $C_{1}=1.982$. Since $\gamma_{6}$ is close to zero, $C_{6}$ is
approximately the range of the outcome. It is attained by a firm with $y=0$,
without which it would be $12.415$. A function in
$\Lambda_{+,\mathcal{V}}(\gamma_{6},C_{6})$ can thus vary over
$\mathcal{D}$ by as much as the observed outcomes do, which we view as a
conservative choice.

We compare different procedures to construct CIs. The methods in comparison
are the FLCI for the largest space
$\Lambda_{+,\mathcal{V}}(\gamma_{6},C_{6})$ (the conservative CI), the FLCI for
the smallest space
$\Lambda_{+,\mathcal{V}}(\gamma_{1},C_{1})$ (the restricted CI), the naive
Bonferroni CI adapting to $\{\gamma_j\}_{j = 1}^6$ (AdaptBonf), AdaptCalib,
AdaptMono, and the CI of DHZ. By construction, AdaptBonf, AdaptCalib and the
conservative CI maintain correct coverage over the largest space
$\Lambda_{+,\mathcal{V}}(\gamma_{6},C_{6})$, and AdaptMono over all monotone
functions, so that its coverage does not depend on $C_{6}$. However, the
restricted CI is valid only over the smallest space
$\Lambda_{+,\mathcal{V}}(\gamma_{1},C_{1})$. We compute the CI of DHZ as in
Section~\ref{sec:sim-both}, with $\hat{\sigma}$ in place of $\sigma$.

For AdaptCalib, Algorithm~\ref{alg:adaptcalib} gives $\tau^{\ast}=0.00717$
and $K_{\tau}=0$. This reduces the critical value from $z_{1-0.05/12}=2.638$ to
$z_{1-\tau^{\ast}}=2.449$. Here, $\gamma_{6}$ is close to zero and $C_{6}$ is
large. Hence, the lower ends that direct power to the five smallest spaces use only
the design points in the nonpositive quadrant around $(k_{0},\ell_{0})$, and
the upper ends only those in the nonnegative quadrant, as under the coverage class of all monotone functions. The reduction thus
comes from the correlation among the one-sided CIs only. These ends are the
same for AdaptCalib and AdaptMono. Since the calibration gives the same
$\tau^{\ast}$, and neither end of the CI comes from the one-sided CIs that
direct power to the largest space, the two CIs coincide.

\EmpiricalPhaseTwoTable
Table~\ref{tab:empirical-phase2} shows the 95\% confidence intervals for
$f(k_{0},\ell_{0})$ produced by different inference methods.
First of all, the adaptive CIs are much shorter than the conservative CI.
AdaptMono is $93\%$ shorter than the conservative CI and $9\%$ shorter than
AdaptBonf. The
restricted CI is shorter still, but it is valid only when we are confident
that the true regression function lies in the smallest space we consider,
$\Lambda_{+,\mathcal{V}}(\gamma_{1},C_{1})$. AdaptMono is $44\%$ longer than
this CI, while it maintains coverage over all monotone functions. This is
consistent with Theorem~\ref{thm:rate-G}. Strong adaptation is possible when
the regression function is monotone in all the variables, even with coverage
over all monotone functions. The CI of DHZ is the shortest, but its critical
value is calibrated for regression functions whose partial derivatives at
$(k_{0},\ell_{0})$ are nonzero, and it can undercover at functions in our
parameter spaces, as shown in Section~\ref{sec:sim-both}.

The lower end of
AdaptMono comes from the one-sided CI that directs power to the parameter
space with $\gamma_{4}=0.4$. This one-sided CI puts positive weight on only $7$
observations, with effective sample size $(\sum_{i}w_{i}^{2})^{-1}$ of about
$4$, where the $w_{i}$ are its weights. The upper end comes from the one-sided
CI with $\gamma_{1}=1$, which puts positive weight on $61$ observations, with
effective sample size of about $36$. The normal approximation in
Theorem~\ref{thm:asy-valid} is thus not necessarily accurate for the lower end in
this illustration.

We next take $\mathcal{V}=\{1\}$ and $\mathcal{V}=\{2\}$, assuming that the
production function is nondecreasing in only one of the inputs, to illustrate
the CIs when the regression function is monotone in some of the variables. In
this case, little adaptation is possible when $\gamma_{J}$ is close to zero.
The minimax rate of adaptation of $\Lambda_{+,\mathcal{V}}(\gamma_{1},C_{1})$
over $\Lambda_{+,\mathcal{V}}(\gamma_{6},C_{6})$ in
Theorem~\ref{thm:rate_result} is $n^{-1/1003}$. In the data, with $J=6$, the
largest space restricts $f$ along the other input almost entirely through its
range. Hence, only the firms whose other input equals its median, of which
there are one or two, are informative about $f(k_{0},\ell_{0})$ without a bias
of the order of $C_{6}$.

For this reason, we adapt to the four smallest spaces only, with $J=4$, so that
$\gamma_{J}=0.4$ and $C_{J}=6.579$. The rate of adaptation in
Theorem~\ref{thm:rate_result} is then $n^{-2/11}$, compared with $n^{-1/7}$ for the
conservative CI. The adaptive CIs are then $22\%$ to $26\%$ shorter than the
conservative CI. The adaptive CIs are more than three times as long as the restricted CI,
and more than twice as long as AdaptMono under monotonicity in both inputs,
whose rate is $n^{-1/4}$, even though AdaptMono assumes less smoothness. Here,
AdaptMono, which maintains coverage over
$\mathcal{G}_{\mathcal{V}}(\gamma_{4},C_{4})$, is $3\%$ to $4\%$ shorter than
AdaptCalib, but it is longer in the simulations of
Section~\ref{sec:sim-three} with one monotone variable.

In this illustration, $C_{6}$ is chosen using the same data that we use to
construct the CIs, while the coverage results in
Sections~\ref{sec:adapt-conf-interv} and \ref{sec:unknown-sigma} take it as
known. We thus interpret the CIs as illustrating the procedures for this value
of $C_{6}$. However, under monotonicity in both inputs, the coverage of
AdaptMono does not depend on it. The monotonicity of
$\E[y|k,\ell]$ in the inputs is an assumption on the conditional mean, which is consistent with, but not implied by, the
monotonicity of the production function, because the inputs can respond to
productivity. Together with the simulation results in the previous section,
the illustration shows that, under monotonicity in all the variables, the
adaptive CIs can be much shorter than the conservative CI while maintaining
coverage over the largest parameter space. Moreover, AdaptMono can do so while
maintaining coverage over all monotone functions.
With monotonicity in only some of the variables, the gain is smaller.

\begin{acks}[Acknowledgments]
We thank our advisors Donald Andrews and Timothy Armstrong for discussions and guidance.
Xiaohong Chen, Yuichi Kitamura, and participants at the Yale Econometrics Prospectus
Lunch provided valuable feedback. We thank David Tom\'as Jacho-Ch\'avez for providing
us with the dataset used in Section~\ref{sec:empir-illustr}.
\end{acks}

\begin{supplement}
\stitle{Supplement to ``Adaptive Inference in Multivariate Nonparametric Regression \mbox{Models} Under Monotonicity''}
\sdescription{Proofs, extensions and further remarks.}
\end{supplement}

\begin{supplement}
\stitle{Replication Package}
\sdescription{Code, archived outputs and tests for Sections~\ref{sec:Simulation} and \ref{sec:empir-illustr}.}
\end{supplement}

\bibliographystyle{imsart-nameyear}
\bibliography{ModulusPartBib,KK_rfap,Phase1Refs}

\end{document}


\begin{frontmatter}
\title{Supplement to ``Adaptive Inference in Multivariate Nonparametric Regression Models Under Monotonicity''}
\pdftitle{Supplement to \textquotedblleft Adaptive Inference in Multivariate Nonparametric Regression Models Under Monotonicity\textquotedblright}
\runtitle{Supplement: adaptive inference under monotonicity}
\begin{aug}
\author[A]{\fnms{Koohyun}~\snm{Kwon}\thanksref{t1}\ead[label=e1]{koohyun.kwon@capitalone.com}}
\author[B]{\fnms{Soonwoo}~\snm{Kwon}\ead[label=e2]{soonwoo\_kwon@brown.edu}}
\address[A]{Capital One, McLean, Virginia, USA\printead[presep={,\ }]{e1}}
\address[B]{Brown University, Providence, Rhode Island, USA\printead[presep={,\ }]{e2}}
\thankstext{t1}{The views expressed are those of the author and do not
necessarily reflect those of Capital One.}
\runauthor{K. Kwon and S. Kwon}
\end{aug}

\begin{abstract}
This supplement contains the definition of the optimal upper CI and the proofs
omitted from the main text. We also provide results on the parameter spaces of
Section~4.3, the limit of the calibration, adapting to an interval
of exponents, estimating the noise level, heteroskedasticity and adaptation
under only monotonicity. We then give the extensions to coordinate-specific
smoothness, to uncertainty about the monotonicity set and to boundary points,
and the proof of the rate result for general monotonicity cones and norms.
Finally, we collect further remarks. We refer to equations, sections and
results in the main text by their main-text numbers.
\end{abstract}

\end{frontmatter}

\section{Definition of the Optimal Upper CI}
\label{sec:defint-optim-upper}

The following corollary is the analogue of Lemma~3.2 for the
upper one-sided CI.
\begin{corollary}
  \label{cor:cor-AK-ub}Let $\delta=z_{\beta}+z_{1-\alpha}$. Suppose that the
  ordered modulus $\omega(\cdot,\mathcal{F}_{j},\mathcal{F}_{J})$ is finite and
  differentiable at $\delta$, and that it is attained at $\delta$ by a solution
  $(f_{j,\delta}^{*},g_{J,\delta}^{*})\in\mathcal{F}_{j}\times\mathcal{F}_{J}$
  for which the constraint binds:
  \begin{align*}
    & \sum_{i=1}^{n}\left((g_{J,\delta}^{*}(x_{i})-f_{j,\delta}^{*}(x_{i}))/\sigma(x_{i})\right)^{2}=\delta^{2},\text{ and}\\
    & Lg_{J,\delta}^{*}-Lf_{j,\delta}^{*}=\omega(\delta,\mathcal{F}_{j},\mathcal{F}_{J}).
  \end{align*}
  Define
  \begin{align*}
    \hat{L}_{\delta}^{u,j}  &= \frac{Lf_{j,\delta}^{*}+Lg_{J,\delta}^{*}}{2} \\
    &\quad +\frac{\omega'(\delta,\mathcal{F}_{j},\mathcal{F}_{J})}{\delta}\sum_{i=1}^{n}\frac{g_{J,\delta}^{*}(x_{i})-f_{j,\delta}^{*}(x_{i})}{\sigma^{2}(x_{i})}\left(y_{i}-\frac{f_{j,\delta}^{*}(x_{i})+g_{J,\delta}^{*}(x_{i})}{2}\right).
  \end{align*}
  Then,
  \[
    \hat{c}_{\alpha,\delta}^{u,j}:=\hat{L}_{\delta}^{u,j}+\frac{1}{2}\omega(\delta,\mathcal{F}_{j},\mathcal{F}_{J})-\frac{1}{2}\delta\omega'(\delta,\mathcal{F}_{j},\mathcal{F}_{J})+z_{1-\alpha}\omega'(\delta,\mathcal{F}_{j},\mathcal{F}_{J})
  \]
  satisfies $(-\infty,\hat{c}_{\alpha,\delta}^{u,j}]\in\mathcal{I}_{\alpha,u}^{J}$
  and solves
  \begin{equation*}
    \min_{\hat{c}:(-\infty,\hat{c}]\in\mathcal{I}_{\alpha,u}^{J}}\;\sup_{f\in\mathcal{F}_{j}}\,q_{\beta,f}(\hat{c}-Lf),
  \end{equation*}
  with
  $\sup_{f\in\mathcal{F}_{j}}q_{\beta,f}(\hat{c}_{\alpha,\delta}^{u,j}-Lf)=\omega(\delta,\mathcal{F}_{j},\mathcal{F}_{J})$.
  In particular, when $\beta=1/2$, we have
  \begin{equation}
    \underset{f\in\mathcal{F}_{j}}{\sup}\E_{f}(\hat{c}_{\alpha}^{u,j}-Lf)=\omega(z_{1-\alpha},\mathcal{F}_{j},\mathcal{F}_{J}),\label{eq:exp_ex_len_U}
  \end{equation}
  where $\hat{c}_{\alpha}^{u,j}:=\hat{c}_{\alpha,z_{1-\alpha}}^{u,j}$.
\end{corollary}
We prove the corollary in Section~\ref{sec:proof-cor-AK-ub}.

\section{Proofs and Additional Results for Sections~3 to 5}
\label{appdx-sec:Proofs_lemmas}
Throughout the supplement, for a function $h$ on $\mathcal{D}$, we write
\[
  Kh:=(h(x_{1})/\sigma(x_{1}),\dots,h(x_{n})/\sigma(x_{n}))'\in\mathbb{R}^{n}
\]
and $\lVert h\rVert_{\sigma}:=\lVert Kh\rVert_{2}$, so that the constraints in the
moduli of continuity can be written as $\lVert f_{2}-f_{1}\rVert_{\sigma}\leq\delta$.
We also write $Z_{i}:=y_{i}/\sigma(x_{i})$ and
$\varepsilon_{i}:=u_{i}/\sigma(x_{i})$, so that
$Z=Kf+\varepsilon$ with $\varepsilon\sim N(0,I_{n})$. For nonempty convex
$\mathcal{F}$ and $\mathcal{G}$ with $\mathcal{F}\cap\mathcal{G}\neq\emptyset$,
the ordered modulus $\omega(\cdot,\mathcal{F},\mathcal{G})$ is nonnegative,
nondecreasing and concave on $[0,\infty)$. This is because a convex combination
of feasible pairs at $\delta_{1}$ and $\delta_{2}$ is feasible at the
corresponding convex combination of $\delta_{1}$ and $\delta_{2}$. Hence,
$\delta\,\omega'(\delta,\mathcal{F},\mathcal{G})\leq\omega(\delta,\mathcal{F},\mathcal{G})$
wherever the derivative exists.

\subsection{Proof of Lemma~3.1}
\label{sec:proof-CL}
\begin{proof}
  We first show the lower bound, which is Theorem~1 of
  \cite{cai2004adaptation}. Take any $f_{1}\in\mathcal{F}_{j}$ and
  $f_{2}\in\mathcal{F}_{J}$ such that
  $d:=\lVert f_{2}-f_{1}\rVert_{\sigma}\leq z_{1-\alpha}$, and let
  $CI\in\mathcal{I}_{\alpha,2}^{J}$. Consider the events
  $A=\{Lf_{2}\in CI\}$ and $B=\{Lf_{1}\in CI\}$. Since
  $f_{1},f_{2}\in\mathcal{F}_{J}$, we have
  $\p_{f_{2}}(A)\geq1-\alpha$ and $\p_{f_{1}}(B)\geq1-\alpha$. If $d>0$, the
  likelihood ratio between $\p_{f_{2}}$ and $\p_{f_{1}}$ depends on the data
  only through $T=(K(f_{2}-f_{1}))'(Z-Kf_{1})/d$, which follows $N(0,1)$
  under $f_{1}$ and $N(d,1)$ under $f_{2}$. Hence, by the Neyman--Pearson
  lemma, any event with probability at least $1-\alpha$ under $f_{2}$ has
  probability at least $\p(N(0,1)\geq d-z_{1-\alpha})=1-\Phi(d-z_{1-\alpha})$ under $f_{1}$.
  This holds trivially if $d=0$. Since $d\leq z_{1-\alpha}$, it follows that
  $\p_{f_{1}}(A)\geq1/2$, and thus
  \[
    \p_{f_{1}}(A\cap B)\geq\frac{1}{2}+(1-\alpha)-1=\frac{1}{2}-\alpha.
  \]
  On $A\cap B$, $CI$ contains both $Lf_{1}$ and $Lf_{2}$, so that
  $\mu(CI)\geq\lvert Lf_{2}-Lf_{1}\rvert$. Therefore,
  \[
    \sup_{f\in\mathcal{F}_{j}}\E_{f}\mu(CI)\geq\E_{f_{1}}\mu(CI)\geq\left(\frac{1}{2}-\alpha\right)\lvert Lf_{2}-Lf_{1}\rvert.
  \]
  Taking the supremum over such $(f_{1},f_{2})$ and then the infimum over
  $CI\in\mathcal{I}_{\alpha,2}^{J}$ gives the lower bound. Note that the
  argument does not use convexity, nor does it require the modulus to be
  attained.

  For the upper bound, we verify that the construction of
  \cite{cai2004adaptation} applies to our setting. The bound is trivial if
  $\omega_{+}(z_{1-\alpha/2},\mathcal{F}_{j},\mathcal{F}_{J})=\infty$, so
  suppose it is finite. Let $\Theta_{1}=\{(Kf,Lf):f\in\mathcal{F}_{j}\}$ and
  $\Theta_{2}=\{(Kf,Lf):f\in\mathcal{F}_{J}\}$, which are nonempty convex
  subsets of $\mathbb{R}^{n}\times\mathbb{R}$ with
  $\Theta_{1}\subseteq\Theta_{2}$. Then, $Z=\widetilde{K}\theta+\varepsilon$,
  where $\widetilde{K}(u,t)=u$ and the parameter of interest is
  $\ell\theta=t$ for $\theta=(u,t)$. This differs from the Gaussian models
  considered by \cite{cai2004adaptation} in that $\ell\theta$ need not be a
  function of the mean of $Z$. However, their proof of the upper bound (their
  Lemma~1, equation (17), for the CI $CI_{1,\alpha}^{\ast}$ in their
  equation (16)) uses the model only through affine estimators with the
  properties in their equations (13)--(15), which we now construct. Fix $\delta>0$, and
  consider the ordered pairs
  $(a,c)\in\{(1,1),(1,2),(2,1)\}$, which are the pairs used in their
  construction for the CI that adapts to $\Theta_{1}$. The pair $(2,2)$ is not
  used, so that the upper bound holds even if
  $\omega(\cdot,\mathcal{F}_{J},\mathcal{F}_{J})$ is infinite. For these pairs, the
  ordered moduli
  $\omega_{ac}:=\sup\{\ell(\theta_{c}-\theta_{a}):\lVert\widetilde{K}(\theta_{c}-\theta_{a})\rVert_{2}\leq\delta,\theta_{a}\in\Theta_{a},\theta_{c}\in\Theta_{c}\}$
  are bounded by the between class modulus, and thus finite. Let
  $\mathcal{C}_{ac}=\{(\widetilde{K}h,s):h\in\Theta_{c}-\Theta_{a},\,s\leq\ell h\}$,
  which is convex and disjoint from the open convex set
  $\{(u,s):\lVert u\rVert_{2}<\delta,\,s>\omega_{ac}\}$ by the definition of
  $\omega_{ac}$. By the separating hyperplane theorem, there exists
  $(w,\lambda)\neq0$ such that $\lambda s-w'u$ is bounded above on
  $\mathcal{C}_{ac}$ by its infimum over the open set. Since
  $\mathcal{C}_{ac}$ is closed under decreasing $s$, $\lambda\geq0$, and since
  $\Theta_{a}\cap\Theta_{c}\neq\emptyset$ implies $(0,0)\in\mathcal{C}_{ac}$,
  $\lambda=0$ would force $w=0$. Normalizing $\lambda=1$ and evaluating the infimum, we
  obtain a vector $w_{ac}$ such that
  \begin{equation*}
    \ell h-w_{ac}'\widetilde{K}h\leq\omega_{ac}-\delta\lVert w_{ac}\rVert_{2}\quad\text{for all }h\in\Theta_{c}-\Theta_{a},\qquad\lVert w_{ac}\rVert_{2}\leq\omega_{ac}/\delta,
  \end{equation*}
  where the second inequality follows from taking $h=0$. Define
  \[
    \kappa_{ac}=\frac{1}{2}\left\{ \sup_{\theta\in\Theta_{c}}(\ell\theta-w_{ac}'\widetilde{K}\theta)+\inf_{\theta\in\Theta_{a}}(\ell\theta-w_{ac}'\widetilde{K}\theta)\right\},
  \]
  where both terms are finite and their difference is at most $\omega_{ac}$
  by the preceding display. Then, the affine estimator
  $\widehat{T}_{ac}=\kappa_{ac}+w_{ac}'Z$ has standard deviation at most
  $\omega_{ac}/\delta$, its bias is at most $\omega_{ac}/2$ over
  $\Theta_{a}$ and at least $-\omega_{ac}/2$ over $\Theta_{c}$. These are
  exactly the properties (13)--(15) of the affine estimators used in the
  proof of Lemma~1 of \cite{cai2004adaptation}, and the rest of their proof
  is a calculation for Gaussian random variables that applies without
  change.
  Taking $\delta=z_{1-\alpha/2}$, it gives a CI with coverage at least
  $1-\alpha$ over $\Theta_{1}\cup\Theta_{2}=\Theta_{2}$, whose expected
  length over $\Theta_{1}$ is at most
  $(4+9/z_{1-\alpha/2})\,\omega_{+}(z_{1-\alpha/2},\mathcal{F}_{j},\mathcal{F}_{J})$.
  This proves the upper bound.
\end{proof}

  \subsection{Proof of Corollary~\ref{cor:cor-AK-ub}}
  \label{sec:proof-cor-AK-ub}
  \begin{proof}
    Let $\widetilde{L}f=-Lf$. The ordered modulus for $\widetilde{L}$ that
    is relevant for a lower CI with coverage over $\mathcal{F}_{J}$ and power
    directed at $\mathcal{F}_{j}$ satisfies
    \[
      \sup\{ \widetilde{L}f_{j}-\widetilde{L}f_{J}:\lVert f_{j}-f_{J}\rVert_{\sigma}\leq\delta,f_{J}\in\mathcal{F}_{J},f_{j}\in\mathcal{F}_{j}\} =\omega(\delta,\mathcal{F}_{j},\mathcal{F}_{J}),
    \]
    with solution $(g_{J,\delta}^{*},f_{j,\delta}^{*})\in\mathcal{F}_{J}\times\mathcal{F}_{j}$.
    With this solution, (4) and (5) for
    $\widetilde{L}$ equal $-\hat{L}_{\delta}^{u,j}$ and
    $-\hat{c}_{\alpha,\delta}^{u,j}$, respectively. Hence, applying Lemma~3.2 with $\widetilde{L}$ in place of $L$, the lower
    one-sided CI $[-\hat{c}_{\alpha,\delta}^{u,j},\infty)$ for
    $\widetilde{L}f$ has coverage at least $1-\alpha$ over $\mathcal{F}_{J}$
    and is optimal. Since
    $\widetilde{L}f\in[-\hat{c}_{\alpha,\delta}^{u,j},\infty)$ if and only if
    $Lf\in(-\infty,\hat{c}_{\alpha,\delta}^{u,j}]$ and
    $\widetilde{L}f-(-\hat{c}_{\alpha,\delta}^{u,j})=\hat{c}_{\alpha,\delta}^{u,j}-Lf$,
    we get the desired result. For $\beta=1/2$, the excess length is normally
    distributed, so that its median equals its mean, which gives
    \eqref{eq:exp_ex_len_U}.
  \end{proof}

  \subsection{Proof of Theorem~3.1}
  \label{sec:proof-BonfCI}
  \begin{proof}
    We first show the coverage. Fix $f\in\mathcal{F}_{J}$. By Lemma~3.2 and Corollary~\ref{cor:cor-AK-ub}, each of the $2J$
    one-sided CIs $[\hat{c}_{\alpha/(2J)}^{\ell,j},\infty)$ and
    $(-\infty,\hat{c}_{\alpha/(2J)}^{u,j}]$ has coverage at least
    $1-\alpha/(2J)$ over $\mathcal{F}_{J}$. Hence, by the union bound, the
    event $\mathcal{E}_{f}=\cap_{j=1}^{J}\{\hat{c}_{\alpha/(2J)}^{\ell,j}\leq
    Lf\leq\hat{c}_{\alpha/(2J)}^{u,j}\}$ has probability at least $1-\alpha$.
    On $\mathcal{E}_{f}$, we have $A_{\alpha}\leq Lf\leq B_{\alpha}$, so that
    $Lf\in CI_{\alpha}^{Bon,J}$.

    Next, we bound the expected length. Fix $j$ and $f\in\mathcal{F}_{j}$,
    and write $\ell_{j}=\hat{c}_{\alpha/(2J)}^{\ell,j}$ and
    $u_{j}=\hat{c}_{\alpha/(2J)}^{u,j}$. Since an empty intersection is replaced
    by a singleton, we have
    \[
      \mu(CI_{\alpha}^{Bon,J})=(B_{\alpha}-A_{\alpha})_{+}\leq(u_{j}-\ell_{j})_{+}.
    \]
    Note that $u_{j}-\ell_{j}$ is normally distributed, being an affine
    function of the data. By (6) and
    \eqref{eq:exp_ex_len_U} applied with $\alpha/(2J)$ in place of $\alpha$,
    its mean is
    \[
      m_{j,f}:=\E_{f}[u_{j}-Lf]+\E_{f}[Lf-\ell_{j}]\leq\omega(\delta_{J},\mathcal{F}_{j},\mathcal{F}_{J})+\omega(\delta_{J},\mathcal{F}_{J},\mathcal{F}_{j})=:M_{j}.
    \]
    Its standard deviation, $r_{j}$, is at most the sum of the standard
    deviations of $u_{j}$ and $\ell_{j}$, which are
    $\omega'(\delta_{J},\mathcal{F}_{j},\mathcal{F}_{J})$ and
    $\omega'(\delta_{J},\mathcal{F}_{J},\mathcal{F}_{j})$, since the
    constraints bind. By the concavity of the ordered moduli and
    $\omega(0,\cdot,\cdot)\geq0$, we have
    $\omega'(\delta_{J},\cdot,\cdot)\leq\omega(\delta_{J},\cdot,\cdot)/\delta_{J}$,
    so that $r_{j}\leq M_{j}/\delta_{J}$. For $m\in\mathbb{R}$ and $r\geq0$,
    define $\Psi(m,r)=\E(m+r\xi)_{+}$ with $\xi\sim N(0,1)$, which equals
    $r\phi(m/r)+m\Phi(m/r)$ for $r>0$. Since
    $\partial\Psi/\partial m=\Phi(m/r)\geq0$ and
    $\partial\Psi/\partial r=\phi(m/r)\geq0$, $\Psi$ is nondecreasing in both
    arguments on $\mathbb{R}\times(0,\infty)$. Moreover, $\Psi(m,0)=m_{+}$,
    $\Psi(m,\cdot)$ is continuous on $[0,\infty)$, and
    $\Psi(m,r)\geq(m+r\E\xi)_{+}=m_{+}$ by Jensen's inequality, so that $\Psi$
    is nondecreasing in both arguments on $\mathbb{R}\times[0,\infty)$. (The
    case $r_{j}=0$ occurs, for example, for $j=J$ in Section~5.1, where $u_{J}-\ell_{J}$ is nonrandom.) It follows
    that
    \begin{align*}
      \E_{f}\mu(CI_{\alpha}^{Bon,J}) & \leq\Psi(m_{j,f},r_{j})\leq\Psi(M_{j},M_{j}/\delta_{J})=M_{j}\left\{ \Phi(\delta_{J})+\frac{\phi(\delta_{J})}{\delta_{J}}\right\} \\
       & \leq2\left\{ \Phi(\delta_{J})+\frac{\phi(\delta_{J})}{\delta_{J}}\right\} \omega_{+}(\delta_{J},\mathcal{F}_{j},\mathcal{F}_{J}),
    \end{align*}
    for any $f\in\mathcal{F}_{j}$. Finally, we note that, for $0<\delta\leq\delta'$,
    \begin{equation}
      \omega_{+}(\delta',\mathcal{F}_{j},\mathcal{F}_{J})\leq\frac{\delta'}{\delta}\omega_{+}(\delta,\mathcal{F}_{j},\mathcal{F}_{J}).\label{eq:contraction}
    \end{equation}
    To see this, take any $(f_{j},f_{J})\in\mathcal{F}_{j}\times\mathcal{F}_{J}$
    with $\lVert f_{J}-f_{j}\rVert_{\sigma}\leq\delta'$. Then,
    $\widetilde{f}_{J}=(1-\delta/\delta')f_{j}+(\delta/\delta')f_{J}\in\mathcal{F}_{J}$
    by the convexity of $\mathcal{F}_{J}$ and
    $\mathcal{F}_{j}\subseteq\mathcal{F}_{J}$,
    $\lVert\widetilde{f}_{J}-f_{j}\rVert_{\sigma}\leq\delta$, and
    $\lvert L\widetilde{f}_{J}-Lf_{j}\rvert=(\delta/\delta')\lvert Lf_{J}-Lf_{j}\rvert$.
    Taking the supremum gives \eqref{eq:contraction}. Applying
    \eqref{eq:contraction} with $\delta=z_{1-\alpha/2}$ and $\delta'=\delta_{J}$,
    we obtain the desired result.

    Lastly, we discuss the case where one of the derivatives is zero, say
    $\omega'(\delta_{J},\mathcal{F}_{J},\mathcal{F}_{j})=0$, as mentioned in
    the footnote to the theorem. Let $(f_{J}^{*},g_{j}^{*})$ attain
    $\omega(\delta_{J},\mathcal{F}_{J},\mathcal{F}_{j})$. By concavity and
    monotonicity, $\omega(t,\mathcal{F}_{J},\mathcal{F}_{j})\leq\omega(\delta_{J},\mathcal{F}_{J},\mathcal{F}_{j})$
    for all $t\geq0$, so that $Lg_{j}^{*}-Lf\leq Lg_{j}^{*}-Lf_{J}^{*}$ and
    $Lg-Lf_{J}^{*}\leq Lg_{j}^{*}-Lf_{J}^{*}$ for all $f\in\mathcal{F}_{J}$
    and $g\in\mathcal{F}_{j}$. Hence, $Lf\geq Lf_{J}^{*}$ for all
    $f\in\mathcal{F}_{J}$, and the lower end of the CI reduces to the
    nonrandom value $Lf_{J}^{*}$, which has coverage one over
    $\mathcal{F}_{J}$ and excess length at most
    $\omega(\delta_{J},\mathcal{F}_{J},\mathcal{F}_{j})$ over $\mathcal{F}_{j}$.
    The argument above goes through with this nonrandom end.
\end{proof}

\subsection{Proof of Lemma~3.3}
\label{sec:proof-lemma}
\begin{proof}
  Fix $f\in\mathcal{F}_{J}$ and write $\delta=\delta_{\tau}$. Since the
  constraints bind, the random part of $\hat{L}_{\delta}^{\ell,j}$ is
  $\hat{L}_{\delta}^{\ell,j}-\E_{f}\hat{L}_{\delta}^{\ell,j}=(s_{j}^{-}/\delta)\,d_{j}^{-}(\tau)'\varepsilon=s_{j}^{-}V_{j}(\tau)$.
  Since $z_{1-\tau}=\delta$, the definition (5) gives
  $\hat{c}_{\tau}^{\ell,j}=\hat{L}_{\delta}^{\ell,j}-\frac{1}{2}\omega(\delta,\mathcal{F}_{J},\mathcal{F}_{j})-\frac{1}{2}\delta s_{j}^{-}$,
  and thus, by the definition of $a_{j}(f)$,
  \begin{align*}
    \hat{c}_{\tau}^{\ell,j}-Lf & =s_{j}^{-}V_{j}(\tau)+(\E_{f}\hat{L}_{\delta}^{\ell,j}-Lf)-\frac{1}{2}\omega(\delta,\mathcal{F}_{J},\mathcal{F}_{j})-\frac{1}{2}\delta s_{j}^{-}\\
     & =s_{j}^{-}\{V_{j}(\tau)-z_{1-\tau}-a_{j}(f)\}.
  \end{align*}
  Likewise, the random part of $\hat{L}_{\delta}^{u,j}$ is
  $(s_{j}^{+}/\delta)\,d_{j}^{+}(\tau)'\varepsilon=-s_{j}^{+}W_{j}(\tau)$, and
  Corollary~\ref{cor:cor-AK-ub} gives
  $\hat{c}_{\tau}^{u,j}=\hat{L}_{\delta}^{u,j}+\frac{1}{2}\omega(\delta,\mathcal{F}_{j},\mathcal{F}_{J})+\frac{1}{2}\delta s_{j}^{+}$,
  so that
  \begin{equation*}
    Lf-\hat{c}_{\tau}^{u,j}=s_{j}^{+}\{W_{j}(\tau)-z_{1-\tau}-c_{j}(f)\}.
  \end{equation*}
  By Lemma A.1 of \cite{armstrong2018optimal}, the first terms in the
  brackets in the definitions of $a_{j}(f)$ and $c_{j}(f)$ are the maximum
  biases over $\mathcal{F}_{J}$ in the respective directions, so that
  $a_{j}(f),c_{j}(f)\geq0$. Since $s_{j}^{-},s_{j}^{+}>0$, the lower end that
  directs power to $\mathcal{F}_{j}$ fails to cover $Lf$ if and only if
  $V_{j}(\tau)>z_{1-\tau}+a_{j}(f)$, and the upper end fails if and only if
  $W_{j}(\tau)>z_{1-\tau}+c_{j}(f)$. If
  $\hat{c}_{\tau}^{\ell,j}\leq Lf\leq\hat{c}_{\tau}^{u,j}$ for all $j$, then
  the intersection of the CIs is nonempty and contains $Lf$.

  Now let $t=\min_{j\leq J}a_{j}(f)$. Then, $a_{j}(f)\geq t$ for all $j$, and,
  by (11), $c_{\ell}(f)\geq K_{\tau}-t$ for all $\ell$. Therefore,
  \[
    \p_{f}(Lf\notin CI_{\tau}^{\mathcal{J}})\leq\p\Big(\max_{j\leq J}V_{j}(\tau)>z_{1-\tau}+t\ \text{ or }\ \max_{\ell\leq J}W_{\ell}(\tau)>z_{1-\tau}+\max\{K_{\tau}-t,0\}\Big).
  \]
  If $t\leq K_{\tau}$, the right-hand side is at most $p_{\tau}(K_{\tau})$ by
  definition. If $t>K_{\tau}$, it is at most the probability in
  (12) with $K=K_{\tau}$ and $t=K_{\tau}$, and hence again at most
  $p_{\tau}(K_{\tau})$. Hence,
  $\p_{f}(Lf\notin CI_{\tau}^{\mathcal{J}})\leq p_{\tau}(K_{\tau})\leq\alpha$
  by (13). Since $f\in\mathcal{F}_{J}$ was arbitrary, this
  proves the lemma.

  We now verify the claims about $p_{\tau}$ made in Section~3.2.
  First, at $\tau=\alpha/(2J)$, the union bound gives
  $p_{\tau}(0)=\p(\max\{V(\tau)',W(\tau)'\}>z_{1-\tau})\leq2J\{1-\Phi(z_{1-\tau})\}=\alpha$.
  Second, $p_{\tau}(K)$ is nonincreasing in $K$. To see this, let $K'>K$ and
  $t\in[0,K']$, and let $t'=\min\{t,K\}$. If $t\leq K$, then $t'=t$ and
  $K'-t>K-t'$. If $t>K$, then $t'=K<t$ and $K'-t\geq0=K-t'$. In either case,
  the event in (12) with $(K',t)$ is contained in the event with
  $(K,t')$, so that $p_{\tau}(K')\leq p_{\tau}(K)$.

  Finally, consider the case $J=1$, and suppose the same solution of
  $\omega(\delta,\mathcal{F}_{1},\mathcal{F}_{1})$ is used for both ends. Then,
  $\hat{L}_{\delta}^{\ell,1}=\hat{L}_{\delta}^{u,1}=:\hat{L}_{\delta}$,
  $s_{1}^{-}=s_{1}^{+}=s_{\delta}$ and $d_{1}^{-}(\tau)=d_{1}^{+}(\tau)$, so that
  $W_{1}(\tau)=-V_{1}(\tau)$, and the bias of $\hat{L}_{\delta}$ cancels in
  $a_{1}(f)+c_{1}(f)$, which equals $2B_{\delta}/s_{\delta}$ for every $f$, where
  $B_{\delta}:=\frac{1}{2}\{\omega(\delta,\mathcal{F}_{1},\mathcal{F}_{1})-\delta s_{\delta}\}$
  is the maximum bias of $\hat{L}_{\delta}$ over $\mathcal{F}_{1}$. Hence,
  (11) holds with equality for every $f$ with
  $K_{\tau}=2B_{\delta}/s_{\delta}$. With $K=2B_{\delta}/s_{\delta}$
  and $z=z_{1-\tau}$, the events $\{V_{1}(\tau)>z+t\}$ and
  $\{-V_{1}(\tau)>z+K-t\}$ are disjoint, so that the probability in
  (12) equals $\Phi(-z-t)+\Phi(-z-K+t)$. This is a convex function
  of $t$ on $[0,K]$, since $z>0$, so that its maximum is attained at $t=0$ or
  $t=K$, and $p_{\tau}(K)=\Phi(-z)+\Phi(-z-K)$. On the other hand, the bias
  $b_{f}$ of $\hat{L}_{\delta}$ at $f\in\mathcal{F}_{1}$ lies in
  $[-B_{\delta},B_{\delta}]$, again by Lemma A.1 of
  \cite{armstrong2018optimal}, so that
  $\hat{L}_{\delta}\pm(B_{\delta}+zs_{\delta})$ fails to cover $Lf$ with
  probability $\p(\lvert N(b_{f}/s_{\delta},1)\rvert>B_{\delta}/s_{\delta}+z)$,
  which is largest when $\lvert b_{f}\rvert=B_{\delta}$, where it equals
  $\Phi(-z)+\Phi(-z-2B_{\delta}/s_{\delta})$. Hence, when $\tau^{\ast}$ solves
  $p_{\tau}(K_{\tau})=\alpha$, $CI_{\tau^{\ast}}^{\{1\}}$ is the fixed-length
  CI of \citet{Donoho1994} based on $\hat{L}_{\delta}$ with
  $\delta=\delta_{\tau^{\ast}}$, whose half-length is $s_{\delta}$ times the
  $1-\alpha$ quantile of $\lvert N(B_{\delta}/s_{\delta},1)\rvert$, whereas
  $K_{\tau}=0$ gives $\tau^{\ast}=\alpha/2$.
\end{proof}

\subsection{A Computable Bound on \texorpdfstring{$p_{\tau}(K)$}{p(K)}}
The following lemma gives the upper bound on $p_{\tau}(K_{\tau})$ that we use
in Sections~6 and 7. It requires
only univariate and bivariate normal probabilities.
\begin{lemma}\label{lem:hunter-bound}
  Let $z=z_{1-\tau}$ and $K\geq0$, let
  $0=t_{0}\leq t_{1}\leq\cdots\leq t_{m}=K$ with $m\geq1$, and let $G=(V(\tau)',W(\tau)')'$. For $u,v\in\mathbb{R}$, let $A_{q}(u,v)$, for
  $q=1,\dots,2J$, be the event that the $q$th element of $G$ exceeds $u$ if
  $q\leq J$ and $v$ if $q>J$, and let
  \[
    H(u,v)=\sum_{q=1}^{2J}\p(A_{q}(u,v))-\max_{T}\sum_{(q,q')\in T}\p(A_{q}(u,v)\cap A_{q'}(u,v)),
  \]
  where the maximum is over the spanning trees $T$ of the complete graph on
  $\{1,\dots,2J\}$. Then,
  $p_{\tau}(K)\leq\max_{1\leq i\leq m}H(z+t_{i-1},z+K-t_{i})$. Moreover,
  $H(u,v)\leq J\{1-\Phi(u)\}+J\{1-\Phi(v)\}$.
\end{lemma}
\begin{proof}
  For $t\in[t_{i-1},t_{i}]$, the event in (12) is contained in
  $\cup_{q}A_{q}(z+t_{i-1},z+K-t_{i})$, since
  $\{\max_{j}V_{j}(\tau)>z+t\}\subseteq\{\max_{j}V_{j}(\tau)>z+t_{i-1}\}$ and
  $\{\max_{\ell}W_{\ell}(\tau)>z+K-t\}\subseteq\{\max_{\ell}W_{\ell}(\tau)>z+K-t_{i}\}$.
  For any events $A_{1},\dots,A_{2J}$ and any spanning tree $T$,
  $\p(\cup_{q}A_{q})\leq\sum_{q}\p(A_{q})-\sum_{(q,q')\in T}\p(A_{q}\cap A_{q'})$
  \citep{hunter1976upper}, so that $\p(\cup_{q}A_{q}(u,v))\leq H(u,v)$. The last
  claim holds since the maximum in the definition of $H$ is nonnegative and
  each element of $G$ has unit variance.
\end{proof}
The maximizing tree can be found by a greedy algorithm, such as Prim's or
Kruskal's, and the bivariate
normal probabilities can be computed to machine precision with the method of
\cite{genz2004numerical}. At $\tau=\alpha/(2J)$, the bound with $K=K_{\tau}$
is at most $2J\{1-\Phi(z_{1-\tau})\}=\alpha$ because every threshold is at
least $z_{1-\tau}$. Hence, the bound can be used in place of
$p_{\tau}(K_{\tau})$ in (13), as described after
Lemma~3.3.

\subsection{Proof of Theorem~3.2}
\begin{proof}
  The coverage follows from Lemma~3.3, since $\tau^{\ast}$
  does not depend on $\{y_{i}\}_{i=1}^{n}$ and satisfies
  (13). For the expected length, the same argument as in the
  proof of Theorem~3.1, with $\delta_{J}$ replaced by
  $\delta^{\ast}$, gives
  \[
    \sup_{f\in\mathcal{F}_{j}}\E_{f}\mu(CI_{\tau^{*}}^{\mathcal{J}})\leq2g(\delta^{\ast})\,\omega_{+}(\delta^{\ast},\mathcal{F}_{j},\mathcal{F}_{J}),\quad\text{where }g(\delta)=\Phi(\delta)+\frac{\phi(\delta)}{\delta},
  \]
  which is the first inequality. For the second inequality, note that
  $g'(\delta)=-\phi(\delta)/\delta^{2}<0$ and that
  $z_{1-\alpha}\leq\delta^{\ast}\leq\delta_{J}$, since
  $\alpha/(2J)\leq\tau^{\ast}\leq\alpha$. Hence,
  $g(\delta^{\ast})\leq g(z_{1-\alpha})$. Moreover, since
  $\omega_{+}(\cdot,\mathcal{F}_{j},\mathcal{F}_{J})$ is nondecreasing,
  \eqref{eq:contraction} gives
  \[
    \omega_{+}(\delta^{\ast},\mathcal{F}_{j},\mathcal{F}_{J})\leq\max\left\{ 1,\frac{\delta^{\ast}}{z_{1-\alpha/2}}\right\} \omega_{+}(z_{1-\alpha/2},\mathcal{F}_{j},\mathcal{F}_{J})\leq\frac{\delta_{J}}{z_{1-\alpha/2}}\omega_{+}(z_{1-\alpha/2},\mathcal{F}_{j},\mathcal{F}_{J}),
  \]
  where the last inequality uses $\delta_{J}\geq z_{1-\alpha/2}$. Combining the
  two inequalities gives the second inequality in
  (14).
\end{proof}

\subsection{Proof of Lemma~4.1}
\begin{proof}
  First, we note that if $z,w\in\mathbb{R}^{k}$ satisfy
  $\lvert z_{v}\rvert\leq\lvert w_{v}\rvert$ for all $v$, then
  $\lVert z\rVert\leq\lVert w\rVert$, by applying Assumption~2.1
  to one coordinate at a time. Now, write $T_{+}x=(x)_{\mathcal{V}+}$ and
  $T_{-}x=(x)_{\mathcal{V}-}$. Since $\lvert\max\{s,0\}-\max\{t,0\}\rvert\leq\lvert s-t\rvert$
  for $s,t\in\mathbb{R}$, each element of $T_{\pm}x-T_{\pm}z$ is bounded in
  magnitude by the corresponding element of $x-z$, so that
  $\lVert T_{\pm}x-T_{\pm}z\rVert\leq\lVert x-z\rVert$. Hence, by the triangle
  inequality and the fact that $\lvert s^{\gamma}-t^{\gamma}\rvert\leq\lvert s-t\rvert^{\gamma}$
  for $s,t\geq0$ and $\gamma\in(0,1]$,
  \[
    \lvert h_{\pm}(x)-h_{\pm}(z)\rvert\leq C\lvert\lVert T_{\pm}x\rVert-\lVert T_{\pm}z\rVert\rvert^{\gamma}\leq C\lVert T_{\pm}x-T_{\pm}z\rVert^{\gamma}\leq C\lVert x-z\rVert^{\gamma},
  \]
  which shows the H\"older continuity of $h_{+}$ and $h_{-}$.

  Lastly, for monotonicity, note that for any $x,z\in\mathcal{D}$ such that
  $x_{v}\leq z_{v}$ for all $v\in\mathcal{V}$ and $x_{v}=z_{v}$ for all
  $v\notin\mathcal{V}$, we have
  $\lvert(T_{+}x)_{v}\rvert\leq\lvert(T_{+}z)_{v}\rvert$ and
  $\lvert(T_{-}x)_{v}\rvert\geq\lvert(T_{-}z)_{v}\rvert$ for all $v$.
  Therefore, we have $h_{+}(z)\geq h_{+}(x)$ and
  $h_{-}(z)=-C\lVert T_{-}z\rVert^{\gamma}\geq-C\lVert T_{-}x\rVert^{\gamma}=h_{-}(x)$.
\end{proof}

\subsection{Proof of Lemma~4.2}
\begin{proof}
  Let $h=\max\{h_{1},h_{2}\}$. First of all, monotonicity easily follows from
  the monotonicity of each $h_{1}$ and $h_{2}$. For the H\"older continuity,
  note that $\lvert\max\{a,b\}-\max\{c,d\}\rvert\leq\max\{\lvert a-c\rvert,\lvert b-d\rvert\}$
  for any real numbers $a,b,c,d$. Hence, for any $x,z\in\mathcal{D}$,
  \[
    \lvert h(x)-h(z)\rvert\leq\max_{m=1,2}\lvert h_{m}(x)-h_{m}(z)\rvert\leq C\lVert x-z\rVert^{\gamma}.
  \]
  The same argument applies to $\min\{h_{1},h_{2}\}$, using the
  corresponding inequality for the minimum.
\end{proof}

\subsection{Proof of Lemma~4.3}
\begin{proof}
  We only prove the claim about the maximum, since the proof for the minimum is
  analogous. First, note that due to Lemma~4.1,
  $f^{*}(x)=f_{0}+C\lVert (x)_{\mathcal{V}+}\rVert^{\gamma}$ is in
  $\Lambda_{+,\mathcal{V}}^{f_{0}}(\gamma,C)$. Now, fix
  $x\in\mathcal{D}$ and take any
  $f\in\Lambda_{+,\mathcal{V}}^{f_{0}}(\gamma,C)$. Let
  $z=(x)_{\mathcal{V}+}$, which is in $\mathcal{D}$ by (2).
  Then, we have $z_{v}\geq x_{v}$ for all $v\in\mathcal{V}$ and $z_{v}=x_{v}$
  for all $v\notin\mathcal{V}$, so we must have $f(z)\geq f(x)$. Hence, by
  H\"older continuity,
  \[
    f(x)\leq f(z)\leq f(0)+C\lVert z\rVert^{\gamma}=f_{0}+C\lVert (x)_{\mathcal{V}+}\rVert^{\gamma}=f^{*}(x).
  \]
  Therefore, $f^{*}(x)$ attains the maximum. For the minimum, take
  $z=-(x)_{\mathcal{V}-}\in\mathcal{D}$ instead, and use the function
  $f_{0}-C\lVert (x)_{\mathcal{V}-}\rVert^{\gamma}$.
\end{proof}

\subsection{Proof of Corollary~4.1}
\label{sec:proof-InvModRes21}
\begin{proof}
  The proof is analogous to that of Proposition~4.1. It is
  without loss of generality to restrict attention to
  $(f_{2},f_{1})\in\Lambda_{+,\mathcal{V}}(\gamma_{2},C_{2})\times\Lambda_{+,\mathcal{V}}(\gamma_{1},C_{1})$
  with $f_{2}(0)=0$ and $f_{1}(0)=b$. For any such pair, Lemma~4.3 gives
  $f_{2}(x_{i})\leq C_{2}\lVert(x_{i})_{\mathcal{V}+}\rVert^{\gamma_{2}}$ and
  $f_{1}(x_{i})\geq b-C_{1}\lVert(x_{i})_{\mathcal{V}-}\rVert^{\gamma_{1}}$,
  so that
  \[
    (f_{1}(x_{i})-f_{2}(x_{i}))^{2}\geq(b-C_{2}\lVert(x_{i})_{\mathcal{V}+}\rVert^{\gamma_{2}}-C_{1}\lVert(x_{i})_{\mathcal{V}-}\rVert^{\gamma_{1}})_{+}^{2}.
  \]
  On the other hand, $g_{1}^{*}\in\Lambda_{+,\mathcal{V}}(\gamma_{1},C_{1})$
  by Lemma~4.1, and since
  $g_{1}^{*}\in\Lambda_{+,\mathcal{V}}(\gamma_{1},C_{1})\subseteq\Lambda_{+,\mathcal{V}}(\gamma_{2},C_{2})$
  and
  $C_{2}\lVert(\cdot)_{\mathcal{V}+}\rVert^{\gamma_{2}}\in\Lambda_{+,\mathcal{V}}(\gamma_{2},C_{2})$,
  Lemma~4.2 gives
  $g_{2}^{*}\in\Lambda_{+,\mathcal{V}}(\gamma_{2},C_{2})$. Moreover,
  $g_{2}^{*}(0)=0$, $g_{1}^{*}(0)=b$, and
  $g_{1}^{*}(x)-g_{2}^{*}(x)=(b-C_{2}\lVert(x)_{\mathcal{V}+}\rVert^{\gamma_{2}}-C_{1}\lVert(x)_{\mathcal{V}-}\rVert^{\gamma_{1}})_{+}$,
  so that $(g_{2}^{*},g_{1}^{*})$ attains the lower bound.
\end{proof}

\subsection{Envelopes of \texorpdfstring{$\mathcal{P}_{\mathcal{V}}(\gamma,C)$ and $\mathcal{G}_{\mathcal{V}}(\gamma,C)$}{P and G}}
The following lemma gives the analogue of Lemma~4.3 for the
parameter spaces of Section~4.3.
\begin{lemma}
  \label{lem:G-envelope}Suppose Assumption~2.1 and
  (2) hold, and let $\gamma\in(0,1]$, $C>0$,
  $f_{0}\in\mathbb{R}$ and $x\in\mathcal{D}$. (i) Over
  $f\in\mathcal{P}_{\mathcal{V}}(\gamma,C)$ with $f(0)=f_{0}$, the maximum and
  the minimum of $f(x)$ are $f_{0}+C\lVert(x)_{\mathcal{V}+}\rVert^{\gamma}$
  and $f_{0}-C\lVert(x)_{\mathcal{V}-}\rVert^{\gamma}$, as in
  Lemma~4.3. (ii) Over
  $f\in\mathcal{G}_{\mathcal{V}}(\gamma,C)$ with $f(0)=f_{0}$, the maximum of
  $f(x)$ is $f_{0}+C\lVert(x)_{\mathcal{V}+}\rVert^{\gamma}$ if
  $x\in\mathcal{D}_{\mathcal{V}}^{-}$, and the supremum is $\infty$ otherwise;
  the minimum of $f(x)$ is $f_{0}-C\lVert(x)_{\mathcal{V}-}\rVert^{\gamma}$ if
  $x\in\mathcal{D}_{\mathcal{V}}^{+}$, and the infimum is $-\infty$ otherwise.
\end{lemma}
Part (i) shows that $\mathcal{P}_{\mathcal{V}}(\gamma,C)$ has the same bounds
as $\Lambda_{+,\mathcal{V}}(\gamma,C)$, which are attained by functions in
$\Lambda_{+,\mathcal{V}}(\gamma,C)$. Part (ii) shows that a function in
$\mathcal{G}_{\mathcal{V}}(\gamma,C)$ is not restricted from above outside
$\mathcal{D}_{\mathcal{V}}^{-}$, or from below outside
$\mathcal{D}_{\mathcal{V}}^{+}$. Hence, the solution to the inverse modulus problem uses
only the design points in one of the two sets when the larger parameter space
is $\mathcal{G}_{\mathcal{V}}(\gamma_{2},C_{2})$.
\begin{proof}
  (i) Take $f\in\mathcal{P}_{\mathcal{V}}(\gamma,C)$ with $f(0)=f_{0}$, and
  let $z=(x)_{\mathcal{V}+}$, which is in $\mathcal{D}$ by
  (2). As in the proof of Lemma~4.3,
  monotonicity gives $f(x)\leq f(z)$, and the restriction in the definition
  of $\mathcal{P}_{\mathcal{V}}(\gamma,C)$ gives
  $f(z)\leq f_{0}+C\lVert z\rVert^{\gamma}$. Likewise,
  $f(x)\geq f(-(x)_{\mathcal{V}-})\geq f_{0}-C\lVert(x)_{\mathcal{V}-}\rVert^{\gamma}$.
  Both bounds are attained by the functions in the proof of
  Lemma~4.3, which are in
  $\Lambda_{+,\mathcal{V}}(\gamma,C)\subseteq\mathcal{P}_{\mathcal{V}}(\gamma,C)$.

  (ii) Take $f\in\mathcal{G}_{\mathcal{V}}(\gamma,C)$ with $f(0)=f_{0}$, and
  suppose $x\in\mathcal{D}_{\mathcal{V}}^{-}$. Then, $z=(x)_{\mathcal{V}+}$ is
  obtained from $x$ by setting the coordinates in $\mathcal{V}$ to $0$, so
  that it lies on the slice
  $\{z\in\mathcal{D}:z_{v}=0\ \forall v\in\mathcal{V}\}$, and it is in
  $\mathcal{D}$ by (2). Hence, the argument in part (i)
  applies, since it uses the restriction on $\lvert f(z)-f(0)\rvert$ only at
  this $z$, and the bound is attained by
  $f_{0}+C\lVert(\cdot)_{\mathcal{V}+}\rVert^{\gamma}\in\Lambda_{+,\mathcal{V}}(\gamma,C)\subseteq\mathcal{G}_{\mathcal{V}}(\gamma,C)$.
  If $x\notin\mathcal{D}_{\mathcal{V}}^{-}$, then $x_{v}>0$ for some
  $v\in\mathcal{V}$. For any $M>0$, the function
  $z\mapsto f_{0}+M\mathds{1}\{z_{v}>0\}$ is nondecreasing in each coordinate
  and equals $f_{0}$ on the slice, so that it is in
  $\mathcal{G}_{\mathcal{V}}(\gamma,C)$, and its value at $x$ is $f_{0}+M$.
  Hence, the supremum is $\infty$. The claims about the minimum follow in the
  same way, with $z=-(x)_{\mathcal{V}-}$ for
  $x\in\mathcal{D}_{\mathcal{V}}^{+}$, and with the functions
  $f_{0}-C\lVert(\cdot)_{\mathcal{V}-}\rVert^{\gamma}$ and
  $z\mapsto f_{0}-M\mathds{1}\{z_{v}<0\}$.
\end{proof}

\subsection{Inverse Moduli Between \texorpdfstring{$\Lambda_{+,\mathcal{V}}(\gamma_{1},C_{1})$ and $\mathcal{G}_{\mathcal{V}}(\gamma_{2},C_{2})$}{Lambda and G}}
\begin{prop}
  \label{prop:G-InvMod}Suppose Assumption~2.1 and
  (2) hold, and that
  $\Lambda_{+,\mathcal{V}}(\gamma_{1},C_{1})\subseteq\Lambda_{+,\mathcal{V}}(\gamma_{2},C_{2})$.
  Write $\mathcal{F}_{1}=\Lambda_{+,\mathcal{V}}(\gamma_{1},C_{1})$ and
  $\mathcal{G}=\mathcal{G}_{\mathcal{V}}(\gamma_{2},C_{2})$. Then, for
  $b\geq0$, the inverse modulus problems that correspond to
  $\omega^{-1}(b,\mathcal{F}_{1},\mathcal{G})$ and
  $\omega^{-1}(b,\mathcal{G},\mathcal{F}_{1})$ have solutions, and
  \begin{align*}
    \omega^{-1}(b,\mathcal{F}_{1},\mathcal{G}) & =\Big(\sum\nolimits_{i:x_{i}\in\mathcal{D}_{\mathcal{V}}^{+}}\big((b-C_{1}\lVert(x_{i})_{\mathcal{V}+}\rVert^{\gamma_{1}}-C_{2}\lVert(x_{i})_{\mathcal{V}-}\rVert^{\gamma_{2}})_{+}/\sigma(x_{i})\big)^{2}\Big)^{1/2},\\
    \omega^{-1}(b,\mathcal{G},\mathcal{F}_{1}) & =\Big(\sum\nolimits_{i:x_{i}\in\mathcal{D}_{\mathcal{V}}^{-}}\big((b-C_{2}\lVert(x_{i})_{\mathcal{V}+}\rVert^{\gamma_{2}}-C_{1}\lVert(x_{i})_{\mathcal{V}-}\rVert^{\gamma_{1}})_{+}/\sigma(x_{i})\big)^{2}\Big)^{1/2}.
  \end{align*}
\end{prop}
That is, the inverse moduli are those in Proposition~4.1 and
Corollary~4.1, with the sums restricted to the design
points in $\mathcal{D}_{\mathcal{V}}^{+}$ and $\mathcal{D}_{\mathcal{V}}^{-}$,
respectively. We give the solutions in the proof.
\begin{proof}
  We consider $\omega^{-1}(b,\mathcal{F}_{1},\mathcal{G})$ first. As in the
  proof of Proposition~4.1, we can restrict attention to
  $f_{1}\in\mathcal{F}_{1}$ and $f_{2}\in\mathcal{G}$ with $f_{1}(0)=0$ and
  $f_{2}(0)=b$, since $\mathcal{G}$ is also invariant to adding constants. By
  Lemmas~4.3 and \ref{lem:G-envelope},
  $f_{1}(x_{i})\leq C_{1}\lVert(x_{i})_{\mathcal{V}+}\rVert^{\gamma_{1}}$ for
  all $i$, and
  $f_{2}(x_{i})\geq b-C_{2}\lVert(x_{i})_{\mathcal{V}-}\rVert^{\gamma_{2}}$ if
  $x_{i}\in\mathcal{D}_{\mathcal{V}}^{+}$. Hence, the objective is bounded from
  below by the square of the right-hand side of the first display. Now, let
  $f_{1}^{\ast}(x)=C_{1}\lVert(x)_{\mathcal{V}+}\rVert^{\gamma_{1}}$, and let
  $f_{2}^{\ast}(x)=\max\{b-C_{2}\lVert(x)_{\mathcal{V}-}\rVert^{\gamma_{2}},f_{1}^{\ast}(x)\}$
  if $x\in\mathcal{D}_{\mathcal{V}}^{+}$ and $f_{2}^{\ast}(x)=f_{1}^{\ast}(x)$
  otherwise. Then, $f_{1}^{\ast}(0)=0$, $f_{2}^{\ast}(0)=b$, and
  $f_{2}^{\ast}(x_{i})-f_{1}^{\ast}(x_{i})$ equals
  $(b-C_{1}\lVert(x_{i})_{\mathcal{V}+}\rVert^{\gamma_{1}}-C_{2}\lVert(x_{i})_{\mathcal{V}-}\rVert^{\gamma_{2}})_{+}$
  if $x_{i}\in\mathcal{D}_{\mathcal{V}}^{+}$ and $0$ otherwise, so that
  $(f_{1}^{\ast},f_{2}^{\ast})$ attains the lower bound, provided that
  $f_{1}^{\ast}\in\mathcal{F}_{1}$ and $f_{2}^{\ast}\in\mathcal{G}$. The
  former holds by Lemma~4.1.

  We now show that $f_{2}^{\ast}\in\mathcal{G}$. Let $x,z\in\mathcal{D}$ be
  such that $x_{v}\geq z_{v}$ for all $v\in\mathcal{V}$ and $x_{v}=z_{v}$ for
  all $v\notin\mathcal{V}$. If $x,z\in\mathcal{D}_{\mathcal{V}}^{+}$, then
  $f_{2}^{\ast}(x)\geq f_{2}^{\ast}(z)$, since both functions in the maximum
  are nondecreasing by Lemma~4.1. If neither is in
  $\mathcal{D}_{\mathcal{V}}^{+}$, this follows from the monotonicity of
  $f_{1}^{\ast}$. If $x\in\mathcal{D}_{\mathcal{V}}^{+}$ and
  $z\notin\mathcal{D}_{\mathcal{V}}^{+}$, then
  $f_{2}^{\ast}(x)\geq f_{1}^{\ast}(x)\geq f_{1}^{\ast}(z)=f_{2}^{\ast}(z)$.
  Finally, $z\in\mathcal{D}_{\mathcal{V}}^{+}$ implies
  $x\in\mathcal{D}_{\mathcal{V}}^{+}$. Hence,
  $f_{2}^{\ast}\in\mathcal{M}_{\mathcal{V}}$. Next, let $x$ be on the slice,
  so that $x\in\mathcal{D}_{\mathcal{V}}^{+}$, $(x)_{\mathcal{V}+}=x$ and
  $(x)_{\mathcal{V}-}=-x$. Then,
  \[
    f_{2}^{\ast}(x)-f_{2}^{\ast}(0)=\max\{-C_{2}\lVert x\rVert^{\gamma_{2}},C_{1}\lVert x\rVert^{\gamma_{1}}-b\},
  \]
  which is at least $-C_{2}\lVert x\rVert^{\gamma_{2}}$ and at most
  $C_{1}\lVert x\rVert^{\gamma_{1}}=f_{1}^{\ast}(x)-f_{1}^{\ast}(0)\leq C_{2}\lVert x\rVert^{\gamma_{2}}$,
  where the last inequality holds since
  $f_{1}^{\ast}\in\mathcal{F}_{1}\subseteq\Lambda_{+,\mathcal{V}}(\gamma_{2},C_{2})$.
  Hence, $f_{2}^{\ast}\in\mathcal{G}$.

  The proof for $\omega^{-1}(b,\mathcal{G},\mathcal{F}_{1})$ is analogous,
  with Corollary~4.1 in place of
  Proposition~4.1. The lower bound follows from
  $f_{2}(x_{i})\leq C_{2}\lVert(x_{i})_{\mathcal{V}+}\rVert^{\gamma_{2}}$ for
  $f_{2}\in\mathcal{G}$ with $f_{2}(0)=0$ and
  $x_{i}\in\mathcal{D}_{\mathcal{V}}^{-}$, and from
  $f_{1}(x_{i})\geq b-C_{1}\lVert(x_{i})_{\mathcal{V}-}\rVert^{\gamma_{1}}$ for
  $f_{1}\in\mathcal{F}_{1}$ with $f_{1}(0)=b$. It is attained by
  $g_{1}^{\ast}(x)=b-C_{1}\lVert(x)_{\mathcal{V}-}\rVert^{\gamma_{1}}$ and by
  $g_{2}^{\ast}(x)=\min\{g_{1}^{\ast}(x),C_{2}\lVert(x)_{\mathcal{V}+}\rVert^{\gamma_{2}}\}$
  if $x\in\mathcal{D}_{\mathcal{V}}^{-}$ and $g_{2}^{\ast}(x)=g_{1}^{\ast}(x)$
  otherwise. That $g_{2}^{\ast}\in\mathcal{G}$ follows as above, since
  $\mathcal{D}_{\mathcal{V}}^{-}$ is closed under decreasing the coordinates
  in $\mathcal{V}$.
\end{proof}

\subsection{Two Properties of \texorpdfstring{$\mathcal{G}_{\mathcal{V}}(\gamma,C)$}{G}}
\label{sec:G-properties}
Write
$\mathcal{S}_{\mathcal{V}}=\mathcal{D}_{\mathcal{V}}^{+}\cap\mathcal{D}_{\mathcal{V}}^{-}=\{x\in\mathcal{D}:x_{v}=0\ \forall v\in\mathcal{V}\}$
for the slice. If $k_{+}\geq1$, $\mathcal{S}_{\mathcal{V}}$ has Lebesgue
measure zero. Then, under the conditions of
Theorem~4.2, no design point lies on
$\mathcal{S}_{\mathcal{V}}$ for almost all realizations of the design
points. If $1\leq k_{+}<k$, the coordinates of $X_{i}$ outside $\mathcal{V}$
have a joint density, so that no two design points have the same coordinates
outside $\mathcal{V}$ for almost all realizations. Hence, if
$k_{+}\geq1$, the conditions of Lemma~\ref{lem:G-slack} hold for almost
all realizations of the design points.

First, suppose that no design point lies on $\mathcal{S}_{\mathcal{V}}$.
Then,
$\omega(\delta,\mathcal{G}_{\mathcal{V}}(\gamma,C),\mathcal{G}_{\mathcal{V}}(\gamma,C))=\infty$
for all $\delta\geq0$. To see this, fix $b>0$, and let
$f_{1}(x)=b\mathds{1}\{x\notin\mathcal{D}_{\mathcal{V}}^{-}\}$ and
$f_{2}(x)=b\mathds{1}\{x\notin\mathcal{D}_{\mathcal{V}}^{-}\text{ or }x\in\mathcal{S}_{\mathcal{V}}\}$.
The sets $\mathcal{D}\setminus\mathcal{D}_{\mathcal{V}}^{-}$ and
$(\mathcal{D}\setminus\mathcal{D}_{\mathcal{V}}^{-})\cup\mathcal{S}_{\mathcal{V}}$
are closed under increasing the coordinates in $\mathcal{V}$, so that $f_{1}$
and $f_{2}$ are in $\mathcal{M}_{\mathcal{V}}$. Both functions are
constant on $\mathcal{S}_{\mathcal{V}}$, which contains $0$. Hence, both are
in $\mathcal{G}_{\mathcal{V}}(\gamma,C)$. Since $f_{2}(0)-f_{1}(0)=b$ and
$f_{2}(x_{i})=f_{1}(x_{i})$ for all $i$, the ordered modulus is at least $b$
for every $\delta\geq0$.

Second, the following lemma shows that (11) holds only with
$K_{\tau}=0$ when the coverage class is
$\mathcal{G}_{\mathcal{V}}(\gamma_{J},C_{J})$.
\begin{lemma}
  \label{lem:G-slack}Let
  $\mathcal{F}_{j}=\Lambda_{+,\mathcal{V}}(\gamma_{j},C_{j})\subseteq\Lambda_{+,\mathcal{V}}(\gamma_{J},C_{J})$
  for $j\leq J$, and consider the CI of Section~5.1
  that maintains coverage over
  $\mathcal{G}=\mathcal{G}_{\mathcal{V}}(\gamma_{J},C_{J})$. Suppose that, if
  $k_{+}<k$, no design point lies on $\mathcal{S}_{\mathcal{V}}$ and no two
  design points have the same coordinates outside $\mathcal{V}$.
  Then, there exists $f^{\ast}\in\mathcal{G}$ such that
  $a_{j}(f^{\ast})=c_{j}(f^{\ast})=0$ for all $j\leq J$, where $a_{j}$ and
  $c_{j}$ are defined in Section~3.2 with
  $\mathcal{G}$ in place of $\mathcal{F}_{J}$. Hence,
  (11) holds only with $K_{\tau}=0$.
\end{lemma}
\begin{proof}
  The estimator of the lower end that directs power to $\mathcal{F}_{j}$ is
  $\sum_{i}w_{j,i}^{-}y_{i}$ plus a constant, where the weights are
  nonnegative, sum to one, and are positive only if
  $x_{i}\in\mathcal{D}_{\mathcal{V}}^{-}$, as described in
  Section~5.1. Hence, its bias at $f\in\mathcal{G}$ is
  $\sum_{i}w_{j,i}^{-}\{f(x_{i})-f(0)\}$ plus a constant. By
  Lemma~\ref{lem:G-envelope}, this bias is maximized over $\mathcal{G}$ by
  every $f\in\mathcal{G}$ with
  $f(x_{i})-f(0)=C_{J}\lVert(x_{i})_{\mathcal{V}+}\rVert^{\gamma_{J}}$ for all
  $x_{i}\in\mathcal{D}_{\mathcal{V}}^{-}$, such as
  $f(0)+C_{J}\lVert(\cdot)_{\mathcal{V}+}\rVert^{\gamma_{J}}$, so that
  $a_{j}(f)=0$ for every such $f$. Similarly, $c_{j}(f)=0$ for every
  $f\in\mathcal{G}$ with
  $f(x_{i})-f(0)=-C_{J}\lVert(x_{i})_{\mathcal{V}-}\rVert^{\gamma_{J}}$ for all
  $x_{i}\in\mathcal{D}_{\mathcal{V}}^{+}$. It thus suffices to find
  $f^{\ast}\in\mathcal{G}$ with these values at the design points. If
  $k_{+}=k$, the function $f^{\ast}=0$ has them for any design, since
  $(x)_{\mathcal{V}+}=0$ on $\mathcal{D}_{\mathcal{V}}^{-}$ and
  $(x)_{\mathcal{V}-}=0$ on $\mathcal{D}_{\mathcal{V}}^{+}$. If $k_{+}<k$,
  for $x\in\mathcal{D}$, let $\bar{x}\in\mathbb{R}^{k}$ denote the vector
  obtained from $x$ by setting the coordinates in $\mathcal{V}$ to $0$.
  Define
  $f^{\ast}(x)=C_{J}\lVert\bar{x}\rVert^{\gamma_{J}}$ if
  $\bar{x}=\bar{x}_{i}$ for a design point
  $x_{i}\in\mathcal{D}_{\mathcal{V}}^{-}$,
  $f^{\ast}(x)=-C_{J}\lVert\bar{x}\rVert^{\gamma_{J}}$ if
  $\bar{x}=\bar{x}_{i}$ for a design point
  $x_{i}\in\mathcal{D}_{\mathcal{V}}^{+}$, and $f^{\ast}(x)=0$ otherwise. By
  assumption, $\bar{x}_{i}\neq\bar{x}_{l}$ for $i\neq l$, and no design point
  is in both $\mathcal{D}_{\mathcal{V}}^{-}$ and
  $\mathcal{D}_{\mathcal{V}}^{+}$, so that $f^{\ast}$ is well defined. Since
  $f^{\ast}(x)$ depends on $x$ only through $\bar{x}$, it is in
  $\mathcal{M}_{\mathcal{V}}$. For $x\in\mathcal{S}_{\mathcal{V}}$,
  $\bar{x}=x$, so that
  $\lvert f^{\ast}(x)\rvert\leq C_{J}\lVert x\rVert^{\gamma_{J}}$, and, in
  particular, $f^{\ast}(0)=0$. Hence, $f^{\ast}\in\mathcal{G}$. Finally,
  $(x_{i})_{\mathcal{V}+}=\bar{x}_{i}$ for
  $x_{i}\in\mathcal{D}_{\mathcal{V}}^{-}$ and
  $(x_{i})_{\mathcal{V}-}=-\bar{x}_{i}$ for
  $x_{i}\in\mathcal{D}_{\mathcal{V}}^{+}$, so that $f^{\ast}$ has the
  required values at the design points.
\end{proof}

\subsection{The CI That Maintains Coverage over \texorpdfstring{$\mathcal{G}_{\mathcal{V}}(\gamma_{J},C_{J})$}{G}}
\label{apdx:adaptmono}
The ends of AdaptMono, the CI of Section~5.1 that
maintains coverage over $\mathcal{G}_{\mathcal{V}}(\gamma_{J},C_{J})$, are
computed from the ordered moduli
$\omega(\delta,\mathcal{G}_{\mathcal{V}}(\gamma_{J},C_{J}),\mathcal{F}_{j})$ and
$\omega(\delta,\mathcal{F}_{j},\mathcal{G}_{\mathcal{V}}(\gamma_{J},C_{J}))$,
obtained from Proposition~\ref{prop:G-InvMod}. The proof of
Corollary~5.1 goes through for them because the solutions
in the proof of Proposition~\ref{prop:G-InvMod} coincide with those of
Proposition~4.1 and Corollary~4.1 at $0$
and at the design points that receive positive weight. If
$\mathcal{D}_{\mathcal{V}}^{+}$ and $\mathcal{D}_{\mathcal{V}}^{-}$ each
contain a design point, both ordered moduli are finite, attained, and
differentiable with positive derivatives, and Lemma~3.3 and
Theorems~3.1 and 3.2 apply with
$\mathcal{G}_{\mathcal{V}}(\gamma_{J},C_{J})$ as the coverage class. Under the
conditions of Theorem~4.2, this holds for large $n$, for
almost all realizations of the design points. Their
proofs use the coverage class only through its convexity, the inclusions
$\mathcal{F}_{j}\subseteq\mathcal{G}_{\mathcal{V}}(\gamma_{J},C_{J})$, and the
ordered moduli between $\mathcal{F}_{j}$ and
$\mathcal{G}_{\mathcal{V}}(\gamma_{J},C_{J})$.

\subsection{Proof of Corollary~5.1}
\begin{proof}
  We consider the lower end of the CI first. Write
  $b=\omega(\delta,\mathcal{F}_{J},\mathcal{F}_{j})$, $D_{i}=D_{Jj,\delta}(x_{i})$
  and $w_{i}=w_{j,i}^{-}$. By Corollary~4.1, with
  $(\gamma_{1},C_{1},\gamma_{2},C_{2})=(\gamma_{j},C_{j},\gamma_{J},C_{J})$,
  the functions $f_{J,\delta}^{*}:=g_{2}^{*}$ and $g_{j,\delta}^{*}:=g_{1}^{*}$
  solve the ordered modulus problem, with
  $g_{j,\delta}^{*}(x_{i})-f_{J,\delta}^{*}(x_{i})=D_{i}$, and the inverse
  modulus is $(\sum_{i=1}^{n}(b-a_{i})_{+}^{2})^{1/2}$ with
  \[
    a_{i}=C_{j}\lVert (x_{i})_{\mathcal{V}-}\rVert^{\gamma_{j}}+C_{J}\lVert (x_{i})_{\mathcal{V}+}\rVert^{\gamma_{J}}.
  \]
  As a function of $b$, the squared inverse modulus has the continuous
  derivative $2\sum_{i}(b-a_{i})_{+}$, which is positive whenever the
  inverse modulus is positive. Hence, for $\delta>0$, the ordered modulus is
  differentiable and
  \begin{equation}\label{eq:LemmaB.3_ArmKol}
    \omega'(\delta,\mathcal{F}_{J},\mathcal{F}_{j})      =
    \frac{\delta}{\sum_{i=1}^{n}D_{i}},
  \end{equation}
  since $2\delta\,d\delta=2\sum_{i=1}^{n}D_{i}\,db$ at $b=\omega(\delta,\mathcal{F}_{J},\mathcal{F}_{j})$.
  Since $\sum_{i=1}^{n}D_{i}^{2}=\delta^{2}$, this also equals
  $(\sum_{i}w_{i}^{2})^{1/2}=s_{j}^{-}$, and
  $\sum_{i}w_{i}D_{i}=\delta^{2}/\sum_{i}D_{i}=\delta s_{j}^{-}$. Therefore, we
  can rewrite $\hat{L}_{\delta}^{\ell,j}$ in (4) as
  \begin{equation*}
    \hat{L}_{\delta}^{\ell,j}  =  \frac{g_{j,\delta}^{*}(0)+f_{J,\delta}^{*}(0)}{2}+\sum_{i=1}^{n}w_{i}\left(y_{i}-\frac{g_{j,\delta}^{*}(x_{i})+f_{J,\delta}^{*}(x_{i})}{2}\right).
  \end{equation*}
  Note that $g_{j,\delta}^{*}(0)+f_{J,\delta}^{*}(0)=b$, and that, whenever
  $D_{i}>0$,
  $g_{j,\delta}^{*}(x_{i})+f_{J,\delta}^{*}(x_{i})=2g_{j,\delta}^{*}(x_{i})-D_{i}=b-C_{j}\lVert (x_{i})_{\mathcal{V}-}\rVert^{\gamma_{j}}+C_{J}\lVert (x_{i})_{\mathcal{V}+}\rVert^{\gamma_{J}}=D_{i}+2C_{J}\lVert (x_{i})_{\mathcal{V}+}\rVert^{\gamma_{J}}$.
  Hence,
  \begin{equation}
    \hat{L}_{\delta}^{\ell,j}=\frac{b}{2}+\sum_{i=1}^{n}w_{i}\big(y_{i}-C_{J}\lVert (x_{i})_{\mathcal{V}+}\rVert^{\gamma_{J}}\big)-\frac{1}{2}\delta s_{j}^{-}.\label{eq:Lhat-weights}
  \end{equation}
  To get the lower end of the CI, we subtract from $\hat{L}_{\delta}^{\ell,j}$
  the maximum bias
  $\frac{1}{2}(\omega(\delta,\mathcal{F}_{J},\mathcal{F}_{j})-\delta s_{j}^{-})$
  plus $\delta s_{j}^{-}$, which gives the formula in the statement. The
  upper end of the CI can be derived in an analogous way, this time using
  Proposition~4.1 and Corollary~\ref{cor:cor-AK-ub}: with
  $b=\omega(\delta,\mathcal{F}_{j},\mathcal{F}_{J})$ and
  $D_{i}=D_{jJ,\delta}(x_{i})$, the solution is
  $(f_{j,\delta}^{*},g_{J,\delta}^{*})=(f_{1}^{*},f_{2}^{*})$ with
  $(\gamma_{1},C_{1},\gamma_{2},C_{2})=(\gamma_{j},C_{j},\gamma_{J},C_{J})$,
  and, whenever $D_{i}>0$,
  \[
    f_{j,\delta}^{*}(x_{i})+g_{J,\delta}^{*}(x_{i})=2C_{j}\lVert (x_{i})_{\mathcal{V}+}\rVert^{\gamma_{j}}+D_{i}=2b-2C_{J}\lVert (x_{i})_{\mathcal{V}-}\rVert^{\gamma_{J}}-D_{i},
  \]
  so that
  \[
    \hat{L}_{\delta}^{u,j}=-\frac{b}{2}+\sum_{i=1}^{n}w_{j,i}^{+}\big(y_{i}+C_{J}\lVert (x_{i})_{\mathcal{V}-}\rVert^{\gamma_{J}}\big)+\frac{1}{2}\delta s_{j}^{+},
  \]
  and adding
  $\frac{1}{2}(\omega(\delta,\mathcal{F}_{j},\mathcal{F}_{J})-\delta s_{j}^{+})+\delta s_{j}^{+}$
  gives the formula for $\hat{c}_{[\delta]}^{u,j}$. For general
  $\sigma(\cdot)$, the same derivation goes through with the weights
  proportional to $D_{i}/\sigma^{2}(x_{i})$, with
  $\omega'(\delta,\mathcal{F}_{J},\mathcal{F}_{j})=\delta/\sum_{i=1}^{n}D_{i}/\sigma^{2}(x_{i})$,
  which equals $(\sum_{i}w_{i}^{2}\sigma^{2}(x_{i}))^{1/2}$, and with
  $\sum_{i}w_{i}D_{i}=\delta\,\omega'(\delta,\mathcal{F}_{J},\mathcal{F}_{j})$.
\end{proof}

\subsection{Proof of Lemma~5.1}
\label{sec:proof-slack}
Since $f(x_{i})-f(0)$ lies between $h_{-}(x_{i})$ and $h_{+}(x_{i})$, the $i$th
term of $\delta\{a_{j}(f)+c_{\ell}(f)\}$ in Lemma~5.1 is at least
$\min\{D_{Jj,\delta}(x_{i}),D_{\ell J,\delta}(x_{i})\}\{h_{+}(x_{i})-h_{-}(x_{i})\}$.
Equality holds when $f(x_{i})-f(0)$ is at the envelope that is least favorable
to the end with the larger value of $D$, that is, at $h_{+}(x_{i})$ if
$D_{Jj,\delta}(x_{i})\geq D_{\ell J,\delta}(x_{i})$ and at $h_{-}(x_{i})$
otherwise. This lower bound is not necessarily attained by a function in
$\mathcal{F}_{J}$, so that $K_{\tau}$ can be conservative. For a general known
$\sigma(\cdot)$, the lemma holds with $D_{Jj,\delta}(x_{i})$ and
$D_{\ell J,\delta}(x_{i})$ replaced by $D_{Jj,\delta}(x_{i})/\sigma^{2}(x_{i})$
and $D_{\ell J,\delta}(x_{i})/\sigma^{2}(x_{i})$, with the same proof. The same
argument applies to $f\in\mathcal{P}_{\mathcal{V}}(\gamma_{J},C_{J})$, by
Lemma~\ref{lem:G-envelope}(i), so that $K_{\tau}$ is also valid with
$\mathcal{P}_{\mathcal{V}}(\gamma_{J},C_{J})$ as the coverage class. For
example, if $\sigma(\cdot)=s$ is constant, $K_{\tau}$ equals the value
computed with $\sigma(\cdot)=1$ at $\delta=s\delta_{\tau}$, divided by $s$.
\begin{proof}
  Consider the lower end that directs power to $\mathcal{F}_{j}$, and write
  $w_{i}=w_{j,i}^{-}$ and $D_{i}=D_{Jj,\delta}(x_{i})$. By
  \eqref{eq:Lhat-weights} and $\sum_{i}w_{i}=1$, for $f\in\mathcal{F}_{J}$,
  \[
    \E_{f}\hat{L}_{\delta}^{\ell,j}-f(0)=\sum_{i=1}^{n}w_{i}\{f(x_{i})-f(0)\}-\sum_{i=1}^{n}w_{i}h_{+}(x_{i})+\frac{1}{2}\{\omega(\delta,\mathcal{F}_{J},\mathcal{F}_{j})-\delta s_{j}^{-}\},
  \]
  so that, by the definition of $a_{j}(f)$,
  $s_{j}^{-}a_{j}(f)=\sum_{i}w_{i}\{h_{+}(x_{i})-f(x_{i})+f(0)\}$. Since
  $w_{i}/s_{j}^{-}=(D_{i}/\sum_{i'}D_{i'})/(\delta/\sum_{i'}D_{i'})=D_{i}/\delta$,
  this gives the formula for $a_{j}(f)$. The formula for $c_{\ell}(f)$
  follows in the same way. Write $g_{i}=f(x_{i})-f(0)$. By
  Lemma~4.3, $h_{-}(x_{i})\leq g_{i}\leq h_{+}(x_{i})$ for all
  $i$. For each $i$, the function
  $g\mapsto D_{Jj,\delta}(x_{i})\{h_{+}(x_{i})-g\}+D_{\ell J,\delta}(x_{i})\{g-h_{-}(x_{i})\}$
  is affine on $[h_{-}(x_{i}),h_{+}(x_{i})]$, so that it is at least the
  smaller of its values at the two endpoints, which is
  $\min\{D_{Jj,\delta}(x_{i}),D_{\ell J,\delta}(x_{i})\}\{h_{+}(x_{i})-h_{-}(x_{i})\}$.
  Summing over $i$ and dividing by $\delta$ gives
  $a_{j}(f)+c_{\ell}(f)\geq K_{\tau}$. For a general known $\sigma(\cdot)$, the
  same argument applies with $w_{i}/s_{j}^{-}=D_{i}/\{\sigma^{2}(x_{i})\delta\}$,
  by the remark after Corollary~5.1.
\end{proof}
The largest valid $K_{\tau}$ is
$\min_{j,\ell}\inf_{f\in\mathcal{F}_{J}}\{a_{j}(f)+c_{\ell}(f)\}$, and
$a_{j}(f)+c_{\ell}(f)$ is an affine function of
$f(0),f(x_{1}),\dots,f(x_{n})$. If $\mathcal{D}$ is a rectangle, every
$f\in\mathcal{F}_{J}$ satisfies
$f(p)-f(q)\leq C_{J}\lVert(p-q)_{\mathcal{V}+}\rVert^{\gamma_{J}}$ for
$p,q\in\{0,x_{1},\dots,x_{n}\}$. This is because $z=q+(p-q)_{\mathcal{V}+}$ then lies
in $\mathcal{D}$, its coordinates in $\mathcal{V}$ are at least those of $p$
and its other coordinates equal those of $p$, so that
$f(p)\leq f(z)\leq f(q)+C_{J}\lVert z-q\rVert^{\gamma_{J}}$. Conversely, if
values $v_{p}$ satisfy the constraints, the function
$\min_{q}\{v_{q}+C_{J}\lVert(\cdot-q)_{\mathcal{V}+}\rVert^{\gamma_{J}}\}$
takes the value $v_{p}$ at each $p$. This function is in $\mathcal{F}_{J}$ by
Lemmas~4.1 and 4.2 because each function
in the minimum is a translate of a function in Lemma~4.1 plus a
constant. Hence, the largest valid $K_{\tau}$ is the smallest, over $j$ and
$\ell$, of the values of the linear programs that minimize
$a_{j}+c_{\ell}$ subject to the constraints.

\subsection{Limit of the Calibration}
\label{sec:proof-Lemma-invariance}
We now describe the behavior of the calibration as $n\to\infty$, which we
discuss in Section~5.1. Here, $\tau^{\ast}$ is the
largest $\tau\in[\alpha/(2J),\alpha]$ that satisfies (13) with
$p_{\tau}(K_{\tau})$ itself and $K_{\tau}$ given in Lemma~5.1. By
Lemma~\ref{lem:tau-limit}, the covariance matrix of $(V(\tau)',W(\tau)')'$
converges as $n\to\infty$ to a limit that does not depend on $\tau$ or on the
design distribution, for almost all realizations of the design points. We
discuss what this implies for $\tau^{\ast}$ in Remark~\ref{rem:dependence-J}.
We compute $\tau^{\ast}$ from the
exact covariance (10), not from its limit, because the convergence
can be very slow. The improvement from
the calibration persists as $n\to\infty$ when the parameter spaces differ only
in $C$. Lemma~\ref{lem:tau} shows that it then does not shrink as $J$ grows if
the constants $C_{j}$ lie in a fixed interval $[\underline{C},C_{J}]$ with
$\underline{C}>0$. Lemma~\ref{lem:tau-limit}(iii) and
Lemma~\ref{lem:tau} do not necessarily hold when $\tau^{\ast}$ is computed on a
grid or from an upper bound $\bar{p}_{\tau}$. For example, the union bound
$\bar{p}_{\tau}=2J\{1-\Phi(z_{1-\tau})\}$ gives
$\tau^{\ast}=\alpha/(2J)$. Since such a $\tau^{\ast}$ is never larger than the
one above, its critical value $z_{1-\tau^{\ast}}$ is never smaller.
\begin{lemma}\label{lem:tau-limit}
  Suppose that, for each $j\leq J$, the conditions of Theorem~4.2 hold with
  $(\gamma_{1},C_{1},\gamma_{2},C_{2})=(\gamma_{j},C_{j},\gamma_{J},C_{J})$,
  and that $\sigma(\cdot)=1$. Then, for almost all realizations
  $\{ x_{i}\} _{i=1}^{\infty}$ of $\{ X_{i}\} _{i=1}^{\infty}$, the following
  holds.
  \begin{enumerate}
  \item[(i)] There is a correlation matrix $\Sigma_{\infty}$, which depends
    only on $(k,\mathcal{V},\lVert\cdot\rVert,\gamma_{1},C_{1},\dots,\gamma_{J},C_{J})$,
    such that the covariance matrix of $(V(\tau)',W(\tau)')'$ converges to
    $\Sigma_{\infty}$ as $n\to\infty$, uniformly over $\tau$ in compact
    subsets of $(0,1/2)$.
  \item[(ii)] If $k_{+}\geq1$ and $\gamma_{1},\dots,\gamma_{J}$ are not all
    equal, then $K_{\tau}$ in Lemma~5.1 converges to
    $K_{\infty}(\tau):=0$. Otherwise, $K_{\tau}$ converges to
    $K_{\infty}(\tau):=k_{\infty}\delta_{\tau}$, where $k_{\infty}\in(0,\infty)$
    depends only on $(k,\mathcal{V},\lVert\cdot\rVert,\gamma_{1},C_{1},\dots,\gamma_{J},C_{J})$.
    In both cases, the convergence is uniform over $\tau$ in compact subsets of
    $(0,1/2)$.
  \item[(iii)] Let $p_{\infty}(\tau)$ be defined as $p_{\tau}(K_{\infty}(\tau))$
    in (12), with $(V(\tau)',W(\tau)')'$ replaced by
    $(V_{\infty}',W_{\infty}')'\sim N(0,\Sigma_{\infty})$. Then, the equation
    $p_{\infty}(\tau)=\alpha$ has a unique solution $\tau_{\infty}$ in
    $[\alpha/(2J),\alpha]$, which lies in $(\alpha/(2J),\alpha)$, and
    $\tau^{\ast}\to\tau_{\infty}$ as $n\to\infty$.
  \end{enumerate}
\end{lemma}
\begin{remark}[Dependence on $J$]\label{rem:dependence-J}
  The proof reveals that when $k_{+}\geq1$ and all $J$ parameter spaces
  correspond to different H\"older exponents (i.e., $\gamma_{1} > \cdots >
  \gamma_{J}$), the elements of $(V_{\infty}',W_{\infty}')'$ are independent,
  except that $V_{\infty,J}=-W_{\infty,J}$, and $K_{\infty}=0$. Hence, the
  dependence of $\tau^{\ast}$ on $J$ does not vanish in the limit. That is,
  $z_{1-\tau_{\infty}}$ is the $1-\alpha$ quantile of the maximum of $2J-2$
  independent standard normal random variables and the absolute value of
  another independent one, which grows like $(2\log J)^{1/2}$ as
  $J\to\infty$, as does the Bonferroni critical value $z_{1-\alpha/(2J)}$.
  Hence, for large $n$, the constant in (14) grows at
  the same rate as the constant in (9), like
  $(\log J)^{1/2}$ for the general bounds and like $(\log J)^{r}$ for the
  sharper bounds based on Theorem~4.2 in
  Section~\ref{apdx:interval}.

  However, the convergence in (i) and (ii) can be very slow. For example, for
  the pair of exponents $(1,4/5)$ in the design of Section~6.1,
  the bound on the corresponding covariance in Step~2 of the following proof
  decreases at the rate $n^{-1/35}$. In the
  design of Section~6.1, with $n=800$ and four different
  H\"older exponents, $\tau$ is selected on a grid with the bound of
  Lemma~\ref{lem:hunter-bound}. At the selected value, the correlations between
  the elements of $V(\tau)$ are between $0.93$ and $0.99$, and $K_{\tau}$ is
  $2.64$.

  When some of the parameter spaces have the
  same H\"older exponent, the improvement can be significant even for large
  $n$. As an extreme case, when $\gamma_{1}= \cdots = \gamma_{J}$ and the
  constants $C_{j}$ lie in a fixed interval $[\underline{C},C_{J}]$ with
  $\underline{C}>0$, $\tau^{\ast}$ is bounded away from $0$ by a constant that
  does not depend on $J$, which is what Lemma~\ref{lem:tau} asserts.
\end{remark}
\begin{proof}[Proof of Lemma~\ref{lem:tau-limit}]
  Throughout the proof, $\sigma(\cdot)=1$. For $\delta>0$ and $j\leq J$,
  write $b_{j}^{-}(\delta)=\omega(\delta,\mathcal{F}_{J},\mathcal{F}_{j})$,
  $b_{j}^{+}(\delta)=\omega(\delta,\mathcal{F}_{j},\mathcal{F}_{J})$,
  \[
    a_{j}^{-}(x)=C_{j}\lVert (x)_{\mathcal{V}-}\rVert^{\gamma_{j}}+C_{J}\lVert (x)_{\mathcal{V}+}\rVert^{\gamma_{J}},\qquad a_{j}^{+}(x)=C_{j}\lVert (x)_{\mathcal{V}+}\rVert^{\gamma_{j}}+C_{J}\lVert (x)_{\mathcal{V}-}\rVert^{\gamma_{J}},
  \]
  and $D_{j,\delta}^{\pm}(x)=(b_{j}^{\pm}(\delta)-a_{j}^{\pm}(x))_{+}$, so
  that $D_{j,\delta}^{-}=D_{Jj,\delta}$ and $D_{j,\delta}^{+}=D_{jJ,\delta}$.
  By Proposition~4.1 and Corollary~4.1,
  $d_{j}^{\pm}(\tau)=(D_{j,\delta}^{\pm}(x_{i}))_{i=1}^{n}$ with
  $\delta=\delta_{\tau}$ (the lower bound in the proof of Proposition~4.1 holds term by term and is attained, so every solution
  of the ordered modulus problem at $\delta>0$ has these differences at the
  design points, whichever solution is fixed in Section~3.2), and
  $\sum_{i=1}^{n}D_{j,\delta}^{\pm}(x_{i})^{2}=\delta^{2}$. Let
  $\mathcal{Q}=\{-,+\}\times\{1,\dots,J\}$. By (10), the
  covariance between the elements of $(V(\tau)',W(\tau)')'$ that correspond to
  $q,q'\in\mathcal{Q}$ is $\pm\sum_{i=1}^{n}D_{q,\delta}(x_{i})D_{q',\delta}(x_{i})/\delta^{2}$,
  with the negative sign if and only if one of $q$ and $q'$ is of the form
  $(+,j)$ and the other of the form $(-,j')$. For $q=(\pm,j)$, we write
  $\gamma_{q}=\gamma_{j}$, $r_{q}=r(\gamma_{j},\gamma_{J})$ and
  $d_{q}=1/r_{q}-2=k_{+}/\gamma_{j}+(k-k_{+})/\gamma_{J}$.

  Let $\Omega_{0}$ be the event on which the conclusion of Theorem~4.2 holds for every pair
  $(\gamma_{j},C_{j},\gamma_{J},C_{J})$, $j\leq J$, and on which the
  countably many almost sure statements below hold. Then, $\Omega_{0}$ has
  probability one, and we work on $\Omega_{0}$ for the rest of the proof. By
  Theorem~4.2 and the remark following it,
  \begin{equation}
    n^{r_{q}}b_{q}(\delta)\to\beta_{q}(\delta):=(\delta^{2}/c_{q}^{\ast})^{r_{q}}\quad\text{for all }\delta>0,\label{eq:tau-b-limit}
  \end{equation}
  where $c_{q}^{\ast}=c_{1,2}^{\ast}=c_{2,1}^{\ast}$ is the constant in
  Theorem~4.2 for the pair
  $(\gamma_{j},C_{j},\gamma_{J},C_{J})$. Note that
  $\Sigma_{n}(\delta):=\sum_{i=1}^{n}D_{q,\delta}(x_{i})D_{q',\delta}(x_{i})$
  is nondecreasing in $\delta$, since $b_{q}(\delta)$ is. Hence, it suffices
  to show the convergence of $\Sigma_{n}(\delta)/\delta^{2}$ below for
  rational $\delta>0$, provided that the limit is continuous in $\delta$.

  \medskip
  \noindent\textbf{Step~1: Bounding boxes.} Let $\kappa_{0}=\min_{v}\lVert e_{v}\rVert$,
  where $e_{v}$ is the $v$th standard basis vector. By Assumption~2.1, $\lVert z\rVert\geq\kappa_{0}\lvert z_{v}\rvert$ for all $z$
  and $v$. Hence, if $D_{j,\delta}^{-}(x)>0$, then for each $v\in\mathcal{V}$,
  $-\kappa_{0}^{-1}(b/C_{j})^{1/\gamma_{j}}<x_{v}<\kappa_{0}^{-1}(b/C_{J})^{1/\gamma_{J}}$,
  and for each $v\notin\mathcal{V}$,
  $\lvert x_{v}\rvert<\kappa_{0}^{-1}(b/C_{J})^{1/\gamma_{J}}$, where
  $b=b_{j}^{-}(\delta)$. That is, the support of $D_{j,\delta}^{-}$ is
  contained in a box $R_{j}^{-}(b)$ whose side lengths are of order
  $b^{1/\gamma_{j}}$ in the negative direction of the coordinates in
  $\mathcal{V}$ and of order $b^{1/\gamma_{J}}$ in all other directions. The
  same holds for $D_{j,\delta}^{+}$, with a box $R_{j}^{+}(b)$ in which the
  roles of the positive and negative directions of the coordinates in
  $\mathcal{V}$ are interchanged. Now, take $q,q'\in\mathcal{Q}$ and let
  $\bar{b}_{q,n}=2\beta_{q}(\delta)n^{-r_{q}}$, which bounds
  $b_{q}(\delta)$ for large $n$ by \eqref{eq:tau-b-limit}. Then, for large
  $n$,
  \begin{equation}
    \Sigma_{n}(\delta)\leq\bar{b}_{q,n}\bar{b}_{q',n}\,N_{n}(\bar{R}_{n}),\label{eq:tau-box-bound}
  \end{equation}
  where $\bar{R}_{n}$ is the intersection of the boxes for $q$ and $q'$ at
  $\bar{b}_{q,n}$ and $\bar{b}_{q',n}$, and $N_{n}(R)$ denotes the number of
  $x_{i}$, $i\leq n$, in $R$. Since $\bar{R}_{n}$ is deterministic and shrinks
  to $\{0\}\subset\mathrm{Int}\,\mathcal{X}$,
  $\pi_{n}:=\p(X_{1}\in\bar{R}_{n})\leq2p_{X}(0)\mathrm{vol}(\bar{R}_{n})$ for
  large $n$. Moreover, Bernstein's inequality gives
  $\p(N_{n}(\bar{R}_{n})>2n\pi_{n}+4\log n)\leq n^{-3/2}$, so that, by the
  Borel--Cantelli lemma, $N_{n}(\bar{R}_{n})\leq4p_{X}(0)\,n\,\mathrm{vol}(\bar{R}_{n})+4\log n$
  for large $n$. Since $\delta$ is rational, there are countably many such
  statements, and they are included in $\Omega_{0}$.

  \medskip
  \noindent\textbf{Step~2: Pairs with vanishing covariance.} Suppose (a) $k_{+}\geq1$
  and $\gamma_{q}\neq\gamma_{q'}$, or (b) $k_{+}\geq1$,
  $\gamma_{q}=\gamma_{q'}>\gamma_{J}$, and one of $q$ and $q'$ corresponds to
  a lower end and the other to an upper end. We show that
  $\Sigma_{n}(\delta)\to0$. In both cases, the term
  $\bar{b}_{q,n}\bar{b}_{q',n}\log n$ in \eqref{eq:tau-box-bound} converges
  to zero, so that it suffices to show that
  $n\bar{b}_{q,n}\bar{b}_{q',n}\mathrm{vol}(\bar{R}_{n})\to0$. In case (a),
  the side lengths of $\bar{R}_{n}$ are of order at most
  $\min\{\bar{b}_{q,n}^{1/\gamma_{q}},\bar{b}_{q',n}^{1/\gamma_{q'}}\}$ in
  the coordinates in $\mathcal{V}$ and
  $\min\{\bar{b}_{q,n},\bar{b}_{q',n}\}^{1/\gamma_{J}}$ in the other
  coordinates. Using $r_{q}(2+d_{q})=r_{q'}(2+d_{q'})=1$, we obtain
  \[
    n\bar{b}_{q,n}\bar{b}_{q',n}\mathrm{vol}(\bar{R}_{n})=O(n^{-\{k_{+}\lvert r_{q}/\gamma_{q}-r_{q'}/\gamma_{q'}\rvert+(k-k_{+})\lvert r_{q}-r_{q'}\rvert/\gamma_{J}\}/2}).
  \]
  The map $\gamma\mapsto r(\gamma,\gamma_{J})/\gamma=(2\gamma+k_{+}+(k-k_{+})\gamma/\gamma_{J})^{-1}$
  is strictly decreasing, so that the exponent is negative when $k_{+}\geq1$
  and $\gamma_{q}\neq\gamma_{q'}$. In case (b), $r_{q}=r_{q'}=:r$, and since
  one box extends to order $b^{1/\gamma_{q}}$ only in the negative direction
  of the coordinates in $\mathcal{V}$ and the other only in the positive
  direction, the side lengths of $\bar{R}_{n}$ are of order $n^{-r/\gamma_{J}}$
  in every coordinate. Hence,
  \[
    n\bar{b}_{q,n}\bar{b}_{q',n}\mathrm{vol}(\bar{R}_{n})=O(n^{1-r(2+k/\gamma_{J})})=O(n^{-rk_{+}(1/\gamma_{J}-1/\gamma_{q})})\to0.
  \]

  \medskip
  \noindent\textbf{Step~3: Pairs with a nondegenerate limit.} Suppose
  neither (a) nor (b) of Step~2 holds. Then, $r_{q}=r_{q'}=:r$, and either
  $q$ and $q'$ both correspond to lower ends (or both to upper ends) with
  $\gamma_{q}=\gamma_{q'}$, or $\gamma_{q}=\gamma_{q'}=\gamma_{J}$, or
  $k_{+}=0$. In each of these cases, $a_{q}$ and $a_{q'}$ have the same
  scaling in each orthant, in the sense of the change of variables in the
  proof of Theorem~4.2. (If $k_{+}=0$, then
  $(x)_{\mathcal{V}+}=x$ and $\lVert(x)_{\mathcal{V}-}\rVert=\lVert x\rVert$,
  so that $a_{q}(x)=C_{j}\lVert x\rVert^{\gamma_{j}}+C_{J}\lVert x\rVert^{\gamma_{J}}$,
  and all the coordinates are scaled by $t^{1/\gamma_{J}}$.) Write
  $d=1/r-2$ and $t_{n}=n^{-r}$, so that $nt_{n}^{2+d}=1$. For $\beta,\beta'>0$,
  define
  \[
    P_{n}(\beta,\beta')=\frac{1}{nt_{n}^{d}}\sum_{i=1}^{n}\left(\beta-\frac{a_{q}(x_{i})}{t_{n}}\right)_{+}\left(\beta'-\frac{a_{q'}(x_{i})}{t_{n}}\right)_{+}.
  \]
  Applying the same change of variables as in the proof of Theorem~4.2 on each orthant, we can write
  $t^{-d}\int(\beta-a_{q}(x)/t)_{+}(\beta'-a_{q'}(x)/t)_{+}dx$ as a sum over
  the orthants of integrals multiplied by $t^{e_{O}-d}$, where $e_{O}\geq d$
  is the exponent of the Jacobian on orthant $O$. The integrands converge
  pointwise as $t\to0$ and are bounded by $\beta\beta'$ times the indicator
  of a fixed bounded set. Hence, the limit $\kappa_{qq'}(\beta,\beta')$ of
  this expression exists by the dominated convergence theorem, and it is
  continuous in $(\beta,\beta')$. Hence, Lemma~\ref{lem:local-sums} in
  Section~\ref{sec: proof main thm} gives $P_{n}(\beta,\beta')\to p_{X}(0)\kappa_{qq'}(\beta,\beta')$
  for each rational $\beta,\beta'>0$. Since $P_{n}$ is nondecreasing in each
  argument and $\kappa_{qq'}$ is continuous, this convergence holds for all
  $\beta,\beta'>0$. Together with \eqref{eq:tau-b-limit}, it gives
  \[
    \Sigma_{n}(\delta)=P_{n}(n^{r}b_{q}(\delta),n^{r}b_{q'}(\delta))\to p_{X}(0)\kappa_{qq'}(\beta_{q}(\delta),\beta_{q'}(\delta)),
  \]
  since, for any $\epsilon\in(0,\min\{\beta_{q}(\delta),\beta_{q'}(\delta)\})$
  and large $n$,
  $P_{n}(\beta_{q}(\delta)-\epsilon,\beta_{q'}(\delta)-\epsilon)\leq\Sigma_{n}(\delta)\leq P_{n}(\beta_{q}(\delta)+\epsilon,\beta_{q'}(\delta)+\epsilon)$,
  and $\kappa_{qq'}$ is continuous.
  Finally, since $(s\beta-a/t)_{+}=s(\beta-a/(st))_{+}$, a change of variables
  from $t$ to $st$ gives $\kappa_{qq'}(s\beta,s\beta')=s^{2+d}\kappa_{qq'}(\beta,\beta')$
  for $s>0$. Since $\beta_{q}(\delta)=\delta^{2r}(c_{q}^{\ast})^{-r}$ and
  $r(2+d)=1$,
  \[
    \frac{\Sigma_{n}(\delta)}{\delta^{2}}\to p_{X}(0)\kappa_{qq'}((c_{q}^{\ast})^{-r},(c_{q'}^{\ast})^{-r}),
  \]
  which does not depend on $\delta$. Together with Step~2, this shows that
  the covariance matrix of $(V(\tau)',W(\tau)')'$ converges to a limit that
  does not depend on $\tau$. The convergence is uniform over $\delta$ in
  compact intervals $[\delta_{1},\delta_{2}]\subset(0,\infty)$, since each
  $\Sigma_{n}(\delta)$ is nondecreasing in $\delta$ and its limit
  $p_{X}(0)\kappa_{qq'}((c_{q}^{\ast})^{-r},(c_{q'}^{\ast})^{-r})\delta^{2}$ is
  continuous. Given $\epsilon>0$, take a finite grid
  $\delta_{1}=u_{0}<\cdots<u_{m}=\delta_{2}$ on which the limit varies by less
  than $\epsilon$ between adjacent points. For
  $\delta\in[u_{i-1},u_{i}]$, monotonicity gives
  $\Sigma_{n}(u_{i-1})\leq\Sigma_{n}(\delta)\leq\Sigma_{n}(u_{i})$, and the
  convergence at the finitely many grid points gives the claim. The same
  argument applies to the pairs in Step~2. Since
  $\delta_{\tau}=z_{1-\tau}$ is continuous in $\tau$, this proves (i).

  \medskip
  \noindent\textbf{Step~4: Limit of $K_{\tau}$.} Fix $j,\ell\leq J$, and let
  $q=(-,j)$ and $q'=(+,\ell)$. By Lemma~5.1,
  $K_{\tau}=\min_{j,\ell\leq J}Q_{j\ell}(\delta_{\tau})$, where
  \[
    Q_{j\ell}(\delta)=\frac{1}{\delta}\sum_{i=1}^{n}\min\{D_{q,\delta}(x_{i}),D_{q',\delta}(x_{i})\}\{h_{+}(x_{i})-h_{-}(x_{i})\}.
  \]
  If $D_{q,\delta}(x)>0$, then
  $h_{+}(x)=C_{J}\lVert (x)_{\mathcal{V}+}\rVert^{\gamma_{J}}\leq a_{j}^{-}(x)<b_{q}(\delta)$,
  and if $D_{q',\delta}(x)>0$, then
  $-h_{-}(x)=C_{J}\lVert (x)_{\mathcal{V}-}\rVert^{\gamma_{J}}\leq a_{\ell}^{+}(x)<b_{q'}(\delta)$.
  Moreover,
  $\min\{D_{q,\delta}(x_{i}),D_{q',\delta}(x_{i})\}\leq\min\{b_{q}(\delta),b_{q'}(\delta)\}$.
  Note also that $\delta Q_{j\ell}(\delta)$ is nondecreasing in $\delta$, since
  $D_{q,\delta}(x_{i})$ and $D_{q',\delta}(x_{i})$ are.

  Suppose first that $k_{+}\geq1$ and $\gamma_{j}\neq\gamma_{\ell}$, which is
  possible for some $j,\ell$ if the exponents are not all equal. By the
  bounds above,
  \[
    0\leq\delta Q_{j\ell}(\delta)\leq2\,b_{q}(\delta)b_{q'}(\delta)N_{n}(\bar{R}_{n}),
  \]
  with $\bar{R}_{n}$ as in Step~1, and Step~2 shows that the right-hand side
  converges to zero (case (a)). Since $\delta Q_{j\ell}(\delta)$ is
  nondecreasing, $Q_{j\ell}(\delta)\leq\delta_{2}Q_{j\ell}(\delta_{2})/\delta_{1}$
  for $\delta\in[\delta_{1},\delta_{2}]$, so that $Q_{j\ell}\to0$ uniformly on
  compact subsets of $(0,\infty)$, and hence $K_{\tau}\to0$ uniformly over
  $\tau$ in compact subsets of $(0,1/2)$.

  Suppose next that $k_{+}=0$ or $\gamma_{1}=\cdots=\gamma_{J}$. Then,
  $r_{q}=r$ for all $q$, and $a_{q}$, $a_{q'}$ and $h_{+}-h_{-}$ have the same
  scaling in each orthant, in the sense of Step~3. With $t_{n}=n^{-r}$, so that
  $nt_{n}^{d+2}=1$ with $d=1/r-2$, the argument in Step~3 gives
  \[
    \delta Q_{j\ell}(\delta)=\sum_{i=1}^{n}\min\{D_{q,\delta}(x_{i}),D_{q',\delta}(x_{i})\}\{h_{+}(x_{i})-h_{-}(x_{i})\}\to p_{X}(0)\mu_{j\ell}(\delta),
  \]
  where $\mu_{j\ell}(\delta)$ is the integral, after the change of variables,
  of the pointwise limit of the summand. The limit is positive, since the
  integrand is positive on a neighborhood of $0$ with $0$ removed. As in
  Step~3, scaling $\beta_{q}(\delta)$ and $\beta_{q'}(\delta)$ by
  $\delta^{2r}$ multiplies $\mu_{j\ell}$ by $\delta^{2r(2+d)}=\delta^{2}$, so that
  $\mu_{j\ell}(\delta)=\mu_{j\ell}(1)\delta^{2}$, and
  $Q_{j\ell}(\delta)\to k_{j\ell}\delta$ with
  $k_{j\ell}=p_{X}(0)\mu_{j\ell}(1)\in(0,\infty)$. Since
  $\delta Q_{j\ell}(\delta)$ is nondecreasing in $\delta$ and its limit is
  continuous, the argument at the end of Step~3 shows that the convergence is
  uniform over $\delta$ in compact subsets of $(0,\infty)$. Hence,
  $K_{\tau}\to k_{\infty}\delta_{\tau}$ with $k_{\infty}=\min_{j,\ell}k_{j\ell}$,
  uniformly over $\tau$ in compact subsets of $(0,1/2)$. Since the limits of
  $\Sigma_{n}(\delta)$ and of the sums above are scale invariant in the sense of
  Step~3, $\Sigma_{\infty}$ and $k_{\infty}$ do not depend on $p_{X}(0)$. This
  proves (ii).

  \medskip
  \noindent\textbf{Step~5: Convergence of $\tau^{\ast}$.} Let
  $p_{n}(\tau)=p_{\tau}(K_{\tau})$. For a correlation matrix $\Sigma$ of
  size $2J$ and $z_{1},z_{2}\in\mathbb{R}$, let
  $\pi(\Sigma,z_{1},z_{2})=\p(\max_{j\leq J}G_{j}>z_{1}\text{ or }\max_{\ell\leq J}G_{J+\ell}>z_{2})$
  with $G\sim N(0,\Sigma)$. Since each coordinate of $G$ has a continuous
  distribution, the boundary of this event has probability zero, so that
  $\pi$ is continuous by the portmanteau theorem, and hence uniformly
  continuous on compact sets. Since
  $p_{\tau}(K)=\sup_{0\leq t\leq K}\pi(\Sigma_{n}(\tau),z_{1-\tau}+t,z_{1-\tau}+K-t)$,
  where $\Sigma_{n}(\tau)$ is the covariance matrix of $(V(\tau)',W(\tau)')'$,
  parts (i) and (ii) imply that $p_{n}\to p_{\infty}$ uniformly over
  $[\alpha/(2J),\alpha]$, and that $p_{\infty}$ is continuous.

  Next, $p_{\infty}$ is strictly increasing. To see this, write
  $K_{\infty}(\tau)=k\delta_{\tau}$, with $k=0$ or $k=k_{\infty}$, so that
  $p_{\infty}(\tau)=\sup_{0\leq v\leq k}\pi(\Sigma_{\infty},\delta_{\tau}(1+v),\delta_{\tau}(1+k-v))$.
  For each $v$, the map
  $\delta\mapsto\pi(\Sigma_{\infty},\delta(1+v),\delta(1+k-v))$ is strictly
  decreasing on $(0,\infty)$: for $0<\delta'<\delta$, the difference between
  its values at $\delta'$ and $\delta$ is at least the probability that
  $V_{\infty,1}\in(\delta'(1+v),\delta(1+v))$, that the other elements of
  $V_{\infty}$ are at most $\delta(1+v)$, and that the elements of
  $W_{\infty}$ are at most $\delta(1+k-v)$. Conditional on
  $V_{\infty,1}=x\in(\delta'(1+v),\delta(1+v))$, the other elements of
  $V_{\infty}$ are jointly Gaussian with means in $[0,x]$, since they are
  nonnegatively correlated with $V_{\infty,1}$, and the elements of
  $W_{\infty}$ have nonpositive means, so that all of them lie below their
  thresholds with positive probability. Since $\delta_{\tau}$ is strictly
  decreasing in $\tau$ and the supremum is over a compact set,
  $p_{\infty}$ is strictly increasing.

  We now show that $p_{\infty}(\alpha/(2J))<\alpha<p_{\infty}(\alpha)$. We use
  Slepian's inequality \citep{slepian1962onesided}, in the form of
  \citet[Theorem~7.2.1 and Lemma~7.2.8]{vershynin2018high}: for centered Gaussian
  vectors $(G_{1},G_{2})$ with unit variances,
  $\p(G_{1}>z_{1},G_{2}>z_{2})$ is nondecreasing in their correlation, and
  equals $\{1-\Phi(z_{1})\}\{1-\Phi(z_{2})\}$ at correlation zero. For the
  first inequality, let $z=z_{1-\alpha/(2J)}$. If $J\geq2$, then
  $p_{\infty}(\alpha/(2J))\leq\p(\max\{V_{\infty}',W_{\infty}'\}>z)<2J\{1-\Phi(z)\}=\alpha$,
  where the strict inequality holds because
  $\p(V_{\infty,1}>z,V_{\infty,2}>z)\geq\{1-\Phi(z)\}^{2}>0$, as
  $\mathrm{Cov}(V_{\infty,1},V_{\infty,2})\geq0$. If $J=1$, then $k_{\infty}>0$
  by (ii), and, with $K=k_{\infty}z$, the union bound and the convexity
  argument at the end of the proof of Lemma~3.3 give
  $p_{\infty}(\alpha/2)\leq\sup_{0\leq t\leq K}[\{1-\Phi(z+t)\}+\{1-\Phi(z+K-t)\}]=\{1-\Phi(z)\}+\{1-\Phi(z+K)\}<2\{1-\Phi(z)\}=\alpha$.
  For the second inequality, let $z=z_{1-\alpha}$ and take $t=0$ in
  (12). Since $\mathrm{Cov}(V_{\infty,J},W_{\infty,J})\leq0$,
  \[
    p_{\infty}(\alpha)\geq\p(V_{\infty,J}>z\text{ or }W_{\infty,J}>z+kz)\geq\alpha+\{1-\Phi(z+kz)\}(1-\alpha)>\alpha.
  \]
  Hence, $\tau_{\infty}$ is the unique solution of $p_{\infty}(\tau)=\alpha$
  on $[\alpha/(2J),\alpha]$, and it lies in $(\alpha/(2J),\alpha)$. For any
  $\epsilon>0$, $p_{\infty}(\tau_{\infty}-\epsilon)<\alpha$ and
  $p_{\infty}(\tau)\geq p_{\infty}(\tau_{\infty}+\epsilon)>\alpha$ for
  $\tau\geq\tau_{\infty}+\epsilon$. By the uniform convergence, for large
  $n$, $p_{n}(\tau_{\infty}-\epsilon)\leq\alpha$ and $p_{n}(\tau)>\alpha$ for
  all $\tau\in[\tau_{\infty}+\epsilon,\alpha]$, so that
  $\tau_{\infty}-\epsilon\leq\tau^{\ast}<\tau_{\infty}+\epsilon$. This proves
  (iii).
\end{proof}

\subsection{Adapting to the H\"older Constant}
\label{sec:adapt-C}
\begin{lemma}\label{lem:tau}
  Suppose that, for each $j\leq J$, the conditions of Theorem~4.2 hold with
  $(\gamma_{1},C_{1},\gamma_{2},C_{2})=(\gamma_{j},C_{j},\gamma_{J},C_{J})$,
  that $\sigma(\cdot)=1$, and that $\gamma_{1}= \cdots = \gamma_{J}=\gamma$.
  Fix $\underline{C}\in(0,C_{J}]$. Then, there exists $\eta>0$ such that,
  for almost all realizations $\{ x_{i}\} _{i=1}^{\infty}$ of
  $\{ X_{i}\} _{i=1}^{\infty}$, there exists $N$ such that
  $\tau^{\ast}\geq\eta$ for all $n\geq N$, all $J\geq1$, and all
  $C_{1},\dots,C_{J-1}\in[\underline{C},C_{J}]$. Here, $\eta$ depends only on
  $(\alpha,k,\mathcal{V},\lVert\cdot\rVert,\gamma,\underline{C},C_{J})$, and
  neither the set of realizations nor $N$ depends on $J$ or
  $C_{1},\dots,C_{J-1}$.
\end{lemma}
Hence, by Theorem~3.2, the constant of the CI does not
grow to infinity as $J \to \infty$ in this case. This is especially useful when
one wishes to adapt to $C$ while keeping $\gamma$ fixed. For example,
\cite{kwon2020InferenceRegressionDiscontinuity} take $\gamma_{j} = 1$ and
consider the problem of adapting to the Lipschitz constant in a regression
discontinuity setting. The constant $\eta$ comes from an entropy bound in the
proof and is not explicit. Numerically, the improvement is large even when
$\tau^{\ast}$ is computed from the bound of Lemma~\ref{lem:hunter-bound} on a
grid, as in Section~6. In the design of
Section~6.1 with $\gamma=1$ and $J=17$ values of $C$
spaced geometrically in $[1/4,4]$, $\tau^{\ast}=0.0195$, compared with $\alpha/(2J)=0.0015$.
\begin{proof}[Proof of Lemma~\ref{lem:tau}]
    We use the notation of the proof of Lemma~\ref{lem:tau-limit}. Suppose
  $\gamma_{1}=\cdots=\gamma_{J}=\gamma$, so that $r(\gamma,\gamma)=(2+k/\gamma)^{-1}=:r$.
  For $\delta>0$ and $C\in[\underline{C},C_{J}]$, let
  \begin{align*}
    b^{-}(C) & =\omega(\delta,\Lambda_{+,\mathcal{V}}(\gamma,C_{J}),\Lambda_{+,\mathcal{V}}(\gamma,C)),\\
    D_{C}^{-}(x) & =(b^{-}(C)-C\lVert (x)_{\mathcal{V}-}\rVert^{\gamma}-C_{J}\lVert (x)_{\mathcal{V}+}\rVert^{\gamma})_{+},\text{ and}\\
    G^{-}(C) & =\sum\nolimits_{i=1}^{n}D_{C}^{-}(x_{i})\varepsilon_{i}/\delta,
  \end{align*}
  and define $b^{+}(C)$, $D_{C}^{+}$ and
  $G^{+}(C)=-\sum_{i=1}^{n}D_{C}^{+}(x_{i})\varepsilon_{i}/\delta$
  analogously, using the ordered modulus
  $\omega(\delta,\Lambda_{+,\mathcal{V}}(\gamma,C),\Lambda_{+,\mathcal{V}}(\gamma,C_{J}))$.
  Note that
  $\Lambda_{+,\mathcal{V}}(\gamma,C)\subseteq\Lambda_{+,\mathcal{V}}(\gamma,C_{J})$
  for $C\leq C_{J}$. For any
  $\{C_{1},\dots,C_{J}\}\subseteq[\underline{C},C_{J}]$ and
  $\delta=\delta_{\tau}$, the vector $(V(\tau)',W(\tau)')'$ has the same
  distribution as
  \[
    (G^{-}(C_{1}),\dots,G^{-}(C_{J}),G^{+}(C_{1}),\dots,G^{+}(C_{J}))'.
  \]
  Hence,
  \begin{equation}
    p_{\tau}(0)\leq\p(S_{\delta}>\delta),\qquad S_{\delta}:=\sup_{C\in[\underline{C},C_{J}]}\max\{G^{-}(C),G^{+}(C)\},\label{eq:tau-sup-bound}
  \end{equation}
  where each $G^{\pm}(C)$ has unit variance. We now show that the process is
  Lipschitz in $C$ in the canonical metric. Since $b^{-}(C)$ solves
  $\sum_{i=1}^{n}(b-C\lVert (x_{i})_{\mathcal{V}-}\rVert^{\gamma}-C_{J}\lVert (x_{i})_{\mathcal{V}+}\rVert^{\gamma})_{+}^{2}=\delta^{2}$,
  and the left-hand side is continuously differentiable in $(b,C)$ with a
  positive derivative in $b$, the implicit function theorem gives
  \[
    \frac{\partial b^{-}(C)}{\partial C}=\frac{\sum_{i=1}^{n}\lVert (x_{i})_{\mathcal{V}-}\rVert^{\gamma}D_{C}^{-}(x_{i})}{\sum_{i=1}^{n}D_{C}^{-}(x_{i})},
  \]
  which lies between $0$ and $\max\{\lVert (x_{i})_{\mathcal{V}-}\rVert^{\gamma}:D_{C}^{-}(x_{i})>0\}$.
  Since $\Lambda_{+,\mathcal{V}}(\gamma,C)$ increases with $C$, $b^{-}(\cdot)$
  is nondecreasing, so that $b^{-}(C)\leq b^{-}(C_{J})$, and we have
  $\lVert (x_{i})_{\mathcal{V}-}\rVert^{\gamma}\leq b^{-}(C)/C\leq b^{-}(C_{J})/\underline{C}$
  whenever $D_{C}^{-}(x_{i})>0$. Since both $\partial b^{-}(C)/\partial C$ and
  $\lVert (x_{i})_{\mathcal{V}-}\rVert^{\gamma}$ lie in
  $[0,b^{-}(C_{J})/\underline{C}]$ whenever $D_{C}^{-}(x_{i})>0$, the
  derivative of
  $b^{-}(C)-C\lVert (x_{i})_{\mathcal{V}-}\rVert^{\gamma}$ with respect to $C$
  is bounded in absolute value by $b^{-}(C_{J})/\underline{C}$ on this set.
  Hence, each $D_{C}^{-}(x_{i})$ is Lipschitz
  in $C\in[\underline{C},C_{J}]$ with constant $b^{-}(C_{J})/\underline{C}$,
  and it is zero unless $x_{i}$ is in
  $A_{n}=\{x:\underline{C}\lVert (x)_{\mathcal{V}-}\rVert^{\gamma}+C_{J}\lVert (x)_{\mathcal{V}+}\rVert^{\gamma}<b^{-}(C_{J})\}$.
  It follows that
  \[
    \E(G^{-}(C)-G^{-}(C'))^{2}=\frac{1}{\delta^{2}}\sum_{i=1}^{n}(D_{C}^{-}(x_{i})-D_{C'}^{-}(x_{i}))^{2}\leq\frac{b^{-}(C_{J})^{2}N_{n}(A_{n})}{\underline{C}^{2}\delta^{2}}(C-C')^{2},
  \]
  and similarly for $G^{+}$. By Theorem~4.2 applied with
  $(\gamma_{1},C_{1},\gamma_{2},C_{2})=(\gamma,C_{J},\gamma,C_{J})$,
  $n^{r}b^{\pm}(C_{J})\to\beta:=(\delta^{2}/c_{JJ}^{\ast})^{r}$, where
  $c_{JJ}^{\ast}$ is the corresponding constant. Moreover, since $A_{n}$ is
  contained in the deterministic set
  \[
    \bar{A}_{n}=\left\{ x:\underline{C}\lVert (x)_{\mathcal{V}-}\rVert^{\gamma}+C_{J}\lVert (x)_{\mathcal{V}+}\rVert^{\gamma}<2\beta n^{-r}\right\}
  \]
  for large $n$, whose volume is
  $(2\beta n^{-r})^{k/\gamma}\mathrm{vol}(\bar{A}^{\circ})$, where
  $\bar{A}^{\circ}$ is defined as $\bar{A}_{n}$ with $2\beta n^{-r}$ replaced
  by $1$, the counting argument in Step~1 gives
  \[
    N_{n}(A_{n})\leq4p_{X}(0)(2\beta)^{k/\gamma}n^{1-rk/\gamma}\mathrm{vol}(\bar{A}^{\circ})+4\log n
  \]
  for large $n$. Since $n^{1-r(2+k/\gamma)}=1$ and $\beta^{2+k/\gamma}=\delta^{2}/c_{JJ}^{\ast}$,
  this implies that, for large $n$, the canonical metric of the process is
  bounded by $\bar{L}\lvert C-C'\rvert$ within each of $G^{-}$ and $G^{+}$,
  where
  \[
    \bar{L}^{2}=\frac{2^{5+k/\gamma}p_{X}(0)\mathrm{vol}(\bar{A}^{\circ})}{\underline{C}^{2}c_{JJ}^{\ast}}
  \]
  does not depend on $\delta$, $J$, or $C_{1},\dots,C_{J-1}$. The same
  bound applies to $G^{+}$ after reflecting $x\mapsto-x$, with the count for
  $-\bar{A}_{n}$ handled by a separate application of the Borel--Cantelli
  lemma. Since the process
  has unit variances, the diameter of its index set
  $[\underline{C},C_{J}]\times\{-,+\}$ in the canonical metric is at most
  $2$, and the number of points in the index set that are at least $u$ apart
  in the canonical metric is at most $2+2\bar{L}(C_{J}-\underline{C})/u$.
  Hence, by Dudley's entropy bound \citep[see, e.g., Section
  2.2 of][]{vandervaart1996WeakConvergenceEmpirical},
  \[
    \E S_{\delta}\leq\E\sup_{C,\pm}\lvert G^{\pm}(C)\rvert\leq1+K\int_{0}^{2}\sqrt{\log(2+2\bar{L}(C_{J}-\underline{C})/u)}du=:K_{0}
  \]
  for a universal constant $K$. Moreover, by Borell's inequality \citep[see,
  e.g., Appendix~A.2 of][]{vandervaart1996WeakConvergenceEmpirical}, we have
  $\p(S_{\delta}>K_{0}+u)\leq2\exp(-u^{2}/2)$ for $u>0$. Now, let
  $\eta\in(0,\alpha]$ be such that
  $z_{1-\eta}\geq K_{0}+(2\log(2/\alpha))^{1/2}$, which does not depend on
  $J$ or $C_{1},\dots,C_{J-1}$, and apply the above with
  $\delta=z_{1-\eta}$. Then, by \eqref{eq:tau-sup-bound}, for large $n$,
  $p_{\eta}(0)\leq\alpha$ for any $\{C_{1},\dots,C_{J}\}\subseteq[\underline{C},C_{J}]$.
  Since $p_{\eta}(K_{\eta})\leq p_{\eta}(0)\leq\alpha$, (13) then holds at $\tau=\eta$. If $\eta\geq\alpha/(2J)$, this implies $\tau^{\ast}\geq\eta$ by the
  definition of $\tau^{\ast}$. Otherwise, $\tau^{\ast}\geq\alpha/(2J)>\eta$.
  Finally, the almost sure statements used in this step are the conclusion of
  Theorem~4.2 for $(\gamma,C_{J},\gamma,C_{J})$ and the two
  applications of the Borel--Cantelli lemma for $\bar{A}_{n}$ and
  $-\bar{A}_{n}$ with $\delta=z_{1-\eta}$. None of them, and hence neither the
  event of probability one nor the $N$ beyond which the above inequalities
  hold, depends on $J$ or $C_{1},\dots,C_{J-1}$. This proves the lemma.
\end{proof}

\section{Proofs of Theorems~4.2 and 4.3 and Corollary~4.3}
\label{sec: proof main thm}
We first state a lemma on sums over the design points near $0$, which we use
in the proof of Theorem~4.2 and Lemma~\ref{lem:tau-limit}.
\begin{lemma}
  \label{lem:local-sums}Let $\{ X_{i}\} _{i=1}^{\infty}$ satisfy
  the conditions in Theorem~4.2. Let $d>0$,
  $r=(2+d)^{-1}$, and $t_{n}=an^{-r}$ for some $a>0$. Let
  $\{g_{t}:0<t\leq t_{0}\}$ be measurable functions from $\mathbb{R}^{k}$ to $[0,M]$
  such that $g_{t}(x)=0$ whenever $\lVert x\rVert>\rho(t)$, where
  $\rho(t)\to0$ as $t\to0$, and such that
  $t^{-d}\int g_{t}(x)dx\to\kappa\in[0,\infty)$ as $t\to0$. Then,
  \[
    \frac{1}{nt_{n}^{d}}\sum_{i=1}^{n}g_{t_{n}}(X_{i})\to p_{X}(0)\kappa
  \]
  almost surely.
\end{lemma}
\begin{proof}
  Take an arbitrary $\varepsilon\in(0,p_{X}(0))$. Due to the regularity
  conditions on $p_{X}(\cdot)$, there exists a neighborhood
  $\mathcal{N}_{\varepsilon}\subseteq\mathcal{X}$ of $0$ such that
  $\lvert p_{X}(x) - p_{X}(0)\rvert \leq \varepsilon $ for all
  $x \in \mathcal{N}_{\varepsilon}$. Since $\rho(t_{n})\to0$ and the interior
  of $\mathcal{X}$ contains $0$, there exists $N_{\varepsilon}$ such that the
  support of $g_{t_{n}}$ is contained in $\mathcal{N}_{\varepsilon}$ for all
  $n\geq N_{\varepsilon}$. Hence, for $n\geq N_{\varepsilon}$, we have
  \begin{equation}
    (p_{X}(0)-\varepsilon)\int g_{t_{n}}(x)dx\leq\E g_{t_{n}}(X_{1})\leq(p_{X}(0)+\varepsilon)\int g_{t_{n}}(x)dx.\label{eq:EWin_bound}
  \end{equation}
  Dividing all sides by $t_{n}^{d}$, taking $n\to\infty,$ and then taking
  $\varepsilon\to0$, we have
  \begin{equation}
    \lim_{n\to\infty}t_{n}^{-d}\E g_{t_{n}}(X_{1})=p_{X}(0)\kappa.\label{eq:Integralrate}
  \end{equation}
  Now, since $0\leq g_{t_{n}}\leq M$, we have
  $\mathrm{var}(g_{t_{n}}(X_{1}))\leq\E g_{t_{n}}^{2}(X_{1})\leq M\E g_{t_{n}}(X_{1})\leq Kt_{n}^{d}$
  for large $n$ and some constant $K>0$. Define
  $\widetilde{W}_{n}:= \frac{1}{n} \sum_{i=1}^{n} ( g_{t_{n}}(X_{i}) - \E g_{t_{n}}(X_{i}))$
  and $\varepsilon_{n} = \varepsilon t_{n}^{d}$. By Bernstein's inequality, we
  have, for large $n$,
  \begin{equation*}
    \p ( \lvert \widetilde{W}_{n} \rvert > \varepsilon_{n}   )
     \leq 2 \exp \left( - \frac{1}{2} \frac{n \varepsilon_{n}^{2}}{  \mathrm{var}(g_{t_{n}}(X_{1}))
      + M\varepsilon_{n}/3}  \right)
     \leq 2 \exp \left( - \frac{1}{2} \frac{\varepsilon^{2}n t_{n}^{d}}{  K
      + M\varepsilon/3}  \right).
  \end{equation*}
  Since $nt_{n}^{d}=a^{d}n^{2r}$, the right-hand side is summable over $n$. By
  the Borel--Cantelli lemma, $\lvert t_{n}^{-d}\widetilde{W}_{n}\rvert\leq\varepsilon$
  for all large $n$, almost surely. Applying this with $\varepsilon=1/m$ for
  each large $m\in\mathbb{N}$, it follows that
  $t_{n}^{-d}\widetilde{W}_{n}\to0$ almost surely. Combining this with
  \eqref{eq:Integralrate} proves the lemma.
\end{proof}

Now, we prove Theorem~4.2.
\begin{proof}[Proof of Theorem~4.2]
  For simplicity, write $r=r(\gamma_{1},\gamma_{2})$ and
  $d=k_{+}/\gamma_{1}+(k-k_{+})/\gamma_{2}$, so that $r=(2+d)^{-1}$. For
  $b>0$, write $b_{n} = b n^{-r}$, and define
  \begin{equation*}
    W_{n}(x;b):=\left(1-\frac{C_{1}}{b_{n}}\lVert (x)_{\mathcal{V}+}\rVert^{\gamma_{1}}-\frac{C_{2}}{b_{n}}\lVert (x)_{\mathcal{V}-}\rVert^{\gamma_{2}}\right)_{+}^{2}.
  \end{equation*}
  Note that $b_{n} \to 0$ and $n b_{n}^{d}=b^{d}n^{2r}\to\infty$. First, we
  show that, for a constant $c_{1,2}^{\ast}\in(0,\infty)$ that does not
  depend on $b$,
  \begin{equation}
    \lim_{n\to\infty}\frac{1}{nb_{n}^{d}}\sum_{i=1}^{n}W_{n}(X_{i};b)=c_{1,2}^{\ast}\label{eq:rate-local-sum}
  \end{equation}
  almost surely, for each $b>0$. To do so, we apply Lemma~\ref{lem:local-sums}
  with $g_{t}(x)=(1-(C_{1}/t)\lVert (x)_{\mathcal{V}+}\rVert^{\gamma_{1}}-(C_{2}/t)\lVert (x)_{\mathcal{V}-}\rVert^{\gamma_{2}})_{+}^{2}$,
  which takes values in $[0,1]$. Note that if $g_{t}(x)>0$, then
  $\lVert (x)_{\mathcal{V}+}\rVert\leq(t/C_{1})^{1/\gamma_{1}}$ and
  $\lVert (x)_{\mathcal{V}-}\rVert\leq(t/C_{2})^{1/\gamma_{2}}$. Since
  $\lvert x_{v}\rvert\leq\lvert[(x)_{\mathcal{V}+}]_{v}\rvert+\lvert[(x)_{\mathcal{V}-}]_{v}\rvert$
  for every $v$, Assumption~2.1 and the triangle inequality give
  \[
    \lVert x\rVert\leq(t/C_{1})^{1/\gamma_{1}}+(t/C_{2})^{1/\gamma_{2}}=:\rho(t)\to0.
  \]
  It remains to compute the limit of $t^{-d}\int g_{t}(x)dx$.

  Let $\mathcal{O}$ be the collection of the $2^{k}$ orthants on
  $\mathbb{R}^{k}.$ Then, we can write
  \begin{equation}
      \int g_{t}(x)dx = \sum_{O\in\mathcal{O}}\int_{O}\left(1-\frac{C_{1}}{t}\left\lVert
          (x)_{\mathcal{V}+}\right\rVert
        ^{\gamma_{1}}-\frac{C_{2}}{t}\left\lVert
          (x)_{\mathcal{V}-}\right\rVert^{\gamma_{2}}\right)_{+}^{2}dx.
\label{eq:sum_orth}
  \end{equation}
  Now, consider an orthant $O$, and let $S\subseteq\mathcal{V}$ be the set of
  monotone coordinates that take positive values on $O$, and $m=\lvert S\rvert$.
  For $U\subseteq\{1,\dots,k\}$ and $z\in\mathbb{R}^{k}$, let $z^{U}$ denote the
  vector whose $v$th element is $z_{v}$ if $v\in U$ and $0$ otherwise, and
  write $\mathcal{V}^{c}=\{1,\dots,k\}\setminus\mathcal{V}$. For $x\in O$, we
  have $(x)_{\mathcal{V}+}=x^{S\cup\mathcal{V}^{c}}$ and
  $(x)_{\mathcal{V}-}=(-x)^{(\mathcal{V}\setminus S)\cup\mathcal{V}^{c}}$. Consider
  the change of variables $x=T_{t}z$, where $(T_{t}z)_{v}=t^{1/\gamma_{1}}z_{v}$
  for $v\in S$ and $(T_{t}z)_{v}=t^{1/\gamma_{2}}z_{v}$ for $v\notin S$, which
  maps $O$ onto itself and has Jacobian $t^{m/\gamma_{1}+(k-m)/\gamma_{2}}$.
  Since the norm is homogeneous, it gives
  \begin{multline*}
    \int_{O}\left(1-\frac{C_{1}}{t}\lVert (x)_{\mathcal{V}+}\rVert^{\gamma_{1}}-\frac{C_{2}}{t}\lVert (x)_{\mathcal{V}-}\rVert^{\gamma_{2}}\right)_{+}^{2}dx\\
    =t^{m/\gamma_{1}+(k-m)/\gamma_{2}}\int_{O}\Big(1-C_{1}\lVert
    z^{S}+t^{1/\gamma_{2}-1/\gamma_{1}}z^{\mathcal{V}^{c}}\rVert^{\gamma_{1}}-C_{2}\lVert (-z)^{(\mathcal{V}\setminus S)\cup\mathcal{V}^{c}}\rVert^{\gamma_{2}}\Big)_{+}^{2}dz.
  \end{multline*}
  No relabeling of the coordinates, and hence no modification of the norm, is
  needed. By Assumption~2.1, the integrand is bounded by the
  indicator of the set
  $\{z:C_{1}\lVert z^{S}\rVert^{\gamma_{1}}\leq1,\ C_{2}\lVert (-z)^{(\mathcal{V}\setminus S)\cup\mathcal{V}^{c}}\rVert^{\gamma_{2}}\leq1\}$,
  which is bounded since every coordinate is in $S$ or in
  $(\mathcal{V}\setminus S)\cup\mathcal{V}^{c}$. Hence, by Lebesgue's dominated
  convergence theorem, the integral in the last expression can be written as
  $c_{O}(C_{1},C_{2})+o(1)$, where
  \begin{align*}
    &c_{O}(C_{1},C_{2})  \\
    &:=\begin{cases}
      \int_{O}\left(1-C_{1}\lVert z^{S}\rVert
        ^{\gamma_{1}}-C_{2}\lVert (-z)^{(\mathcal{V}\setminus S)\cup\mathcal{V}^{c}}\rVert
        ^{\gamma_{2}}\right)_{+}^{2}dz &\text{if }\gamma_{1}>\gamma_{2}\\
      \int_{O}\left(1-C_{1}\lVert z^{S\cup\mathcal{V}^{c}}\rVert
        ^{\gamma_{1}}-C_{2}\lVert (-z)^{(\mathcal{V}\setminus S)\cup\mathcal{V}^{c}}\rVert
        ^{\gamma_{2}}\right)_{+}^{2}dz &\text{if }\gamma_{1}=\gamma_{2}.
    \end{cases}
  \end{align*}
  Moreover, note that $c_{O}(C_{1},C_{2})\in(0,\infty)$. If $\gamma_{1}>\gamma_{2}$, the
  exponent $m/\gamma_{1}+(k-m)/\gamma_{2}$ is uniquely minimized at
  $m=k_{+}$, so that the integrals that correspond to the orthants where
  $m=k_{+}$ determine the rate at which the entire integral goes to $0$. If
  $\gamma_{1}=\gamma_{2}$, note that the exponent of $t$ is always
  $k/\gamma_{1}$ and thus the integral is of the same order on all the
  orthants. Let $\mathcal{O_{+}}$ denote the collection of those orthants with
  $m=k_{+},$ and write
  $c_{+}(C_{1},C_{2})=\sum_{O\in\mathcal{O}_{+}}c_{O}(C_{1},C_{2})$ if
  $\gamma_{1}>\gamma_{2}$ and
  $c_{+}(C_{1},C_{2})=\sum_{O\in\mathcal{O}}c_{O}(C_{1},C_{2})$ if
  $\gamma_{1}=\gamma_{2}$. Then, it follows that
  \[
    t^{-d}\int g_{t}(x)dx\to c_{+}(C_{1},C_{2}),
  \]
  and Lemma~\ref{lem:local-sums} gives \eqref{eq:rate-local-sum} with
  $c_{1,2}^{\ast}=c_{+}(C_{1},C_{2})p_{X}(0)$, which does not depend on $b$. For
  example, if $k=2$, $\mathcal{V}=\{1\}$, the norm is Euclidean,
  $(\gamma_{1},\gamma_{2})=(1,1/2)$ and $C_{1}=C_{2}=1$, then the leading
  orthants are the two with $x_{1}>0$, and
  \[
    c_{+}(1,1)=\int_{0}^{1}\int_{-\infty}^{\infty}(1-z_{1}-\lvert z_{2}\rvert^{1/2})_{+}^{2}dz_{2}dz_{1}=\int_{0}^{1}\frac{(1-z_{1})^{4}}{3}dz_{1}=\frac{1}{15},
  \]
  since $\int_{-\infty}^{\infty}(a-\lvert z_{2}\rvert^{1/2})_{+}^{2}dz_{2}=4\int_{0}^{a}(a-v)^{2}v\,dv=a^{4}/3$
  for $a\geq0$. Similarly, $c_{+}=\pi/15$ if $\mathcal{V}=\emptyset$ and
  $c_{+}=\pi/24$ if $\mathcal{V}=\{1,2\}$.

  The same argument gives
  \begin{equation}
    \lim_{n\to\infty}\frac{1}{nb_{n}^{d}}\sum_{i=1}^{n}\widetilde{W}_{n}(X_{i};b)=c_{2,1}^{\ast},\label{eq:rate-local-sum-21}
  \end{equation}
  almost surely, where $\widetilde{W}_{n}(x;b)$ is defined as $W_{n}(x;b)$
  with $(\gamma_{1},C_{1})$ and $(\gamma_{2},C_{2})$ interchanged. The change
  of variables we previously used should be modified so that the coordinates
  in $\mathcal{V}\setminus S$ are scaled by $t^{1/\gamma_{1}}$ and the other
  coordinates by $t^{1/\gamma_{2}}$, and the leading orthants are those with
  $m=0$ when $\gamma_{1}>\gamma_{2}$. In fact, $c_{2,1}^{\ast}=c_{1,2}^{\ast}$. To see
  this, note that $\widetilde{W}_{n}(-x;b)=W_{n}(x;b)$, since
  $(-x)_{\mathcal{V}+}=(x)_{\mathcal{V}-}$ and
  $(-x)_{\mathcal{V}-}=(x)_{\mathcal{V}+}$, so that the two integrals in the
  definition of $c_{1,2}^{\ast}$ and $c_{2,1}^{\ast}$ coincide after the
  change of variables $x\mapsto-x$.

  Now, let $\Omega_{0}$ be the intersection of the probability one events on
  which \eqref{eq:rate-local-sum} and \eqref{eq:rate-local-sum-21} hold for
  each positive rational $b$, which also has probability one. We show that
  the conclusion of the theorem holds for any realization
  $\{x_{i}\}_{i=1}^{\infty}$ in $\Omega_{0}$. We prove the result for only the
  ordered modulus in the first display of the theorem, because the proof for
  the other one is essentially the same. Write
  \begin{align*}
    H_{n}(b) & =\omega^{-1}(b,\Lambda_{+,\mathcal{V}}(\gamma_{1},C_{1}),\Lambda_{+,\mathcal{V}}(\gamma_{2},C_{2})),\quad\text{and}\\
    \omega_{n}(\delta) & =\omega(\delta,\Lambda_{+,\mathcal{V}}(\gamma_{1},C_{1}),\Lambda_{+,\mathcal{V}}(\gamma_{2},C_{2})).
  \end{align*}
  By Proposition~4.1,
  \begin{equation*}
    H_{n}(bn^{-r})^{2}=\sum_{i=1}^{n}(b_{n}-C_{1}\lVert (x_{i})_{\mathcal{V}+}\rVert^{\gamma_{1}}-C_{2}\lVert (x_{i})_{\mathcal{V}-}\rVert^{\gamma_{2}})_{+}^{2}=nb_{n}^{2+d}\frac{1}{nb_{n}^{d}}\sum_{i=1}^{n}W_{n}(x_{i};b),
  \end{equation*}
  and $nb_{n}^{2+d}=b^{2+d}$. Hence, $H_{n}(bn^{-r})^{2}\to c_{1,2}^{\ast}b^{2+d}$
  for each positive rational $b$. Since $H_{n}$ is nondecreasing and the
  limit is continuous in $b$, this convergence holds for all $b>0$: for
  rational $b'<b<b''$, $H_{n}(b'n^{-r})\leq H_{n}(bn^{-r})\leq H_{n}(b''n^{-r})$,
  and we can take $b'$ and $b''$ arbitrarily close to $b$.

  Now, fix $\delta>0$ and let $b_{\delta}=(\delta^{2}/c_{1,2}^{\ast})^{r}$,
  so that $c_{1,2}^{\ast}b_{\delta}^{2+d}=\delta^{2}$. For any
  $0<b'<b_{\delta}<b''$, the previous display gives
  $H_{n}(b'n^{-r})<\delta<H_{n}(b''n^{-r})$ for large $n$. Since the inverse
  modulus is attained by Proposition~4.1, the ordered modulus
  satisfies $\omega_{n}(\delta)=\sup\{b\geq0:H_{n}(b)\leq\delta\}$, and since
  $H_{n}$ is nondecreasing, it follows that
  $b'\leq n^{r}\omega_{n}(\delta)\leq b''$ for large $n$. Taking
  $b'\uparrow b_{\delta}$ and $b''\downarrow b_{\delta}$ gives
  $n^{r}\omega_{n}(\delta)\to(\delta^{2}/c_{1,2}^{\ast})^{r}$, which is the
  desired result.
\end{proof}

The last two paragraphs of the proof only use the fact that
\eqref{eq:rate-local-sum} and \eqref{eq:rate-local-sum-21} hold along the
given sequence of design points, which is the condition in Remark~\ref{rem:deterministic-design}.

\begin{proof}[Proof of Theorem~4.3]
  Let $\mathcal{H}^{+}=\{x\in\mathbb{R}^{k}:x_{v}\geq0\ \forall v\in\mathcal{V}\}$
  and $\mathcal{H}^{-}=-\mathcal{H}^{+}$, so that
  $\mathcal{D}_{\mathcal{V}}^{\pm}=\mathcal{D}\cap\mathcal{H}^{\pm}$. By
  Proposition~\ref{prop:G-InvMod}, $\omega^{-1}(bn^{-r},\mathcal{F}_{1},\mathcal{G})^{2}$
  is the sum in the proof of Theorem~4.2, with the $i$th
  summand multiplied by $\mathds{1}\{x_{i}\in\mathcal{H}^{+}\}$. Hence, the
  proof goes through with $g_{t}(x)$ replaced by
  $g_{t}(x)\mathds{1}\{x\in\mathcal{H}^{+}\}$, which satisfies the conditions
  of Lemma~\ref{lem:local-sums}, and it remains to compute the limit of
  $t^{-d}\int_{\mathcal{H}^{+}}g_{t}(x)dx$. Up to a set of Lebesgue measure
  zero, $\mathcal{H}^{+}$ is the union of the orthants on which all the
  coordinates in $\mathcal{V}$ are positive, that is, of the orthants with
  $m=k_{+}$ in the notation of that proof. Hence,
  \[
    t^{-d}\int_{\mathcal{H}^{+}}g_{t}(x)dx\to c_{+}^{\mathcal{G}}(C_{1},C_{2}):=\sum_{O\in\mathcal{O}_{+}}c_{O}(C_{1},C_{2}),
  \]
  where $\mathcal{O}_{+}$ is the collection of the orthants with $m=k_{+}$,
  and $c_{O}(C_{1},C_{2})\in(0,\infty)$ is defined in that proof. If
  $\gamma_{1}>\gamma_{2}$, this is $c_{+}(C_{1},C_{2})$, since the orthants
  with $m<k_{+}$ do not contribute to the limit in
  Theorem~4.2. If $\gamma_{1}=\gamma_{2}$, it is the sum over
  a subset of the orthants that enter $c_{+}(C_{1},C_{2})$. This gives the
  first limit with $c_{\mathcal{G}}^{\ast}=p_{X}(0)c_{+}^{\mathcal{G}}(C_{1},C_{2})$.
  The second limit follows in the same way, with $\mathcal{H}^{-}$ and the
  orthants on which all the coordinates in $\mathcal{V}$ are negative, and
  the two constants coincide by the change of variables $x\mapsto-x$, which
  maps $\mathcal{H}^{+}$ onto $\mathcal{H}^{-}$, as in the proof of
  Theorem~4.2.

  For the last claim, note that $\mathcal{G}$ is convex and contains
  $\mathcal{F}_{1}$, and that the proof of Lemma~3.1 in
  Section~\ref{appdx-sec:Proofs_lemmas} uses only the ordered moduli
  $\omega(\cdot,\mathcal{F}_{1},\mathcal{F}_{1})$,
  $\omega(\cdot,\mathcal{F}_{1},\mathcal{G})$ and
  $\omega(\cdot,\mathcal{G},\mathcal{F}_{1})$, which are finite, and not
  $\omega(\cdot,\mathcal{G},\mathcal{G})$, which is infinite. Indeed, the CI
  $CI_{1,\alpha}^{\ast}$ in equation (16) of \cite{cai2004adaptation} is built
  from the estimators that correspond to the first three moduli only. Hence, the
  lemma applies with $\mathcal{F}_{j}=\mathcal{F}_{1}$ and
  $\mathcal{F}_{J}=\mathcal{G}$, and the limits above give the conclusion of
  Corollary~4.2 in the same way as
  Theorem~4.2 does.
\end{proof}

\begin{proof}[Proof of Corollary~4.3]
  (i) The ordered moduli are nondecreasing in each of their two parameter
  spaces, so that
  $\omega(\delta,\mathcal{F}_{1},\Lambda_{+,\mathcal{V}}(\gamma_{2},C_{2}))\leq\omega(\delta,\mathcal{F}_{1},\mathcal{C})\leq\omega(\delta,\mathcal{F}_{1},\mathcal{G}_{\mathcal{V}}(\gamma_{2},C_{2}'))$
  for all $\delta\geq0$, and likewise in the reverse order. The conditions of
  Theorem~4.2 also hold with $C_{2}'$ in place of $C_{2}$,
  since
  $\mathcal{F}_{1}\subseteq\Lambda_{+,\mathcal{V}}(\gamma_{2},C_{2})\subseteq\Lambda_{+,\mathcal{V}}(\gamma_{2},C_{2}')$.
  Hence, by Theorems~4.2 and 4.3, for almost all
  realizations of the design points,
  $n^{r}\omega(\delta,\mathcal{F}_{1},\mathcal{C})$ and
  $n^{r}\omega(\delta,\mathcal{C},\mathcal{F}_{1})$ are bounded away from zero
  and infinity for every $\delta>0$, and so is
  $n^{r}\omega_{+}(\delta,\mathcal{F}_{1},\mathcal{C})$. Since $\mathcal{C}$ is
  convex and contains $\mathcal{F}_{1}$, and since the proof of
  Lemma~3.1 uses only
  $\omega(\cdot,\mathcal{F}_{1},\mathcal{F}_{1})$,
  $\omega(\cdot,\mathcal{F}_{1},\mathcal{C})$ and
  $\omega(\cdot,\mathcal{C},\mathcal{F}_{1})$, which are finite for large $n$,
  the lemma applies with $\mathcal{F}_{j}=\mathcal{F}_{1}$ and
  $\mathcal{F}_{J}=\mathcal{C}$, and its bounds give the claim.

  (ii) The conditions of Theorem~4.2 hold with
  $(\gamma_{2},C_{2})$ replaced by $(\gamma_{2}',C')$, since
  $\gamma_{2}'<\gamma_{2}\leq\gamma_{1}$ and
  $\mathcal{F}_{1}\subseteq\Lambda_{+,\mathcal{V}}(\gamma_{2}',C')$. By
  Theorem~4.3 with these constants, for almost all realizations
  of the design points,
  $n^{r'}\omega(z_{1-\alpha},\mathcal{F}_{1},\mathcal{G}_{\mathcal{V}}(\gamma_{2}',C'))$
  converges to a positive limit. Since $\mathcal{C}$ is convex and contains
  $\mathcal{G}_{\mathcal{V}}(\gamma_{2}',C')$, which contains $\mathcal{F}_{1}$,
  the lower bound in Lemma~3.1, which does not require the
  moduli to be finite, applies with $\mathcal{F}_{j}=\mathcal{F}_{1}$ and
  $\mathcal{F}_{J}=\mathcal{C}$, and
  $\omega_{+}(z_{1-\alpha},\mathcal{F}_{1},\mathcal{C})\geq\omega(z_{1-\alpha},\mathcal{F}_{1},\mathcal{G}_{\mathcal{V}}(\gamma_{2}',C'))$.
  This gives the claim. Finally, $r'<r$, since $k-k_{+}>0$ and
  $\gamma_{2}'<\gamma_{2}$.

  (iii) First,
  $\mathcal{P}_{\mathcal{V}}(\gamma_{1},C_{1})\subseteq\mathcal{P}_{\mathcal{V}}(\gamma_{2},C_{2})$.
  Applying
  $\Lambda_{+,\mathcal{V}}(\gamma_{1},C_{1})\subseteq\Lambda_{+,\mathcal{V}}(\gamma_{2},C_{2})$
  to the functions of Lemma~4.1 with $(\gamma,C)=(\gamma_{1},C_{1})$
  at the points $(x)_{\mathcal{V}+}$ and $-(x)_{\mathcal{V}-}$, which are in
  $\mathcal{D}$ by (2), gives
  $C_{1}\lVert(x)_{\mathcal{V}\pm}\rVert^{\gamma_{1}}\leq C_{2}\lVert(x)_{\mathcal{V}\pm}\rVert^{\gamma_{2}}$.
  Hence, for $f\in\mathcal{P}_{\mathcal{V}}(\gamma_{1},C_{1})$,
  $f(x)-f(0)\leq f((x)_{\mathcal{V}+})-f(0)\leq C_{2}\lVert(x)_{\mathcal{V}+}\rVert^{\gamma_{2}}\leq C_{2}\lVert x\rVert^{\gamma_{2}}$
  by Assumption~2.1, and likewise from below. The same argument
  shows that
  $\mathcal{P}_{\mathcal{V}}(\gamma_{j},C_{j})\subseteq\mathcal{P}_{\mathcal{V}}(\gamma_{J},C_{J})$
  under the conditions of Section~4.1. Next, by
  Lemma~\ref{lem:G-envelope}(i),
  $\mathcal{P}_{\mathcal{V}}(\gamma,C)$ has the same bounds on $f(x)$ given
  $f(0)$ as $\Lambda_{+,\mathcal{V}}(\gamma,C)$. Hence, the lower bounds on the
  objective in the proofs of Proposition~4.1,
  Corollary~4.1 and Proposition~\ref{prop:G-InvMod}
  continue to hold when $\Lambda_{+,\mathcal{V}}(\gamma_{1},C_{1})$ and
  $\Lambda_{+,\mathcal{V}}(\gamma_{2},C_{2})$ are replaced by
  $\mathcal{P}_{\mathcal{V}}(\gamma_{1},C_{1})$ and
  $\mathcal{P}_{\mathcal{V}}(\gamma_{2},C_{2})$, and they are attained by the
  same functions, which lie in the original parameter spaces and thus in the
  pointwise ones. The inverse moduli, and with them the ordered moduli, are
  therefore unchanged, and so are the limits in Theorems~4.2
  and 4.3. Since the pointwise parameter spaces are convex and
  nested, Lemma~3.1 then gives the conclusion of
  Corollary~4.2 as before.
\end{proof}

\section{Adapting to an Interval of Exponents, Estimating the Noise Level, and Proof of Theorem~5.1}
\label{apdx:asy}
\subsection{Adapting to an Interval of Exponents}
\label{apdx:interval}
The bound in (9) uses only the concavity of the ordered
moduli, and Theorem~4.2 sharpens it. By the proof of
Theorem~3.1, the worst-case expected length is at
most $2g(\delta_{J})\,\omega_{+}(\delta_{J},\mathcal{F}_{j},\mathcal{F}_{J})$,
where $g(\delta)=\Phi(\delta)+\phi(\delta)/\delta$, and, by
Theorem~4.2,
\[
  \omega_{+}(\delta_{J},\mathcal{F}_{j},\mathcal{F}_{J})/\omega_{+}(z_{1-\alpha/2},\mathcal{F}_{j},\mathcal{F}_{J})\to(\delta_{J}/z_{1-\alpha/2})^{2r(\gamma_{j},\gamma_{J})}.
\]
Hence, for large $n$, the constant grows
only like $(\log J)^{r(\gamma_{j},\gamma_{J})}$. For example, for $\alpha=0.05$
and $r(\gamma_{j},\gamma_{J})=1/5$, it is about $2.2$ for $J=4$ and $2.7$ for
$J=1{,}000$, compared with $2.55$ and $4.14$ in (9). The
same holds for $CI_{\tau^{\ast}}^{\mathcal{J}}$ by
(14). Since the constant grows so slowly with $J$, we
can adapt to every exponent in an interval by letting the grid of
exponents become finer as $n$ grows.
\begin{corollary}[Adapting to an interval of exponents]
  \label{cor:interval}Suppose that $\sigma(\cdot)=1$, that
  Assumption~2.1 and (2) hold, that the diameter of
  $\mathcal{D}$ with respect to $\lVert\cdot\rVert$ is at most $1$, and that
  $\{X_{i}\}_{i=1}^{\infty}$ satisfies the conditions in
  Theorem~4.2. Let $0<\gamma_{\min}<1$ and $C>0$. For
  $n\geq3$, let $J_{n}=\lceil\log n\rceil+1$ and
  $\gamma_{n,j}=1-(j-1)(1-\gamma_{\min})/(J_{n}-1)$ for $j=1,\dots,J_{n}$, and
  let $CI_{n}$ denote the CI in (8) with $J=J_{n}$ and
  $\mathcal{F}_{j}=\Lambda_{+,\mathcal{V}}(\gamma_{n,j},C)$. Then, $CI_{n}$ has
  coverage at least $1-\alpha$ over $\Lambda_{+,\mathcal{V}}(\gamma_{\min},C)$,
  and, for almost all realizations $\{x_{i}\}_{i=1}^{\infty}$ of
  $\{X_{i}\}_{i=1}^{\infty}$, there exists $K<\infty$ such that, for all
  $n\geq3$,
  \[
    \sup_{\gamma\in[\gamma_{\min},1]}\,(\log\log n)^{-r(\gamma,\gamma_{\min})}\,n^{r(\gamma,\gamma_{\min})}\sup_{f\in\Lambda_{+,\mathcal{V}}(\gamma,C)}\E_{f}\mu(CI_{n})\leq K.
  \]
\end{corollary}
By Corollary~4.2, for each $\gamma\in[\gamma_{\min},1]$
and almost all realizations of the design points, no CI that has coverage at
least $1-\alpha$ over $\Lambda_{+,\mathcal{V}}(\gamma_{\min},C)$ can have
worst-case expected length over $\Lambda_{+,\mathcal{V}}(\gamma,C)$ of smaller
order than $n^{-r(\gamma,\gamma_{\min})}$. Hence, $CI_{n}$ adapts to every
H\"older exponent in $[\gamma_{\min},1]$, up to the factor
$(\log\log n)^{r(\gamma,\gamma_{\min})}$, which is at most
$(\log\log n)^{1/(2+k)}$. This factor comes from the Bonferroni correction
over $J_{n}$ parameter spaces. The critical value $z_{1-\alpha/(2J_{n})}$ grows
like $(\log\log n)^{1/2}$, but the moduli grow only like
$\delta^{2r(\gamma,\gamma_{\min})}$ in $\delta$ by Theorem~4.2.
The grid is fine enough that
$n^{r(\gamma,\gamma_{\min})-r(\gamma_{n,j},\gamma_{\min})}$ remains bounded
when $\gamma_{n,j}$ is the largest grid point that does not exceed $\gamma$. By
the first inequality in (14), the same conclusion holds
for $CI_{\tau^{\ast}}^{\mathcal{J}}$ constructed with the same grid. The proof
of the corollary is given below.

As discussed in Section~5.1, the factor
$(\log\log n)^{r(\gamma,\gamma_{\min})}$ is not intrinsic to the problem. The
CI of \cite{cai2004adaptation} for nested parameter spaces, whose constant
does not depend on the number of parameter spaces, attains the minimax rate of
adaptation over the whole interval without it.
\begin{corollary}
  \label{cor:CL-interval}Suppose that the conditions of
  Corollary~\ref{cor:interval} hold and that $\alpha\leq0.2$, and let
  $CI_{n}^{CL}$ denote the CI of \citet[Theorem~2]{cai2004adaptation} for the
  nested parameter spaces
  $\mathcal{F}_{1}\subseteq\cdots\subseteq\mathcal{F}_{J_{n}}$ in
  Corollary~\ref{cor:interval}. Then, $CI_{n}^{CL}$ has coverage at least
  $1-\alpha$ over $\Lambda_{+,\mathcal{V}}(\gamma_{\min},C)$, and, for almost
  all realizations $\{x_{i}\}_{i=1}^{\infty}$ of $\{X_{i}\}_{i=1}^{\infty}$,
  there exists $K<\infty$ such that, for all $n\geq3$,
  \[
    \sup_{\gamma\in[\gamma_{\min},1]}\,n^{r(\gamma,\gamma_{\min})}\sup_{f\in\Lambda_{+,\mathcal{V}}(\gamma,C)}\E_{f}\mu(CI_{n}^{CL})\leq K.
  \]
\end{corollary}
The proof is given below. Together with Corollary~4.2, this shows that, for each
$\gamma\in[\gamma_{\min},1]$, the minimax rate of adaptation is
$n^{-r(\gamma,\gamma_{\min})}$, and that the CI of \cite{cai2004adaptation}
attains it uniformly in $\gamma$, extending their result for univariate
monotone H\"older classes \citep[Section~3.2]{cai2004adaptation}. However,
as discussed in Section~6.1, their construction retains
few parameter spaces at reasonable sample sizes because the between class
moduli change slowly with $\gamma$. Our CI has a larger constant for very large $J$,
but it uses every parameter space in the grid.

\begin{proof}[Proof of Corollary~\ref{cor:interval}]
  Since the diameter of $\mathcal{D}$ is at most $1$,
  $\Lambda_{+,\mathcal{V}}(\gamma,C)\subseteq\Lambda_{+,\mathcal{V}}(\gamma',C)$
  whenever $\gamma'\leq\gamma$, so that
  $\mathcal{F}_{1}\subseteq\cdots\subseteq\mathcal{F}_{J_{n}}=\Lambda_{+,\mathcal{V}}(\gamma_{\min},C)$.
  By Proposition~4.1 and
  Corollary~4.1, the ordered moduli between
  $\mathcal{F}_{j}$ and $\mathcal{F}_{J_{n}}$ are attained, and the explicit
  inverse moduli show that the ordered moduli are differentiable with positive
  derivatives, as noted in Section~5.1. Hence, Theorem~3.1
  applies for every $n$, so that $CI_{n}$ has coverage at least $1-\alpha$
  over $\Lambda_{+,\mathcal{V}}(\gamma_{\min},C)$, and the display before
  \eqref{eq:contraction} in its proof gives
  \begin{equation}
    \sup_{f\in\mathcal{F}_{j}}\E_{f}\mu(CI_{n})\leq2g(\delta_{J_{n}})\,\omega_{+}(\delta_{J_{n}},\mathcal{F}_{j},\mathcal{F}_{J_{n}})\leq2g(z_{1-\alpha})\,\omega_{+}(\delta_{J_{n}},\mathcal{F}_{j},\mathcal{F}_{J_{n}}),\label{eq:interval-bon}
  \end{equation}
  for each $j\leq J_{n}$, where $\delta_{J}=z_{1-\alpha/(2J)}$ and
  $g(\delta)=\Phi(\delta)+\phi(\delta)/\delta$, which is decreasing. Since
  $1-\Phi(z)\leq\frac{1}{2}e^{-z^{2}/2}$ for $z\geq0$ and
  $J_{n}\leq\log n+2$, we have
  $\delta_{J_{n}}^{2}\leq2\log(J_{n}/\alpha)\leq K'\log\log n$ for all
  $n\geq3$ and some $K'\geq1$.

  Write $r_{n,j}=r(\gamma_{n,j},\gamma_{\min})$ and
  $d_{n,j}=1/r_{n,j}-2=k_{+}/\gamma_{n,j}+(k-k_{+})/\gamma_{\min}$, and let
  $\underline{r}=r(\gamma_{\min},\gamma_{\min})$, so that $r_{n,j}\geq \underline{r}>0$ and
  $d_{n,j}\geq k$. We show that there exists $K_{1}\geq1$ such that, for
  almost all realizations of the design points,
  \begin{equation}
    \omega_{+}(\delta_{J_{n}},\mathcal{F}_{j},\mathcal{F}_{J_{n}})\leq K_{1}\delta_{J_{n}}^{2r_{n,j}}n^{-r_{n,j}}\quad\text{for all }j\leq J_{n},\label{eq:interval-unif}
  \end{equation}
  for all large $n$. Fix $n$ and $j$, and let
  $a(x)=C\lVert (x)_{\mathcal{V}+}\rVert^{\gamma_{n,j}}+C\lVert (x)_{\mathcal{V}-}\rVert^{\gamma_{\min}}$
  and
  $\tilde{a}(x)=C\lVert (x)_{\mathcal{V}+}\rVert^{\gamma_{\min}}+C\lVert (x)_{\mathcal{V}-}\rVert^{\gamma_{n,j}}$.
  By Proposition~4.1 and Corollary~4.1,
  the inverse moduli of $\omega(\cdot,\mathcal{F}_{J_{n}},\mathcal{F}_{j})$
  and $\omega(\cdot,\mathcal{F}_{j},\mathcal{F}_{J_{n}})$ are
  $H(b)=(\sum_{i=1}^{n}(b-\tilde{a}(x_{i}))_{+}^{2})^{1/2}$ and
  $\tilde{H}(b)=(\sum_{i=1}^{n}(b-a(x_{i}))_{+}^{2})^{1/2}$, respectively, so
  that both ordered moduli at $\delta_{J_{n}}$, and hence
  $\omega_{+}(\delta_{J_{n}},\mathcal{F}_{j},\mathcal{F}_{J_{n}})$, are at most
  $b$ whenever $H(b)\geq\delta_{J_{n}}$ and $\tilde{H}(b)\geq\delta_{J_{n}}$.

  Let $\kappa_{1}=\max_{v}\lVert e_{v}\rVert$, so that
  $\lVert z\rVert\leq\kappa_{1}k\max_{v}\lvert z_{v}\rvert$. For
  $0<t_{2}\leq t_{1}$, let $R(t_{1},t_{2})$ be the set of $x$ with
  $-t_{2}\leq x_{v}\leq t_{1}$ for $v\in\mathcal{V}$ and
  $\lvert x_{v}\rvert\leq t_{2}$ for $v\notin\mathcal{V}$. For
  $x\in R(t_{1},t_{2})$, every element of $(x)_{\mathcal{V}+}$ is at most
  $t_{1}$ and every element of $(x)_{\mathcal{V}-}$ is at most $t_{2}$ in
  absolute value. Given $b\leq4C$, let
  $t_{1}=(b/(4C))^{1/\gamma_{n,j}}/(\kappa_{1}k)$ and
  $t_{2}=(b/(4C))^{1/\gamma_{\min}}/(\kappa_{1}k)$, so that $t_{2}\leq t_{1}$
  and $a(x)\leq b/2$ on $R(t_{1},t_{2})$. Hence,
  $\tilde{H}(b)^{2}\geq(b/2)^{2}N_{n}(R(t_{1},t_{2}))$, where $N_{n}(R)$
  denotes the number of $x_{i}$, $i\leq n$, in $R$. Moreover,
  $\mathrm{vol}(R(t_{1},t_{2}))\geq t_{1}^{k_{+}}t_{2}^{k-k_{+}}=(\kappa_{1}k)^{-k}(b/(4C))^{d_{n,j}}$.
  The same bounds hold for $H(b)$, with $R(t_{1},t_{2})$ replaced by
  $-R(t_{1},t_{2})$.

  Now let $b_{n,j}=K_{1}\delta_{J_{n}}^{2r_{n,j}}n^{-r_{n,j}}$ with
  $K_{1}\geq4C$, and let $R_{n,j}$ be the set $R(t_{1},t_{2})$ above with
  $b=b_{n,j}$, which is deterministic. Since $2r_{n,j}\leq1$ and
  $\delta_{J_{n}}\geq1$ for large $n$,
  $b_{n,j}\leq K_{1}\delta_{J_{n}}n^{-\underline{r}}\to0$ uniformly in $j$. Hence,
  for large $n$, $b_{n,j}\leq4C$, the sets $\pm R_{n,j}$ are contained in a
  neighborhood of $0$ on which $p_{X}\geq p_{X}(0)/2$, and
  \[
    n\p(X_{1}\in\pm R_{n,j})\geq c_{0}\Big(\frac{b_{n,j}}{4C}\Big)^{d_{n,j}}n\geq c_{0}\Big(\frac{K_{1}}{4C}\Big)^{k}\delta_{J_{n}}^{2r_{n,j}d_{n,j}}n^{2r_{n,j}}=:m_{n,j}\geq c_{0}n^{2\underline{r}},
  \]
  where $c_{0}=(p_{X}(0)/2)(\kappa_{1}k)^{-k}$, and we used
  $r_{n,j}(2+d_{n,j})=1$. By the Chernoff bound,
  \[
    \p\big(N_{n}(\pm R_{n,j})<m_{n,j}/2\big)\leq\exp(-m_{n,j}/8)\leq\exp(-c_{0}n^{2\underline{r}}/8).
  \]
  Since $\sum_{n}2J_{n}\exp(-c_{0}n^{2\underline{r}}/8)<\infty$, the Borel--Cantelli
  lemma shows that, almost surely, $N_{n}(\pm R_{n,j})\geq m_{n,j}/2$ for all
  $j\leq J_{n}$ and all large $n$. On this event, for large $n$ and all
  $j\leq J_{n}$,
  \[
    \tilde{H}(b_{n,j})^{2}\geq\frac{b_{n,j}^{2}}{4}\cdot\frac{m_{n,j}}{2}\geq\frac{c_{0}K_{1}^{2}}{8}\Big(\frac{K_{1}}{4C}\Big)^{k}\delta_{J_{n}}^{2},
  \]
  and the same bound holds for $H(b_{n,j})^{2}$. Taking $K_{1}$ large enough
  that $c_{0}K_{1}^{2}(K_{1}/(4C))^{k}\geq8$ proves \eqref{eq:interval-unif}.

  Finally, fix $\gamma\in[\gamma_{\min},1]$, and let $j$ be such that
  $\gamma_{n,j}$ is the largest grid point that does not exceed $\gamma$, so
  that $0\leq\gamma-\gamma_{n,j}\leq(1-\gamma_{\min})/\log n$ and
  $r_{n,j}\leq r(\gamma,\gamma_{\min})$. Since
  $\Lambda_{+,\mathcal{V}}(\gamma,C)\subseteq\mathcal{F}_{j}$ and
  $\delta_{J_{n}}\geq1$ for large $n$, \eqref{eq:interval-bon} and
  \eqref{eq:interval-unif} give
  \[
    n^{r(\gamma,\gamma_{\min})}\sup_{f\in\Lambda_{+,\mathcal{V}}(\gamma,C)}\E_{f}\mu(CI_{n})\leq2g(z_{1-\alpha})K_{1}K'(\log\log n)^{r(\gamma,\gamma_{\min})}\,n^{r(\gamma,\gamma_{\min})-r_{n,j}}
  \]
  for large $n$, where we used
  $\delta_{J_{n}}^{2r_{n,j}}\leq\delta_{J_{n}}^{2r(\gamma,\gamma_{\min})}\leq K'(\log\log n)^{r(\gamma,\gamma_{\min})}$.
  If $k_{+}\geq1$, then
  $r(\gamma_{1},\gamma_{\min})/\gamma_{1}=(2\gamma_{1}+k_{+}+(k-k_{+})\gamma_{1}/\gamma_{\min})^{-1}\leq1/k_{+}$,
  so that
  $\partial r(\gamma_{1},\gamma_{\min})/\partial\gamma_{1}=k_{+}\{r(\gamma_{1},\gamma_{\min})/\gamma_{1}\}^{2}\leq1/k_{+}$
  and $n^{r(\gamma,\gamma_{\min})-r_{n,j}}\leq e^{(1-\gamma_{\min})/k_{+}}$. If
  $k_{+}=0$, this factor equals one. This proves the claim for large $n$,
  uniformly in $\gamma$. For each fixed $n\geq3$, the left-hand side of the
  display in the corollary is finite, since the ordered moduli are finite as
  soon as there is a design point and the CI has finite expected length.
  Enlarging $K$ to cover the finitely many remaining $n\geq3$ completes the
  proof.
\end{proof}

\begin{proof}[Proof of Corollary~\ref{cor:CL-interval}]
  As in the proof of Lemma~3.1, affine estimators with the
  properties in equations (13)--(15) of \cite{cai2004adaptation} exist for
  every ordered pair of the parameter spaces
  $\mathcal{F}_{1}\subseteq\cdots\subseteq\mathcal{F}_{J_{n}}$, and the proof of
  their Theorem~2 uses the model only through these estimators. Hence, their
  Theorem~2, together with their Lemma~2 and Remark~4 for $\alpha\leq0.2$,
  shows that, for each $n$, $CI_{n}^{CL}$ has coverage at least $1-\alpha$ over
  $\mathcal{F}_{J_{n}}=\Lambda_{+,\mathcal{V}}(\gamma_{\min},C)$, and, by their
  equation (37),
  \[
    \sup_{f\in\mathcal{F}_{j}}\E_{f}\mu(CI_{n}^{CL})\leq16\,\omega_{+}(z_{1-\alpha/2},\mathcal{F}_{j},\mathcal{F}_{J_{n}})\quad\text{for all }j\leq J_{n},
  \]
  where the constant $16$ does not depend on $J_{n}$. The argument that proves
  \eqref{eq:interval-unif} in the proof of Corollary~\ref{cor:interval} applies
  with $\delta_{J_{n}}$ replaced by the fixed value $z_{1-\alpha/2}\geq1$, and it
  shows that, for almost all realizations of the design points,
  $\omega_{+}(z_{1-\alpha/2},\mathcal{F}_{j},\mathcal{F}_{J_{n}})\leq K_{1}z_{1-\alpha/2}^{2r_{n,j}}n^{-r_{n,j}}$
  for all $j\leq J_{n}$ and all large $n$. Now fix
  $\gamma\in[\gamma_{\min},1]$, and let $j$ be such that $\gamma_{n,j}$ is the
  largest grid point that does not exceed $\gamma$. Since
  $\Lambda_{+,\mathcal{V}}(\gamma,C)\subseteq\mathcal{F}_{j}$ and
  $z_{1-\alpha/2}^{2r_{n,j}}\leq z_{1-\alpha/2}$, we obtain, for large $n$,
  \[
    n^{r(\gamma,\gamma_{\min})}\sup_{f\in\Lambda_{+,\mathcal{V}}(\gamma,C)}\E_{f}\mu(CI_{n}^{CL})\leq16K_{1}z_{1-\alpha/2}\,n^{r(\gamma,\gamma_{\min})-r_{n,j}},
  \]
  and the last paragraph of the proof of Corollary~\ref{cor:interval} shows
  that $n^{r(\gamma,\gamma_{\min})-r_{n,j}}$ is bounded uniformly in $\gamma$.
  For each fixed $n\geq3$, the left-hand side is finite, as in the proof of
  Corollary~\ref{cor:interval}, and enlarging $K$ to cover the finitely many
  remaining $n$ completes the proof.
\end{proof}

\subsection{Estimating the Noise Level}
\label{apdx:sigma-hat}
Theorem~5.1 requires only that $\hat{\sigma}$ be consistent uniformly over
$\mathcal{F}_{J}$ and $\mathcal{U}_{\eta,M}$. For example, for $n\geq2$, let
$i^{\ast}\neq i$ denote the index of a nearest neighbor of $x_{i}$ among the
other design points $x_{l}$, $l\neq i$, with respect to the Euclidean norm.
Ties are broken by a rule that depends only on the design points. Let
$\hat{\sigma}^{2}=(2n)^{-1}\sum_{i=1}^{n}(y_{i}-y_{i^{\ast}})^{2}$.
Lemma~\ref{lem:nn-variance} shows
that $\hat{\sigma}^{2}$ is consistent for $\sigma^{2}$ uniformly over
$\Lambda(\gamma_{J},C_{J})\supseteq\mathcal{F}_{J}$ and over the distributions
of $u_{i}$ with $\E u_{i}=0$, $\E u_{i}^{2}=1$ and $\E u_{i}^{4}\leq M'$, for
any $M'<\infty$, if the design points are distinct. The design points are distinct
almost surely under the conditions of Theorem~4.2.
Hence, $\hat{\sigma}$ satisfies the condition of the theorem if the fourth
moments are also bounded over $\mathcal{U}_{\eta,M}$, because
$\lvert\hat{\sigma}-\sigma\rvert\leq\lvert\hat{\sigma}^{2}-\sigma^{2}\rvert/\sigma$.

The proof of the theorem uses the coverage class only through the maximum
biases of the ends, $K_{\tau}$ and the consistency of
$\hat{\sigma}$. Let $\gamma'\in(0,1]$ and $C'>0$, and suppose that
$\hat{\sigma}$ is consistent uniformly over
$\Lambda(\gamma_{J},C_{J})\cup\Lambda(\gamma',C')$ and $\mathcal{U}_{\eta,M}$,
as the nearest-neighbor estimator is when the fourth moments are bounded. Then, the same
conclusions hold for $CI_{\alpha}^{Bon,J}(\hat{\sigma})$ and
$CI_{\tau^{\ast}(\hat{\sigma})}^{\mathcal{J}}(\hat{\sigma})$ with coverage over
$\mathcal{P}_{\mathcal{V}}(\gamma_{J},C_{J})\cap\Lambda(\gamma',C')$, and for
AdaptMono with coverage over
$\mathcal{G}_{\mathcal{V}}(\gamma_{J},C_{J})\cap\Lambda(\gamma',C')$ and the
rates in Theorem~4.3. The space $\Lambda(\gamma',C')$ is needed
only for $\hat{\sigma}$, which is not necessarily consistent uniformly over
$\mathcal{P}_{\mathcal{V}}(\gamma_{J},C_{J})$ or
$\mathcal{G}_{\mathcal{V}}(\gamma_{J},C_{J})$. This is because the two spaces restrict
$f$ only at $0$ or only along the slice through $0$. The theorem is for a fixed
number $J$ of parameter spaces, and its coverage is asymptotic, uniformly over $\mathcal{F}_{J}$ and $\mathcal{U}_{\eta,M}$, while
the coverage results in Section~3 hold in finite
samples under Gaussian errors with a known variance function.

The following lemma concerns the nearest-neighbor variance estimator
defined above.
\begin{lemma}
  \label{lem:nn-variance}Suppose that $y_{i}=f(x_{i})+\sigma u_{i}$, where
  $\{u_{i}\}_{i=1}^{\infty}$ are i.i.d. with $\E u_{i}=0$, $\E u_{i}^{2}=1$
  and $\E u_{i}^{4}\leq M'$ for some $M'<\infty$, and let $\{x_{i}\}_{i=1}^{\infty}$ be a sequence
  of distinct points in the bounded set $\mathcal{D}$. For $n\geq2$, let
  $i^{\ast}\neq i$ denote the index of a nearest neighbor of $x_{i}$ among
  $\{x_{l}:l\leq n,\,l\neq i\}$ with respect to the Euclidean norm, with ties
  broken by a rule that depends only on the design points, and let
  $\hat{\sigma}^{2}=(2n)^{-1}\sum_{i=1}^{n}(y_{i}-y_{i^{\ast}})^{2}$. Then,
  for any $\gamma\in(0,1]$, $C>0$ and $\epsilon>0$,
  $\sup_{P}\sup_{f\in\Lambda(\gamma,C)}\p_{f,P}(\lvert\hat{\sigma}^{2}-\sigma^{2}\rvert>\epsilon)\to0$,
  where the first supremum is over the distributions $P$ of $u_{i}$ with such
  properties.
\end{lemma}
\begin{proof}
  Write $\Delta_{i}=f(x_{i})-f(x_{i^{\ast}})$ and
  $\rho_{i}=\lVert x_{i}-x_{i^{\ast}}\rVert_{2}$. Then,
  $\hat{\sigma}^{2}-\sigma^{2}=A_{n}+B_{n}+E_{n}$, where
  $A_{n}=(2n)^{-1}\sum_{i}\Delta_{i}^{2}$,
  $B_{n}=n^{-1}\sigma\sum_{i}\Delta_{i}(u_{i}-u_{i^{\ast}})$ and
  $E_{n}=\sigma^{2}\{(2n)^{-1}\sum_{i}(u_{i}-u_{i^{\ast}})^{2}-1\}$. We use
  that, since the points are distinct, each $x_{l}$ is the chosen nearest
  neighbor of at most $\kappa$ other points, where $\kappa$ depends only on
  $k$. To see this, note that if $x_{l}$ is the chosen nearest neighbor of two
  distinct points $x_{i}$
  and $x_{i'}$, then $\lVert x_{i}-x_{i'}\rVert_{2}$ is the largest side of
  the triangle with vertices $x_{l},x_{i},x_{i'}$, so that the angle at
  $x_{l}$ is at least $60$ degrees. The number of directions with pairwise
  angles of at least $60$ degrees is bounded by a constant that depends only on
  $k$.

  First, $\lvert\Delta_{i}\rvert\leq C\kappa_{2}^{\gamma}\rho_{i}^{\gamma}$,
  where $\kappa_{2}$ is such that $\lVert z\rVert\leq\kappa_{2}\lVert z\rVert_{2}$.
  We show that $n^{-1}\sum_{i}\rho_{i}^{2\gamma}\to0$. Take $h>0$ and
  partition a cube that contains $\mathcal{D}$ into $N_{h}$ cubes with
  diameter $h$. If the cube that contains $x_{i}$ contains another $x_{l}$,
  $l\leq n$, then $\rho_{i}\leq h$, and at most $N_{h}$ indices $i\leq n$ do
  not have this property. Hence,
  $n^{-1}\sum_{i}\rho_{i}^{2\gamma}\leq h^{2\gamma}+N_{h}\bar{\rho}^{2\gamma}/n$,
  where $\bar{\rho}$ is the Euclidean diameter of $\mathcal{D}$, and the
  claim follows since $h$ is arbitrary. Therefore,
  $\sup_{f\in\Lambda(\gamma,C)}A_{n}\to0$. Next, $\E B_{n}=0$ and
  $\sum_{i}\Delta_{i}(u_{i}-u_{i^{\ast}})=\sum_{l}w_{l}u_{l}$ with
  $w_{l}=\Delta_{l}-\sum_{i:i^{\ast}=l}\Delta_{i}$, so that, by the
  Cauchy--Schwarz inequality,
  $\mathrm{var}(B_{n})=n^{-2}\sigma^{2}\sum_{l}w_{l}^{2}\leq n^{-2}\sigma^{2}(1+\kappa)\sum_{l}(\Delta_{l}^{2}+\sum_{i:i^{\ast}=l}\Delta_{i}^{2})\leq4(1+\kappa)\sigma^{2}A_{n}/n$.
  Finally, $\E E_{n}=0$, since $i^{\ast}\neq i$. Each summand
  $(u_{i}-u_{i^{\ast}})^{2}$ is correlated with at most $2+2\kappa$ other
  summands, namely those whose index pair intersects $\{i,i^{\ast}\}$, and has
  variance at most $16\E u_{i}^{4}$, so that
  $\mathrm{var}(E_{n})\leq\sigma^{4}(3+2\kappa)16\E u_{i}^{4}/(4n)$. Since
  neither bound depends on $f$, and both depend on the distribution of $u_{i}$
  only through $\E u_{i}^{4}\leq M'$, Chebyshev's inequality completes the
  proof.
\end{proof}

\subsection{Proof of Theorem~5.1}
\label{apdx:asy-proof}
Throughout the proof, we fix a
realization of the design points for which the conclusions of
Theorem~4.2 and Lemma~\ref{lem:tau-limit} hold, together with
the countably many almost sure statements used below. Here, $\p$ and $\E$
refer
to the distribution of $\{u_{i}\}_{i=1}^{\infty}$ when the distribution of
$u_{i}$ is $P\in\mathcal{U}_{\eta,M}$. All the bounds below hold uniformly over
$P\in\mathcal{U}_{\eta,M}$. The bounds use the distribution of $u_{i}$ only through
$\E u_{i}^{2}=1$ in Step~2, through the central limit theorem in Step~3, which
holds uniformly over $\mathcal{U}_{\eta,M}$, and through the uniform
consistency of $\hat{\sigma}$. The coverage is uniform over
$f\in\mathcal{F}_{J}$ because the noise in the ends of the CIs, the processes
$S_{q}(\cdot)$ below, does not depend on $f$, while the biases enter only
through deterministic bounds.

\begin{proof}[Proof of Theorem~5.1]
  We use the notation of the proof of Lemma~\ref{lem:tau-limit}. For
  $q=(\pm,j)\in\mathcal{Q}$ and $t>0$, let $b_{q}(t)$ be the ordered modulus
  in that proof at $\delta=t$, computed with $\sigma(\cdot)=1$, so that
  $\sum_{i=1}^{n}D_{q,t}(x_{i})^{2}=t^{2}$, and let
  $\tilde{c}_{q}(t)=(D_{q,t}(x_{i})/t)_{i=1}^{n}$, which is a unit vector.
  Let $S_{q}(t)=\tilde{c}_{q}(t)'u$ if $q=(-,j)$ and
  $S_{q}(t)=-\tilde{c}_{q}(t)'u$ if $q=(+,j)$, where
  $u=(u_{1},\dots,u_{n})'$. Computing the CI as if $\sigma(\cdot)=s$ at
  $\delta$ amounts to solving the ordered moduli with $\sigma(\cdot)=1$ at
  $t=s\delta$, since
  $\sum_{i}((g-f)(x_{i})/s)^{2}\leq\delta^{2}$ if and only if
  $\sum_{i}(g-f)(x_{i})^{2}\leq(s\delta)^{2}$.

  \medskip
  \noindent\textbf{Step~1: Inclusion.} Fix $s>0$ and
  $\tau\in[\alpha/(2J),\alpha]$, let $\delta=\delta_{\tau}$ and $t=s\delta$,
  and let $\hat{c}_{\tau}^{\ell,j}(s)$ and $\hat{c}_{\tau}^{u,j}(s)$ denote the
  ends computed as if $\sigma(\cdot)=s$. As in the proof of
  Lemma~3.3, the maximum bias of the estimator over
  $\mathcal{F}_{J}$, given in Lemma A.1 of \cite{armstrong2018optimal}, does
  not depend on the distribution of the errors, while its random part is now
  $(s_{j}^{-}/\delta)\sum_{i}D_{q,t}(x_{i})\sigma u_{i}/s^{2}=s_{j}^{-}(\sigma/s)S_{q}(t)$
  for $q=(-,j)$, where $s_{j}^{-}=\omega'(\delta,\mathcal{F}_{J},\mathcal{F}_{j})>0$
  is computed as if $\sigma(\cdot)=s$. Hence, for every $f\in\mathcal{F}_{J}$,
  $\hat{c}_{\tau}^{\ell,j}(s)-f(0)\leq s_{j}^{-}\{(\sigma/s)S_{q}(t)-\delta\}$,
  and similarly for the upper end. Therefore, for every $f\in\mathcal{F}_{J}$,
  \begin{equation}
    \{f(0)\notin\cap_{j=1}^{J}[\hat{c}_{\tau}^{\ell,j}(s),\hat{c}_{\tau}^{u,j}(s)]\}\subseteq\big\{\max_{q\in\mathcal{Q}}S_{q}(s\delta_{\tau})>s\delta_{\tau}/\sigma\big\}.\label{eq:asy-inclusion}
  \end{equation}
  The convention for an empty intersection does not affect this inclusion.
  Since \eqref{eq:asy-inclusion} holds for every realization of $u$, it also
  holds when $s$ and $\tau$ are replaced by functions of the data, such as
  $\hat{\sigma}$ and $\tau^{\ast}(\hat{\sigma})$.

  \medskip
  \noindent\textbf{Step~2: Continuity in $t$.} Let
  $t_{\mathrm{lo}}=\sigma z_{1-\alpha}/2$ and
  $t_{\mathrm{hi}}=2\sigma z_{1-\alpha/(2J)}$. We show that there exists
  $K_{0}$ such that, for large $n$,
  \begin{equation}
    \E\sup_{t,t'\in I}\lvert S_{q}(t)-S_{q}(t')\rvert\leq K_{0}\lvert I\rvert\quad\text{for every interval }I\subseteq[t_{\mathrm{lo}},t_{\mathrm{hi}}]\text{ and }q\in\mathcal{Q}.\label{eq:asy-equicont}
  \end{equation}
  Since $\sum_{i}(b-a_{q}(x_{i}))_{+}^{2}$ is continuously differentiable in
  $b$, with a positive derivative where it is positive, $b_{q}(t)$ is
  continuously differentiable with
  $b_{q}'(t)=t/\sum_{i}D_{q,t}(x_{i})$. Hence, $\tilde{c}_{q}(\cdot)$ is
  Lipschitz on $[t_{\mathrm{lo}},t_{\mathrm{hi}}]$, with derivative
  $\tilde{c}_{q}'(t)=(\mathds{1}\{a_{q}(x_{i})<b_{q}(t)\})_{i=1}^{n}b_{q}'(t)/t-\tilde{c}_{q}(t)/t$
  almost everywhere, and
  \[
    \lVert\tilde{c}_{q}'(t)\rVert\leq\frac{N_{n}(\{x:a_{q}(x)<b_{q}(t)\})^{1/2}}{\sum_{i}D_{q,t}(x_{i})}+\frac{1}{t}.
  \]
  By \eqref{eq:tau-b-limit}, $n^{r_{q}}b_{q}(t)\to\beta_{q}(t)$, and since
  $b_{q}(\cdot)$ is nondecreasing and $\beta_{q}(\cdot)$ is continuous, the
  convergence is uniform over $t\in[t_{\mathrm{lo}},t_{\mathrm{hi}}]$.
  Applying Lemma~\ref{lem:local-sums} on each orthant after the change of
  variables in the proof of Theorem~4.2, with
  $\mathds{1}\{a_{q}(x)/b<1\}$ and $(1-a_{q}(x)/b)_{+}$ in place of
  $(1-a_{q}(x)/b)_{+}^{2}$ and with $b=\beta n^{-r_{q}}$ for each rational
  $\beta>0$, we obtain the limits of the three sums below. Since the sums are
  nondecreasing in $b$ and their limits are continuous in $\beta$, these
  limits hold uniformly over $\beta$ in compact subsets of $(0,\infty)$. Hence,
  there are constants $0<\lambda_{q,0},\lambda_{q,1},\lambda_{q,2}<\infty$ such
  that
  $N_{n}(\{a_{q}<b_{q}(t)\})\sim p_{X}(0)\lambda_{q,0}nb_{q}(t)^{d_{q}}$,
  $\sum_{i}D_{q,t}(x_{i})\sim p_{X}(0)\lambda_{q,1}nb_{q}(t)^{d_{q}+1}$ and
  $t^{2}=\sum_{i}D_{q,t}(x_{i})^{2}\sim p_{X}(0)\lambda_{q,2}nb_{q}(t)^{d_{q}+2}$,
  uniformly over $t\in[t_{\mathrm{lo}},t_{\mathrm{hi}}]$. Hence, for large
  $n$, $\lVert\tilde{c}_{q}'(t)\rVert\leq K_{0}$ on
  $[t_{\mathrm{lo}},t_{\mathrm{hi}}]$ with
  $K_{0}=\{2(\lambda_{q,0}\lambda_{q,2})^{1/2}/\lambda_{q,1}+1\}/t_{\mathrm{lo}}$
  (maximized over $q$). Since
  $S_{q}(t)-S_{q}(t')=\pm\int_{t'}^{t}\tilde{c}_{q}'(v)'u\,dv$, we have
  $\sup_{t,t'\in I}\lvert S_{q}(t)-S_{q}(t')\rvert\leq\int_{I}\lvert\tilde{c}_{q}'(v)'u\rvert dv$,
  and \eqref{eq:asy-equicont} follows from
  $\E\lvert\tilde{c}_{q}'(v)'u\rvert\leq\lVert\tilde{c}_{q}'(v)\rVert$.

  \medskip
  \noindent\textbf{Step~3: Central limit theorem.} Fix
  $t\in[t_{\mathrm{lo}},t_{\mathrm{hi}}]$, and let $\Sigma_{n}(t)$ be the
  covariance matrix of $(S_{q}(t))_{q\in\mathcal{Q}}$, which equals the
  covariance matrix in (10) with $\sigma(\cdot)=1$ and
  $\delta_{\tau}=t$, and which does not depend on the distribution of $u_{i}$.
  Let $P_{n}\in\mathcal{U}_{\eta,M}$ be an arbitrary sequence of distributions
  of $u_{i}$. For a fixed vector $\lambda$, write
  $\lambda'(S_{q}(t))_{q\in\mathcal{Q}}=\sum_{i}a_{i}u_{i}$. Since
  $\max_{i}\lvert\tilde{c}_{q,i}(t)\rvert\leq b_{q}(t)/t\to0$, we have
  $\max_{i}\lvert a_{i}\rvert\to0$, while
  $\sum_{i}a_{i}^{2}=\lambda'\Sigma_{n}(t)\lambda$ is bounded, so that
  \[
    \sum_{i}\lvert a_{i}\rvert^{2+\eta}\E_{P_{n}}\lvert u_{i}\rvert^{2+\eta}\leq M\max_{i}\lvert a_{i}\rvert^{\eta}\sum_{i}a_{i}^{2}\to0.
  \]
  Hence, by the Lyapunov central limit theorem and the Cram\'er--Wold device,
  $(S_{q}(t))_{q}\rightsquigarrow N(0,\Sigma)$ under $P_{n}$ along any
  subsequence along which $\Sigma_{n}(t)\to\Sigma$. (If $\lambda'\Sigma\lambda=0$,
  then $\mathrm{var}(\sum_{i}a_{i}u_{i})=\lambda'\Sigma_{n}(t)\lambda\to0$ along
  the subsequence, and $\sum_{i}a_{i}u_{i}\to0$ in probability by Chebyshev's
  inequality.) A centered Gaussian vector with unit
  variances has continuous marginal distributions, so that its (multivariate)
  distribution function is continuous, and the set of correlation matrices is
  compact. Therefore, considering subsequences and applying the multivariate
  version of P\'olya's theorem, which gives uniform convergence of
  distribution functions to a continuous limit, we obtain, with
  $\mathcal{Q}^{-}=\{-\}\times\{1,\dots,J\}$ and
  $\mathcal{Q}^{+}=\{+\}\times\{1,\dots,J\}$,
  \begin{equation}
    \sup_{z_{1},z_{2}\in\mathbb{R}}\big\lvert\p(\max_{q\in\mathcal{Q}^{-}}S_{q}(t)>z_{1}\text{ or }\max_{q\in\mathcal{Q}^{+}}S_{q}(t)>z_{2})-\pi(\Sigma_{n}(t),z_{1},z_{2})\big\rvert\to0,\label{eq:asy-clt}
  \end{equation}
  where $\pi$ is defined in Step~5 of the proof of Lemma~\ref{lem:tau-limit}.
  The same argument gives the analogous statements for each $S_{q}(t)$
  separately and for $\max_{q\in\mathcal{Q}}S_{q}(t)$, which we also refer
  to as \eqref{eq:asy-clt}. Since the sequence $P_{n}$ is arbitrary,
  \eqref{eq:asy-clt} holds uniformly over $P\in\mathcal{U}_{\eta,M}$:
  otherwise, there would be a sequence $P_{n}\in\mathcal{U}_{\eta,M}$ and a
  subsequence along which the left-hand side of \eqref{eq:asy-clt} under
  $P_{n}$ is bounded away from zero, which contradicts the above.

  \medskip
  \noindent\textbf{Step~4: Coverage of $CI_{\alpha}^{Bon,J}(\hat{\sigma})$.}
  Let $t_{0}=\sigma\delta_{J}$, with $\delta_{J}=z_{1-\alpha/(2J)}$, and for
  $0<\epsilon<1/2$, let $I_{\epsilon}=[(1-\epsilon)t_{0},(1+\epsilon)t_{0}]$
  and $\mathcal{E}_{\epsilon}=\{\lvert\hat{\sigma}-\sigma\rvert\leq\epsilon\sigma\}$.
  On $\mathcal{E}_{\epsilon}$, $\hat{t}:=\hat{\sigma}\delta_{J}\in I_{\epsilon}$,
  so that, by \eqref{eq:asy-inclusion}, for any $\vartheta>0$ and
  $f\in\mathcal{F}_{J}$,
  \begin{align*}
    \{f(0)\notin CI_{\alpha}^{Bon,J}(\hat{\sigma})\}\cap\mathcal{E}_{\epsilon} & \subseteq\big\{\max_{q}S_{q}(t_{0})>(1-\epsilon)\delta_{J}-\vartheta\big\}\\
     & \quad\cup\big\{\max_{q}\sup_{t\in I_{\epsilon}}\lvert S_{q}(t)-S_{q}(t_{0})\rvert>\vartheta\big\}.
  \end{align*}
  By the union bound, \eqref{eq:asy-clt} applied to each $S_{q}(t_{0})$,
  \eqref{eq:asy-equicont} and Markov's inequality,
  \[
    \limsup_{n\to\infty}\sup_{f\in\mathcal{F}_{J}}\p_{f}(f(0)\notin CI_{\alpha}^{Bon,J}(\hat{\sigma}))\leq2J\{1-\Phi((1-\epsilon)\delta_{J}-\vartheta)\}+\frac{2JK_{0}\cdot2\epsilon t_{0}}{\vartheta},
  \]
  where we used $\sup_{f}\p_{f}(\mathcal{E}_{\epsilon}^{c})\to0$. Letting
  $\epsilon\to0$ and then $\vartheta\to0$ gives the bound
  $2J\{1-\Phi(\delta_{J})\}=\alpha$.

  \medskip
  \noindent\textbf{Step~5: Coverage of $CI_{\tau^{\ast}(\hat{\sigma})}^{\mathcal{J}}(\hat{\sigma})$.}
  We first refine \eqref{eq:asy-inclusion} to account for the bias slacks.
  For $f\in\mathcal{F}_{J}$, $t>0$ and $j,\ell\leq J$, let
  \begin{align*}
    A_{j}(f;t) & =\sum_{i=1}^{n}D_{(-,j),t}(x_{i})\{h_{+}(x_{i})-f(x_{i})+f(0)\},\quad\text{and}\\
    B_{\ell}(f;t) & =\sum_{i=1}^{n}D_{(+,\ell),t}(x_{i})\{f(x_{i})-f(0)-h_{-}(x_{i})\},
  \end{align*}
  with $h_{\pm}$ as in Lemma~5.1. Both are nonnegative by
  Lemma~4.3, and both are nondecreasing in $t$, since
  $D_{q,t}(x_{i})=(b_{q}(t)-a_{q}(x_{i}))_{+}$ is. Fix $s>0$ and $\tau$, and
  let $t=s\delta_{\tau}$. By Corollary~5.1 applied with
  $\sigma(\cdot)=1$ at $\delta=t$, the lower end that directs power to
  $\mathcal{F}_{j}$, computed as if $\sigma(\cdot)=s$, fails to cover $f(0)$
  if and only if $\sigma S_{(-,j)}(t)>t+A_{j}(f;t)/t$, and the upper end that
  directs power to $\mathcal{F}_{\ell}$ fails if and only if
  $\sigma S_{(+,\ell)}(t)>t+B_{\ell}(f;t)/t$. Moreover, $A_{j}(f;t)/t$ and
  $B_{\ell}(f;t)/t$ are the bias slacks $a_{j}(f)$ and $c_{\ell}(f)$
  computed with $\sigma(\cdot)=1$ at $\delta=t$, so that, by
  Lemma~5.1,
  \begin{equation}
    \frac{A_{j}(f;t)+B_{\ell}(f;t)}{t}\geq K^{(1)}(t)\quad\text{for all }j,\ell\leq J,\label{eq:asy-slack}
  \end{equation}
  where $K^{(1)}(t)$ denotes $K_{\tau}$ computed with $\sigma(\cdot)=1$ at
  $\delta_{\tau}=t$. Computing $K_{\tau}$ as if $\sigma(\cdot)=s$ gives
  $K^{(1)}(t)/s$, and the covariance matrix of $(V(\tau)',W(\tau)')'$
  computed as if $\sigma(\cdot)=s$ is $\Sigma_{n}(t)$. Hence, the left-hand
  side of (13) computed as if $\sigma(\cdot)=s$ is
  \[
    p_{n}(s,\tau)=\sup_{0\leq v\leq K^{(1)}(s\delta_{\tau})}\pi\big(\Sigma_{n}(s\delta_{\tau}),\delta_{\tau}+v/s,\delta_{\tau}+(K^{(1)}(s\delta_{\tau})-v)/s\big).
  \]
  By Lemma~\ref{lem:tau-limit}(i) and (ii), applied with $t$ in place of
  $\delta$, $\Sigma_{n}(t)\to\Sigma_{\infty}$ and $K^{(1)}(t)\to kt$
  uniformly over $t\in[t_{\mathrm{lo}},t_{\mathrm{hi}}]$, where $k=0$ or
  $k=k_{\infty}$. Since $K^{(1)}(s\delta)/s\to k\delta$, it follows that
  $p_{n}(s,\tau)\to p_{\infty}(\tau)$ uniformly over $s\in[\sigma/2,2\sigma]$
  and $\tau\in[\alpha/(2J),\alpha]$, and the argument in Step~5 of the proof
  of Lemma~\ref{lem:tau-limit} shows that
  $\sup_{s\in[\sigma/2,2\sigma]}\lvert\tau^{\ast}(s)-\tau_{\infty}\rvert\to0$.

  Let $t_{\infty}=\sigma z_{1-\tau_{\infty}}$, let $0<\epsilon<1/6$, and let
  $I=[t_{1},t_{2}]$ with $t_{1}=(1-3\epsilon)t_{\infty}$ and
  $t_{2}=(1+3\epsilon)t_{\infty}$, so that
  $I\subseteq[t_{\mathrm{lo}},t_{\mathrm{hi}}]$. On $\mathcal{E}_{\epsilon}$,
  $\hat{t}:=\hat{\sigma}\delta_{\tau^{\ast}(\hat{\sigma})}$ lies in $I$ for
  large $n$. Fix $f\in\mathcal{F}_{J}$, and let $\rho=t_{1}/t_{2}$. For
  $t\in I$, monotonicity gives
  $A_{j}(f;t)/t\geq A_{j}(f;t_{1})/t_{2}=\rho A_{j}(f;t_{1})/t_{1}$, and
  likewise for $B_{\ell}$. Let $v_{f}=\rho\min_{j\leq J}A_{j}(f;t_{1})/t_{1}$,
  which does not depend on the data. By \eqref{eq:asy-slack} at $t_{1}$,
  $\rho B_{\ell}(f;t_{1})/t_{1}\geq\rho K^{(1)}(t_{1})-v_{f}$ for all $\ell$.
  Moreover, $v_{f}$ is bounded uniformly over $f$: since the two $J$th ends
  use the same weights, $D_{(-,J),t}=D_{(+,J),t}$, so that
  $v_{f}\leq\rho\{A_{J}(f;t_{1})+B_{J}(f;t_{1})\}/t_{1}=\rho t_{1}^{-1}\sum_{i}D_{(-,J),t_{1}}(x_{i})\{h_{+}(x_{i})-h_{-}(x_{i})\}$,
  which does not depend on $f$ and is bounded for large $n$ by the argument in
  Step~4 of the proof of Lemma~\ref{lem:tau-limit}.
  Hence, for any $\vartheta>0$, on $\mathcal{E}_{\epsilon}$ and for large $n$, the
  event $\{f(0)\notin CI_{\tau^{\ast}(\hat{\sigma})}^{\mathcal{J}}(\hat{\sigma})\}$
  is contained in the union of
  \[
    \Big\{\max_{q\in\mathcal{Q}^{-}}\sigma S_{q}(t_{1})>t_{1}+v_{f}-\sigma\vartheta\ \text{ or }\max_{q\in\mathcal{Q}^{+}}\sigma S_{q}(t_{1})>t_{1}+(\rho K^{(1)}(t_{1})-v_{f})_{+}-\sigma\vartheta\Big\}
  \]
  and $\{\max_{q}\sup_{t\in I}\lvert S_{q}(t)-S_{q}(t_{1})\rvert>\vartheta\}$. By
  \eqref{eq:asy-equicont} and Markov's inequality, the probability of the
  latter event is at most $2JK_{0}\cdot6\epsilon t_{\infty}/\vartheta$. By
  \eqref{eq:asy-clt}, $\Sigma_{n}(t_{1})\to\Sigma_{\infty}$,
  $K^{(1)}(t_{1})\to kt_{1}$ and the uniform continuity of $\pi$ on compact
  sets, the probability of the former event is at most
  \[
    \sup_{v\geq0}\pi\big(\Sigma_{\infty},(t_{1}+v)/\sigma-\vartheta,(t_{1}+(\rho kt_{1}-v)_{+})/\sigma-\vartheta\big)+o(1),
  \]
  uniformly over $f\in\mathcal{F}_{J}$. Together with
  $\sup_{f}\p_{f}(\mathcal{E}_{\epsilon}^{c})\to0$, this gives a bound on
  the limit superior of the largest noncoverage probability of the CI over
  $\mathcal{F}_{J}$.
  Letting $\epsilon\to0$, so that $t_{1}\to t_{\infty}$ and $\rho\to1$, and
  then $\vartheta\to0$, and using the continuity of $\pi$, the bound becomes
  \[
    \sup_{v\geq0}\pi\big(\Sigma_{\infty},z_{1-\tau_{\infty}}+v/\sigma,z_{1-\tau_{\infty}}+(kz_{1-\tau_{\infty}}-v/\sigma)_{+}\big)=p_{\infty}(\tau_{\infty})=\alpha,
  \]
  where the first equality uses $t_{\infty}/\sigma=z_{1-\tau_{\infty}}$ and
  the fact that the expression for $v/\sigma>kz_{1-\tau_{\infty}}$ is at most
  its value at $v/\sigma=kz_{1-\tau_{\infty}}$.

  Consider the CI that we compute in
  Sections~6 and 7, in which
  $\tau^{\ast}(s)$ is the largest point of a fixed finite grid
  $G\subset[\alpha/(2J),\alpha]$ that contains $\alpha/(2J)$ at which an upper
  bound $\bar{p}_{n}(s,\tau)\geq p_{n}(s,\tau)$ is at most $\alpha$. The same
  argument applies to this CI, provided
  that $\bar{p}_{n}(s,\tau)$ converges, uniformly over $s\in[\sigma/2,2\sigma]$
  and $\tau\in G$, to a limit $\bar{p}_{\infty}(\tau)$ with
  $\bar{p}_{\infty}(g)\neq\alpha$ for all $g\in G$. Then, for large $n$,
  $\tau^{\ast}(s)=g^{\ast}:=\max\{g\in G:\bar{p}_{\infty}(g)<\alpha\}$ for
  all $s\in[\sigma/2,2\sigma]$, and the argument above goes through with
  $\tau_{\infty}$ replaced by $g^{\ast}$, giving the bound
  $p_{\infty}(g^{\ast})\leq\bar{p}_{\infty}(g^{\ast})<\alpha$. The bound in
  Lemma~\ref{lem:hunter-bound} has such a limit, obtained by replacing the
  covariance matrix and $K_{\tau}$ by their limits, since it is a continuous
  function of the covariance matrix and of the thresholds.

  \medskip
  \noindent\textbf{Step~6: Length.} Fix $j$ and $f\in\mathcal{F}_{j}$. On
  $\mathcal{E}_{1/2}$, the realized $s=\hat{\sigma}$ and $\tau$ satisfy
  $t=s\delta_{\tau}\in[t_{\mathrm{lo}},t_{\mathrm{hi}}]$, and
  $\mu(CI_{n})\leq(\hat{c}_{\tau}^{u,j}(s)-\hat{c}_{\tau}^{\ell,j}(s))_{+}$.
  By Lemma A.1 of \cite{armstrong2018optimal}, the bias of
  $\hat{L}_{\delta}^{\ell,j}$ is at least
  $-\frac{1}{2}\{\omega(\delta,\mathcal{F}_{J},\mathcal{F}_{j})-\delta s_{j}^{-}\}$
  over $\mathcal{F}_{j}$. Since computing the moduli as if $\sigma(\cdot)=s$
  at $\delta$ gives the moduli with $\sigma(\cdot)=1$ at $t$, writing
  $\omega_{1}^{-}(t)=\omega(t,\mathcal{F}_{J},\mathcal{F}_{j})$ with
  $\sigma(\cdot)=1$, we obtain
  $f(0)-\hat{c}_{\tau}^{\ell,j}(s)\leq\omega_{1}^{-}(t)+\sigma\omega_{1}^{-\prime}(t)\lvert S_{(-,j)}(t)\rvert$,
  and $\omega_{1}^{-\prime}(t)\leq\omega_{1}^{-}(t)/t$ by concavity. The upper
  end is analogous. Hence,
  \[
    \mu(CI_{n})\leq\omega_{+}(t_{\mathrm{hi}},\mathcal{F}_{j},\mathcal{F}_{J})\Big(2+\frac{\sigma}{t_{\mathrm{lo}}}\sum_{q\in\{(-,j),(+,j)\}}\sup_{t\in[t_{\mathrm{lo}},t_{\mathrm{hi}}]}\lvert S_{q}(t)\rvert\Big),
  \]
  where $\omega_{+}$ is computed with $\sigma(\cdot)=1$. The supremum is
  $O_{p}(1)$ by \eqref{eq:asy-equicont} and Step~3, uniformly over $f$, and
  $\omega_{+}(t_{\mathrm{hi}},\mathcal{F}_{j},\mathcal{F}_{J})\leq\max\{1,t_{\mathrm{hi}}/z_{1-\alpha/2}\}\,\omega_{+}(z_{1-\alpha/2},\mathcal{F}_{j},\mathcal{F}_{J})$
  by \eqref{eq:contraction} and monotonicity, which is $O(n^{-r(\gamma_{j},\gamma_{J})})$ by
  Theorem~4.2. Together with
  $\sup_{f}\p_{f}(\mathcal{E}_{1/2}^{c})\to0$, this proves the second claim.

  \medskip
  \noindent\textbf{Other coverage classes.} The proof uses the coverage class
  $\mathcal{F}_{J}$ only through the bounds in Lemma~4.3,
  which give the maximum biases of the ends and the nonnegativity of
  $A_{j}(f;t)$ and $B_{\ell}(f;t)$, and through the uniform consistency of
  $\hat{\sigma}$. For
  $\mathcal{P}_{\mathcal{V}}(\gamma_{J},C_{J})\cap\Lambda(\gamma',C')$, these
  bounds are the same, by Lemma~\ref{lem:G-envelope}(i), and $\hat{\sigma}$ is
  consistent uniformly over this set by Lemma~\ref{lem:nn-variance}, so that
  the proof applies without change. For AdaptMono, with coverage over
  $\mathcal{G}_{\mathcal{V}}(\gamma_{J},C_{J})\cap\Lambda(\gamma',C')$ and
  $k_{+}\geq1$, the ends use the ordered moduli of
  Proposition~\ref{prop:G-InvMod}, and the bounds in
  Lemma~\ref{lem:G-envelope}(ii) are finite at the design points that receive
  positive weight. Theorem~4.3 replaces
  Theorem~4.2, and the proof of Lemma~\ref{lem:tau-limit},
  with the sums restricted to $\mathcal{D}_{\mathcal{V}}^{+}$ and
  $\mathcal{D}_{\mathcal{V}}^{-}$ as in the proof of Theorem~4.3,
  gives the limits of the covariance matrices, in which the lower and upper
  ends are independent, and the convergence of $\tau^{\ast}(s)$, uniformly over
  $s\in[\sigma/2,2\sigma]$. For $J=1$, Step~5 of that proof uses the
  independence of the two ends in place of $k_{\infty}>0$: with $K_{\tau}=0$,
  $p_{\infty}(\tau)=1-(1-\tau)^{2}$, so that
  $p_{\infty}(\alpha/2)<\alpha<p_{\infty}(\alpha)$. Since $K_{\tau}=0$, Step~5
  of the present proof simplifies. By
  \eqref{eq:asy-inclusion}, the argument in Step~4, with $t_{0}$ replaced by
  $t_{\infty}$ and the union bound replaced by \eqref{eq:asy-clt}, bounds the
  limit superior of the largest noncoverage probability by
  $p_{\infty}(\tau_{\infty})=\alpha$. Hence, the bound on $v_{f}$, which uses
  that the two $J$th ends have the same weights, is not needed. Finally,
  $\hat{\sigma}$ is consistent uniformly over
  $\mathcal{G}_{\mathcal{V}}(\gamma_{J},C_{J})\cap\Lambda(\gamma',C')$ by
  Lemma~\ref{lem:nn-variance}, since this set is contained in
  $\Lambda(\gamma',C')$.
\end{proof}

\section{Heteroskedasticity}\label{apdx:Heteroskedasiticy}

In Theorem~4.2, we assume $\sigma(\cdot)=1$. However,
allowing for general heteroskedasticity does not change the result as long as
$\sigma(\cdot)$ is continuous at $0$ and $\sigma(0)>0.$ All proofs follow with
minor changes. The solution to the inverse modulus problem remains unchanged,
except that each summand is divided by $\sigma^{2}(x_{i})$. For
Theorem~4.2, we apply the argument in the proof of Lemma~\ref{lem:local-sums} to $g_{t}(x)/\sigma^{2}(x)$, which is bounded for small
$t$ because $\sigma(\cdot)$ is continuous at $0$ and the support of $g_{t}$
shrinks to $0$. In this argument, we can take
$\varepsilon\in(0,\min\{p_{X}(0),\sigma(0)\})$ and replace the terms
$p_{X}(0)-\varepsilon$ and $p_{X}(0)+\varepsilon$ by
$(p_{X}(0)-\varepsilon)/(\sigma(0)+\varepsilon)^{2}$ and
$(p_{X}(0)+\varepsilon)/(\sigma(0)-\varepsilon)^{2}$ in
(\ref{eq:EWin_bound}). Accordingly, the right-hand side of
(\ref{eq:Integralrate}) becomes $\kappa p_{X}(0)/\sigma^{2}(0),$ and the
result of the theorem remains the same with $p_{X}(0)$ replaced by
$p_{X}(0)/\sigma^{2}(0)$ in the definition of the constant terms.

\begin{remark}[Heteroskedasticity and an estimated noise level]\label{rem:hetero-feasible}
  Suppose instead that $y_{i}=f(x_{i})+\sigma(x_{i})u_{i}$, where
  $\sigma(\cdot)$ is continuous at $0$ with $\sigma(0)>0$, and the $u_{i}$ are
  as above. Since only the design points in a shrinking neighborhood of $0$
  receive positive weight, the CIs computed as if $\sigma(\cdot)$ were a
  constant remain asymptotically valid when this constant is a consistent
  estimator of $\sigma(0)$. For example, let $m_{n}$ be integers such that
  $m_{n}\to\infty$ and $m_{n}/n\to0$, such as $m_{n}=\lceil n^{4/5}\rceil$.
  Let $\mathcal{N}_{n}$ be the set of indices of the $m_{n}$ design points
  closest to $0$ in the Euclidean norm, with ties broken by a rule that
  depends only on the design points, and let
  $\hat{\sigma}_{0}^{2}=(2m_{n})^{-1}\sum_{i\in\mathcal{N}_{n}}(y_{i}-y_{i^{\ast}})^{2}$.
  Then, the conclusions of Theorem~5.1 hold with
  $\hat{\sigma}$ replaced by $\hat{\sigma}_{0}$ and $\sigma$ by $\sigma(0)$,
  if the fourth moments are also bounded over $\mathcal{U}_{\eta,M}$,
  as shown below. The global
  estimator $\hat{\sigma}$ above is not consistent for $\sigma(0)$ in
  general, and the resulting CIs can undercover. Standard errors that are
  robust to heteroskedasticity, as in
  \citet[Section~3.3.3]{armstrong2020simple}, give
  another route. We do not pursue it here.
\end{remark}

\begin{proof}[Proof of the claim in Remark~\ref{rem:hetero-feasible}]
Fix a realization
of the design points for which the conclusions of
Theorem~4.2 and Lemma~\ref{lem:tau-limit} hold, the design
points are distinct, and
$\rho_{n}=\max_{i\in\mathcal{N}_{n}}\lVert x_{i}\rVert_{2}\to0$. The last two
properties hold for almost all realizations because $m_{n}/n\to0$ and
$p_{X}(0)>0$. We first show that $\hat{\sigma}_{0}^{2}\to\sigma^{2}(0)$
in probability, uniformly over $f\in\Lambda(\gamma_{J},C_{J})$ and over the
distributions of $u_{i}$ with $\E u_{i}=0$, $\E u_{i}^{2}=1$ and
$\E u_{i}^{4}\leq M'$. For $i\in\mathcal{N}_{n}$, the nearest neighbor
$x_{i^{\ast}}$ is within $2\rho_{n}$ of $x_{i}$ because the other points of
$\mathcal{N}_{n}$ are. Then, both points lie within $3\rho_{n}$ of $0$.
Hence,
$\lvert f(x_{i})-f(x_{i^{\ast}})\rvert\leq C_{J}\kappa_{2}^{\gamma_{J}}(2\rho_{n})^{\gamma_{J}}$,
with $\kappa_{2}$ as in the proof of Lemma~\ref{lem:nn-variance}, and
\[
  \frac{1}{2}\E\big(\sigma(x_{i})u_{i}-\sigma(x_{i^{\ast}})u_{i^{\ast}}\big)^{2}=\frac{\sigma^{2}(x_{i})+\sigma^{2}(x_{i^{\ast}})}{2},
\]
which converges to $\sigma^{2}(0)$ uniformly over $i\in\mathcal{N}_{n}$ because
$\sigma(\cdot)$ is continuous at $0$. The argument in the proof of
Lemma~\ref{lem:nn-variance}, applied to the $m_{n}$ summands with
$i\in\mathcal{N}_{n}$, shows that the remaining terms have variance of order
$1/m_{n}$ because $\sigma(\cdot)$ is bounded near $0$. This gives the
consistency of $\hat{\sigma}_{0}^{2}$.

Next, the proof of
Theorem~5.1 goes through with $\sigma$ replaced by
$\sigma(0)$ and $S_{q}(t)$ by
$\tilde{S}_{q}(t)=\pm\sum_{i}\tilde{c}_{q,i}(t)\sigma(x_{i})u_{i}/\sigma(0)$,
with the same signs as $S_{q}(t)$. The inclusion \eqref{eq:asy-inclusion}
holds with $\sigma S_{q}$ replaced by $\sigma(0)\tilde{S}_{q}$ because the
maximum biases do not depend on the errors. For large $n$, the supports of
$\tilde{c}_{q}(t)$ and of its derivative in $t$ lie in a small neighborhood of
$0$, uniformly over $t\in[t_{\mathrm{lo}},t_{\mathrm{hi}}]$. Hence,
\eqref{eq:asy-equicont} holds with $K_{0}$ multiplied by the supremum of
$\sigma(x)/\sigma(0)$ over this neighborhood. Finally, by the Cauchy--Schwarz
inequality, each element of the covariance matrix of
$(\tilde{S}_{q}(t))_{q}$ differs from the corresponding element of
$\Sigma_{n}(t)$ by at most the largest value of
$\lvert\sigma^{2}(x_{i})/\sigma^{2}(0)-1\rvert$ over the design points with
$\tilde{c}_{q,i}(t)\neq0$ for some $q$. This value tends to zero, so that Step~3 and
the rest of the proof apply.
\end{proof}

\section{Adaptation Under Only Monotonicity}
\label{apdx:Onlymon}
Define $\mathcal{M}_{\mathcal{V}}$ as the space of functions on
$\mathcal{D}$ that are monotone with respect to those variables whose indices
lie in $\mathcal{V}$, without any smoothness restriction. Specifically,
\[
  \mathcal{M}_{\mathcal{V}}:=\left\{
    f:\mathcal{D}\to\mathbb{R}:f(x)\geq f(z)\,\,\text{if }x_{i}\geq
    z_{i}\text{ }\forall i\in\mathcal{V}\text{ and
    }x_{i}=z_{i}\,\,\forall i\notin\mathcal{V}\right\} .
\]
Here, we consider the problem of adapting to $\Lambda_{+,\mathcal{V}}(\gamma,C)$
while maintaining coverage over $\mathcal{M}_{\mathcal{V}}$, which is
a convex set that contains $\Lambda_{+,\mathcal{V}}(\gamma,C)$. If
$\mathcal{V}=\{1,\dots,k\}$, then
$\mathcal{M}_{\mathcal{V}}=\mathcal{G}_{\mathcal{V}}(\gamma',C')$ for every
$\gamma'\in(0,1]$ and $C'>0$, and Proposition~\ref{prop:onlymon} and
Corollary~\ref{cor:onlymon-rate} are special cases of
Proposition~\ref{prop:G-InvMod} and Theorem~4.3. On the other
hand, if $k_{+}<k$, then $\mathcal{M}_{\mathcal{V}}$ is much larger than
$\mathcal{G}_{\mathcal{V}}(\gamma',C')$, and the following results show that
it is too large for adaptation. Define
$A_{\mathcal{V}}:=\{ x\in\mathcal{D}:x_{j}\geq0\,\,\forall j\in\mathcal{V}\text{ and }x_{j}=0\,\,\forall j\notin\mathcal{V}\}$,
which is the set of points that are at least as large as $0$ in the partial order used
to define $\mathcal{M}_{\mathcal{V}}$.
\begin{prop}
  \label{prop:onlymon}Suppose Assumption~2.1 and
  (2) hold, and let $\gamma\in(0,1]$ and $C>0$. Then, for
  $b\geq0$,
  \begin{align*}
    \omega^{-1}\left(b,\Lambda_{+,\mathcal{V}}(\gamma,C),\mathcal{M}_{\mathcal{V}}\right) & =\Big(\sum\nolimits_{i:x_{i}\in A_{\mathcal{V}}}((b-C\lVert x_{i}\rVert^{\gamma})_{+}/\sigma(x_{i}))^{2}\Big)^{1/2},\text{ and}\\
    \omega^{-1}\left(b,\mathcal{M}_{\mathcal{V}},\Lambda_{+,\mathcal{V}}(\gamma,C)\right) & =\Big(\sum\nolimits_{i:x_{i}\in-A_{\mathcal{V}}}((b-C\lVert x_{i}\rVert^{\gamma})_{+}/\sigma(x_{i}))^{2}\Big)^{1/2}.
  \end{align*}
\end{prop}
\begin{proof}
  We only prove the first claim, since the second claim can be shown
  analogously. As in the proof of Proposition~4.1, we can
  restrict attention to $f_{1}\in\Lambda_{+,\mathcal{V}}(\gamma,C)$ and
  $f_{2}\in\mathcal{M}_{\mathcal{V}}$ with $f_{1}(0)=0$ and
  $f_{2}(0)=b$. If $x_{i}\in A_{\mathcal{V}}$, then
  $f_{2}(x_{i})\geq f_{2}(0)=b$ by monotonicity, and
  $f_{1}(x_{i})\leq C\lVert(x_{i})_{\mathcal{V}+}\rVert^{\gamma}=C\lVert x_{i}\rVert^{\gamma}$
  by Lemma~4.3. Hence,
  $(f_{2}(x_{i})-f_{1}(x_{i}))^{2}\geq(b-C\lVert x_{i}\rVert^{\gamma})_{+}^{2}$
  for $x_{i}\in A_{\mathcal{V}}$, which gives a lower bound on the objective.
  Now, let
  \begin{align*}
    f_{1}^{*}(x)=\min\left\{ C\left\lVert
        (x)_{\mathcal{V}+}\right\rVert^{\gamma},b\right\},\qquad
    f_{2}^{\ast}(x) =\begin{cases}
      b & \text{if }x\in A_{\mathcal{V}}\\
      f_{1}^{*}(x)  & \text{otherwise}.
    \end{cases}
  \end{align*}
  Note that $f_{1}^{*}\in\Lambda_{+,\mathcal{V}}(\gamma,C)$ by Lemmas~4.1 and 4.2, since the constant
  function $b$ is in $\Lambda_{+,\mathcal{V}}(\gamma,C)$. Next, we argue that
  $f_{2}^{\ast}\in\mathcal{M}_{\mathcal{V}}$. To show this, we must
  show that for any $x,z\in\mathcal{D}$,
  \[
    f_{2}^{\ast}(x)\geq f_{2}^{\ast}(z)\,\,\text{if }x_{j}\geq z_{j}\text{
    }\forall j\in\mathcal{V}\text{ and }x_{j}=z_{j}\,\,\forall
    j\notin\mathcal{V}.
  \]
  Note that this clearly holds if both $x$ and $z$ are in $A_{\mathcal{V}}$,
  or if neither of them is, by the monotonicity of $f_{1}^{*}$. If
  $x\in A_{\mathcal{V}}$ and $z\notin A_{\mathcal{V}}$, the monotonicity
  condition holds because $f_{2}^{\ast}(z)\leq b=f_{2}^{\ast}(x)$. Finally, it
  cannot be the case that $z\in A_{\mathcal{V}}$ and
  $x\notin A_{\mathcal{V}}$, since $z\in A_{\mathcal{V}}$ implies
  $x\in A_{\mathcal{V}}$ under the premise. Since $f_{1}^{*}(0)=0$,
  $f_{2}^{*}(0)=b$, and $f_{2}^{*}(x_{i})-f_{1}^{*}(x_{i})$ equals
  $(b-C\lVert x_{i}\rVert^{\gamma})_{+}$ if $x_{i}\in A_{\mathcal{V}}$ and
  $0$ otherwise, $(f_{1}^{*},f_{2}^{*})$ attains the lower bound.
\end{proof}

If $\mathcal{V}\subsetneq\{ 1,\dots,k\}$, then $A_{\mathcal{V}}$
is contained in a lower-dimensional linear subspace of $\mathbb{R}^{k}$, and
thus has Lebesgue measure zero. Hence, under the
assumption that the design points are a realization of a random variable that
admits a pdf with respect to the Lebesgue measure, we may assume that
$x_{i}\notin A_{\mathcal{V}}$ and $x_{i}\notin-A_{\mathcal{V}}$ for all
$i=1,\dots,n$. That is, the inverse moduli are zero for all $b\geq0$, so that
the ordered moduli, and thus the between class modulus, are infinite for every
$\delta\geq0$. The lower bound in Lemma~3.1 does not
require the modulus to be finite. Hence, for almost all
realizations of the design points, every CI that has coverage at least
$1-\alpha$ over $\mathcal{M}_{\mathcal{V}}$ has infinite worst-case
expected length over $\Lambda_{+,\mathcal{V}}(\gamma,C)$. In fact, the
expected length is infinite at every $f\in\mathcal{M}_{\mathcal{V}}$.
For $b>0$, the function $f+b\mathds{1}\{x\in A_{\mathcal{V}}\}$ is also in
$\mathcal{M}_{\mathcal{V}}$ and coincides with $f$ at every design
point, so that the CI contains both $f(0)$ and $f(0)+b$ with probability at
least $1-2\alpha$ under $f$. In other words, there is no scope for
adaptation in this case.

On the other hand, if $\mathcal{V}=\{ 1,\dots,k\}$, then
$A_{\mathcal{V}}$ is the intersection of $\mathcal{D}$ and the nonnegative
orthant. Hence, adaptation is possible, and the resulting CIs end up using
only those data with design points that lie in either the nonnegative or
nonpositive orthant. In fact, the following corollary shows that strong
adaptation is possible in this case.
\begin{corollary}
  \label{cor:onlymon-rate}Suppose $\mathcal{V}=\{1,\dots,k\}$, and that
  Assumption~2.1, (2), and the conditions on
  $\{X_{i}\}_{i=1}^{\infty}$ and $\sigma(\cdot)$ in Theorem~4.2 hold. Then, for almost all realizations
  $\{ x_{i}\} _{i=1}^{\infty}$ of
  $\{ X_{i}\} _{i=1}^{\infty}$, we have, for all $\delta>0$,
  \begin{align*}
    & \lim_{n\to\infty}n^{1/(2+k/\gamma)}\omega(\delta,\Lambda_{+,\mathcal{V}}(\gamma,C),\mathcal{M}_{\mathcal{V}})\\
    & \quad=\lim_{n\to\infty}n^{1/(2+k/\gamma)}\omega(\delta,\mathcal{M}_{\mathcal{V}},\Lambda_{+,\mathcal{V}}(\gamma,C))=(\delta^{2}/c_{0}^{\ast})^{1/(2+k/\gamma)},
  \end{align*}
  where
  $c_{0}^{\ast}=p_{X}(0)\int_{\{u:u\geq0\}}(1-C\lVert u\rVert^{\gamma})_{+}^{2}du$.
\end{corollary}
\begin{proof}
  The proof is the same as that of Theorem~4.2, using
  Proposition~\ref{prop:onlymon} in place of Proposition~4.1
  and Corollary~4.1. That is, we apply Lemma~\ref{lem:local-sums} with
  $d=k/\gamma$ and
  \[
    g_{t}(x)=(1-(C/t)\lVert x\rVert^{\gamma})_{+}^{2}\mathds{1}\{x\geq0\}.
  \] A change of variables $x=t^{1/\gamma}u$ gives
  $t^{-k/\gamma}\int g_{t}(x)dx=\int_{\{u\geq0\}}(1-C\lVert u\rVert^{\gamma})_{+}^{2}du$,
  and the reverse ordering is handled by the change of variables
  $x\mapsto-x$.
\end{proof}
Since
$\omega(\delta,\mathcal{M}_{\mathcal{V}},\mathcal{M}_{\mathcal{V}})=\infty$
almost surely, Theorem~3.1 does not apply directly with
$\mathcal{F}_{J}=\mathcal{M}_{\mathcal{V}}$.\footnote{For example,
  $f_{1}(x)=b\mathds{1}\{x\geq0,x\neq0\}$ and $f_{2}(x)=b\mathds{1}\{x\geq0\}$
  are both in $\mathcal{M}_{\mathcal{V}}$, $f_{2}(0)-f_{1}(0)=b$,
  and $f_{2}(x_{i})=f_{1}(x_{i})$ if $x_{i}\neq0$.} However, consider
the interval $[\hat{c}_{\alpha/2}^{\ell,1},\hat{c}_{\alpha/2}^{u,1}]$
constructed from the ordered moduli between
$\mathcal{F}_{1}=\Lambda_{+,\mathcal{V}}(\gamma,C)$ and
$\mathcal{M}_{\mathcal{V}}$. For large $n$, both $A_{\mathcal{V}}$
and $-A_{\mathcal{V}}$ contain design points almost surely. Then, by
Proposition~\ref{prop:onlymon}, the two moduli are attained, and differentiable
with positive derivatives. This interval has coverage at least $1-\alpha$
over $\mathcal{M}_{\mathcal{V}}$ by the union bound, and the argument
in the proof of Theorem~3.1 bounds its worst-case expected
length over $\mathcal{F}_{1}$ by
\[
  2\{ \Phi(z_{1-\alpha/2})+\phi(z_{1-\alpha/2})/z_{1-\alpha/2}\} \omega_{+}(z_{1-\alpha/2},\mathcal{F}_{1},\mathcal{M}_{\mathcal{V}}).
\]
The minimax rate over $\Lambda_{+,\mathcal{V}}(\gamma,C)$ is
$n^{-1/(2+k/\gamma)}$ by Theorem~4.2. Hence,
Corollary~\ref{cor:onlymon-rate} and Lemma~3.1 show that this CI
achieves the minimax rate over $\Lambda_{+,\mathcal{V}}(\gamma,C)$ while
maintaining coverage over all functions that are monotone in every
coordinate. This is consistent with the
limit of $r(\gamma_{1},\gamma_{2})$ as $\gamma_{2}\to0$ discussed in Section~4.2.

More generally, let $\mathcal{F}_{j}=\Lambda_{+,\mathcal{V}}(\gamma_{j},C_{j})$
for $j\leq J$, and consider the CI $CI_{\tau}^{\mathcal{J}}$ of
Section~3.2 with the coverage class $\mathcal{F}_{J}$
replaced by $\mathcal{M}_{\mathcal{V}}$. That is, it intersects the one-sided
CIs that direct power to $\mathcal{F}_{j}$, $j\leq J$, while maintaining
coverage over $\mathcal{M}_{\mathcal{V}}$. Here, $V(\tau)$, $W(\tau)$,
$a_{j}(f)$, $c_{j}(f)$ and (11) are defined with the ordered moduli
$\omega(\cdot,\mathcal{M}_{\mathcal{V}},\mathcal{F}_{j})$ and
$\omega(\cdot,\mathcal{F}_{j},\mathcal{M}_{\mathcal{V}})$, and with
$f\in\mathcal{M}_{\mathcal{V}}$. When $A_{\mathcal{V}}$ and
$-A_{\mathcal{V}}$ contain design points, both moduli are attained, and
differentiable with positive derivatives, by Proposition~\ref{prop:onlymon}.
Since the proof of Lemma~3.3 uses only such properties and
the convexity of the parameter spaces, the lemma continues to hold. Hence,
the calibrated CI has coverage at least $1-\alpha$ over
$\mathcal{M}_{\mathcal{V}}$. For this coverage class, the largest $K_{\tau}$
in (11) is $0$. By Proposition~\ref{prop:onlymon}, the lower ends
use only the design points in $-A_{\mathcal{V}}$, and the upper ends only those
in $A_{\mathcal{V}}$, so that a constant function attains the maximum bias of
every end at once. If $\mathcal{V}=\{1,\dots,k\}$, then
$\mathcal{M}_{\mathcal{V}}=\mathcal{G}_{\mathcal{V}}(\gamma_{J},C_{J})$, and this
CI is AdaptMono, which Section~5.1 defines for a general
$\mathcal{V}$. We use it with $J=4$ in Section~6.

\section{Extensions}
\label{apdx:extensions}

\subsection{Coordinate-Specific Smoothness}
\label{apdx:anisotropic}
In this section, we allow the smoothness of the regression function to differ
across coordinates. For
$\boldsymbol{\gamma}=(\gamma_{1},\dots,\gamma_{k})\in(0,1]^{k}$ and
$\mathbf{C}=(C_{1},\dots,C_{k})\in(0,\infty)^{k}$, define
$\rho_{\boldsymbol{\gamma},\mathbf{C}}(x,z)=\sum_{v=1}^{k}C_{v}\lvert x_{v}-z_{v}\rvert^{\gamma_{v}}$,
which is a metric on $\mathbb{R}^{k}$ because
$\lvert a-c\rvert^{\gamma}\leq\lvert a-b\rvert^{\gamma}+\lvert b-c\rvert^{\gamma}$
for $\gamma\in(0,1]$, and let
\begin{align*}
  \Lambda_{+,\mathcal{V}}(\boldsymbol{\gamma},\mathbf{C}):=\Big\{ f:\mathcal{D}\to\mathbb{R}: & \lvert f(x)-f(z)\rvert\leq\rho_{\boldsymbol{\gamma},\mathbf{C}}(x,z)\text{ and }\\
   & f(x)\geq f(z)\text{ if }x\geq_{\mathcal{V}}z,\text{ for all }x,z\in\mathcal{D}\Big\},
\end{align*}
where we write $x\geq_{\mathcal{V}}z$ if $x_{j}\geq z_{j}$ for all
$j\in\mathcal{V}$ and $x_{j}=z_{j}$ for all $j\notin\mathcal{V}$. That is,
the functions in $\Lambda_{+,\mathcal{V}}(\boldsymbol{\gamma},\mathbf{C})$
are monotone in the coordinates in $\mathcal{V}$ and Lipschitz with constant
at most one with respect to $\rho_{\boldsymbol{\gamma},\mathbf{C}}$. When all
the exponents are equal to one, this is the class
$\Lambda_{+,\mathcal{V}}(1,1)$ in Section~2 with the weighted
$\ell_{1}$ norm $\lVert z\rVert=\sum_{v}C_{v}\lvert z_{v}\rvert$, but the two
types of classes differ otherwise.

For $j=1,2$, let
$\mathcal{F}_{j}=\Lambda_{+,\mathcal{V}}(\boldsymbol{\gamma}_{j},\mathbf{C}_{j})$,
where $\boldsymbol{\gamma}_{j}$ and $\mathbf{C}_{j}$ have elements
$\gamma_{jv}$ and $C_{jv}$, $v=1,\dots,k$, and define
\begin{align*}
  A_{j}(x) & =\sum\nolimits_{v\in\mathcal{V}}C_{jv}(x_{v})_{+}^{\gamma_{jv}}+\sum\nolimits_{v\notin\mathcal{V}}C_{jv}\lvert x_{v}\rvert^{\gamma_{jv}},\quad\text{and}\\
  B_{j}(x) & =\sum\nolimits_{v\in\mathcal{V}}C_{jv}(-x_{v})_{+}^{\gamma_{jv}}+\sum\nolimits_{v\notin\mathcal{V}}C_{jv}\lvert x_{v}\rvert^{\gamma_{jv}}.
\end{align*}
The following lemma is the analogue of Lemmas~4.1 to
4.3.
\begin{lemma}
  \label{lem:aniso-geometry}(i) $\mathcal{F}_{j}$ is convex, invariant to
  adding constants, and closed under the maximum and minimum operators. (ii)
  If $\rho_{\boldsymbol{\gamma}_{1},\mathbf{C}_{1}}(x,z)\leq\rho_{\boldsymbol{\gamma}_{2},\mathbf{C}_{2}}(x,z)$
  for all $x,z\in\mathcal{D}$, then $\mathcal{F}_{1}\subseteq\mathcal{F}_{2}$.
  A sufficient condition is that, for each $v$, $\gamma_{2v}\leq\gamma_{1v}$
  and $C_{1v}\Delta_{v}^{\gamma_{1v}-\gamma_{2v}}\leq C_{2v}$, where
  $\Delta_{v}=\sup_{x,z\in\mathcal{D}}\lvert x_{v}-z_{v}\rvert$. (iii)
  Suppose (2) holds. Then, every $f\in\mathcal{F}_{j}$ with
  $f(0)=f_{0}$ satisfies $f_{0}-B_{j}(x)\leq f(x)\leq f_{0}+A_{j}(x)$ for all
  $x\in\mathcal{D}$, and the bounds are attained by $f_{0}-B_{j}$ and
  $f_{0}+A_{j}$, respectively, which are in $\mathcal{F}_{j}$.
\end{lemma}
\begin{prop}
  \label{prop:aniso-exact}Suppose (2) holds and
  $\mathcal{F}_{1}\subseteq\mathcal{F}_{2}$. Then, for $b\geq0$,
  \begin{align*}
    \omega^{-1}(b,\mathcal{F}_{1},\mathcal{F}_{2}) & =\Big(\sum\nolimits_{i=1}^{n}\left((b-A_{1}(x_{i})-B_{2}(x_{i}))_{+}/\sigma(x_{i})\right)^{2}\Big)^{1/2},\text{ and}\\
    \omega^{-1}(b,\mathcal{F}_{2},\mathcal{F}_{1}) & =\Big(\sum\nolimits_{i=1}^{n}\left((b-A_{2}(x_{i})-B_{1}(x_{i}))_{+}/\sigma(x_{i})\right)^{2}\Big)^{1/2},
  \end{align*}
  which are attained by $(A_{1},\max\{b-B_{2},A_{1}\})\in\mathcal{F}_{1}\times\mathcal{F}_{2}$
  and $(\min\{b-B_{1},A_{2}\},b-B_{1})\in\mathcal{F}_{2}\times\mathcal{F}_{1}$,
  respectively.
\end{prop}
\begin{theorem}
  \label{thm:aniso-rate}Suppose (2) holds,
  $0<\gamma_{2v}\leq\gamma_{1v}\leq1$ for all $v$, and
  $\mathcal{F}_{1}\subseteq\mathcal{F}_{2}$. Suppose the conditions on
  $\{X_{i}\}_{i=1}^{\infty}$ and $\sigma(\cdot)$ in Theorem~4.2 hold, and let
  \[
    \tilde{r}=\left(2+\sum\nolimits_{v\in\mathcal{V}}1/\gamma_{1v}+\sum\nolimits_{v\notin\mathcal{V}}1/\gamma_{2v}\right)^{-1}.
  \]
  Then, for almost all realizations $\{ x_{i}\} _{i=1}^{\infty}$
  of $\{ X_{i}\} _{i=1}^{\infty}$, we have, for all $\delta>0$,
  \[
    \lim_{n\to\infty}n^{\tilde{r}}\omega(\delta,\mathcal{F}_{1},\mathcal{F}_{2})=\lim_{n\to\infty}n^{\tilde{r}}\omega(\delta,\mathcal{F}_{2},\mathcal{F}_{1})=(\delta^{2}/\tilde{c}^{\ast})^{\tilde{r}},
  \]
  where $\tilde{c}^{\ast}\in(0,\infty)$ is given in the proof. Hence, by
  Lemma~3.1, the minimax expected length over
  $\mathcal{F}_{1}$ among CIs with coverage at least $1-\alpha$ over
  $\mathcal{F}_{2}$ is of order $n^{-\tilde{r}}$, almost surely.
\end{theorem}
That is, the rate uses the smoother exponent for each monotone coordinate and
the rougher exponent for each of the other coordinates. Hence, the
smoothness of the larger parameter space in the monotone coordinates does not
affect the rate.
The rate in Theorem~\ref{thm:aniso-rate} parallels the local rates
$n^{-1/(2+\sum_{v}\alpha_{v}^{-1})}$ that \citet[Theorem~1]{han2020limit} obtain
for block estimators in multiple isotonic regression under a random design.
Here, $\alpha_{v}$ is the order of the first nonzero partial derivative in the
$v$th coordinate at the point of interest.
Applying the theorem with $\mathcal{F}_{1}=\mathcal{F}_{2}$ shows that,
regardless of $\mathcal{V}$, the minimax rate over
$\Lambda_{+,\mathcal{V}}(\boldsymbol{\gamma},\mathbf{C})$ is
$n^{-1/(2+\sum_{v}1/\gamma_{v})}$. Hence, as in Section~4.2, strong adaptation is possible if $\mathcal{V}=\{1,\dots,k\}$,
and there is no scope for adaptation if $\mathcal{V}=\emptyset$. Otherwise, the
rate is in between, unless the exponents of $\mathcal{F}_{1}$ and
$\mathcal{F}_{2}$ coincide for all the coordinates in $\mathcal{V}$ or for all
the other coordinates.

\begin{proof}[Proof of Lemma~\ref{lem:aniso-geometry}]
  (i) Convexity and invariance to adding constants are immediate. Since
  $\lvert\max\{a,b\}-\max\{c,d\}\rvert\leq\max\{\lvert a-c\rvert,\lvert b-d\rvert\}$,
  and the same inequality holds for the minimum, the closedness under the
  maximum and minimum operators follows as in the proof of Lemma~4.2. (ii) The first claim is immediate from the
  definition. For the sufficient condition, note that, for
  $0\leq s\leq\Delta_{v}$,
  $C_{1v}s^{\gamma_{1v}}\leq C_{1v}\Delta_{v}^{\gamma_{1v}-\gamma_{2v}}s^{\gamma_{2v}}\leq C_{2v}s^{\gamma_{2v}}$.
  (iii) Since $\lvert(s)_{+}^{\gamma}-(t)_{+}^{\gamma}\rvert\leq\lvert s-t\rvert^{\gamma}$
  and $\lvert\lvert s\rvert^{\gamma}-\lvert t\rvert^{\gamma}\rvert\leq\lvert s-t\rvert^{\gamma}$
  for $\gamma\in(0,1]$, $A_{j}$ and $-B_{j}$ are Lipschitz with constant at
  most one with respect to $\rho_{\boldsymbol{\gamma}_{j},\mathbf{C}_{j}}$,
  and they are monotone in the coordinates in $\mathcal{V}$, so that
  $f_{0}+A_{j},f_{0}-B_{j}\in\mathcal{F}_{j}$. Now, take $f\in\mathcal{F}_{j}$
  with $f(0)=f_{0}$ and $x\in\mathcal{D}$, and let $x^{+}=(x)_{\mathcal{V}+}$
  and $x^{-}=-(x)_{\mathcal{V}-}$, which are in $\mathcal{D}$ by
  (2). Since $x^{+}\geq_{\mathcal{V}}x\geq_{\mathcal{V}}x^{-}$,
  $\rho_{\boldsymbol{\gamma}_{j},\mathbf{C}_{j}}(x^{+},0)=A_{j}(x)$ and
  $\rho_{\boldsymbol{\gamma}_{j},\mathbf{C}_{j}}(x^{-},0)=B_{j}(x)$, we
  have $f(x)\leq f(x^{+})\leq f_{0}+A_{j}(x)$ and
  $f(x)\geq f(x^{-})\geq f_{0}-B_{j}(x)$.
\end{proof}
\begin{proof}[Proof of Proposition~\ref{prop:aniso-exact}]
  The proof is the same as that of Proposition~4.1 and
  Corollary~4.1, using Lemma~\ref{lem:aniso-geometry} in
  place of Lemmas~4.1 to 4.3. For the first
  claim, it suffices to consider $f_{1}(0)=0$ and $f_{2}(0)=b$, in which case
  $f_{1}(x_{i})\leq A_{1}(x_{i})$ and $f_{2}(x_{i})\geq b-B_{2}(x_{i})$, which
  gives a lower bound on the objective. The pair
  $(A_{1},\max\{b-B_{2},A_{1}\})$ is feasible by part (i) of Lemma~\ref{lem:aniso-geometry} and $\mathcal{F}_{1}\subseteq\mathcal{F}_{2}$, and
  its difference is $(b-A_{1}-B_{2})_{+}$, which attains the lower bound. The
  second claim follows analogously.
\end{proof}
\begin{proof}[Proof of Theorem~\ref{thm:aniso-rate}]
  We prove the result for $\omega(\delta,\mathcal{F}_{1},\mathcal{F}_{2})$.
  The reverse ordering is analogous. Let $\eta_{v}=\gamma_{1v}$ for
  $v\in\mathcal{V}$ and $\eta_{v}=\gamma_{2v}$ for $v\notin\mathcal{V}$, and
  $d=\sum_{v}1/\eta_{v}$, so that $\tilde{r}=(2+d)^{-1}$. Partition the
  coordinates into $S=\{v\in\mathcal{V}:\gamma_{1v}>\gamma_{2v}\}$,
  $T=\{v\in\mathcal{V}:\gamma_{1v}=\gamma_{2v}\}$,
  $R=\{v\notin\mathcal{V}:\gamma_{1v}>\gamma_{2v}\}$, and
  $E=\{v\notin\mathcal{V}:\gamma_{1v}=\gamma_{2v}\}$. Let
  $a(x)=A_{1}(x)+B_{2}(x)$, and apply Lemma~\ref{lem:local-sums} with
  $g_{t}(x)=(1-a(x)/t)_{+}^{2}$, using any norm on $\mathbb{R}^{k}$ in its
  support condition, since all norms on $\mathbb{R}^{k}$ are equivalent. By the change of variables
  $x_{v}=t^{1/\eta_{v}}u_{v}$, we have
  $t^{-d}\int g_{t}(x)dx=\int(1-a(s_{t}u)/t)_{+}^{2}du$ with
  $(s_{t}u)_{v}=t^{1/\eta_{v}}u_{v}$. As $t\to0$, the integrand converges to
  $\mathds{1}\{u\in\mathcal{K}\}(1-q(u))_{+}^{2}$ for almost every $u$, where
  $\mathcal{K}=\{u:u_{v}\geq0\ \forall v\in S\}$ and
  \begin{align*}
    q(u) &= \sum_{v\in S}C_{1v}u_{v}^{\gamma_{1v}}+\sum_{v\in T}\left\{ C_{1v}(u_{v})_{+}^{\gamma_{1v}}+C_{2v}(-u_{v})_{+}^{\gamma_{1v}}\right\} \\
     &\quad +\sum_{v\in R}C_{2v}\lvert u_{v}\rvert^{\gamma_{2v}}+\sum_{v\in E}(C_{1v}+C_{2v})\lvert u_{v}\rvert^{\gamma_{1v}}.
  \end{align*}
  To see this, note that for $v\in S$, the term
  $C_{2v}t^{\gamma_{2v}/\gamma_{1v}-1}(-u_{v})_{+}^{\gamma_{2v}}$ diverges
  unless $u_{v}\geq0$, and that for $v\in R$, the term
  $C_{1v}t^{\gamma_{1v}/\gamma_{2v}-1}\lvert u_{v}\rvert^{\gamma_{1v}}$
  vanishes. Moreover, the integrand is bounded by the indicator of a fixed
  bounded set for $t\leq1$. Hence, by the dominated convergence theorem and
  Lemma~\ref{lem:local-sums},
  $\frac{1}{nt_{n}^{d}}\sum_{i=1}^{n}g_{t_{n}}(X_{i})\to \tilde{c}^{\ast}:=p_{X}(0)\int_{\mathcal{K}}(1-q(u))_{+}^{2}du\in(0,\infty)$
  almost surely for $t_{n}=bn^{-\tilde{r}}$, for each $b>0$. For the
  reverse ordering, $A_{2}(-x)+B_{1}(-x)=A_{1}(x)+B_{2}(x)$, so that the
  change of variables $x\mapsto-x$ gives the same constant. The rest of the
  proof is the same as the last two paragraphs of the proof of Theorem~4.2, using Proposition~\ref{prop:aniso-exact}.
\end{proof}

\subsection{Uncertainty About the Monotonicity Set}
\label{apdx:finite-family}
In this section, we consider the case where the researcher is not sure which
coordinates the regression function is monotone in. Let $\mathfrak{V}$ be a
nonempty finite family of subsets of $\{1,\dots,k\}$, specified before
observing the data, and suppose that
$\mathcal{V}_{0}:=\cap_{\mathcal{W}\in\mathfrak{V}}\mathcal{W}$ is itself an
element of $\mathfrak{V}$. We require the CI to have correct coverage over
$\cup_{\mathcal{W}\in\mathfrak{V}}\Lambda_{+,\mathcal{W}}(\gamma_{2},C_{2})$,
and evaluate its expected length over
$\mathcal{F}_{1,\mathcal{V}}:=\Lambda_{+,\mathcal{V}}(\gamma_{1},C_{1})$ for
each $\mathcal{V}\in\mathfrak{V}$. Note that a function that is monotone with
respect to the coordinates in $\mathcal{W}$ is also monotone with respect to
the coordinates in $\mathcal{V}_{0}\subseteq\mathcal{W}$, in the sense of the
definition in Section~2. Since
$\mathcal{V}_{0}\in\mathfrak{V}$, it follows that
\begin{equation}
  \cup_{\mathcal{W}\in\mathfrak{V}}\Lambda_{+,\mathcal{W}}(\gamma_{2},C_{2})=\Lambda_{+,\mathcal{V}_{0}}(\gamma_{2},C_{2})=:\mathcal{F}_{2,0},\label{eq:rough-union-identity}
\end{equation}
which is convex. Without the condition $\mathcal{V}_{0}\in\mathfrak{V}$, this
union is not necessarily convex. We assume that
$\mathcal{F}_{1,\mathcal{V}}\subseteq\mathcal{F}_{2,0}$ for all
$\mathcal{V}\in\mathfrak{V}$, which holds if
$\Lambda(\gamma_{1},C_{1})\subseteq\Lambda(\gamma_{2},C_{2})$, and that
(2) holds with $\mathcal{V}$ replaced by each element of
$\mathfrak{V}$.
\begin{lemma}
  \label{lem:cross-V-inverse}Suppose Assumption~2.1 and the
  conditions above hold, and let $\mathcal{V}\in\mathfrak{V}$. Then, for
  $b\geq0$,
  \begin{align*}
    \omega^{-1}(b,\mathcal{F}_{1,\mathcal{V}},\mathcal{F}_{2,0}) & =\Big(\sum\nolimits_{i=1}^{n}\left((b-C_{1}\lVert (x_{i})_{\mathcal{V}+}\rVert^{\gamma_{1}}-C_{2}\lVert (x_{i})_{\mathcal{V}_{0}-}\rVert^{\gamma_{2}})_{+}/\sigma(x_{i})\right)^{2}\Big)^{1/2},\text{ and}\\
    \omega^{-1}(b,\mathcal{F}_{2,0},\mathcal{F}_{1,\mathcal{V}}) & =\Big(\sum\nolimits_{i=1}^{n}\left((b-C_{2}\lVert (x_{i})_{\mathcal{V}_{0}+}\rVert^{\gamma_{2}}-C_{1}\lVert (x_{i})_{\mathcal{V}-}\rVert^{\gamma_{1}})_{+}/\sigma(x_{i})\right)^{2}\Big)^{1/2}.
  \end{align*}
\end{lemma}
\begin{theorem}
  \label{thm:finite-family-robustness}Suppose the conditions of Lemma~\ref{lem:cross-V-inverse} hold, $0<\gamma_{2}\leq\gamma_{1}\leq1$, and the
  conditions on $\{X_{i}\}_{i=1}^{\infty}$ and $\sigma(\cdot)$ in Theorem~4.2 hold. Let
  $r_{\mathcal{V}_{0}}=(2+\lvert\mathcal{V}_{0}\rvert/\gamma_{1}+(k-\lvert\mathcal{V}_{0}\rvert)/\gamma_{2})^{-1}$.
  Then, for almost all realizations $\{ x_{i}\} _{i=1}^{\infty}$
  of $\{ X_{i}\} _{i=1}^{\infty}$, the following holds for all
  $\mathcal{V}\in\mathfrak{V}$ and $\delta>0$:
  \[
    \lim_{n\to\infty}n^{r_{\mathcal{V}_{0}}}\omega(\delta,\mathcal{F}_{1,\mathcal{V}},\mathcal{F}_{2,0})=\lim_{n\to\infty}n^{r_{\mathcal{V}_{0}}}\omega(\delta,\mathcal{F}_{2,0},\mathcal{F}_{1,\mathcal{V}})=(\delta^{2}/c_{\mathcal{V}}^{\ast})^{r_{\mathcal{V}_{0}}},
  \]
  where $c_{\mathcal{V}}^{\ast}\in(0,\infty)$ is given in the proof.
  Consequently, for each $\mathcal{V}\in\mathfrak{V}$, the minimax expected
  length over $\mathcal{F}_{1,\mathcal{V}}$ among CIs that have coverage at
  least $1-\alpha$ over the union
  $\cup_{\mathcal{W}\in\mathfrak{V}}\Lambda_{+,\mathcal{W}}(\gamma_{2},C_{2})$
  is of order $n^{-r_{\mathcal{V}_{0}}}$, and this rate is achieved simultaneously for all
  $\mathcal{V}\in\mathfrak{V}$ by the CI in (8)
  constructed with the parameter spaces
  $\{\mathcal{F}_{1,\mathcal{V}}:\mathcal{V}\in\mathfrak{V}\}$ and
  $\mathcal{F}_{2,0}$, the latter playing the role of the largest parameter
  space.
\end{theorem}
Hence, only the coordinates that are monotone under all the candidate sets in
$\mathfrak{V}$ help, and the researcher cannot adapt to the monotonicity set
in this sense. The theorem considers a prespecified finite family
that contains the intersection of its elements, with the direction of the
monotonicity fixed. It does not cover data-dependent choices of the
monotonicity set or unknown directions of monotonicity. If the direction of
monotonicity in a coordinate is unknown, the coordinate can simply be left out
of $\mathcal{V}_{0}$.

\begin{proof}[Proof of Lemma~\ref{lem:cross-V-inverse}]
  The proof is the same as that of Proposition~4.1 and
  Corollary~4.1. For the first claim, take
  $f_{1}\in\mathcal{F}_{1,\mathcal{V}}$ and $f_{2}\in\mathcal{F}_{2,0}$ with
  $f_{1}(0)=0$ and $f_{2}(0)=b$. Lemma~4.3, applied with
  $\mathcal{V}$ and $\mathcal{V}_{0}$, respectively, gives
  $f_{1}(x_{i})\leq C_{1}\lVert (x_{i})_{\mathcal{V}+}\rVert^{\gamma_{1}}$
  and
  $f_{2}(x_{i})\geq b-C_{2}\lVert (x_{i})_{\mathcal{V}_{0}-}\rVert^{\gamma_{2}}$,
  which gives the lower bound. The pair
  \[
    \left(C_{1}\lVert (x)_{\mathcal{V}+}\rVert^{\gamma_{1}},\,\max\{b-C_{2}\lVert (x)_{\mathcal{V}_{0}-}\rVert^{\gamma_{2}},C_{1}\lVert (x)_{\mathcal{V}+}\rVert^{\gamma_{1}}\}\right)
  \]
  is feasible since $\mathcal{F}_{1,\mathcal{V}}\subseteq\mathcal{F}_{2,0}$,
  and it attains the lower bound. The second claim follows analogously, using
  the pair
  \[
    \left(\min\{b-C_{1}\lVert (x)_{\mathcal{V}-}\rVert^{\gamma_{1}},C_{2}\lVert (x)_{\mathcal{V}_{0}+}\rVert^{\gamma_{2}}\},\,b-C_{1}\lVert (x)_{\mathcal{V}-}\rVert^{\gamma_{1}}\right).\qedhere
  \]
\end{proof}
\begin{proof}[Proof of Theorem~\ref{thm:finite-family-robustness}]
  Fix $\mathcal{V}\in\mathfrak{V}$, and let
  $a(x)=C_{1}\lVert (x)_{\mathcal{V}+}\rVert^{\gamma_{1}}+C_{2}\lVert (x)_{\mathcal{V}_{0}-}\rVert^{\gamma_{2}}$
  and $d_{0}=1/r_{\mathcal{V}_{0}}-2$. We apply Lemma~\ref{lem:local-sums} with
  $g_{t}(x)=(1-a(x)/t)_{+}^{2}$. Suppose $\gamma_{1}>\gamma_{2}$, and
  partition $\mathbb{R}^{k}$ according to the signs of the coordinates in
  $\mathcal{V}_{0}$. On a set in which $m$ of these coordinates are positive,
  the change of variables that scales these $m$ coordinates by
  $t^{1/\gamma_{1}}$ and all the other coordinates by $t^{1/\gamma_{2}}$ has
  Jacobian $t^{m/\gamma_{1}+(k-m)/\gamma_{2}}$, whose exponent is uniquely
  minimized at $m=\lvert\mathcal{V}_{0}\rvert$. The coordinates in
  $\mathcal{V}\setminus\mathcal{V}_{0}$ are scaled by $t^{1/\gamma_{2}}$, so
  that they vanish from the first term of $a(x)/t$ in the limit. By the
  dominated convergence theorem, we get
  $t^{-d_{0}}\int g_{t}(x)dx\to c_{\mathcal{V}}^{\ast}/p_{X}(0)$, where
  \[
    c_{\mathcal{V}}^{\ast}=p_{X}(0)\int_{\{u:u_{v}\geq0\ \forall v\in\mathcal{V}_{0}\}}(1-C_{1}\lVert u_{\mathcal{V}_{0}}\rVert^{\gamma_{1}}-C_{2}\lVert -u_{\mathcal{V}_{0}^{c}}\rVert^{\gamma_{2}})_{+}^{2}du,
  \]
  with $u_{S}$ denoting the vector obtained from $u$ by setting the
  coordinates not in $S$ to zero. If $\gamma_{1}=\gamma_{2}=\gamma$, all the
  coordinates are scaled by $t^{1/\gamma}$, and
  \[
    c_{\mathcal{V}}^{\ast}=p_{X}(0)\int(1-C_{1}\lVert (u)_{\mathcal{V}+}\rVert^{\gamma}-C_{2}\lVert (u)_{\mathcal{V}_{0}-}\rVert^{\gamma})_{+}^{2}du.
  \]
  For the reverse ordering, the function
  $C_{2}\lVert (x)_{\mathcal{V}_{0}+}\rVert^{\gamma_{2}}+C_{1}\lVert (x)_{\mathcal{V}-}\rVert^{\gamma_{1}}$
  equals $a(-x)$, so that the change of variables $x\mapsto-x$ gives the same
  constant. The rest of the proof of the first claim is the same as the last
  two paragraphs of the proof of Theorem~4.2, using Lemma~\ref{lem:cross-V-inverse}, and since $\mathfrak{V}$ is finite, the
  probability one event can be taken to be the same for all
  $\mathcal{V}\in\mathfrak{V}$.

  For the second claim, \eqref{eq:rough-union-identity} shows that the
  coverage requirement is the same as requiring coverage over the convex set
  $\mathcal{F}_{2,0}$. Hence, Lemma~3.1 and the first claim
  give the minimax rate. For the CI, Lemma~\ref{lem:cross-V-inverse} and
  Proposition~4.1 show that the ordered moduli between
  $\mathcal{F}_{1,\mathcal{V}}$ and $\mathcal{F}_{2,0}$, and between
  $\mathcal{F}_{2,0}$ and itself, are attained and differentiable with
  positive derivatives, by the same argument as in Section~5.1. Hence, Theorem~3.1 applies, and
  its bound is of order $n^{-r_{\mathcal{V}_{0}}}$ for every
  $\mathcal{F}_{1,\mathcal{V}}$ by the first claim.
\end{proof}

\subsection{Boundary Points}
\label{apdx:boundary}
Here, we prove Theorem~4.4 and its extension to supports that are,
near $0$, unions of orthants.
\begin{proof}[Proof of Theorem~4.4]
  Let
  $T=\{x\in\mathbb{R}^{k}:x_{v}\geq0\ \forall v\in\mathcal{B}_{+},\ x_{v}\leq0\ \forall v\in\mathcal{B}_{-}\}$,
  so that $\mathcal{X}\cap B_{\rho}=T\cap B_{\rho}$. Lemma~\ref{lem:local-sums}
  continues to hold if the condition $0\in\mathrm{Int}\,\mathcal{X}$ is
  replaced by this condition, and $\int g_{t}(x)dx$ by $\int_{T}g_{t}(x)dx$ in
  its statement. Indeed, take $\varepsilon\in(0,p_{X}(0))$ and a neighborhood
  $\mathcal{N}_{\varepsilon}\subseteq B_{\rho}$ of $0$ such that
  $\lvert p_{X}(x)-p_{X}(0)\rvert\leq\varepsilon$ for
  $x\in\mathcal{X}\cap\mathcal{N}_{\varepsilon}$. Once the support of
  $g_{t_{n}}$ is contained in $\mathcal{N}_{\varepsilon}$, we have
  $\E g_{t_{n}}(X_{1})=\int_{T}g_{t_{n}}(x)p_{X}(x)dx$, so that
  \eqref{eq:EWin_bound} holds with the integrals over $T$, and the rest of the
  proof is unchanged.

  We use this version of the lemma in the proof of
  Theorem~4.2. Up to a set of Lebesgue measure zero, $T$ is
  the union of the orthants on which the coordinates in $\mathcal{B}_{+}$ are
  positive and those in $\mathcal{B}_{-}$ are negative, so that only these
  orthants appear in \eqref{eq:sum_orth}, and each of them gives the same
  integral as before. For the first ordered modulus, an orthant $O$ on which
  $m$ of the monotone coordinates are positive contributes
  $t^{m/\gamma_{1}+(k-m)/\gamma_{2}}\{c_{O}(C_{1},C_{2})+o(1)\}$, where
  $c_{O}(C_{1},C_{2})\in(0,\infty)$ is defined in that proof. On the orthants
  in $T$, the coordinates in $\mathcal{V}\cap\mathcal{B}_{+}$ are positive and
  those in $\mathcal{V}\cap\mathcal{B}_{-}$ are negative, while the other
  monotone coordinates can take either sign, so that the largest value of $m$
  among these orthants is $k_{+}-\ell_{-}$. Hence, if
  $\gamma_{1}>\gamma_{2}$, the exponent of $t$ is smallest when
  $m=k_{+}-\ell_{-}$, and
  \[
    t^{-d_{-}}\int_{T}g_{t}(x)dx\to\sum_{O}c_{O}(C_{1},C_{2})\in(0,\infty),\qquad d_{-}=\frac{k_{+}-\ell_{-}}{\gamma_{1}}+\frac{k-k_{+}+\ell_{-}}{\gamma_{2}},
  \]
  where the sum is over the orthants $O$ in $T$ with $m=k_{+}-\ell_{-}$. If
  $\gamma_{1}=\gamma_{2}$, every orthant in $T$ contributes, the exponent is
  $k/\gamma_{1}$, which equals $d_{-}$, and the limit is the sum of
  $c_{O}(C_{1},C_{2})$ over all of them. In both cases, the rest of the proof
  of Theorem~4.2 gives the first limit, with
  $r_{\ell_{-}}=(2+d_{-})^{-1}$ and $\bar{c}_{1,2}$ equal to $p_{X}(0)$ times
  this sum. The second limit follows in the same way. For the reverse
  ordering, an orthant on which $m'$ of the monotone coordinates are negative
  contributes an integral of order $t^{m'/\gamma_{1}+(k-m')/\gamma_{2}}$, and
  the largest value of $m'$ among the orthants in $T$ is $k_{+}-\ell_{+}$.
\end{proof}
The proof uses $T$ only through the fact that, up to a set of Lebesgue
measure zero, it is a union of orthants. Suppose, instead, that
$\mathcal{X}\cap B_{\rho}=U\cap B_{\rho}$ for a union $U$ of closed orthants,
and let $m_{+}$ and $m_{-}$ be the largest numbers of monotone coordinates that
are positive, and negative, on an orthant contained in $U$. Then, the same
argument, with the sums over the orthants contained in $U$, gives the limits
in Theorem~4.4 with $r_{\ell_{-}}$ replaced by
$r_{k_{+}-m_{+}}$ and $r_{\ell_{+}}$ by $r_{k_{+}-m_{-}}$. If $U=T$, then
$m_{+}=k_{+}-\ell_{-}$ and $m_{-}=k_{+}-\ell_{+}$.

\subsection{General Monotonicity Cones and Norms}
\label{apdx:cones}
In this section, we use the notation of
Section~4.2: $\lVert\cdot\rVert$ is an arbitrary
norm on $\mathbb{R}^{k}$, Assumption~2.1 is not imposed, $K$ is a
closed convex cone, $\Lambda_{K}(\gamma,C)$ and $K_{\mathcal{V}}$ are defined
there, and $d(x,A)=\inf_{a\in A}\lVert x-a\rVert$. In place of
(2), we assume the domain condition (17). It holds if
$\mathcal{D}$ is a ball with respect to $\lVert\cdot\rVert$ centered at $0$,
because the infima are attained, $K$ being closed, and
$d(x,\pm K)\leq\lVert x\rVert$. If $K=K_{\mathcal{V}}$ and
Assumption~2.1 holds, it is implied by (2) because,
by part (iii) of the following lemma, $x+h$ and $x-h'$ can then be taken to be
$(x)_{\mathcal{V}+}$ and $-(x)_{\mathcal{V}-}$.
\begin{lemma}
  \label{lem:cone-envelope}Let $\gamma\in(0,1]$, $C>0$ and
  $f_{0}\in\mathbb{R}$. (i) The functions $f_{0}+Cd(\cdot,-K)^{\gamma}$ and
  $f_{0}-Cd(\cdot,K)^{\gamma}$ are in $\Lambda_{K}(\gamma,C)$, and every
  $f\in\Lambda_{K}(\gamma,C)$ with $f(0)=f_{0}$ satisfies
  $f_{0}-Cd(x,K)^{\gamma}\leq f(x)\leq f_{0}+Cd(x,-K)^{\gamma}$ for all
  $x\in\mathcal{D}$. (ii) $\Lambda_{K}(\gamma,C)$ is closed under the maximum
  and minimum operators. (iii) If $K=K_{\mathcal{V}}$ and
  Assumption~2.1 holds, then
  $d(x,-K)=\lVert(x)_{\mathcal{V}+}\rVert$ and
  $d(x,K)=\lVert(x)_{\mathcal{V}-}\rVert$.
\end{lemma}
\begin{proof}
  (i) Since $x\mapsto d(x,A)$ is Lipschitz with constant one and
  $\lvert s^{\gamma}-t^{\gamma}\rvert\leq\lvert s-t\rvert^{\gamma}$ for
  $s,t\geq0$, both functions are in $\Lambda(\gamma,C)$. Let $h\in K$. For any
  $\kappa\in K$, $h+\kappa\in K$, so that $\lVert x+h+\kappa\rVert\geq d(x,-K)$,
  and thus $d(x+h,-K)\geq d(x,-K)$. Likewise,
  $d(x+h,K)\leq\lVert x+h-(\kappa+h)\rVert=\lVert x-\kappa\rVert$ for any
  $\kappa\in K$, and thus $d(x+h,K)\leq d(x,K)$. Hence, both functions are
  nondecreasing along $K$. Now, take $f\in\Lambda_{K}(\gamma,C)$ with
  $f(0)=f_{0}$ and $x\in\mathcal{D}$, and let $h$ and $h'$ be as above. Then,
  $f(x)\leq f(x+h)\leq f_{0}+C\lVert x+h\rVert^{\gamma}=f_{0}+Cd(x,-K)^{\gamma}$,
  and, likewise, $f(x)\geq f(x-h')\geq f_{0}-Cd(x,K)^{\gamma}$. (ii) This
  follows as in the proof of Lemma~4.2, which does not
  use Assumption~2.1, since the maximum and the minimum of two
  functions that are nondecreasing along $K$ are nondecreasing along $K$.
  (iii) For $h\in-K_{\mathcal{V}}$, the elements of $x-h$ are $x_{v}-h_{v}$,
  with $h_{v}\leq0$, for $v\in\mathcal{V}$, and $x_{v}$ for
  $v\notin\mathcal{V}$. Under Assumption~2.1, $\lVert x-h\rVert$
  is minimized by making each $\lvert x_{v}-h_{v}\rvert$ as small as possible,
  that is, by $h_{v}=\min\{x_{v},0\}$, which gives $x-h=(x)_{\mathcal{V}+}$.
  The second claim follows in the same way.
\end{proof}
\begin{prop}
  \label{prop:cone-InvMod}Write $\mathcal{F}_{1}^{K}=\Lambda_{K}(\gamma_{1},C_{1})$
  and $\mathcal{F}_{2}^{K}=\Lambda_{K}(\gamma_{2},C_{2})$. If
  $\mathcal{F}_{1}^{K}\subseteq\mathcal{F}_{2}^{K}$, then, for $b\geq0$,
  \begin{align*}
    \omega^{-1}(b,\mathcal{F}_{1}^{K},\mathcal{F}_{2}^{K}) & =\Big(\sum\nolimits_{i=1}^{n}\big((b-C_{1}d(x_{i},-K)^{\gamma_{1}}-C_{2}d(x_{i},K)^{\gamma_{2}})_{+}/\sigma(x_{i})\big)^{2}\Big)^{1/2},\\
    \omega^{-1}(b,\mathcal{F}_{2}^{K},\mathcal{F}_{1}^{K}) & =\Big(\sum\nolimits_{i=1}^{n}\big((b-C_{2}d(x_{i},-K)^{\gamma_{2}}-C_{1}d(x_{i},K)^{\gamma_{1}})_{+}/\sigma(x_{i})\big)^{2}\Big)^{1/2},
  \end{align*}
  and the corresponding inverse modulus problems have solutions.
\end{prop}
\begin{proof}
  The proof is the same as those of Proposition~4.1 and
  Corollary~4.1, with Lemma~\ref{lem:cone-envelope} in
  place of Lemmas~4.1 to 4.3. The first inverse
  modulus is attained by $C_{1}d(\cdot,-K)^{\gamma_{1}}$ and
  $\max\{b-C_{2}d(\cdot,K)^{\gamma_{2}},C_{1}d(\cdot,-K)^{\gamma_{1}}\}$, and
  the second by
  $\min\{b-C_{1}d(\cdot,K)^{\gamma_{1}},C_{2}d(\cdot,-K)^{\gamma_{2}}\}$ and
  $b-C_{1}d(\cdot,K)^{\gamma_{1}}$.
\end{proof}
Hence, the CIs of Section~5.1 can be computed in the
same way, with $\lVert(x_{i})_{\mathcal{V}+}\rVert$ and
$\lVert(x_{i})_{\mathcal{V}-}\rVert$ replaced by $d(x_{i},-K)$ and
$d(x_{i},K)$. Assumption~2.1 is not needed. For
example, monotonicity in the coordinates can be combined with the Mahalanobis
distance without changing the basis as in Remark~\ref{rem:norm}.
Theorem~4.1 in Section~4.2 shows
that the rate depends on $K$ only through the dimension of its linear span.
By Lemma~3.1, the minimax rate of adaptation is then
$n^{-r_{K}}$, with $r_{K}=r_{K}(\gamma_{1},\gamma_{2})$.
\begin{proof}[Proof of Theorem~4.1]
  By Proposition~\ref{prop:cone-InvMod} and the argument in the proof of
  Theorem~4.2, it suffices to show that the support of
  \[
    g_{t}(x)=\left(1-\frac{C_{1}}{t}d(x,-K)^{\gamma_{1}}-\frac{C_{2}}{t}d(x,K)^{\gamma_{2}}\right)_{+}^{2}
  \]
  shrinks to $0$ as $t\to0$ and that $t^{-d_{K}}\int g_{t}(x)dx$ converges to
  a positive and finite limit $c_{K}$, where
  $d_{K}=m/\gamma_{1}+(k-m)/\gamma_{2}$. The reverse ordering then follows
  from the change of variables $x\mapsto-x$, since $d(-x,-K)=d(x,K)$, and, as
  in that proof, Lemma~\ref{lem:local-sums} gives the factor $p_{X}(0)$ in the
  limits of Theorem~4.1. We use two facts. First,
  $d(\lambda x,\pm K)=\lambda d(x,\pm K)$ for $\lambda>0$, since $K$ is a cone.
  Second, since $K$ is closed and $K\cap(-K)=\{0\}$, the continuous function
  $d(\cdot,-K)$ is positive on the compact set
  $\{\kappa\in K:\lVert\kappa\rVert=1\}$, so that there exists $c_{0}>0$ such
  that $d(\kappa,-K)\geq c_{0}\lVert\kappa\rVert$ for all $\kappa\in K$.

  If $\gamma_{1}=\gamma_{2}$, the change of variables $x=t^{1/\gamma_{1}}z$
  gives $t^{-d_{K}}\int g_{t}(x)dx=\int g_{1}(z)dz$, which is positive. It is
  also finite, and the support of $g_{t}$ shrinks to $0$, since
  $\{z:g_{1}(z)>0\}$ is bounded: if $g_{1}(z)>0$, then $z=\kappa+e$ with
  $\kappa\in K$ and $\lVert e\rVert<C_{2}^{-1/\gamma_{2}}$, and
  $c_{0}\lVert\kappa\rVert\leq d(\kappa,-K)\leq d(z,-K)+\lVert e\rVert<C_{1}^{-1/\gamma_{1}}+C_{2}^{-1/\gamma_{2}}$.

  Now, suppose $\gamma_{1}>\gamma_{2}$. Let $S_{K}$ be the linear span of $K$,
  $S_{K}^{\perp}$ its orthogonal complement, and $P$ the orthogonal projection
  onto $S_{K}^{\perp}$.
  Writing $x=t^{1/\gamma_{1}}u+t^{1/\gamma_{2}}v$ with $u\in S_{K}$ and $v\in S_{K}^{\perp}$,
  and $\lambda=t^{1/\gamma_{1}-1/\gamma_{2}}$, we have
  \[
    t^{-d_{K}}\int g_{t}(x)dx=\int_{S_{K}}\int_{S_{K}^{\perp}}\left(1-C_{1}d(u+v/\lambda,-K)^{\gamma_{1}}-C_{2}d(\lambda u+v,K)^{\gamma_{2}}\right)_{+}^{2}dv\,du,
  \]
  where the integrals are with respect to the Lebesgue measures on $S_{K}$ and
  $S_{K}^{\perp}$. As $t\to0$, $\lambda\to\infty$. For $u$ in the relative interior of
  $K$, the sets $K-\lambda u$ increase with $\lambda$, since
  $K+(\lambda'-\lambda)u\subseteq K$ for $\lambda'\geq\lambda$, and their union
  is $S_{K}$. Hence, $d(\lambda u+v,K)=d(v,K-\lambda u)$ decreases to $d(v,S_{K})$.
  For $u\in S_{K}\setminus K$,
  $d(\lambda u+v,K)\geq\lambda d(u,K)-\lVert v\rVert\to\infty$. Since the
  relative boundary of $K$ has Lebesgue measure zero in $S_{K}$, the integrand
  converges, for almost all $(u,v)$, to
  $(1-C_{1}d(u,-K)^{\gamma_{1}}-C_{2}d(v,S_{K})^{\gamma_{2}})_{+}^{2}\mathds{1}\{u\in K\}$.
  The integrand is bounded by one, and, for $t\leq1$, so that $\lambda\geq1$,
  it vanishes outside a bounded set. Indeed, if it is positive, then
  $\lambda u+v=\kappa+e$ for some $\kappa\in K$ and some $e$ with
  $\lVert e\rVert<C_{2}^{-1/\gamma_{2}}$, so that $v=Pe$, since
  $u,\kappa\in S_{K}$, and $\lVert v\rVert\leq\lVert P\rVert C_{2}^{-1/\gamma_{2}}$,
  where $\lVert P\rVert$ is the operator norm with respect to
  $\lVert\cdot\rVert$. Moreover, $u+v/\lambda=(\kappa+e)/\lambda$, so that
  $c_{0}\lVert\kappa\rVert/\lambda\leq d(\kappa/\lambda,-K)<C_{1}^{-1/\gamma_{1}}+C_{2}^{-1/\gamma_{2}}$,
  and
  $\lVert u\rVert\leq\lVert\kappa\rVert/\lambda+\lVert e\rVert+\lVert v\rVert$
  is bounded. The same bounds show that the support of $g_{t}$ shrinks to
  $0$. By the dominated convergence theorem,
  \[
    t^{-d_{K}}\int g_{t}(x)dx\to c_{K}:=\int_{K}\int_{S_{K}^{\perp}}\left(1-C_{1}d(u,-K)^{\gamma_{1}}-C_{2}d(v,S_{K})^{\gamma_{2}}\right)_{+}^{2}dv\,du,
  \]
  which is finite, and positive since $K$ has a nonempty interior relative to
  $S_{K}$ and the integrand is positive near $(u,v)=(0,0)$.
\end{proof}
For $K=K_{\mathcal{V}}$ and a norm that satisfies
Assumption~2.1, $m=k_{+}$, $d(u,-K)=\lVert u\rVert$ for $u\in K$
and $d(v,S_{K})=\lVert v\rVert$ for $v\in S_{K}^{\perp}$, so that $c_{K}$ is the
constant $c_{+}(C_{1},C_{2})$ in the proof of Theorem~4.2
when $\gamma_{1}>\gamma_{2}$. When $\gamma_{1}=\gamma_{2}$, $c_{K}$ is the
integral of
$(1-C_{1}\lVert(x)_{\mathcal{V}+}\rVert^{\gamma_{1}}-C_{2}\lVert(x)_{\mathcal{V}-}\rVert^{\gamma_{1}})_{+}^{2}$
over $\mathbb{R}^{k}$, which is the sum of $c_{O}(C_{1},C_{2})$ over all the
orthants, that is, again $c_{+}(C_{1},C_{2})$.

\section{Further Remarks}
\label{apdx:remarks}
In this section, we collect the remarks referred to in the main text.

\begin{remark}[Comparison with the constant of Cai and Low]\label{rem:CL-constant}
  For nonnested parameter spaces, the CI of \citet{cai2004adaptation} is also an
  intersection of CIs with a Bonferroni correction. In terms of
  $\omega_{+}(z_{1-\alpha/2},\mathcal{F}_{j},\mathcal{F}_{J})$, their bound on
  the constant is $12z_{1-\alpha/(2J)}/z_{1-\alpha/2}$, about six times the
  constant in (9). On the other hand, their result requires
  only convex parameter spaces with nonempty pairwise intersections, rather
  than a convex parameter space that contains all the others.
\end{remark}

\begin{remark}[Adaptive one-sided CIs]\label{rem:one-sided}
  Suppose one is interested in constructing the one-sided CI in an adaptive
  way. Note that Lemma~3.2 with $\beta=1/2$ implies that any
  one-sided CI $[\hat{c},\infty)$ with coverage probability at least
  $1-\alpha$ over $\mathcal{F}_{J}$ should satisfy
  \[
    \underset{f\in\mathcal{F}_{j}}{\sup}q_{1/2,f}(Lf-\hat{c})\geq\omega(z_{1-\alpha},\mathcal{F}_{J},\mathcal{F}_{j}),
  \]
  under the conditions of the lemma at $\delta=z_{1-\alpha}$. Suppose, in
  addition, that the conditions of the lemma with $\beta=1/2$ hold at
  $\delta=z_{1-\alpha/J}$ for each $j\leq J$, and define
  $\hat{c}_{\alpha}^{\ell,\max}=\max _{j}\hat{c}_{\alpha/J}^{\ell,j}$, which
  has coverage at least $1-\alpha$ over $\mathcal{F}_{J}$ by the union bound.
  Since
  $Lf-\hat{c}_{\alpha}^{\ell,\max}\leq Lf-\hat{c}_{\alpha/J}^{\ell,j}$, Lemma~3.2 and the concavity of the ordered
  modulus\footnote{By the concavity shown at the start of
    Section~\ref{appdx-sec:Proofs_lemmas} \citep[see also][Theorem~1]{cai2004minimax},
    $\omega(\delta',\mathcal{F},\mathcal{G})\leq(\delta'/\delta)\,\omega(\delta,\mathcal{F},\mathcal{G})$
    for $0<\delta\leq\delta'$.}
  give
  \[
    \underset{f\in\mathcal{F}_{j}}{\sup}q_{1/2,f}(Lf-\hat{c}_{\alpha}^{\ell,\max})\leq\omega(z_{1-\alpha/J},\mathcal{F}_{J},\mathcal{F}_{j})\leq\frac{z_{1-\alpha/J}}{z_{1-\alpha}}\omega(z_{1-\alpha},\mathcal{F}_{J},\mathcal{F}_{j}).
  \]
  Therefore, $[\hat{c}_{\alpha}^{\ell,\max},\infty)$ is an adaptive
  one-sided CI in the sense defined in Section~3.1.2.
\end{remark}

\begin{remark}[Role of Assumption~2.1]\label{rem:norm} Proposition~4.1 requires Assumption~2.1 due to the specific
  form of monotonicity we consider. By considering coordinate-wise monotonicity, we
  must take a norm that is ``aligned'' with this direction of monotonicity. The
  assumption precisely imposes this. This is a unique feature that arises in the
  multivariate setting. For example, the $\ell_{p}$ norms and their versions
  with coordinate-specific weights satisfy Assumption~2.1, while a
  norm of the form $\lVert Az\rVert_{2}$ for a nondiagonal matrix $A$ does not
  necessarily satisfy it.

  To allow for more general norms, let $\mathcal{B}$ be an
  orthonormal basis of $\mathbb{R}^{k}$, and denote by $z^{\mathcal{B}}$ the
  coordinate vector of $z \in \mathbb{R}^{k}$ with respect to $\mathcal{B}$ and
  $z^{\mathcal{B}}_{j}$ its $j$th component. Suppose the regression function is
  monotone in the coefficients with respect to this basis $\mathcal{B}$, so that
  the monotone H\"older class is given as
  \begin{equation*}
  \Lambda_{+,\mathcal{V}}(\gamma,C):=\left\{ f\in\Lambda(\gamma,C):f(x)\geq
    f(z)\,\,\text{if }x^{\mathcal{B}}_{j}\geq z^{\mathcal{B}}_{j}\text{ }\forall
    j \in\mathcal{V} \text{ and }x^{\mathcal{B}}_{j}= z^{\mathcal{B}}_{j}\,\,\forall
    j\notin\mathcal{V}\right\}.
\end{equation*}
Then, the condition we want to impose on the norm $\lVert \cdot \rVert$ is
monotonicity with respect to the magnitude of $z^{\mathcal{B}}_{j}$. All
the results go through once all vectors, including those in the domain
condition (2), are expressed in terms of their coordinates with
respect to $\mathcal{B}$. A special case is the Mahalanobis distance
$\lVert z\rVert=(z'\Sigma^{-1}z)^{1/2}$ when $\mathcal{B}$ consists of the
eigenvectors of $\Sigma$.

In Section~\ref{apdx:cones}, we replace the truncated vectors by distances to
the cone of monotone directions. We show that the solution to the modulus
problem and the rate result then hold for any norm, under the domain condition (17),
without aligning the norm with the directions of monotonicity. If the norm does not satisfy
Assumption~2.1, the bounds in Lemma~4.3 still hold
under (2) because
$f(x)\leq f((x)_{\mathcal{V}+})\leq f(0)+C\lVert(x)_{\mathcal{V}+}\rVert^{\gamma}$,
and likewise from below. However, the bounds are not necessarily attained
because $d(x,-K_{\mathcal{V}})\leq\lVert(x)_{\mathcal{V}+}\rVert$. The CIs of
Corollary~5.1 computed with such bounds then remain valid,
by the argument in the proof of Lemma~3.3 with $a_{j}(f)$ and
$c_{\ell}(f)$ defined relative to their bias corrections. This is because the
bias corrections and $K_{\tau}$ in Lemma~5.1 use only valid bounds.
Nonetheless, such CIs can be conservative.
\end{remark}

\begin{remark}[Deterministic design points]
  \label{rem:deterministic-design}The proof of Theorem~4.2
  uses the randomness of the design points only to show that, for every
  $a>0$,
  \begin{equation}
    \frac{1}{n t_{n}^{d}}\sum_{i=1}^{n}\left(1-\frac{C_{1}}{t_{n}}\lVert (x_{i})_{\mathcal{V}+}\rVert^{\gamma_{1}}-\frac{C_{2}}{t_{n}}\lVert (x_{i})_{\mathcal{V}-}\rVert^{\gamma_{2}}\right)_{+}^{2}\to c_{1,2},\label{eq:deterministic-design}
  \end{equation}
  where $t_{n}=an^{-r(\gamma_{1},\gamma_{2})}$,
  $d=k_{+}/\gamma_{1}+(k-k_{+})/\gamma_{2}$, and $c_{1,2}\in(0,\infty)$ does
  not depend on $a$, and that the same holds, with a limit $c_{2,1}$, when
  $(C_{1},\gamma_{1})$ and $(C_{2},\gamma_{2})$ are interchanged inside the
  summand. Hence, the conclusion of the theorem holds for any deterministic
  sequence of design points that satisfies the two conditions, with $c_{1,2}$
  and $c_{2,1}$ in place of $c_{1,2}^{\ast}$ and $c_{2,1}^{\ast}$. Note,
  however, that the comparisons with the minimax rates over each parameter
  space that we make after Theorem~4.2 in
  Section~4.2 also require such conditions
  when both parameter
  spaces are $\Lambda_{+,\mathcal{V}}(\gamma_{1},C_{1})$, and when both are
  $\Lambda_{+,\mathcal{V}}(\gamma_{2},C_{2})$.
\end{remark}

\begin{remark}[Kernel representation]\label{rem:kernel}
  In the formulas for $\hat{c}_{[\delta]}^{\ell,j}$ and
  $\hat{c}_{[\delta]}^{u,j}$ in Corollary~5.1, the random
  terms are the weighted averages $\sum_{i}w_{j,i}^{-}y_{i}$ and
  $\sum_{i}w_{j,i}^{+}y_{i}$, and the remaining terms are nonrandom. If $\mathcal{V}=\{1,\dots,k\}$ (so
  the function is monotone in every coordinate), the random terms can be viewed
  as a kernel estimator with a bandwidth that depends on the design points and
  on the direction. To see this, if we
  define
  \[
    k(x)=\left[1-C_{j}\lVert (x)_{\mathcal{V}-}\rVert^{\gamma_{j}}-C_{J}\lVert (x)_{\mathcal{V}+}\rVert^{\gamma_{J}}\right]_{+},
  \]
  and, for $m=1,\dots,k$,
  \[
    h_{mn}(x)=\begin{cases}
      \omega(\delta,\mathcal{F}_{J},\mathcal{F}_{j}){}^{1/\gamma_{J}} & \text{if the }m\text{th coordinate of }x\geq0\\
      \omega(\delta,\mathcal{F}_{J},\mathcal{F}_{j}){}^{1/\gamma_{j}}
      & \text{otherwise},
    \end{cases}
  \]
  we have
  \[
    \frac{\sum_{i=1}^{n}D_{Jj,\delta}(x_{i})y_{i}}{\sum_{i=1}^{n}D_{Jj,\delta}(x_{i})}=\frac{\sum_{i=1}^{n}k(x_{1i}/h_{1n}(x_{i}),\dots,x_{ki}/h_{kn}(x_{i}))y_{i}}{\sum_{i=1}^{n}k(x_{1i}/h_{1n}(x_{i}),\dots,x_{ki}/h_{kn}(x_{i}))},
  \]
  where $x_{i}=(x_{1i},\dots,x_{ki})$. Hence, the CI can be considered to be
  based on a Nadaraya--Watson type estimator, correcting for the bias.
\end{remark}

\begin{remark}[H\"older exponents above one]\label{rem:gamma-above-1}
  We consider only exponents $\gamma\leq1$ because monotonicity yields no gain
  beyond the Lipschitz rate. For $\gamma\in(1,2]$, let
  $\Lambda^{(1)}(\gamma,C)$ denote the set of functions
  $f:\mathcal{D}\to\mathbb{R}$ that are continuously differentiable on an open
  set containing $\mathcal{D}$ and such that
  $\lVert\nabla f(x)-\nabla f(z)\rVert_{\ast}\leq C\lVert x-z\rVert^{\gamma-1}$
  for all $x,z\in\mathcal{D}$, where $\lVert\cdot\rVert_{\ast}$ is the dual
  norm of $\lVert\cdot\rVert$. Suppose that we adapt to
  $\mathcal{F}_{1}=\Lambda_{+,\mathcal{V}}(\gamma_{2},C_{2})\cap\Lambda^{(1)}(\gamma_{1},C_{1})$,
  with $\gamma_{1}>1$, while maintaining coverage over
  $\mathcal{F}_{2}=\Lambda_{+,\mathcal{V}}(\gamma_{2},C_{2})$. Then,
  $\mathcal{F}_{1}$ contains linear functions that increase in the coordinates
  in $\mathcal{V}$. A function in $\mathcal{F}_{2}$ whose value at $0$
  exceeds that of such a function by $b$ can coincide with it outside a region
  whose size in the monotone coordinates is of order $b$, rather than
  $b^{1/\gamma_{1}}$.

  Suppose Assumption~2.1 and (2)
  hold, $\mathcal{D}$ is convex, $0<\gamma_{2}\leq1<\gamma_{1}\leq2$, and the
  conditions on $\{X_{i}\}_{i=1}^{\infty}$ and $\sigma(\cdot)$ in
  Theorem~4.2 hold. Then, for almost all realizations of the
  design points and all $\delta>0$,
  $n^{r(1,\gamma_{2})}\omega(\delta,\mathcal{F}_{1},\mathcal{F}_{2})$ and
  $n^{r(1,\gamma_{2})}\omega(\delta,\mathcal{F}_{2},\mathcal{F}_{1})$ are
  bounded away from zero and infinity. Hence, by Lemma~3.1,
  the minimax rate of adaptation is $n^{-r(1,\gamma_{2})}$, the rate with
  $\gamma_{1}=1$. Even when $k_{+}=k$, it is slower than
  $n^{-1/(2+k/\gamma_{1})}$, the rate attained by a CI that is valid only over
  $\mathcal{F}_{1}$.

  In the univariate case,
  \citet[Section~5.1]{cai2013AdaptiveConfidenceIntervals} make a related
  observation. That is, with coverage over all monotone functions, the worst-case
  expected length over the monotone functions in a H\"older class with
  exponent above one is infinite because such a class contains linear functions
  of arbitrary slope. Their CIs adapt to H\"older exponents in $(0,1]$ under
  monotonicity, and to exponents in $[1,2]$ under convexity.
\end{remark}
\begin{proof}[Proof of the rate in Remark~\ref{rem:gamma-above-1}]
  Write $d=k_{+}+(k-k_{+})/\gamma_{2}$ and $r=r(1,\gamma_{2})=(2+d)^{-1}$. Let
  $\rho>0$ be such that $\{x:\lVert x\rVert\leq\rho\}\subseteq\mathcal{D}$,
  which exists since $0\in\mathrm{Int}\,\mathcal{X}$, and let $\Delta$ be the
  diameter of $\mathcal{D}$.

  We first show that the functions in $\mathcal{F}_{1}$ are Lipschitz with a
  common constant. Take $f\in\mathcal{F}_{1}$, let $g=\nabla f(0)$, and let
  $e\in\mathbb{R}^{k}$ with $\lVert e\rVert=1$. Since
  $\lvert\langle\nabla f(se)-g,e\rangle\rvert\leq\lVert\nabla f(se)-g\rVert_{\ast}\leq C_{1}s^{\gamma_{1}-1}$
  for $s\in[0,\rho]$, we have
  \[
    C_{2}\rho^{\gamma_{2}}\geq f(\rho e)-f(0)=\int_{0}^{\rho}\langle\nabla f(se),e\rangle ds\geq\rho\langle g,e\rangle-C_{1}\rho^{\gamma_{1}}/\gamma_{1}.
  \]
  Since $e$ is arbitrary,
  $\lVert g\rVert_{\ast}\leq C_{2}\rho^{\gamma_{2}-1}+C_{1}\rho^{\gamma_{1}-1}/\gamma_{1}$,
  and thus
  $\lVert\nabla f(x)\rVert_{\ast}\leq M:=C_{2}\rho^{\gamma_{2}-1}+C_{1}\rho^{\gamma_{1}-1}/\gamma_{1}+C_{1}\Delta^{\gamma_{1}-1}$
  for all $x\in\mathcal{D}$. Since $\mathcal{D}$ is convex, $f$ is Lipschitz
  with constant $M$, so that $\mathcal{F}_{1}\subseteq\Lambda_{+,\mathcal{V}}(1,M)$.

  We now show the upper bounds. By Lemma~4.3, applied to
  $\Lambda_{+,\mathcal{V}}(1,M)$ and to $\mathcal{F}_{2}$, every
  $f_{1}\in\mathcal{F}_{1}$ and $f_{2}\in\mathcal{F}_{2}$ with $f_{1}(0)=0$ and
  $f_{2}(0)=b$ satisfy
  $f_{2}(x_{i})-f_{1}(x_{i})\geq b-M\lVert(x_{i})_{\mathcal{V}+}\rVert-C_{2}\lVert(x_{i})_{\mathcal{V}-}\rVert^{\gamma_{2}}$.
  Hence, $\omega^{-1}(b,\mathcal{F}_{1},\mathcal{F}_{2})\geq\underline{H}_{n}(b)$,
  where $\underline{H}_{n}(b)$ denotes the right-hand side of
  (16) with $(\gamma_{1},C_{1})$ replaced by $(1,M)$. The
  proof of Theorem~4.2 uses the nesting of the parameter
  spaces only to show that this lower bound is attained, and it shows that,
  for almost all realizations of the design points,
  $\underline{H}_{n}(bn^{-r})^{2}\to\underline{c}\,b^{2+d}$ for all $b>0$, for
  some $\underline{c}\in(0,\infty)$. Since
  $\omega(\delta,\mathcal{F}_{1},\mathcal{F}_{2})\leq\sup\{b\geq0:\underline{H}_{n}(b)\leq\delta\}$,
  the last paragraph of that proof gives
  $\limsup_{n\to\infty}n^{r}\omega(\delta,\mathcal{F}_{1},\mathcal{F}_{2})\leq(\delta^{2}/\underline{c})^{r}$.
  The same argument, with
  $f_{1}(x)\geq f_{1}(0)-M\lVert(x)_{\mathcal{V}-}\rVert$ for
  $f_{1}\in\mathcal{F}_{1}$, gives the upper bound for
  $\omega(\delta,\mathcal{F}_{2},\mathcal{F}_{1})$.

  For the lower bounds, let $\ell(x)=s\sum_{v\in\mathcal{V}}x_{v}$, where $s>0$
  is small enough that
  $\lvert\ell(x)-\ell(z)\rvert\leq C_{2}\lVert x-z\rVert^{\gamma_{2}}$ for all
  $x,z\in\mathcal{D}$, and that
  $s\sum_{v\in\mathcal{V}}(\mp x_{v})_{+}\leq(C_{2}/2)\lVert(x)_{\mathcal{V}\mp}\rVert^{\gamma_{2}}$
  for all $x\in\mathcal{D}$. Such an $s$ exists since $\mathcal{D}$ is bounded
  and all norms on $\mathbb{R}^{k}$ are equivalent. Then, $\ell\in\mathcal{F}_{1}$,
  since its gradient is constant, and, for $b>0$, the function
  $f_{2}=\max\{b-C_{2}\lVert(\cdot)_{\mathcal{V}-}\rVert^{\gamma_{2}},\ell\}$ is
  in $\mathcal{F}_{2}$ by Lemmas~4.1 and
  4.2. Since $f_{2}(0)-\ell(0)=b$, the pair
  $(\ell,f_{2})$ is feasible in the modulus problem at $\delta$ whenever
  $\overline{H}_{n}(b)\leq\delta$, where
  $\overline{H}_{n}(b)=\big(\sum_{i=1}^{n}(b-C_{2}\lVert(x_{i})_{\mathcal{V}-}\rVert^{\gamma_{2}}-\ell(x_{i}))_{+}^{2}\big)^{1/2}$.
  We apply Lemma~\ref{lem:local-sums} to
  $g_{t}(x)=(1-(C_{2}/t)\lVert(x)_{\mathcal{V}-}\rVert^{\gamma_{2}}-\ell(x)/t)_{+}^{2}\mathds{1}\{x\in\mathcal{D}\}$,
  which takes values in $[0,1]$ by the choice of $s$, and whose indicator does
  not change the sums over the design points, since
  $\mathcal{X}\subseteq\mathcal{D}$. If $g_{t}(x)>0$, then
  $(C_{2}/2)\lVert(x)_{\mathcal{V}-}\rVert^{\gamma_{2}}<t$ and
  $s\sum_{v\in\mathcal{V}}(x_{v})_{+}<t+s\sum_{v\in\mathcal{V}}(-x_{v})_{+}<2t$,
  so that the support of $g_{t}$ shrinks to $0$. On an orthant on which the
  monotone coordinates in $S\subseteq\mathcal{V}$ are positive and the other
  monotone coordinates are negative, consider the change of variables
  $x=T_{t}z$, where $T_{t}$ scales the coordinates in $S$ by $t$ and the other
  coordinates by $t^{1/\gamma_{2}}$. Then, $(C_{2}/t)\lVert(x)_{\mathcal{V}-}\rVert^{\gamma_{2}}$
  does not change, and $\ell(x)/t$ converges to $s\sum_{v\in S}z_{v}$ if
  $\gamma_{2}<1$ and equals $s\sum_{v\in\mathcal{V}}z_{v}$ if $\gamma_{2}=1$.
  Since the integrand is bounded by one, vanishes outside a bounded set by the
  bounds above, and $\mathds{1}\{T_{t}z\in\mathcal{D}\}\to1$ because
  $0\in\mathrm{Int}\,\mathcal{D}$, the dominated convergence theorem shows that
  $t^{-(\lvert S\rvert+(k-\lvert S\rvert)/\gamma_{2})}$ times the integral of
  $g_{t}$ over this orthant converges to a positive and finite limit. The
  exponent is smallest when $S=\mathcal{V}$ if $\gamma_{2}<1$, and it does not
  depend on $S$ if $\gamma_{2}=1$, so that $t^{-d}\int g_{t}(x)dx$ converges
  to a positive and finite limit. As in the proof of
  Theorem~4.2, it follows that, for almost all realizations
  of the design points, $\overline{H}_{n}(bn^{-r})^{2}\to\bar{c}\,b^{2+d}$ for
  all $b>0$, for some $\bar{c}\in(0,\infty)$, and thus
  $\liminf_{n\to\infty}n^{r}\omega(\delta,\mathcal{F}_{1},\mathcal{F}_{2})\geq(\delta^{2}/\bar{c})^{r}$.
  For the reverse ordering, the same argument applies to the pair
  $(\min\{C_{2}\lVert(\cdot)_{\mathcal{V}+}\rVert^{\gamma_{2}},b+\ell\},b+\ell)\in\mathcal{F}_{2}\times\mathcal{F}_{1}$.
\end{proof}

\bibliographystyle{imsart-nameyear}
\bibliography{ModulusPartBib,KK_rfap,Phase1Refs}

%% file: tables/simulation_tables.tex
\newcommand{\RateDiagnosticTable}{%
\begin{table}[t]
  \caption{Average between class modulus of continuity ($k=2$, $(\gamma_{1},\gamma_{2})=(1,1/2)$, $C_{1}=C_{2}=1$, $\delta=z_{0.975}$)}
  \label{tab:rate-diagnostic}
  \centering
  \begin{tabular}{@{\extracolsep{5pt}} lccc} \\[-1.8ex]\hline
    \hline \\[-1.8ex] & $k_{+}=0$ & $k_{+}=1$ & $k_{+}=2$ \\
    \hline \\[-1.8ex]
    $n=250$ & $0.708$ & $0.625$ & $0.469$ \\
    $n=500$ & $0.619$ & $0.544$ & $0.397$ \\
    $n=1{,}000$ & $0.542$ & $0.474$ & $0.336$ \\
    $n=2{,}000$ & $0.476$ & $0.413$ & $0.284$ \\
    $n=5{,}000$ & $0.401$ & $0.345$ & $0.228$ \\
    $n=10{,}000$ & $0.353$ & $0.300$ & $0.192$ \\
    $n=20{,}000$ & $0.311$ & $0.262$ & $0.162$ \\
    \hline \\[-1.8ex]
    Fitted rate & $0.185$ & $0.197$ & $0.244$ \\
    $r(\gamma_{1},\gamma_{2})$ & $1/6=0.167$ & $1/5=0.200$ & $1/4=0.250$ \\
    $n^{r}\bar\omega$ at $n=20{,}000$ & $1.619$ & $1.899$ & $1.928$ \\
    Limit $(z_{0.975}^{2}/c^{\ast})^{r}$ & $1.447$ & $1.958$ & $1.957$ \\
    \hline \\[-1.8ex]
  \end{tabular}
  \par\smallskip
  \begin{minipage}{0.92\linewidth}\footnotesize
  $\bar\omega$ is the average over 1{,}000 uniform designs on $[-1/(2\sqrt{2}),1/(2\sqrt{2})]^{2}$ of $\omega_{+}(z_{0.975},\allowbreak\Lambda_{+,\mathcal{V}}(1,1),\allowbreak\Lambda_{+,\mathcal{V}}(1/2,1))$, and $r=r(\gamma_{1},\gamma_{2})$. The fitted rate is minus the slope of $\log\bar\omega$ on $\log n$ for $n\geq2{,}000$. The last row is the limit of $n^{r}\bar\omega$ given by the rate theorem, with $c^{\ast}=p_{X}(0)c_{+}$, $p_{X}(0)=2$ and the Euclidean norm: $c^{\ast}=2\pi/15$, $2/15$ and $\pi/12$ for $k_{+}=0,1,2$ (computed by numerical integration from its definition; it agrees with these closed forms to $10^{-15}$). The Monte Carlo standard errors of $n^{r}\bar\omega$ are at most $0.0008$ at $n=20{,}000$ and $0.0019$ at every $n$. The convergence is slow: $n^{r}\bar\omega$ is $1.777$, $1.690$ and $1.619$ for $k_{+}=0$; $1.886$, $1.887$ and $1.899$ for $k_{+}=1$; $1.865$, $1.902$ and $1.928$ for $k_{+}=2$ at $n=250$, $2{,}000$ and $20{,}000$.
  \end{minipage}
\end{table}}

\newcommand{\SimulationUniformTable}{%
\begin{table}[t]
  \caption{Average lengths and coverage probabilities of CIs under monotonicity in one variable ($\mathcal{V}=\{1\}$, $n=800$, $J=4$, $\mathcal{F}_{j}=\Lambda_{+,\mathcal{V}}(\gamma_{j},1)$ with $(\gamma_{1},\dots,\gamma_{4})=(1,4/5,2/3,1/2)$)}
  \label{tab:sim-one}
  \centering
  \begin{tabular}{@{\extracolsep{5pt}} lccccc} \\[-1.8ex]\hline
    \hline \\[-1.8ex] & AdaptBonf & AdaptCalib & CL & Conservative & Oracle \\
    \hline \\[-1.8ex]
    \multicolumn{6}{@{}l}{\textit{Average length}} \\
    $f_{1}$ (zero) & $0.869$ & $0.747$ & $2.566$ & $0.773$ & $0.413$ \\
    $f_{2}$ (linear) & $0.908$ & $0.787$ & $2.566$ & $0.773$ & $0.413$ \\
    $f_{3}$ (cusp) & $0.903$ & $0.784$ & $2.566$ & $0.773$ & $0.773$ \\
    $f_{4}$ (LF lower) & $0.907$ & $0.790$ & $2.566$ & $0.773$ & $0.773$ \\
    $f_{5}$ (LF upper) & $0.913$ & $0.800$ & $2.566$ & $0.773$ & $0.773$ \\
    $h_{+}$ & $0.909$ & $0.790$ & $2.566$ & $0.773$ & $0.773$ \\
    $h_{-}$ & $0.916$ & $0.802$ & $2.566$ & $0.773$ & $0.773$ \\
    \hline \\[-1.8ex]
    \multicolumn{6}{@{}l}{\textit{Coverage probability}} \\
    $f_{1}$ (zero) & $1.000$ & $0.998$ & $1.000$ & $1.000$ & $0.999$ \\
    $f_{2}$ (linear) & $1.000$ & $0.999$ & $1.000$ & $1.000$ & $0.999$ \\
    $f_{3}$ (cusp) & $0.997$ & $0.980$ & $1.000$ & $0.987$ & $0.987$ \\
    $f_{4}$ (LF lower) & $0.992$ & $0.955$ & $1.000$ & $0.959$ & $0.959$ \\
    $f_{5}$ (LF upper) & $0.994$ & $0.952$ & $1.000$ & $0.960$ & $0.960$ \\
    $h_{+}$ & $0.992$ & $0.955$ & $1.000$ & $0.952$ & $0.952$ \\
    $h_{-}$ & $0.993$ & $0.952$ & $1.000$ & $0.950$ & $0.950$ \\
    \hline \\[-1.8ex]
  \end{tabular}
  \par\smallskip
  \begin{minipage}{0.92\linewidth}\footnotesize
  Uniform design. Conservative and Oracle are the minimax affine fixed-length CIs of \citet{Donoho1994} for $\Lambda_{+,\mathcal{V}}(1/2,1)$ and for the smallest parameter space $\Lambda_{+,\mathcal{V}}(\gamma_{j},1)$ that contains the regression function. AdaptBonf uses $\tau=0.05/8$ ($z=2.498$), and AdaptCalib uses $\tau^{\ast}=0.03700$ ($z=1.787$), the largest grid point at which the bound of Lemma~S2.1 on $p_{\tau}(K_{\tau})$ is at most $0.05$. $f_{4}$ and $f_{5}$ (LF lower and LF upper) are the least favorable functions for the lower and the upper end of AdaptCalib that direct power to $\Lambda_{+,\mathcal{V}}(1,1)$, and $h_{+}(x)=\lVert(x)_{\mathcal{V}+}\rVert^{1/2}$ and $h_{-}(x)=-\lVert(x)_{\mathcal{V}-}\rVert^{1/2}$ are the envelopes of $\Lambda_{+,\mathcal{V}}(1/2,1)$. Each entry uses the same 10{,}000 draws of the noise. Monte Carlo standard errors are at most $0.0006$ for the average lengths and $0.0022$ for the coverage probabilities. No intersection was empty.
  \end{minipage}
\end{table}}

\newcommand{\SimulationLengthsTable}{%
\begin{table}[t]
  \caption{Average lengths of CIs under monotonicity in one variable ($\mathcal{V}=\{1\}$, $n=800$, $J=4$, $\mathcal{F}_{j}=\Lambda_{+,\mathcal{V}}(\gamma_{j},1)$ with $(\gamma_{1},\dots,\gamma_{4})=(1,4/5,2/3,1/2)$)}
  \label{tab:sim-lengths}
  \centering
  \begin{tabular}{@{\extracolsep{5pt}} lccccc} \\[-1.8ex]\hline
    \hline \\[-1.8ex] & AdaptBonf & AdaptCalib & CL & Conservative & Oracle \\
    \hline \\[-1.8ex]
    \multicolumn{6}{@{}l}{\textit{Uniform design}} \\
    $f_{1}$ (zero) & $0.869$ & $0.747$ & $2.566$ & $0.773$ & $0.413$ \\
    $f_{2}$ (linear) & $0.908$ & $0.787$ & $2.566$ & $0.773$ & $0.413$ \\
    $f_{3}$ (cusp) & $0.903$ & $0.784$ & $2.566$ & $0.773$ & $0.773$ \\
    $f_{4}$ (LF lower) & $0.907$ & $0.790$ & $2.566$ & $0.773$ & $0.773$ \\
    $f_{5}$ (LF upper) & $0.913$ & $0.800$ & $2.566$ & $0.773$ & $0.773$ \\
    $h_{+}$ & $0.909$ & $0.790$ & $2.566$ & $0.773$ & $0.773$ \\
    $h_{-}$ & $0.916$ & $0.802$ & $2.566$ & $0.773$ & $0.773$ \\
    \hline \\[-1.8ex]
    \multicolumn{6}{@{}l}{\textit{Center-concentrated design}} \\
    $f_{1}$ (zero) & $0.782$ & $0.676$ & $2.359$ & $0.703$ & $0.351$ \\
    $f_{2}$ (linear) & $0.815$ & $0.709$ & $2.359$ & $0.703$ & $0.351$ \\
    $f_{3}$ (cusp) & $0.813$ & $0.710$ & $2.359$ & $0.703$ & $0.703$ \\
    $f_{4}$ (LF lower) & $0.816$ & $0.715$ & $2.359$ & $0.703$ & $0.703$ \\
    $f_{5}$ (LF upper) & $0.817$ & $0.717$ & $2.359$ & $0.703$ & $0.703$ \\
    $h_{+}$ & $0.818$ & $0.717$ & $2.359$ & $0.703$ & $0.703$ \\
    $h_{-}$ & $0.821$ & $0.719$ & $2.359$ & $0.703$ & $0.703$ \\
    \hline \\[-1.8ex]
  \end{tabular}
  \par\smallskip
  \begin{minipage}{0.92\linewidth}\footnotesize
  Conservative and Oracle are the minimax affine fixed-length CIs of \citet{Donoho1994} for $\Lambda_{+,\mathcal{V}}(1/2,1)$ and for the smallest parameter space $\Lambda_{+,\mathcal{V}}(\gamma_{j},1)$ that contains the regression function. AdaptBonf uses $\tau=0.05/8$ ($z=2.498$). AdaptCalib uses $\tau^{\ast}=0.03700$ ($z_{1-\tau^{\ast}}=1.787$) on both designs, the largest grid point at which the bound of Lemma~S2.1 on $p_{\tau}(K_{\tau})$ is at most $0.05$. $f_{4}$ and $f_{5}$ (LF lower and LF upper) are the least favorable functions for the lower and the upper end of AdaptCalib that direct power to $\Lambda_{+,\mathcal{V}}(1,1)$, and $h_{+}(x)=\lVert(x)_{\mathcal{V}+}\rVert^{1/2}$ and $h_{-}(x)=-\lVert(x)_{\mathcal{V}-}\rVert^{1/2}$ are the envelopes of $\Lambda_{+,\mathcal{V}}(1/2,1)$. Each entry uses the same 10{,}000 draws of the noise. Monte Carlo standard errors of the average lengths are at most $0.0006$. No intersection was empty.
  \end{minipage}
\end{table}}

\newcommand{\SimulationCoverageTable}{%
\begin{table}[t]
  \caption{Coverage probabilities of CIs under monotonicity in one variable ($\mathcal{V}=\{1\}$, $n=800$, $J=4$, $\mathcal{F}_{j}=\Lambda_{+,\mathcal{V}}(\gamma_{j},1)$ with $(\gamma_{1},\dots,\gamma_{4})=(1,4/5,2/3,1/2)$)}
  \label{tab:sim-coverage}
  \centering
  \begin{tabular}{@{\extracolsep{5pt}} lccccc} \\[-1.8ex]\hline
    \hline \\[-1.8ex] & AdaptBonf & AdaptCalib & CL & Conservative & Oracle \\
    \hline \\[-1.8ex]
    \multicolumn{6}{@{}l}{\textit{Uniform design}} \\
    $f_{1}$ (zero) & $1.000$ & $0.998$ & $1.000$ & $1.000$ & $0.999$ \\
    $f_{2}$ (linear) & $1.000$ & $0.999$ & $1.000$ & $1.000$ & $0.999$ \\
    $f_{3}$ (cusp) & $0.997$ & $0.980$ & $1.000$ & $0.987$ & $0.987$ \\
    $f_{4}$ (LF lower) & $0.992$ & $0.955$ & $1.000$ & $0.959$ & $0.959$ \\
    $f_{5}$ (LF upper) & $0.994$ & $0.952$ & $1.000$ & $0.960$ & $0.960$ \\
    $h_{+}$ & $0.992$ & $0.955$ & $1.000$ & $0.952$ & $0.952$ \\
    $h_{-}$ & $0.993$ & $0.952$ & $1.000$ & $0.950$ & $0.950$ \\
    \hline \\[-1.8ex]
    \multicolumn{6}{@{}l}{\textit{Center-concentrated design}} \\
    $f_{1}$ (zero) & $1.000$ & $1.000$ & $1.000$ & $1.000$ & $0.998$ \\
    $f_{2}$ (linear) & $1.000$ & $1.000$ & $1.000$ & $1.000$ & $0.998$ \\
    $f_{3}$ (cusp) & $0.999$ & $0.982$ & $1.000$ & $0.991$ & $0.991$ \\
    $f_{4}$ (LF lower) & $0.993$ & $0.956$ & $1.000$ & $0.959$ & $0.959$ \\
    $f_{5}$ (LF upper) & $0.992$ & $0.952$ & $1.000$ & $0.959$ & $0.959$ \\
    $h_{+}$ & $0.992$ & $0.956$ & $1.000$ & $0.952$ & $0.952$ \\
    $h_{-}$ & $0.992$ & $0.952$ & $1.000$ & $0.952$ & $0.952$ \\
    \hline \\[-1.8ex]
  \end{tabular}
  \par\smallskip
  \begin{minipage}{0.92\linewidth}\footnotesize
  Conservative and Oracle are the minimax affine fixed-length CIs of \citet{Donoho1994} for $\Lambda_{+,\mathcal{V}}(1/2,1)$ and for the smallest parameter space $\Lambda_{+,\mathcal{V}}(\gamma_{j},1)$ that contains the regression function. AdaptBonf uses $\tau=0.05/8$ ($z=2.498$). AdaptCalib uses $\tau^{\ast}=0.03700$ ($z_{1-\tau^{\ast}}=1.787$) on both designs, the largest grid point at which the bound of Lemma~S2.1 on $p_{\tau}(K_{\tau})$ is at most $0.05$. $f_{4}$ and $f_{5}$ (LF lower and LF upper) are the least favorable functions for the lower and the upper end of AdaptCalib that direct power to $\Lambda_{+,\mathcal{V}}(1,1)$, and $h_{+}(x)=\lVert(x)_{\mathcal{V}+}\rVert^{1/2}$ and $h_{-}(x)=-\lVert(x)_{\mathcal{V}-}\rVert^{1/2}$ are the envelopes of $\Lambda_{+,\mathcal{V}}(1/2,1)$. Each entry uses the same 10{,}000 draws of the noise. Monte Carlo standard errors of the coverage probabilities are at most $0.0022$. No intersection was empty.
  \end{minipage}
\end{table}}

%% file: tables/aos_numerics_tables.tex
\newcommand{\FullyMonotoneLatticeTable}{%
\begin{table}[t]
  \caption{Average lengths and coverage probabilities of CIs under monotonicity in both variables (lattice design, $\mathcal{V}=\{1,2\}$, $n=841$, $J=4$, $\mathcal{F}_{j}=\Lambda_{+,\mathcal{V}}(\gamma_{j},1)$ with $(\gamma_{1},\dots,\gamma_{4})=(1,4/5,2/3,1/2)$)}
  \label{tab:fm-lattice}
  \centering
  \setlength{\tabcolsep}{3pt}
  \begin{tabular}{@{\extracolsep{5pt}} lccccccc} \\[-1.8ex]\hline
    \hline \\[-1.8ex] & AdaptBonf & AdaptCalib & AdaptMono & CL & Conservative & Oracle & DHZ \\
    \hline \\[-1.8ex]
    \multicolumn{8}{@{}l}{\textit{Average length}} \\
    $f_{1}$ (zero) & $0.458$ & $0.401$ & $0.373$ & $2.032$ & $0.608$ & $0.356$ & $0.230$\rlap{\smash{$^{\dagger}$}} \\
    $f_{2}$ (linear) & $0.650$ & $0.575$ & $0.603$ & $2.032$ & $0.608$ & $0.356$ & $0.423$ \\
    $f_{3}$ (cusp) & $0.685$ & $0.618$ & $0.685$ & $2.032$ & $0.608$ & $0.608$ & $0.485$\rlap{\smash{$^{\dagger}$}} \\
    $f_{4}$ (LF lower) & $0.681$ & $0.622$ & $0.690$ & $2.032$ & $0.608$ & $0.608$ & $0.482$\rlap{\smash{$^{\dagger}$}} \\
    $f_{5}$ (LF upper) & $0.681$ & $0.622$ & $0.690$ & $2.032$ & $0.608$ & $0.608$ & $0.485$\rlap{\smash{$^{\dagger}$}} \\
    $h_{+}$ & $0.716$ & $0.649$ & $0.762$ & $2.032$ & $0.608$ & $0.608$ & $0.546$\rlap{\smash{$^{\dagger}$}} \\
    $h_{-}$ & $0.716$ & $0.649$ & $0.762$ & $2.032$ & $0.608$ & $0.608$ & $0.545$\rlap{\smash{$^{\dagger}$}} \\
    \hline \\[-1.8ex]
    \multicolumn{8}{@{}l}{\textit{Coverage probability}} \\
    $f_{1}$ (zero) & $0.998$ & $0.981$ & $0.952$ & $1.000$ & $1.000$ & $0.999$ & $0.955$\rlap{\smash{$^{\dagger}$}} \\
    $f_{2}$ (linear) & $1.000$ & $0.999$ & $0.998$ & $1.000$ & $1.000$ & $0.999$ & $0.955$ \\
    $f_{3}$ (cusp) & $0.994$ & $0.972$ & $0.975$ & $1.000$ & $0.988$ & $0.988$ & $0.702$\rlap{\smash{$^{\dagger}$}} \\
    $f_{4}$ (LF lower) & $0.992$ & $0.962$ & $0.975$ & $1.000$ & $0.986$ & $0.986$ & $0.681$\rlap{\smash{$^{\dagger}$}} \\
    $f_{5}$ (LF upper) & $0.992$ & $0.965$ & $0.976$ & $1.000$ & $0.987$ & $0.987$ & $0.682$\rlap{\smash{$^{\dagger}$}} \\
    $h_{+}$ & $0.989$ & $0.958$ & $0.975$ & $1.000$ & $0.949$ & $0.949$ & $0.677$\rlap{\smash{$^{\dagger}$}} \\
    $h_{-}$ & $0.991$ & $0.959$ & $0.976$ & $1.000$ & $0.949$ & $0.949$ & $0.678$\rlap{\smash{$^{\dagger}$}} \\
    \hline \\[-1.8ex]
  \end{tabular}
  \par\smallskip
  \begin{minipage}{0.92\linewidth}\footnotesize
  $^{\dagger}$The critical value 1.80, which \citet{deng2021confidence} suggest for $k=2$, is calibrated for regression functions whose first partial derivatives at $x_{0}$ are nonzero, the case of their Theorem 2; among $f_{1}$ to $f_{5}$ and $h_{\pm}$, this holds only for $f_{2}$. Their Theorem 3 covers the other functions that satisfy their Assumption A, including the constant $f_{1}$, but with a larger critical value that they do not tabulate. AdaptMono directs power to the same four parameter spaces and maintains coverage over all functions that are monotone in both variables. Conservative and Oracle are the minimax affine fixed-length CIs of \citet{Donoho1994} for $\Lambda_{+,\mathcal{V}}(1/2,1)$ and for the smallest parameter space $\Lambda_{+,\mathcal{V}}(\gamma_{j},1)$ that contains the regression function. DHZ is the CI of \citet{deng2021confidence} with critical value $1.80$ and known $\sigma$. The lattice contains $x_{0}=0$. AdaptBonf uses $\tau=0.05/8$ ($z=2.498$), AdaptCalib $\tau^{\ast}=0.02375$ ($z=1.982$) and AdaptMono $\tau^{\ast}=0.02075$ ($z=2.038$). $h_{+}(x)=\lVert(x)_{\mathcal{V}+}\rVert^{1/2}$ and $h_{-}(x)=-\lVert(x)_{\mathcal{V}-}\rVert^{1/2}$ are the envelopes of $\Lambda_{+,\mathcal{V}}(1/2,1)$. Monte Carlo standard errors are at most $0.0029$ for the average lengths and $0.0047$ for the coverage probabilities. Empty intersections (length zero, counted as not covering) occurred in at most 0.19\% of the draws.
  \end{minipage}
\end{table}}

\newcommand{\DHZBalancedTable}{%
\begin{table}[t]
  \caption{Average lengths and coverage probabilities for a balanced set of regression functions (lattice design, $\mathcal{V}=\{1,2\}$, $n=841$, $J=4$, $\mathcal{F}_{j}=\Lambda_{+,\mathcal{V}}(\gamma_{j},1)$ with $(\gamma_{1},\dots,\gamma_{4})=(1,4/5,2/3,1/2)$)}
  \label{tab:dhz-balanced}
  \centering
  \setlength{\tabcolsep}{3pt}
  \begin{tabular}{@{\extracolsep{3pt}} lcccccccc} \\[-1.8ex]\hline
    \hline \\[-1.8ex] & & & \multicolumn{2}{c}{AdaptCalib} & \multicolumn{2}{c}{AdaptMono} & \multicolumn{2}{c}{DHZ} \\
    \cline{4-5} \cline{6-7} \cline{8-9} \\[-1.8ex]
    & $\gamma_{j}$ & DHZ case & Length & Cov. & Length & Cov. & Length & Cov. \\
    \hline \\[-1.8ex]
    (a) $0.6x_{1}$ & $1$ & Asm.~A & $0.501$ & $0.997$ & $0.499$ & $0.989$ & $0.343$ & $0.944$ \\
    (b) $0.6x_{1}+0.05x_{2}$ & $1$ & Thm.~2 & $0.508$ & $0.997$ & $0.510$ & $0.990$ & $0.351$ & $0.945$ \\
    (c) $1.5(x_{1}^{3}+x_{2}^{3})$ & $1$ & Asm.~A & $0.418$ & $0.985$ & $0.395$ & $0.962$ & $0.261$ & $0.955$ \\
    (d) softplus, $s=0.02$ & $1$ & Thm.~2 & $0.516$ & $0.992$ & $0.520$ & $0.983$ & $0.347$ & $0.909$ \\
    (e) softplus, $s=0.05$ & $1$ & Thm.~2 & $0.516$ & $0.995$ & $0.521$ & $0.987$ & $0.350$ & $0.934$ \\
    (f) $\lVert(x)_{\mathcal{V}+}\rVert$ & $1$ & no & $0.530$ & $0.988$ & $0.540$ & $0.975$ & $0.368$ & $0.847$ \\
    (g) $(x_{1})_{+}+(x_{2})_{+}$ & $1/2$ & no & $0.557$ & $0.987$ & $0.580$ & $0.975$ & $0.385$ & $0.855$ \\
    (h) $2\lVert(x)_{\mathcal{V}+}\rVert^{1/2}$ & none & no & $0.689$ & $0.559$ & $1.139$ & $0.975$ & $0.809$ & $0.647$ \\
    (i) $0.8\lVert(x)_{\mathcal{V}+}\rVert^{1/4}$ & none & no & $0.674$ & $0.877$ & $0.863$ & $0.975$ & $0.656$ & $0.579$ \\
    \hline \\[-1.8ex]
  \end{tabular}
  \par\smallskip
  \begin{minipage}{0.92\linewidth}\footnotesize
  Softplus is $f_{s}(x)=(s/\sqrt{2})\sum_{k}\log(1+e^{x_{k}/s})$ minus its value at $0$. $\gamma_{j}$: the largest $\gamma_{j}$ such that $\Lambda_{+,\mathcal{V}}(\gamma_{j},1)$ contains $f$ (checked numerically on the design and on an $81\times81$ grid), or none. DHZ case: Thm.~2 if both first partial derivatives of $f$ at $x_{0}$ are positive, the case of Theorem 2 of \citet{deng2021confidence}, for which the critical value $1.80$ is calibrated; Asm.~A if $f$ satisfies their Assumption A but not this condition, so that their Theorem 3 applies with a larger, untabulated critical value; no if $f$ violates Assumption A (a kink or a cusp at $x_{0}$). AdaptCalib and AdaptMono are those of the lattice design with $f_{1}$ to $f_{5}$ ($\tau^{\ast}=0.02375$ and $0.02075$); the conservative fixed-length CI has length $0.608$. DHZ uses the critical value $1.80$ and known $\sigma$. Monte Carlo standard errors are at most $0.0041$ for the average lengths and $0.0050$ for the coverage probabilities. Empty intersections (length zero, counted as not covering) occurred in at most 0.12\% of the draws.
  \end{minipage}
\end{table}}

\newcommand{\SimulationDetailsLengthsTable}{%
\begin{table}[t]
  \caption{Average lengths of all CIs under monotonicity in one variable ($\mathcal{V}=\{1\}$, $n=800$, $J=4$, $\mathcal{F}_{j}=\Lambda_{+,\mathcal{V}}(\gamma_{j},1)$ with $(\gamma_{1},\dots,\gamma_{4})=(1,4/5,2/3,1/2)$)}
  \label{tab:sim-details-lengths}
  \centering
  \setlength{\tabcolsep}{4pt}
  \begin{tabular}{@{\extracolsep{5pt}} lccccccc} \\[-1.8ex]\hline
    \hline \\[-1.8ex] & $f_{1}$ & $f_{2}$ & $f_{3}$ & $f_{4}$ & $f_{5}$ & $h_{+}$ & $h_{-}$ \\
    \hline \\[-1.8ex]
    \multicolumn{8}{@{}l}{\textit{Uniform design}} \\
    AdaptBonf & $0.869$ & $0.908$ & $0.903$ & $0.907$ & $0.913$ & $0.909$ & $0.916$ \\
    AdaptCalib & $0.747$ & $0.787$ & $0.784$ & $0.790$ & $0.800$ & $0.790$ & $0.802$ \\
    AdaptCalib ($K=0$) & $0.802$ & $0.842$ & $0.837$ & $0.842$ & $0.851$ & $0.844$ & $0.854$ \\
    AdaptCalib ($\delta=2z$) & $0.747$ & $0.773$ & $0.768$ & $0.771$ & $0.772$ & $0.770$ & $0.771$ \\
    CL & $2.566$ & $2.566$ & $2.566$ & $2.566$ & $2.566$ & $2.566$ & $2.566$ \\
    Conservative & $0.773$ & $0.773$ & $0.773$ & $0.773$ & $0.773$ & $0.773$ & $0.773$ \\
    Oracle & $0.413$ & $0.413$ & $0.773$ & $0.773$ & $0.773$ & $0.773$ & $0.773$ \\
    Conservative ($J=1$) & $0.873$ & $0.873$ & $0.873$ & $0.873$ & $0.873$ & $0.873$ & $0.873$ \\
    Oracle ($J=1$) & $0.481$ & $0.481$ & $0.873$ & $0.873$ & $0.873$ & $0.873$ & $0.873$ \\
    \hline \\[-1.8ex]
    \multicolumn{8}{@{}l}{\textit{Center-concentrated design}} \\
    AdaptBonf & $0.782$ & $0.815$ & $0.813$ & $0.816$ & $0.817$ & $0.818$ & $0.821$ \\
    AdaptCalib & $0.676$ & $0.709$ & $0.710$ & $0.715$ & $0.717$ & $0.717$ & $0.719$ \\
    AdaptCalib ($K=0$) & $0.724$ & $0.757$ & $0.757$ & $0.761$ & $0.762$ & $0.763$ & $0.765$ \\
    CL & $2.359$ & $2.359$ & $2.359$ & $2.359$ & $2.359$ & $2.359$ & $2.359$ \\
    Conservative & $0.703$ & $0.703$ & $0.703$ & $0.703$ & $0.703$ & $0.703$ & $0.703$ \\
    Oracle & $0.351$ & $0.351$ & $0.703$ & $0.703$ & $0.703$ & $0.703$ & $0.703$ \\
    Conservative ($J=1$) & $0.787$ & $0.787$ & $0.787$ & $0.787$ & $0.787$ & $0.787$ & $0.787$ \\
    Oracle ($J=1$) & $0.405$ & $0.405$ & $0.787$ & $0.787$ & $0.787$ & $0.787$ & $0.787$ \\
    \hline \\[-1.8ex]
  \end{tabular}
  \par\smallskip
  \begin{minipage}{0.92\linewidth}\footnotesize
  Conservative and Oracle are the minimax affine fixed-length CIs of \citet{Donoho1994} for $\Lambda_{+,\mathcal{V}}(1/2,1)$ and for the smallest parameter space $\Lambda_{+,\mathcal{V}}(\gamma_{j},1)$ that contains the regression function. The rows marked $J=1$ use the construction of AdaptBonf with $J=1$. AdaptBonf uses $\tau=0.05/8$ ($z=2.498$). AdaptCalib uses $\tau^{\ast}=0.03700$ ($z=1.787$) on both designs. AdaptCalib ($K=0$) sets $K_{\tau}=0$, which gives $\tau^{\ast}_{0}=0.01850$ on the uniform design and $\tau^{\ast}_{0}=0.01825$ on the center-concentrated design. AdaptCalib ($\delta=2z$) evaluates the moduli at twice the critical value and has its own $\tau^{\ast}=0.04050$ ($z=1.745$); it was computed for the uniform design only. $f_{4}$ and $f_{5}$ (LF lower and LF upper) are the least favorable functions for the lower and the upper end of AdaptCalib that direct power to $\Lambda_{+,\mathcal{V}}(1,1)$, and $h_{\pm}$ are the envelopes of $\Lambda_{+,\mathcal{V}}(1/2,1)$. Monte Carlo standard errors of the average lengths are at most $0.0006$. No intersection was empty.
  \end{minipage}
\end{table}}

\newcommand{\SimulationDetailsCoverageTable}{%
\begin{table}[t]
  \caption{Coverage probabilities of all CIs under monotonicity in one variable ($\mathcal{V}=\{1\}$, $n=800$, $J=4$, $\mathcal{F}_{j}=\Lambda_{+,\mathcal{V}}(\gamma_{j},1)$ with $(\gamma_{1},\dots,\gamma_{4})=(1,4/5,2/3,1/2)$)}
  \label{tab:sim-details-coverage}
  \centering
  \setlength{\tabcolsep}{4pt}
  \begin{tabular}{@{\extracolsep{5pt}} lccccccc} \\[-1.8ex]\hline
    \hline \\[-1.8ex] & $f_{1}$ & $f_{2}$ & $f_{3}$ & $f_{4}$ & $f_{5}$ & $h_{+}$ & $h_{-}$ \\
    \hline \\[-1.8ex]
    \multicolumn{8}{@{}l}{\textit{Uniform design}} \\
    AdaptBonf & $1.000$ & $1.000$ & $0.997$ & $0.992$ & $0.994$ & $0.992$ & $0.993$ \\
    AdaptCalib & $0.998$ & $0.999$ & $0.980$ & $0.955$ & $0.952$ & $0.955$ & $0.952$ \\
    AdaptCalib ($K=0$) & $1.000$ & $1.000$ & $0.991$ & $0.978$ & $0.977$ & $0.978$ & $0.977$ \\
    AdaptCalib ($\delta=2z$) & $1.000$ & $1.000$ & $0.990$ & $0.964$ & $0.964$ & $0.954$ & $0.951$ \\
    CL & $1.000$ & $1.000$ & $1.000$ & $1.000$ & $1.000$ & $1.000$ & $1.000$ \\
    Conservative & $1.000$ & $1.000$ & $0.987$ & $0.959$ & $0.960$ & $0.952$ & $0.950$ \\
    Oracle & $0.999$ & $0.999$ & $0.987$ & $0.959$ & $0.960$ & $0.952$ & $0.950$ \\
    Conservative ($J=1$) & $1.000$ & $1.000$ & $0.990$ & $0.978$ & $0.976$ & $0.978$ & $0.975$ \\
    Oracle ($J=1$) & $0.996$ & $0.996$ & $0.990$ & $0.978$ & $0.976$ & $0.978$ & $0.975$ \\
    \hline \\[-1.8ex]
    \multicolumn{8}{@{}l}{\textit{Center-concentrated design}} \\
    AdaptBonf & $1.000$ & $1.000$ & $0.999$ & $0.993$ & $0.992$ & $0.992$ & $0.992$ \\
    AdaptCalib & $1.000$ & $1.000$ & $0.982$ & $0.956$ & $0.952$ & $0.956$ & $0.952$ \\
    AdaptCalib ($K=0$) & $1.000$ & $1.000$ & $0.993$ & $0.980$ & $0.975$ & $0.979$ & $0.975$ \\
    CL & $1.000$ & $1.000$ & $1.000$ & $1.000$ & $1.000$ & $1.000$ & $1.000$ \\
    Conservative & $1.000$ & $1.000$ & $0.991$ & $0.959$ & $0.959$ & $0.952$ & $0.952$ \\
    Oracle & $0.998$ & $0.998$ & $0.991$ & $0.959$ & $0.959$ & $0.952$ & $0.952$ \\
    Conservative ($J=1$) & $1.000$ & $1.000$ & $0.993$ & $0.979$ & $0.976$ & $0.978$ & $0.975$ \\
    Oracle ($J=1$) & $0.997$ & $0.997$ & $0.993$ & $0.979$ & $0.976$ & $0.978$ & $0.975$ \\
    \hline \\[-1.8ex]
  \end{tabular}
  \par\smallskip
  \begin{minipage}{0.92\linewidth}\footnotesize
  Conservative and Oracle are the minimax affine fixed-length CIs of \citet{Donoho1994} for $\Lambda_{+,\mathcal{V}}(1/2,1)$ and for the smallest parameter space $\Lambda_{+,\mathcal{V}}(\gamma_{j},1)$ that contains the regression function. The rows marked $J=1$ use the construction of AdaptBonf with $J=1$. AdaptBonf uses $\tau=0.05/8$ ($z=2.498$). AdaptCalib uses $\tau^{\ast}=0.03700$ ($z=1.787$) on both designs. AdaptCalib ($K=0$) sets $K_{\tau}=0$, which gives $\tau^{\ast}_{0}=0.01850$ on the uniform design and $\tau^{\ast}_{0}=0.01825$ on the center-concentrated design. AdaptCalib ($\delta=2z$) evaluates the moduli at twice the critical value and has its own $\tau^{\ast}=0.04050$ ($z=1.745$); it was computed for the uniform design only. $f_{4}$ and $f_{5}$ (LF lower and LF upper) are the least favorable functions for the lower and the upper end of AdaptCalib that direct power to $\Lambda_{+,\mathcal{V}}(1,1)$, and $h_{\pm}$ are the envelopes of $\Lambda_{+,\mathcal{V}}(1/2,1)$. Monte Carlo standard errors of the coverage probabilities are at most $0.0022$. No intersection was empty.
  \end{minipage}
\end{table}}

\newcommand{\LargerNTable}{%
\begin{table}[t]
  \caption{Average lengths of adaptive CIs relative to the fixed-length CIs, and coverage of AdaptCalib, as $n$ grows ($\mathcal{V}=\{1\}$, uniform design, $J=4$, $\mathcal{F}_{j}=\Lambda_{+,\mathcal{V}}(\gamma_{j},1)$ with $(\gamma_{1},\dots,\gamma_{4})=(1,4/5,2/3,1/2)$)}
  \label{tab:larger-n}
  \centering
  \setlength{\tabcolsep}{4pt}
  \begin{tabular}{@{\extracolsep{5pt}} lccccc} \\[-1.8ex]\hline
    \hline \\[-1.8ex] & \multicolumn{2}{c}{Relative to Conservative} & \multicolumn{2}{c}{Relative to Oracle} & Coverage \\
    \cline{2-3} \cline{4-5} \cline{6-6} \\[-1.8ex]
    & AdaptBonf & AdaptCalib & AdaptBonf & AdaptCalib & AdaptCalib \\
    \hline \\[-1.8ex]
    \multicolumn{6}{@{}l}{\textit{$f_{1}$ (zero)}} \\
    $n=800$ & $1.124$ & $0.966$ & $2.104$ & $1.808$ & $0.998$ \\
    $n=5{,}000$ & $1.053$ & $0.920$ & $2.294$ & $2.003$ & $1.000$ \\
    $n=20{,}000$ & $1.006$ & $0.889$ & $2.484$ & $2.194$ & $1.000$ \\
    \hline \\[-1.8ex]
    \multicolumn{6}{@{}l}{\textit{$f_{2}$ (linear)}} \\
    $n=800$ & $1.174$ & $1.017$ & $2.198$ & $1.904$ & $0.999$ \\
    $n=5{,}000$ & $1.132$ & $0.990$ & $2.465$ & $2.155$ & $1.000$ \\
    $n=20{,}000$ & $1.094$ & $0.964$ & $2.700$ & $2.381$ & $1.000$ \\
    \hline \\[-1.8ex]
    \multicolumn{6}{@{}l}{\textit{$f_{3}$ (cusp)}} \\
    $n=800$ & $1.168$ & $1.013$ & $1.168$ & $1.013$ & $0.980$ \\
    $n=5{,}000$ & $1.141$ & $1.003$ & $1.141$ & $1.003$ & $0.978$ \\
    $n=20{,}000$ & $1.122$ & $0.996$ & $1.122$ & $0.996$ & $0.979$ \\
    \hline \\[-1.8ex]
    \multicolumn{6}{@{}l}{\textit{$f_{4}$ (LF lower)}} \\
    $n=800$ & $1.173$ & $1.021$ & $1.173$ & $1.021$ & $0.955$ \\
    $n=5{,}000$ & $1.142$ & $1.010$ & $1.142$ & $1.010$ & $0.950$ \\
    $n=20{,}000$ & $1.117$ & $0.999$ & $1.117$ & $0.999$ & $0.947$ \\
    \hline \\[-1.8ex]
    \multicolumn{6}{@{}l}{\textit{$f_{5}$ (LF upper)}} \\
    $n=800$ & $1.180$ & $1.035$ & $1.180$ & $1.035$ & $0.952$ \\
    $n=5{,}000$ & $1.147$ & $1.016$ & $1.147$ & $1.016$ & $0.953$ \\
    $n=20{,}000$ & $1.118$ & $1.002$ & $1.118$ & $1.002$ & $0.948$ \\
    \hline \\[-1.8ex]
    \multicolumn{6}{@{}l}{\textit{$h_{+}$}} \\
    $n=800$ & $1.175$ & $1.022$ & $1.175$ & $1.022$ & $0.955$ \\
    $n=5{,}000$ & $1.154$ & $1.017$ & $1.154$ & $1.017$ & $0.950$ \\
    $n=20{,}000$ & $1.138$ & $1.013$ & $1.138$ & $1.013$ & $0.946$ \\
    \hline \\[-1.8ex]
    \multicolumn{6}{@{}l}{\textit{$h_{-}$}} \\
    $n=800$ & $1.185$ & $1.037$ & $1.185$ & $1.037$ & $0.952$ \\
    $n=5{,}000$ & $1.161$ & $1.025$ & $1.161$ & $1.025$ & $0.953$ \\
    $n=20{,}000$ & $1.138$ & $1.015$ & $1.138$ & $1.015$ & $0.948$ \\
    \hline \\[-1.8ex]
  \end{tabular}
  \par\smallskip
  \begin{minipage}{0.92\linewidth}\footnotesize
  Entries in the first four columns are ratios of average lengths. Conservative and Oracle are the minimax affine fixed-length CIs of \citet{Donoho1994} for $\Lambda_{+,\mathcal{V}}(1/2,1)$ and for the smallest parameter space $\Lambda_{+,\mathcal{V}}(\gamma_{j},1)$ that contains the regression function. For $f_{3}$ to $f_{5}$ and $h_{\pm}$ the oracle CI is the conservative CI. At $n=800$, $5{,}000$ and $20{,}000$, the conservative CI has length $0.773$, $0.577$ and $0.460$, the oracle CI for $\Lambda_{+,\mathcal{V}}(1,1)$ has length $0.413$, $0.265$ and $0.186$, the CL interval is $3.32$, $3.35$ and $3.38$ times as long as the conservative CI, and AdaptCalib uses $\tau^{\ast}=0.03700$, $0.03400$ and $0.03225$. The design for $n=800$ is the uniform design of Table~1 with $\mathcal{V}=\{1\}$. Monte Carlo standard errors of the coverage probabilities are at most $0.0023$.
  \end{minipage}
\end{table}}

\newcommand{\AOSCalibrationTable}{%
\begin{table}[t]
  \caption{Calibration: $\tau^{\ast}$, $K_{\tau}$ and critical values ($\alpha=0.05$; $\mathcal{F}_{j}=\Lambda_{+,\mathcal{V}}(\gamma_{j},1)$ with $(\gamma_{1},\dots,\gamma_{4})=(1,4/5,2/3,1/2)$, or $\mathcal{F}_{j}=\Lambda_{+,\mathcal{V}}(1,C_{j})$ for adaptation over $C$)}
  \label{tab:aos-calibration}
  \centering
  \setlength{\tabcolsep}{4pt}
  \begin{tabular}{@{\extracolsep{5pt}} lcccccc} \\[-1.8ex]\hline
    \hline \\[-1.8ex] & $\tau^{\ast}_{0}$ & $z_{1-\tau^{\ast}_{0}}$ & $\tau^{\ast}$ & $K$ & $z_{1-\tau^{\ast}}$ & $z_{\rm Bon}$ \\
    \hline \\[-1.8ex]
    \multicolumn{7}{@{}l}{\textit{$\mathcal{V}=\{1\}$, $J=4$}} \\
    Uniform, $n=800$ & $0.01850$ & $2.086$ & $0.03700$ & $2.642$ & $1.787$ & $2.498$ \\
    Uniform, $n=800$, $\delta=2z$ & $0.02025$ & $2.049$ & $0.04050$ & $5.535$ & $1.745$ & $2.498$ \\
    Center-concentrated, $n=800$ & $0.01825$ & $2.091$ & $0.03700$ & $3.010$ & $1.787$ & $2.498$ \\
    Uniform, $n=5{,}000$ & $0.01700$ & $2.120$ & $0.03400$ & $2.469$ & $1.825$ & $2.498$ \\
    Uniform, $n=20{,}000$ & $0.01575$ & $2.151$ & $0.03225$ & $2.527$ & $1.849$ & $2.498$ \\
    \hline \\[-1.8ex]
    \multicolumn{7}{@{}l}{\textit{$\mathcal{V}=\{1,2\}$, $J=4$}} \\
    Uniform, AdaptCalib & $0.01450$ & $2.183$ & $0.02500$ & $0.711$ & $1.960$ & $2.498$ \\
    Uniform, AdaptMono & $0.02200$ & $2.014$ & $0.02200$ & $0.000$ & $2.014$ & $2.498$ \\
    Lattice, AdaptCalib & $0.01450$ & $2.183$ & $0.02375$ & $0.599$ & $1.982$ & $2.498$ \\
    Lattice, AdaptMono & $0.02075$ & $2.038$ & $0.02075$ & $0.000$ & $2.038$ & $2.498$ \\
    \hline \\[-1.8ex]
    \multicolumn{7}{@{}l}{\textit{Adaptation over $C$ ($\gamma=1$)}} \\
    $\mathcal{V}=\{1\}$, uniform, $J=5$, AdaptCalib & $0.01225$ & $2.249$ & $0.02025$ & $0.669$ & $2.049$ & $2.576$ \\
    $\mathcal{V}=\{1,2\}$, lattice, $J=5$, AdaptCalib & $0.00925$ & $2.355$ & $0.01000$ & $0.061$ & $2.326$ & $2.576$ \\
    $\mathcal{V}=\{1,2\}$, lattice, $J=5$, AdaptMono & $0.01025$ & $2.317$ & $0.01025$ & $0.000$ & $2.317$ & $2.576$ \\
    $\mathcal{V}=\{1\}$, uniform, $J=17$, AdaptCalib & $0.01172$ & $2.266$ & $0.01947$ & $0.676$ & $2.065$ & $2.974$ \\
    $\mathcal{V}=\{1,2\}$, lattice, $J=17$, AdaptCalib & $0.00872$ & $2.377$ & $0.00922$ & $0.064$ & $2.357$ & $2.974$ \\
    $\mathcal{V}=\{1,2\}$, lattice, $J=17$, AdaptMono & $0.00972$ & $2.337$ & $0.00972$ & $0.000$ & $2.337$ & $2.974$ \\
    \hline \\[-1.8ex]
  \end{tabular}
  \par\smallskip
  \begin{minipage}{0.92\linewidth}\footnotesize
  $\tau^{\ast}$ is the largest point of the grid $\alpha/(2J)+0.00025m$ (together with $\alpha$) at which the bound of Lemma~S2.1 on $p_{\tau}(K_{\tau})$ is at most $\alpha$, and $\tau^{\ast}_{0}$ uses $K_{\tau}=0$. The bound uses Hunter's inequality with a maximum-weight spanning tree and exact bivariate normal probabilities, and splits $[0,K]$ into 24 subintervals. No numerical tolerance is added. $K$ is $K_{\tau}$ at $\tau^{\ast}$, given by $K_{\tau}=\delta^{-1}\min_{j,\ell}\sum_{i}\min\{D^{-}_{j,i},D^{+}_{\ell,i}\}\{h_{+}(x_{i})-h_{-}(x_{i})\}$, where $D^{-}_{j}$ and $D^{+}_{l}$ are the differences of the least favorable pairs of the lower and the upper ends, $h_{\pm}$ are the envelopes of the coverage class, and $\delta$ is the argument of the moduli ($\delta=z_{1-\tau}$, or $2z_{1-\tau}$ in the row marked $\delta=2z$); it is $0$ for AdaptMono, whose coverage class contains all monotone functions. $z_{\rm Bon}=z_{1-\alpha/(2J)}$ is the critical value of AdaptBonf.
  \end{minipage}
\end{table}}

\newcommand{\FullyMonotoneLengthsTable}{%
\begin{table}[t]
  \caption{Average lengths of CIs under monotonicity in both variables ($\mathcal{V}=\{1,2\}$, $J=4$, $\mathcal{F}_{j}=\Lambda_{+,\mathcal{V}}(\gamma_{j},1)$ with $(\gamma_{1},\dots,\gamma_{4})=(1,4/5,2/3,1/2)$)}
  \label{tab:fm-lengths}
  \centering
  \setlength{\tabcolsep}{4pt}
  \begin{tabular}{@{\extracolsep{5pt}} lccccccc} \\[-1.8ex]\hline
    \hline \\[-1.8ex] & $f_{1}$ & $f_{2}$ & $f_{3}$ & $f_{4}$ & $f_{5}$ & $h_{+}$ & $h_{-}$ \\
    \hline \\[-1.8ex]
    \multicolumn{8}{@{}l}{\textit{Uniform design ($n=800$)}} \\
    AdaptBonf & $0.470$ & $0.652$ & $0.684$ & $0.681$ & $0.681$ & $0.712$ & $0.714$ \\
    AdaptCalib & $0.407$ & $0.575$ & $0.614$ & $0.619$ & $0.619$ & $0.644$ & $0.646$ \\
    AdaptCalib ($K=0$) & $0.434$ & $0.608$ & $0.644$ & $0.646$ & $0.646$ & $0.674$ & $0.676$ \\
    AdaptMono & $0.373$ & $0.630$ & $0.707$ & $0.705$ & $0.698$ & $0.790$ & $0.787$ \\
    CL & $2.049$ & $2.049$ & $2.049$ & $2.049$ & $2.049$ & $2.049$ & $2.049$ \\
    Conservative & $0.609$ & $0.609$ & $0.609$ & $0.609$ & $0.609$ & $0.609$ & $0.609$ \\
    Oracle & $0.352$ & $0.352$ & $0.609$ & $0.609$ & $0.609$ & $0.609$ & $0.609$ \\
    Conservative ($J=1$) & $0.681$ & $0.681$ & $0.681$ & $0.681$ & $0.681$ & $0.681$ & $0.681$ \\
    Oracle ($J=1$) & $0.408$ & $0.408$ & $0.681$ & $0.681$ & $0.681$ & $0.681$ & $0.681$ \\
    DHZ & $0.242$\rlap{\smash{$^{\dagger}$}} & $0.432$ & $0.497$\rlap{\smash{$^{\dagger}$}} & $0.496$\rlap{\smash{$^{\dagger}$}} & $0.503$\rlap{\smash{$^{\dagger}$}} & $0.564$\rlap{\smash{$^{\dagger}$}} & $0.572$\rlap{\smash{$^{\dagger}$}} \\
    \hline \\[-1.8ex]
    \multicolumn{8}{@{}l}{\textit{Lattice design ($29\times 29$, $n=841$)}} \\
    AdaptBonf & $0.458$ & $0.650$ & $0.685$ & $0.681$ & $0.681$ & $0.716$ & $0.716$ \\
    AdaptCalib & $0.401$ & $0.575$ & $0.618$ & $0.622$ & $0.622$ & $0.649$ & $0.649$ \\
    AdaptCalib ($K=0$) & $0.424$ & $0.605$ & $0.646$ & $0.646$ & $0.646$ & $0.677$ & $0.677$ \\
    AdaptMono & $0.373$ & $0.603$ & $0.685$ & $0.690$ & $0.690$ & $0.762$ & $0.762$ \\
    CL & $2.032$ & $2.032$ & $2.032$ & $2.032$ & $2.032$ & $2.032$ & $2.032$ \\
    Conservative & $0.608$ & $0.608$ & $0.608$ & $0.608$ & $0.608$ & $0.608$ & $0.608$ \\
    Oracle & $0.356$ & $0.356$ & $0.608$ & $0.608$ & $0.608$ & $0.608$ & $0.608$ \\
    Conservative ($J=1$) & $0.687$ & $0.687$ & $0.687$ & $0.687$ & $0.687$ & $0.687$ & $0.687$ \\
    Oracle ($J=1$) & $0.412$ & $0.412$ & $0.687$ & $0.687$ & $0.687$ & $0.687$ & $0.687$ \\
    DHZ & $0.230$\rlap{\smash{$^{\dagger}$}} & $0.423$ & $0.485$\rlap{\smash{$^{\dagger}$}} & $0.482$\rlap{\smash{$^{\dagger}$}} & $0.485$\rlap{\smash{$^{\dagger}$}} & $0.546$\rlap{\smash{$^{\dagger}$}} & $0.545$\rlap{\smash{$^{\dagger}$}} \\
    \hline \\[-1.8ex]
  \end{tabular}
  \par\smallskip
  \begin{minipage}{0.92\linewidth}\footnotesize
  $^{\dagger}$The critical value 1.80, which \citet{deng2021confidence} suggest for $k=2$, is calibrated for regression functions whose first partial derivatives at $x_{0}$ are nonzero, the case of their Theorem 2; among $f_{1}$ to $f_{5}$ and $h_{\pm}$, this holds only for $f_{2}$. Their Theorem 3 covers the other functions that satisfy their Assumption A, including the constant $f_{1}$, but with a larger critical value that they do not tabulate. Conservative and Oracle are the minimax affine fixed-length CIs of \citet{Donoho1994} for $\Lambda_{+,\mathcal{V}}(1/2,1)$ and for the smallest parameter space $\Lambda_{+,\mathcal{V}}(\gamma_{j},1)$ that contains the regression function. The rows marked $J=1$ use the construction of AdaptBonf with $J=1$. AdaptMono maintains coverage over all functions that are monotone in both variables. The calibration gives $\tau^{\ast}=0.02500$, $\tau^{\ast}_{0}=0.01450$ and $\tau^{\ast}_{\rm Mono}=0.02200$ on the uniform design and $\tau^{\ast}=0.02375$, $\tau^{\ast}_{0}=0.01450$ and $\tau^{\ast}_{\rm Mono}=0.02075$ on the lattice. DHZ is the CI of \citet{deng2021confidence} with critical value $1.80$ and known $\sigma$. On the uniform design we cap $\hat u$ and $\hat v$ at $x_{0}$, so that the block $[\hat u,\hat v]$ contains $x_{0}$; this is our addition, needed for feasibility on designs that are not lattices (\citet{deng2021confidence} take $\hat u$ and $\hat v$ from the block max-min and min-max estimators and note that their results hold for every feasible pair). On the uniform design the DHZ block contained no design point ($N=0$, so that the CI is the real line) in 0.5\% to 6.6\% of the draws; these draws count as covering, and the DHZ lengths average over the other draws. Monte Carlo standard errors of the average lengths are at most $0.0037$. Empty intersections (length zero, counted as not covering) occurred in at most 0.27\% of the draws.
  \end{minipage}
\end{table}}

\newcommand{\FullyMonotoneCoverageTable}{%
\begin{table}[t]
  \caption{Coverage probabilities of CIs under monotonicity in both variables ($\mathcal{V}=\{1,2\}$, $J=4$, $\mathcal{F}_{j}=\Lambda_{+,\mathcal{V}}(\gamma_{j},1)$ with $(\gamma_{1},\dots,\gamma_{4})=(1,4/5,2/3,1/2)$)}
  \label{tab:fm-coverage}
  \centering
  \setlength{\tabcolsep}{4pt}
  \begin{tabular}{@{\extracolsep{5pt}} lccccccc} \\[-1.8ex]\hline
    \hline \\[-1.8ex] & $f_{1}$ & $f_{2}$ & $f_{3}$ & $f_{4}$ & $f_{5}$ & $h_{+}$ & $h_{-}$ \\
    \hline \\[-1.8ex]
    \multicolumn{8}{@{}l}{\textit{Uniform design ($n=800$)}} \\
    AdaptBonf & $0.999$ & $1.000$ & $0.996$ & $0.993$ & $0.994$ & $0.991$ & $0.992$ \\
    AdaptCalib & $0.986$ & $1.000$ & $0.976$ & $0.964$ & $0.968$ & $0.959$ & $0.963$ \\
    AdaptCalib ($K=0$) & $0.995$ & $1.000$ & $0.988$ & $0.982$ & $0.982$ & $0.978$ & $0.978$ \\
    AdaptMono & $0.954$ & $0.999$ & $0.975$ & $0.975$ & $0.978$ & $0.975$ & $0.978$ \\
    CL & $1.000$ & $1.000$ & $1.000$ & $1.000$ & $1.000$ & $1.000$ & $1.000$ \\
    Conservative & $1.000$ & $1.000$ & $0.991$ & $0.986$ & $0.987$ & $0.951$ & $0.952$ \\
    Oracle & $0.999$ & $0.998$ & $0.991$ & $0.986$ & $0.987$ & $0.951$ & $0.952$ \\
    Conservative ($J=1$) & $1.000$ & $1.000$ & $0.993$ & $0.985$ & $0.985$ & $0.977$ & $0.977$ \\
    Oracle ($J=1$) & $0.997$ & $0.997$ & $0.993$ & $0.985$ & $0.985$ & $0.977$ & $0.977$ \\
    DHZ & $0.958$\rlap{\smash{$^{\dagger}$}} & $0.955$ & $0.694$\rlap{\smash{$^{\dagger}$}} & $0.667$\rlap{\smash{$^{\dagger}$}} & $0.708$\rlap{\smash{$^{\dagger}$}} & $0.663$\rlap{\smash{$^{\dagger}$}} & $0.723$\rlap{\smash{$^{\dagger}$}} \\
    \hline \\[-1.8ex]
    \multicolumn{8}{@{}l}{\textit{Lattice design ($29\times 29$, $n=841$)}} \\
    AdaptBonf & $0.998$ & $1.000$ & $0.994$ & $0.992$ & $0.992$ & $0.989$ & $0.991$ \\
    AdaptCalib & $0.981$ & $0.999$ & $0.972$ & $0.962$ & $0.965$ & $0.958$ & $0.959$ \\
    AdaptCalib ($K=0$) & $0.991$ & $1.000$ & $0.985$ & $0.978$ & $0.980$ & $0.975$ & $0.976$ \\
    AdaptMono & $0.952$ & $0.998$ & $0.975$ & $0.975$ & $0.976$ & $0.975$ & $0.976$ \\
    CL & $1.000$ & $1.000$ & $1.000$ & $1.000$ & $1.000$ & $1.000$ & $1.000$ \\
    Conservative & $1.000$ & $1.000$ & $0.988$ & $0.986$ & $0.987$ & $0.949$ & $0.949$ \\
    Oracle & $0.999$ & $0.999$ & $0.988$ & $0.986$ & $0.987$ & $0.949$ & $0.949$ \\
    Conservative ($J=1$) & $1.000$ & $1.000$ & $0.988$ & $0.981$ & $0.983$ & $0.975$ & $0.974$ \\
    Oracle ($J=1$) & $0.997$ & $0.997$ & $0.988$ & $0.981$ & $0.983$ & $0.975$ & $0.974$ \\
    DHZ & $0.955$\rlap{\smash{$^{\dagger}$}} & $0.955$ & $0.702$\rlap{\smash{$^{\dagger}$}} & $0.681$\rlap{\smash{$^{\dagger}$}} & $0.682$\rlap{\smash{$^{\dagger}$}} & $0.677$\rlap{\smash{$^{\dagger}$}} & $0.678$\rlap{\smash{$^{\dagger}$}} \\
    \hline \\[-1.8ex]
  \end{tabular}
  \par\smallskip
  \begin{minipage}{0.92\linewidth}\footnotesize
  $^{\dagger}$The critical value 1.80, which \citet{deng2021confidence} suggest for $k=2$, is calibrated for regression functions whose first partial derivatives at $x_{0}$ are nonzero, the case of their Theorem 2; among $f_{1}$ to $f_{5}$ and $h_{\pm}$, this holds only for $f_{2}$. Their Theorem 3 covers the other functions that satisfy their Assumption A, including the constant $f_{1}$, but with a larger critical value that they do not tabulate. Conservative and Oracle are the minimax affine fixed-length CIs of \citet{Donoho1994} for $\Lambda_{+,\mathcal{V}}(1/2,1)$ and for the smallest parameter space $\Lambda_{+,\mathcal{V}}(\gamma_{j},1)$ that contains the regression function. The rows marked $J=1$ use the construction of AdaptBonf with $J=1$. AdaptMono maintains coverage over all functions that are monotone in both variables. The calibration gives $\tau^{\ast}=0.02500$, $\tau^{\ast}_{0}=0.01450$ and $\tau^{\ast}_{\rm Mono}=0.02200$ on the uniform design and $\tau^{\ast}=0.02375$, $\tau^{\ast}_{0}=0.01450$ and $\tau^{\ast}_{\rm Mono}=0.02075$ on the lattice. DHZ is the CI of \citet{deng2021confidence} with critical value $1.80$ and known $\sigma$. On the uniform design we cap $\hat u$ and $\hat v$ at $x_{0}$, so that the block $[\hat u,\hat v]$ contains $x_{0}$; this is our addition, needed for feasibility on designs that are not lattices (\citet{deng2021confidence} take $\hat u$ and $\hat v$ from the block max-min and min-max estimators and note that their results hold for every feasible pair). On the uniform design the DHZ block contained no design point ($N=0$, so that the CI is the real line) in 0.5\% to 6.6\% of the draws; these draws count as covering, and the DHZ lengths average over the other draws. Monte Carlo standard errors of the coverage probabilities are at most $0.0048$. Empty intersections (length zero, counted as not covering) occurred in at most 0.27\% of the draws.
  \end{minipage}
\end{table}}

\newcommand{\DHZBalancedDetailsTable}{%
\begin{table}[t]
  \caption{Balanced set of regression functions: average lengths and coverage probabilities, with the fixed-length CIs ($\mathcal{V}=\{1,2\}$, $J=4$, $\mathcal{F}_{j}=\Lambda_{+,\mathcal{V}}(\gamma_{j},1)$ with $(\gamma_{1},\dots,\gamma_{4})=(1,4/5,2/3,1/2)$)}
  \label{tab:dhz-balanced-details}
  \renewcommand{\arraystretch}{0.9}
  \centering
  \setlength{\tabcolsep}{4pt}
  \begin{tabular}{@{\extracolsep{5pt}} lccccc} \\[-1.8ex]\hline
    \hline \\[-1.8ex] & AdaptCalib & AdaptMono & DHZ & Conservative & Oracle \\
    \hline \\[-1.8ex]
    \multicolumn{6}{@{}l}{\textit{Lattice design, average length}} \\
    (a) $0.6x_{1}$ & $0.501$ & $0.499$ & $0.343$ & $0.608$ & $0.356$ \\
    (b) $0.6x_{1}+0.05x_{2}$ & $0.508$ & $0.510$ & $0.351$ & $0.608$ & $0.356$ \\
    (c) $1.5(x_{1}^{3}+x_{2}^{3})$ & $0.418$ & $0.395$ & $0.261$ & $0.608$ & $0.356$ \\
    (d) softplus, $s=0.02$ & $0.516$ & $0.520$ & $0.347$ & $0.608$ & $0.356$ \\
    (e) softplus, $s=0.05$ & $0.516$ & $0.521$ & $0.350$ & $0.608$ & $0.356$ \\
    (f) $\lVert(x)_{\mathcal{V}+}\rVert$ & $0.530$ & $0.540$ & $0.368$ & $0.608$ & $0.356$ \\
    (g) $(x_{1})_{+}+(x_{2})_{+}$ & $0.557$ & $0.580$ & $0.385$ & $0.608$ & $0.608$ \\
    (h) $2\lVert(x)_{\mathcal{V}+}\rVert^{1/2}$ & $0.689$ & $1.139$ & $0.809$ & $0.608$ & -- \\
    (i) $0.8\lVert(x)_{\mathcal{V}+}\rVert^{1/4}$ & $0.674$ & $0.863$ & $0.656$ & $0.608$ & -- \\
    \hline \\[-1.8ex]
    \multicolumn{6}{@{}l}{\textit{Lattice design, coverage probability}} \\
    (a) $0.6x_{1}$ & $0.997$ & $0.989$ & $0.944$ & $1.000$ & $0.999$ \\
    (b) $0.6x_{1}+0.05x_{2}$ & $0.997$ & $0.990$ & $0.945$ & $1.000$ & $0.999$ \\
    (c) $1.5(x_{1}^{3}+x_{2}^{3})$ & $0.985$ & $0.962$ & $0.955$ & $1.000$ & $0.999$ \\
    (d) softplus, $s=0.02$ & $0.992$ & $0.983$ & $0.909$ & $1.000$ & $0.987$ \\
    (e) softplus, $s=0.05$ & $0.995$ & $0.987$ & $0.934$ & $1.000$ & $0.994$ \\
    (f) $\lVert(x)_{\mathcal{V}+}\rVert$ & $0.988$ & $0.975$ & $0.847$ & $1.000$ & $0.948$ \\
    (g) $(x_{1})_{+}+(x_{2})_{+}$ & $0.987$ & $0.975$ & $0.855$ & $0.999$ & $0.999$ \\
    (h) $2\lVert(x)_{\mathcal{V}+}\rVert^{1/2}$ & $0.559$ & $0.975$ & $0.647$ & $0.070$ & -- \\
    (i) $0.8\lVert(x)_{\mathcal{V}+}\rVert^{1/4}$ & $0.877$ & $0.975$ & $0.579$ & $0.759$ & -- \\
    \hline \\[-1.8ex]
    \multicolumn{6}{@{}l}{\textit{Uniform design, average length}} \\
    (a) $0.6x_{1}$ & $0.504$ & $0.513$ & $0.357$ & $0.609$ & $0.352$ \\
    (b) $0.6x_{1}+0.05x_{2}$ & $0.512$ & $0.525$ & $0.365$ & $0.609$ & $0.352$ \\
    (c) $1.5(x_{1}^{3}+x_{2}^{3})$ & $0.423$ & $0.398$ & $0.270$ & $0.609$ & $0.352$ \\
    (d) softplus, $s=0.02$ & $0.518$ & $0.538$ & $0.361$ & $0.609$ & $0.352$ \\
    (e) softplus, $s=0.05$ & $0.518$ & $0.538$ & $0.363$ & $0.609$ & $0.352$ \\
    (f) $\lVert(x)_{\mathcal{V}+}\rVert$ & $0.531$ & $0.557$ & $0.377$ & $0.609$ & $0.352$ \\
    (g) $(x_{1})_{+}+(x_{2})_{+}$ & $0.557$ & $0.606$ & $0.397$ & $0.609$ & $0.609$ \\
    (h) $2\lVert(x)_{\mathcal{V}+}\rVert^{1/2}$ & $0.681$ & $1.203$ & $0.861$ & $0.609$ & -- \\
    (i) $0.8\lVert(x)_{\mathcal{V}+}\rVert^{1/4}$ & $0.666$ & $0.886$ & $0.656$ & $0.609$ & -- \\
    \hline \\[-1.8ex]
    \multicolumn{6}{@{}l}{\textit{Uniform design, coverage probability}} \\
    (a) $0.6x_{1}$ & $0.998$ & $0.994$ & $0.948$ & $1.000$ & $0.998$ \\
    (b) $0.6x_{1}+0.05x_{2}$ & $0.999$ & $0.995$ & $0.949$ & $1.000$ & $0.998$ \\
    (c) $1.5(x_{1}^{3}+x_{2}^{3})$ & $0.990$ & $0.965$ & $0.958$ & $1.000$ & $0.999$ \\
    (d) softplus, $s=0.02$ & $0.996$ & $0.984$ & $0.920$ & $1.000$ & $0.987$ \\
    (e) softplus, $s=0.05$ & $0.997$ & $0.990$ & $0.940$ & $1.000$ & $0.994$ \\
    (f) $\lVert(x)_{\mathcal{V}+}\rVert$ & $0.992$ & $0.975$ & $0.860$ & $1.000$ & $0.951$ \\
    (g) $(x_{1})_{+}+(x_{2})_{+}$ & $0.991$ & $0.975$ & $0.866$ & $0.999$ & $0.999$ \\
    (h) $2\lVert(x)_{\mathcal{V}+}\rVert^{1/2}$ & $0.467$ & $0.975$ & $0.623$ & $0.053$ & -- \\
    (i) $0.8\lVert(x)_{\mathcal{V}+}\rVert^{1/4}$ & $0.830$ & $0.975$ & $0.498$ & $0.707$ & -- \\
    \hline \\[-1.8ex]
  \end{tabular}
  \par\smallskip
  \begin{minipage}{0.92\linewidth}\footnotesize
  Conservative and Oracle are the minimax affine fixed-length CIs of \citet{Donoho1994} for $\Lambda_{+,\mathcal{V}}(1/2,1)$ and for the smallest parameter space $\Lambda_{+,\mathcal{V}}(\gamma_{j},1)$ that contains the regression function. For (h) and (i), which lie in none of these parameter spaces, there is no oracle CI. The designs are the $29\times29$ lattice ($n=841$) and the uniform design ($n=800$). DHZ uses the critical value $1.80$ and known $\sigma$. On the uniform design we cap $\hat u$ and $\hat v$ at $x_{0}$, so that the block $[\hat u,\hat v]$ contains $x_{0}$; this is our addition, needed for feasibility on designs that are not lattices (\citet{deng2021confidence} take $\hat u$ and $\hat v$ from the block max-min and min-max estimators and note that their results hold for every feasible pair). On the uniform design the DHZ block contained no design point ($N=0$, so that the CI is the real line) in 0.8\% to 2.4\% of the draws at (a) to (g) and in 18.1\% and 11.1\% at (h) and (i). These draws count as covering, and the DHZ lengths average over the other draws. Given $N>0$, DHZ covers with probabilities $0.539$ and $0.436$ at (h) and (i), and within $0.004$ of the table at (a) to (g). Monte Carlo standard errors are at most $0.0060$ for the average lengths and $0.0050$ for the coverage probabilities. Empty intersections (length zero, counted as not covering) occurred in at most 0.18\% of the draws.
  \end{minipage}
\end{table}}

\newcommand{\AdaptCTable}{%
\begin{table}[t]
  \caption{Average lengths of CIs that adapt to the H\"older constant ($\gamma=1$, $\mathcal{F}_{j}=\Lambda_{+,\mathcal{V}}(1,C_{j})$ with $C_{j}$ on a geometric grid on $[1/4,4]$, $J=5$ or $17$)}
  \label{tab:adapt-c}
  \centering
  \setlength{\tabcolsep}{4pt}
  \begin{tabular}{@{\extracolsep{5pt}} lccccccc} \\[-1.8ex]\hline
    \hline \\[-1.8ex] & $z$ & $0$ & $0.25x_{1}$ & $x_{1}$ & $4x_{1}$ & $h_{+}$ & $h_{-}$ \\
    \hline \\[-1.8ex]
    \multicolumn{8}{@{}l}{\textit{$\mathcal{V}=\{1\}$, uniform design ($n=800$)}} \\
    AdaptBonf, $J=5$ & $2.576$ & $0.758$ & $0.811$ & $0.939$ & $1.074$ & $0.933$ & $0.957$ \\
    AdaptCalib, $J=5$ & $2.049$ & $0.639$ & $0.690$ & $0.811$ & $0.949$ & $0.807$ & $0.831$ \\
    AdaptBonf, $J=17$ & $2.974$ & $0.838$ & $0.893$ & $1.022$ & $1.157$ & $1.016$ & $1.038$ \\
    AdaptCalib, $J=17$ & $2.065$ & $0.641$ & $0.692$ & $0.811$ & $0.949$ & $0.807$ & $0.831$ \\
    Conservative ($C=4$) & -- & $0.801$ & $0.801$ & $0.801$ & $0.801$ & $0.801$ & $0.801$ \\
    Oracle & -- & $0.218$ & $0.218$ & $0.413$ & $0.801$ & $0.801$ & $0.801$ \\
    \hline \\[-1.8ex]
    \multicolumn{8}{@{}l}{\textit{$\mathcal{V}=\{1,2\}$, lattice design ($n=841$)}} \\
    AdaptBonf, $J=5$ & $2.576$ & $0.340$ & $0.413$ & $0.608$ & $0.921$ & $0.777$ & $0.777$ \\
    AdaptCalib, $J=5$ & $2.326$ & $0.307$ & $0.379$ & $0.569$ & $0.871$ & $0.727$ & $0.728$ \\
    AdaptMono, $J=5$ & $2.317$ & $0.300$ & $0.374$ & $0.579$ & $1.134$ & $0.876$ & $0.876$ \\
    AdaptBonf, $J=17$ & $2.974$ & $0.391$ & $0.464$ & $0.662$ & $0.989$ & $0.845$ & $0.845$ \\
    AdaptCalib, $J=17$ & $2.357$ & $0.310$ & $0.381$ & $0.569$ & $0.869$ & $0.728$ & $0.728$ \\
    AdaptMono, $J=17$ & $2.337$ & $0.301$ & $0.375$ & $0.578$ & $1.129$ & $0.876$ & $0.876$ \\
    Conservative ($C=4$) & -- & $0.709$ & $0.709$ & $0.709$ & $0.709$ & $0.709$ & $0.709$ \\
    Oracle & -- & $0.189$ & $0.189$ & $0.356$ & $0.709$ & $0.709$ & $0.709$ \\
    \hline \\[-1.8ex]
  \end{tabular}
  \par\smallskip
  \begin{minipage}{0.92\linewidth}\footnotesize
  The coverage class is $\Lambda_{+,\mathcal{V}}(1,4)$, and the grids are $C_{j}=2^{j-3}$ ($J=5$) and $C_{j}=2^{(j-9)/4}$ ($J=17$). $c\,x_{1}$ belongs to $\Lambda_{+,\mathcal{V}}(1,c)$, and $h_{\pm}$ are the envelopes of the coverage class. Conservative and Oracle are the minimax affine fixed-length CIs of \citet{Donoho1994} for $\Lambda_{+,\mathcal{V}}(1,4)$ and for the smallest $\Lambda_{+,\mathcal{V}}(1,C_{j})$ that contains the regression function; AdaptMono maintains coverage over all functions that are monotone in both variables. The column $z$ gives the critical value: $z_{1-\alpha/(2J)}$ for AdaptBonf and $z_{1-\tau^{\ast}}$ for AdaptCalib and AdaptMono. The values of $(\tau^{\ast},\tau^{\ast}_{0})$ for AdaptCalib are $(0.02025,0.01225)$ for $J=5$ on the uniform design, $(0.01947,0.01172)$ for $J=17$ on the uniform design, $(0.01000,0.00925)$ for $J=5$ on the lattice and $(0.00922,0.00872)$ for $J=17$ on the lattice. Monte Carlo standard errors of the average lengths are at most $0.0019$. Empty intersections (length zero, counted as not covering) occurred in at most 0.33\% of the draws.
  \end{minipage}
\end{table}}

\newcommand{\AdaptCCoverageTable}{%
\begin{table}[t]
  \caption{Coverage probabilities of CIs that adapt to the H\"older constant ($\gamma=1$, $\mathcal{F}_{j}=\Lambda_{+,\mathcal{V}}(1,C_{j})$ with $C_{j}$ on a geometric grid on $[1/4,4]$, $J=5$ or $17$)}
  \label{tab:adapt-c-coverage}
  \centering
  \setlength{\tabcolsep}{4pt}
  \begin{tabular}{@{\extracolsep{5pt}} lccccccc} \\[-1.8ex]\hline
    \hline \\[-1.8ex] & $z$ & $0$ & $0.25x_{1}$ & $x_{1}$ & $4x_{1}$ & $h_{+}$ & $h_{-}$ \\
    \hline \\[-1.8ex]
    \multicolumn{8}{@{}l}{\textit{$\mathcal{V}=\{1\}$, uniform design ($n=800$)}} \\
    AdaptBonf, $J=5$ & $2.576$ & $0.999$ & $1.000$ & $1.000$ & $1.000$ & $0.990$ & $0.990$ \\
    AdaptCalib, $J=5$ & $2.049$ & $0.992$ & $0.993$ & $0.995$ & $0.996$ & $0.964$ & $0.961$ \\
    AdaptBonf, $J=17$ & $2.974$ & $1.000$ & $1.000$ & $1.000$ & $1.000$ & $0.997$ & $0.997$ \\
    AdaptCalib, $J=17$ & $2.065$ & $0.992$ & $0.993$ & $0.995$ & $0.996$ & $0.964$ & $0.961$ \\
    Conservative ($C=4$) & -- & $0.998$ & $0.998$ & $0.998$ & $0.998$ & $0.954$ & $0.952$ \\
    Oracle & -- & $0.995$ & $0.995$ & $0.999$ & $0.998$ & $0.954$ & $0.952$ \\
    \hline \\[-1.8ex]
    \multicolumn{8}{@{}l}{\textit{$\mathcal{V}=\{1,2\}$, lattice design ($n=841$)}} \\
    AdaptBonf, $J=5$ & $2.576$ & $0.991$ & $0.998$ & $1.000$ & $1.000$ & $0.985$ & $0.987$ \\
    AdaptCalib, $J=5$ & $2.326$ & $0.978$ & $0.993$ & $0.999$ & $1.000$ & $0.972$ & $0.975$ \\
    AdaptMono, $J=5$ & $2.317$ & $0.950$ & $0.974$ & $0.992$ & $0.999$ & $0.974$ & $0.976$ \\
    AdaptBonf, $J=17$ & $2.974$ & $0.997$ & $1.000$ & $1.000$ & $1.000$ & $0.995$ & $0.995$ \\
    AdaptCalib, $J=17$ & $2.357$ & $0.979$ & $0.994$ & $0.999$ & $1.000$ & $0.972$ & $0.975$ \\
    AdaptMono, $J=17$ & $2.337$ & $0.950$ & $0.974$ & $0.992$ & $0.999$ & $0.974$ & $0.975$ \\
    Conservative ($C=4$) & -- & $0.999$ & $0.999$ & $0.999$ & $0.999$ & $0.950$ & $0.953$ \\
    Oracle & -- & $0.990$ & $0.990$ & $0.999$ & $0.999$ & $0.950$ & $0.953$ \\
    \hline \\[-1.8ex]
  \end{tabular}
  \par\smallskip
  \begin{minipage}{0.92\linewidth}\footnotesize
  The coverage class is $\Lambda_{+,\mathcal{V}}(1,4)$, and the grids are $C_{j}=2^{j-3}$ ($J=5$) and $C_{j}=2^{(j-9)/4}$ ($J=17$). $c\,x_{1}$ belongs to $\Lambda_{+,\mathcal{V}}(1,c)$, and $h_{\pm}$ are the envelopes of the coverage class. Conservative and Oracle are the minimax affine fixed-length CIs of \citet{Donoho1994} for $\Lambda_{+,\mathcal{V}}(1,4)$ and for the smallest $\Lambda_{+,\mathcal{V}}(1,C_{j})$ that contains the regression function; AdaptMono maintains coverage over all functions that are monotone in both variables. The column $z$ gives the critical value: $z_{1-\alpha/(2J)}$ for AdaptBonf and $z_{1-\tau^{\ast}}$ for AdaptCalib and AdaptMono. The values of $(\tau^{\ast},\tau^{\ast}_{0})$ for AdaptCalib are $(0.02025,0.01225)$ for $J=5$ on the uniform design, $(0.01947,0.01172)$ for $J=17$ on the uniform design, $(0.01000,0.00925)$ for $J=5$ on the lattice and $(0.00922,0.00872)$ for $J=17$ on the lattice. Monte Carlo standard errors of the coverage probabilities are at most $0.0022$. Empty intersections (length zero, counted as not covering) occurred in at most 0.33\% of the draws.
  \end{minipage}
\end{table}}

\newcommand{\ThinningTable}{%
\begin{table}[t]
  \caption{Between class moduli $\xi_{j}$ used by the thinning step of \citet{cai2004adaptation} ($\mathcal{V}=\{1\}$, uniform design, $J=4$, $\mathcal{F}_{j}=\Lambda_{+,\mathcal{V}}(\gamma_{j},1)$ with $(\gamma_{1},\dots,\gamma_{4})=(1,4/5,2/3,1/2)$, $\delta=z_{0.975}$)}
  \label{tab:thinning}
  \centering
  \setlength{\tabcolsep}{4pt}
  \begin{tabular}{@{\extracolsep{5pt}} lcccccc} \\[-1.8ex]\hline
    \hline \\[-1.8ex] & $\xi_{1}$ & $\xi_{2}$ & $\xi_{3}$ & $\xi_{4}$ & $\xi_{4}/\xi_{1}$ & Retained \\
    \hline \\[-1.8ex]
    $n=800$ & $0.4768$ & $0.5215$ & $0.5638$ & $0.6415$ & $1.345$ & $\{4\}$ \\
    $n=5{,}000$ & $0.3414$ & $0.3781$ & $0.4141$ & $0.4826$ & $1.414$ & $\{4\}$ \\
    $n=10{,}000$ & $0.2948$ & $0.3292$ & $0.3633$ & $0.4293$ & $1.456$ & $\{4\}$ \\
    $n=20{,}000$ & $0.2597$ & $0.2921$ & $0.3248$ & $0.3885$ & $1.496$ & $\{4\}$ \\
    $n=10^{5}$ & $0.1940$ & $0.2208$ & $0.2480$ & $0.3019$ & $1.557$ & $\{4\}$ \\
    $n=10^{6}$ & $0.1221$ & $0.1418$ & $0.1624$ & $0.2057$ & $1.684$ & $\{4\}$ \\
    $n=10^{7}$ & $0.0766$ & $0.0904$ & $0.1051$ & $0.1373$ & $1.793$ & $\{4\}$ \\
    \hline \\[-1.8ex]
  \end{tabular}
  \par\smallskip
  \begin{minipage}{0.92\linewidth}\footnotesize
  $\xi_{j}$ is the larger of $\omega(z_{0.975},\mathcal{F}_{j},\mathcal{F}_{4})$ and $\omega(z_{0.975},\mathcal{F}_{4},\mathcal{F}_{j})$, computed for one uniform design on $[-1/(2\sqrt{2}),1/(2\sqrt{2})]^{2}$ for each $n$ (the designs are nested). Starting from $\mathcal{F}_{4}$, the thinning step keeps an earlier parameter space only if its modulus is at most half the modulus of the last parameter space kept. Since $\xi_{4}/\xi_{1}<2$ for every $n$ here, only $\mathcal{F}_{4}$ is kept, and the interval of \citet{cai2004adaptation} does not adapt.
  \end{minipage}
\end{table}}

\newcommand{\DHZValidationTable}{%
\begin{table}[t]
  \caption{Quantiles of $T=\sqrt{N}\,\lvert\hat f(x_{0})-f_{0}(x_{0})\rvert$ for the CI of \citet{deng2021confidence} ($f_{0}(x)=2x_{1}+2x_{2}-2$, $x_{0}=(0.5,0.5)$, $10^{6}$ draws)}
  \label{tab:dhz-validation}
  \centering
  \setlength{\tabcolsep}{4pt}
  \begin{tabular}{@{\extracolsep{5pt}} lccccccc} \\[-1.8ex]\hline
    \hline \\[-1.8ex] Quantile level & $0.99$ & $0.98$ & $0.95$ & $0.90$ & $0.85$ & $0.80$ & $0.50$ \\
    \hline \\[-1.8ex]
    DHZ, Table 2 & $2.61$ & $2.26$ & $1.78$ & $1.41$ & $1.19$ & $1.03$ & $0.49$ \\
    \hline \\[-1.8ex]
    \multicolumn{8}{@{}l}{\textit{Lattice $\{i/50:i=1,\dots,50\}^{2}$, $n=2{,}500$}} \\
    Simulated & $2.607$ & $2.253$ & $1.781$ & $1.413$ & $1.191$ & $1.032$ & $0.493$ \\
    (s.e.) & $(0.005)$ & $(0.003)$ & $(0.002)$ & $(0.002)$ & $(0.001)$ & $(0.001)$ & $(0.001)$ \\
    \hline \\[-1.8ex]
    \multicolumn{8}{@{}l}{\textit{Lattice $\{i/50:i=0,\dots,49\}^{2}$, $n=2{,}500$}} \\
    Simulated & $2.607$ & $2.252$ & $1.777$ & $1.409$ & $1.190$ & $1.031$ & $0.493$ \\
    (s.e.) & $(0.005)$ & $(0.003)$ & $(0.002)$ & $(0.001)$ & $(0.001)$ & $(0.001)$ & $(0.001)$ \\
    \hline \\[-1.8ex]
    \multicolumn{8}{@{}l}{\textit{Lattice $\{i/50:i=0,\dots,50\}^{2}$, $n=2{,}601$}} \\
    Simulated & $2.602$ & $2.257$ & $1.782$ & $1.412$ & $1.190$ & $1.030$ & $0.492$ \\
    (s.e.) & $(0.004)$ & $(0.003)$ & $(0.002)$ & $(0.002)$ & $(0.001)$ & $(0.001)$ & $(0.001)$ \\
    \hline \\[-1.8ex]
  \end{tabular}
  \par\smallskip
  \begin{minipage}{0.92\linewidth}\footnotesize
  Errors are $N(0,1)$ with $\sigma=1$ known. Standard errors are from 100 batches. With the critical value $1.80$, the simulated coverage is $0.952$, $0.952$, $0.952$ on the three lattices.
  \end{minipage}
\end{table}}

%% file: tables/phase2_tables.tex
\newcommand{\KThreeLengthsTable}{%
\begin{table}[t]
  \caption{Average lengths of CIs with $k=3$ ($n=800$, uniform design on $[-a,a]^{3}$, $a=1/(2\sqrt{3})$, $J=4$, $\mathcal{F}_{j}=\Lambda_{+,\mathcal{V}}(\gamma_{j},1)$ with $(\gamma_{1},\dots,\gamma_{4})=(1,4/5,2/3,1/2)$)}
  \label{tab:k3-lengths}
  \centering
  \setlength{\tabcolsep}{3pt}
  \begin{tabular}{@{\extracolsep{5pt}} lcccccccc} \\[-1.8ex]\hline
    \hline \\[-1.8ex] & $f_{1}$ & $f_{2}$ & $f_{3}$ & $f_{4}$ & $f_{5}$ & $h_{+}$ & $h_{-}$ & $f_{\mathcal{G}}$ \\
    \hline \\[-1.8ex]
    \multicolumn{9}{@{}l}{\textit{$\mathcal{V}=\{1\}$}} \\
    AdaptBonf & $1.075$ & $1.087$ & $1.082$ & $1.083$ & $1.081$ & $1.082$ & $1.080$ & $1.037$ \\
    AdaptCalib & $0.968$ & $0.985$ & $0.979$ & $0.981$ & $0.978$ & $0.981$ & $0.978$ & $0.916$ \\
    AdaptMono & $1.043$ & $1.170$ & $1.124$ & $1.141$ & $1.111$ & $1.144$ & $1.110$ & $0.418$ \\
    Conservative & $0.955$ & $0.955$ & $0.955$ & $0.955$ & $0.955$ & $0.955$ & $0.955$ & $0.955$ \\
    Oracle & $0.505$ & $0.505$ & $0.955$ & $0.955$ & $0.955$ & $0.955$ & $0.955$ & -- \\
    \hline \\[-1.8ex]
    \multicolumn{9}{@{}l}{\textit{$\mathcal{V}=\{1,2\}$}} \\
    AdaptBonf & $0.949$ & $0.994$ & $0.987$ & $0.992$ & $0.987$ & $0.993$ & $0.988$ & $0.895$ \\
    AdaptCalib & $0.839$ & $0.898$ & $0.891$ & $0.899$ & $0.894$ & $0.900$ & $0.894$ & $0.773$ \\
    AdaptMono & $0.906$ & $1.185$ & $1.111$ & $1.146$ & $1.118$ & $1.163$ & $1.120$ & $0.476$ \\
    Conservative & $0.873$ & $0.873$ & $0.873$ & $0.873$ & $0.873$ & $0.873$ & $0.873$ & $0.873$ \\
    Oracle & $0.456$ & $0.456$ & $0.873$ & $0.873$ & $0.873$ & $0.873$ & $0.873$ & -- \\
    \hline \\[-1.8ex]
    \multicolumn{9}{@{}l}{\textit{$\mathcal{V}=\{1,2,3\}$}} \\
    AdaptBonf & $0.571$ & $0.786$ & $0.791$ & $0.795$ & $0.794$ & $0.817$ & $0.813$ & -- \\
    AdaptCalib & $0.471$ & $0.703$ & $0.715$ & $0.724$ & $0.722$ & $0.746$ & $0.741$ & -- \\
    AdaptMono & $0.437$ & $0.811$ & $0.835$ & $0.853$ & $0.850$ & $0.934$ & $0.912$ & -- \\
    Conservative & $0.710$ & $0.710$ & $0.710$ & $0.710$ & $0.710$ & $0.710$ & $0.710$ & -- \\
    Oracle & $0.401$ & $0.401$ & $0.710$ & $0.710$ & $0.710$ & $0.710$ & $0.710$ & -- \\
    \hline \\[-1.8ex]
  \end{tabular}
  \par\smallskip
  \begin{minipage}{0.92\linewidth}\footnotesize
  $f_{1}=0$, $f_{2}=0.6x_{1}+0.5x_{2}+0.4x_{3}$, $f_{3}=0.8\lVert(x)_{\mathcal{V}+}\rVert^{1/2}$, $f_{4}$ and $f_{5}$ are the least favorable functions for the ends of AdaptCalib that direct power to $\Lambda_{+,\mathcal{V}}(1,1)$, $h_{+}(x)=\lVert(x)_{\mathcal{V}+}\rVert^{1/2}$ and $h_{-}(x)=-\lVert(x)_{\mathcal{V}-}\rVert^{1/2}$. At the design points, $f_{\mathcal{G}}$ equals $\lVert s(x)\rVert^{1/2}$ if $x_{\mathcal{V}}\leq0$, $-\lVert s(x)\rVert^{1/2}$ if $x_{\mathcal{V}}\geq0$ and $0$ otherwise, where $s(x)$ sets the coordinates in $\mathcal{V}$ to $0$. The function $f_{\mathcal{G}}$ lies in $\mathcal{G}_{\mathcal{V}}(1/2,1)$ but in no H\"older class, and it puts every end of AdaptMono at its maximum bias. $f_{\mathcal{G}}$ is defined only if $\mathcal{V}\neq\{1,2,3\}$. AdaptMono directs power to the same four parameter spaces and maintains coverage over $\mathcal{G}_{\mathcal{V}}(1/2,1)$, the functions that are nondecreasing in $x_{\mathcal{V}}$ and satisfy $\lvert f(x)-f(0)\rvert\leq\lVert x\rVert^{1/2}$ when $x_{v}=0$ for all $v\in\mathcal{V}$, and AdaptBonf, AdaptCalib and Conservative over $\Lambda_{+,\mathcal{V}}(1/2,1)$. If $\mathcal{V}=\{1,2,3\}$, $\mathcal{G}_{\mathcal{V}}(1/2,1)$ contains all monotone functions. Conservative and Oracle are the minimax affine fixed-length CIs of \citet{Donoho1994} for $\Lambda_{+,\mathcal{V}}(1/2,1)$ and for the smallest parameter space $\Lambda_{+,\mathcal{V}}(\gamma_{j},1)$ that contains the regression function. AdaptBonf uses $\tau=0.05/8$, and $(\tau^{\ast},\tau^{\ast}_{\rm Mono})$ of AdaptCalib and AdaptMono are $(0.04150,0.02300)$ for $\mathcal{V}=\{1\}$, $(0.03725,0.02375)$ for $\mathcal{V}=\{1,2\}$ and $(0.02675,0.02375)$ for $\mathcal{V}=\{1,2,3\}$. Monte Carlo standard errors of the average lengths are at most $0.0018$. Empty intersections (length zero, counted as not covering) occurred in at most 0.34\% of the draws, only for AdaptMono at $f_{\mathcal{G}}$ and $f_{1}$.
  \end{minipage}
\end{table}}

\newcommand{\KThreeCoverageTable}{%
\begin{table}[t]
  \caption{Coverage probabilities of CIs with $k=3$ ($n=800$, uniform design on $[-a,a]^{3}$, $a=1/(2\sqrt{3})$, $J=4$, $\mathcal{F}_{j}=\Lambda_{+,\mathcal{V}}(\gamma_{j},1)$ with $(\gamma_{1},\dots,\gamma_{4})=(1,4/5,2/3,1/2)$)}
  \label{tab:k3-coverage}
  \centering
  \setlength{\tabcolsep}{3pt}
  \begin{tabular}{@{\extracolsep{5pt}} lcccccccc} \\[-1.8ex]\hline
    \hline \\[-1.8ex] & $f_{1}$ & $f_{2}$ & $f_{3}$ & $f_{4}$ & $f_{5}$ & $h_{+}$ & $h_{-}$ & $f_{\mathcal{G}}$ \\
    \hline \\[-1.8ex]
    \multicolumn{9}{@{}l}{\textit{$\mathcal{V}=\{1\}$}} \\
    AdaptBonf & $1.000$ & $1.000$ & $1.000$ & $0.994$ & $0.993$ & $0.994$ & $0.992$ & $1.000$ \\
    AdaptCalib & $1.000$ & $1.000$ & $0.988$ & $0.955$ & $0.950$ & $0.955$ & $0.950$ & $1.000$ \\
    AdaptMono & $1.000$ & $1.000$ & $0.993$ & $0.977$ & $0.978$ & $0.977$ & $0.977$ & $0.954$ \\
    Conservative & $1.000$ & $1.000$ & $0.997$ & $0.959$ & $0.956$ & $0.953$ & $0.949$ & $1.000$ \\
    Oracle & $1.000$ & $1.000$ & $0.997$ & $0.959$ & $0.956$ & $0.953$ & $0.949$ & -- \\
    \hline \\[-1.8ex]
    \multicolumn{9}{@{}l}{\textit{$\mathcal{V}=\{1,2\}$}} \\
    AdaptBonf & $1.000$ & $1.000$ & $0.999$ & $0.994$ & $0.992$ & $0.993$ & $0.991$ & $1.000$ \\
    AdaptCalib & $1.000$ & $1.000$ & $0.987$ & $0.956$ & $0.953$ & $0.956$ & $0.953$ & $1.000$ \\
    AdaptMono & $1.000$ & $1.000$ & $0.990$ & $0.979$ & $0.979$ & $0.979$ & $0.976$ & $0.955$ \\
    Conservative & $1.000$ & $1.000$ & $0.995$ & $0.961$ & $0.955$ & $0.953$ & $0.950$ & $1.000$ \\
    Oracle & $1.000$ & $1.000$ & $0.995$ & $0.961$ & $0.955$ & $0.953$ & $0.950$ & -- \\
    \hline \\[-1.8ex]
    \multicolumn{9}{@{}l}{\textit{$\mathcal{V}=\{1,2,3\}$}} \\
    AdaptBonf & $1.000$ & $1.000$ & $0.996$ & $0.992$ & $0.991$ & $0.989$ & $0.988$ & -- \\
    AdaptCalib & $0.990$ & $1.000$ & $0.976$ & $0.965$ & $0.962$ & $0.957$ & $0.955$ & -- \\
    AdaptMono & $0.950$ & $1.000$ & $0.975$ & $0.975$ & $0.976$ & $0.975$ & $0.975$ & -- \\
    Conservative & $1.000$ & $1.000$ & $0.993$ & $0.984$ & $0.980$ & $0.950$ & $0.950$ & -- \\
    Oracle & $1.000$ & $1.000$ & $0.993$ & $0.984$ & $0.980$ & $0.950$ & $0.950$ & -- \\
    \hline \\[-1.8ex]
  \end{tabular}
  \par\smallskip
  \begin{minipage}{0.92\linewidth}\footnotesize
  $f_{1}=0$, $f_{2}=0.6x_{1}+0.5x_{2}+0.4x_{3}$, $f_{3}=0.8\lVert(x)_{\mathcal{V}+}\rVert^{1/2}$, $f_{4}$ and $f_{5}$ are the least favorable functions for the ends of AdaptCalib that direct power to $\Lambda_{+,\mathcal{V}}(1,1)$, $h_{+}(x)=\lVert(x)_{\mathcal{V}+}\rVert^{1/2}$ and $h_{-}(x)=-\lVert(x)_{\mathcal{V}-}\rVert^{1/2}$. At the design points, $f_{\mathcal{G}}$ equals $\lVert s(x)\rVert^{1/2}$ if $x_{\mathcal{V}}\leq0$, $-\lVert s(x)\rVert^{1/2}$ if $x_{\mathcal{V}}\geq0$ and $0$ otherwise, where $s(x)$ sets the coordinates in $\mathcal{V}$ to $0$. The function $f_{\mathcal{G}}$ lies in $\mathcal{G}_{\mathcal{V}}(1/2,1)$ but in no H\"older class, and it puts every end of AdaptMono at its maximum bias. $f_{\mathcal{G}}$ is defined only if $\mathcal{V}\neq\{1,2,3\}$. AdaptMono directs power to the same four parameter spaces and maintains coverage over $\mathcal{G}_{\mathcal{V}}(1/2,1)$, the functions that are nondecreasing in $x_{\mathcal{V}}$ and satisfy $\lvert f(x)-f(0)\rvert\leq\lVert x\rVert^{1/2}$ when $x_{v}=0$ for all $v\in\mathcal{V}$, and AdaptBonf, AdaptCalib and Conservative over $\Lambda_{+,\mathcal{V}}(1/2,1)$. If $\mathcal{V}=\{1,2,3\}$, $\mathcal{G}_{\mathcal{V}}(1/2,1)$ contains all monotone functions. Conservative and Oracle are the minimax affine fixed-length CIs of \citet{Donoho1994} for $\Lambda_{+,\mathcal{V}}(1/2,1)$ and for the smallest parameter space $\Lambda_{+,\mathcal{V}}(\gamma_{j},1)$ that contains the regression function. AdaptBonf uses $\tau=0.05/8$, and $(\tau^{\ast},\tau^{\ast}_{\rm Mono})$ of AdaptCalib and AdaptMono are $(0.04150,0.02300)$ for $\mathcal{V}=\{1\}$, $(0.03725,0.02375)$ for $\mathcal{V}=\{1,2\}$ and $(0.02675,0.02375)$ for $\mathcal{V}=\{1,2,3\}$. Monte Carlo standard errors of the coverage probabilities are at most $0.0022$. Empty intersections (length zero, counted as not covering) occurred in at most 0.34\% of the draws, only for AdaptMono at $f_{\mathcal{G}}$ and $f_{1}$.
  \end{minipage}
\end{table}}

\newcommand{\KThreeLargerNTable}{%
\begin{table}[t]
  \caption{Average lengths of adaptive CIs relative to the fixed-length CIs as $n$ grows ($k=3$, uniform design, $J=4$, $\mathcal{F}_{j}=\Lambda_{+,\mathcal{V}}(\gamma_{j},1)$ with $(\gamma_{1},\dots,\gamma_{4})=(1,4/5,2/3,1/2)$)}
  \label{tab:k3-larger-n}
  \centering
  \setlength{\tabcolsep}{4pt}
  \begin{tabular}{@{\extracolsep{5pt}} lccccc} \\[-1.8ex]\hline
    \hline \\[-1.8ex] & \multicolumn{4}{c}{Relative to Conservative} & Relative to Oracle \\
    \cline{2-5} \cline{6-6} \\[-1.8ex]
    & \multicolumn{2}{c}{$f_{1}$} & \multicolumn{2}{c}{$h_{+}$} & $f_{1}$ \\
    & AdaptCalib & AdaptMono & AdaptCalib & AdaptMono & AdaptCalib \\
    \hline \\[-1.8ex]
    \multicolumn{6}{@{}l}{\textit{$\mathcal{V}=\{1\}$}} \\
    $n=800$ & $1.014$ & $1.091$ & $1.027$ & $1.197$ & $1.918$ \\
    $n=5{,}000$ & $0.993$ & $1.064$ & $1.021$ & $1.171$ & $2.137$ \\
    $n=20{,}000$ & $0.975$ & $1.039$ & $1.015$ & $1.162$ & $2.335$ \\
    \hline \\[-1.8ex]
    \multicolumn{6}{@{}l}{\textit{$\mathcal{V}=\{1,2\}$}} \\
    $n=800$ & $0.962$ & $1.038$ & $1.031$ & $1.332$ & $1.841$ \\
    $n=5{,}000$ & $0.904$ & $0.923$ & $1.023$ & $1.236$ & $1.964$ \\
    $n=20{,}000$ & $0.859$ & $0.865$ & $1.018$ & $1.241$ & $2.074$ \\
    \hline \\[-1.8ex]
    \multicolumn{6}{@{}l}{\textit{$\mathcal{V}=\{1,2,3\}$}} \\
    $n=800$ & $0.664$ & $0.616$ & $1.052$ & $1.317$ & $1.175$ \\
    $n=5{,}000$ & $0.536$ & $0.468$ & $1.061$ & $1.234$ & $1.070$ \\
    $n=20{,}000$ & $0.476$ & $0.424$ & $1.059$ & $1.209$ & $1.058$ \\
    \hline \\[-1.8ex]
  \end{tabular}
  \par\smallskip
  \begin{minipage}{0.92\linewidth}\footnotesize
  Entries are ratios of average lengths. Conservative and Oracle are the minimax affine fixed-length CIs of \citet{Donoho1994} for $\Lambda_{+,\mathcal{V}}(1/2,1)$ and for the smallest parameter space $\Lambda_{+,\mathcal{V}}(\gamma_{j},1)$ that contains the regression function. AdaptMono maintains coverage over $\mathcal{G}_{\mathcal{V}}(1/2,1)$, and AdaptCalib and Conservative over $\Lambda_{+,\mathcal{V}}(1/2,1)$. At $n=800/5{,}000/20{,}000$, the conservative and oracle CIs have lengths $0.955/0.741/0.612$ and $0.505/0.344/0.256$ for $\mathcal{V}=\{1\}$, $0.873/0.676/0.559$ and $0.456/0.311/0.232$ for $\mathcal{V}=\{1,2\}$ and $0.710/0.549/0.457$ and $0.401/0.275/0.205$ for $\mathcal{V}=\{1,2,3\}$. Over all $n$, $\mathcal{V}$ and the regression functions in their largest parameter spaces, the coverage probabilities of AdaptCalib and AdaptMono are at least $0.945$.
  \end{minipage}
\end{table}}

\newcommand{\PhaseTwoCalibrationTable}{%
\begin{table}[t]
  \caption{$\tau^{\ast}$, $K_{\tau}$ and critical values of AdaptCalib and AdaptMono ($\alpha=0.05$, $J=4$, $\mathcal{F}_{j}=\Lambda_{+,\mathcal{V}}(\gamma_{j},1)$ with $(\gamma_{1},\dots,\gamma_{4})=(1,4/5,2/3,1/2)$)}
  \label{tab:phase2-calibration}
  \centering
  \setlength{\tabcolsep}{4pt}
  \begin{tabular}{@{\extracolsep{5pt}} lccccccc} \\[-1.8ex]\hline
    \hline \\[-1.8ex] & \multicolumn{4}{c}{AdaptCalib} & \multicolumn{2}{c}{AdaptMono} & \\
    \cline{2-5} \cline{6-7} \\[-1.8ex]
    & $\tau^{\ast}_{0}$ & $\tau^{\ast}$ & $K$ & $z_{1-\tau^{\ast}}$ & $\tau^{\ast}$ & $z_{1-\tau^{\ast}}$ & $z_{\rm Bon}$ \\
    \hline \\[-1.8ex]
    \multicolumn{8}{@{}l}{\textit{$k=2$, $\mathcal{V}=\{1\}$, $n=800$}} \\
    Uniform design & $0.01850$ & $0.03700$ & $2.642$ & $1.787$ & $0.02150$ & $2.024$ & $2.498$ \\
    Center-concentrated design & $0.01825$ & $0.03700$ & $3.010$ & $1.787$ & $0.02100$ & $2.034$ & $2.498$ \\
    \hline \\[-1.8ex]
    \multicolumn{8}{@{}l}{\textit{$k=3$, $\mathcal{V}=\{1\}$}} \\
    $n=800$ & $0.02100$ & $0.04150$ & $5.245$ & $1.734$ & $0.02300$ & $1.995$ & $2.498$ \\
    $n=5{,}000$ & $0.01900$ & $0.03850$ & $4.097$ & $1.768$ & $0.02075$ & $2.038$ & $2.498$ \\
    $n=20{,}000$ & $0.01800$ & $0.03650$ & $4.155$ & $1.793$ & $0.01950$ & $2.064$ & $2.498$ \\
    \hline \\[-1.8ex]
    \multicolumn{8}{@{}l}{\textit{$k=3$, $\mathcal{V}=\{1,2\}$}} \\
    $n=800$ & $0.01875$ & $0.03725$ & $4.369$ & $1.784$ & $0.02375$ & $1.982$ & $2.498$ \\
    $n=5{,}000$ & $0.01700$ & $0.03450$ & $3.156$ & $1.818$ & $0.02125$ & $2.029$ & $2.498$ \\
    $n=20{,}000$ & $0.01550$ & $0.03225$ & $2.713$ & $1.849$ & $0.01975$ & $2.059$ & $2.498$ \\
    \hline \\[-1.8ex]
    \multicolumn{8}{@{}l}{\textit{$k=3$, $\mathcal{V}=\{1,2,3\}$}} \\
    $n=800$ & $0.01450$ & $0.02675$ & $0.914$ & $1.931$ & $0.02375$ & $1.982$ & $2.498$ \\
    $n=5{,}000$ & $0.01350$ & $0.02100$ & $0.519$ & $2.034$ & $0.02150$ & $2.024$ & $2.498$ \\
    $n=20{,}000$ & $0.01250$ & $0.01725$ & $0.312$ & $2.114$ & $0.01975$ & $2.059$ & $2.498$ \\
    \hline \\[-1.8ex]
  \end{tabular}
  \par\smallskip
  \begin{minipage}{0.92\linewidth}\footnotesize
  $\tau^{\ast}$ is the largest point of the grid $\alpha/(2J)+0.00025m$ (together with $\alpha$) at which the bound of Lemma~S2.1 on $p_{\tau}(K_{\tau})$, computed with exact bivariate normal probabilities, is at most $\alpha$, and $\tau^{\ast}_{0}$ uses $K_{\tau}=0$. $K$ is $K_{\tau}$ at $\tau^{\ast}$, given by $K_{\tau}=\delta^{-1}\min_{j,\ell}\sum_{i}\min\{D^{-}_{j,i},D^{+}_{\ell,i}\}\{h_{+}(x_{i})-h_{-}(x_{i})\}$. For AdaptMono, whose largest parameter space is $\mathcal{G}_{\mathcal{V}}(1/2,1)$, $K=0$ because no design point has $x_{v}=0$ for all $v\in\mathcal{V}$, so that its lower ends use only the design points with $x_{\mathcal{V}}\leq0$ and its upper ends only those with $x_{\mathcal{V}}\geq0$. $z_{\rm Bon}=z_{1-\alpha/(2J)}$ is the critical value of AdaptBonf.
  \end{minipage}
\end{table}}

\newcommand{\VOneMonoLengthsTable}{%
\begin{table}[t]
  \caption{Average lengths of CIs under monotonicity in one variable, including AdaptMono ($\mathcal{V}=\{1\}$, $n=800$, $J=4$, $\mathcal{F}_{j}=\Lambda_{+,\mathcal{V}}(\gamma_{j},1)$ with $(\gamma_{1},\dots,\gamma_{4})=(1,4/5,2/3,1/2)$)}
  \label{tab:v1-mono-lengths}
  \centering
  \setlength{\tabcolsep}{3pt}
  \begin{tabular}{@{\extracolsep{5pt}} lcccccccccc} \\[-1.8ex]\hline
    \hline \\[-1.8ex] & $f_{1}$ & $f_{2}$ & $f_{3}$ & $f_{4}$ & $f_{5}$ & $h_{+}$ & $h_{-}$ & $f_{\mathcal{G}}$ & $f_{6}$ & $f_{7}$ \\
    \hline \\[-1.8ex]
    \multicolumn{11}{@{}l}{\textit{Uniform design}} \\
    AdaptBonf & $0.869$ & $0.908$ & $0.903$ & $0.907$ & $0.913$ & $0.909$ & $0.916$ & $0.806$ & $0.896$ & $0.897$ \\
    AdaptCalib & $0.747$ & $0.787$ & $0.784$ & $0.790$ & $0.800$ & $0.790$ & $0.802$ & $0.680$ & $0.775$ & $0.775$ \\
    AdaptMono & $0.802$ & $0.927$ & $0.924$ & $0.948$ & $0.963$ & $0.954$ & $0.974$ & $0.474$ & $0.886$ & $0.885$ \\
    Conservative & $0.773$ & $0.773$ & $0.773$ & $0.773$ & $0.773$ & $0.773$ & $0.773$ & $0.773$ & $0.773$ & $0.773$ \\
    Oracle & $0.413$ & $0.413$ & $0.773$ & $0.773$ & $0.773$ & $0.773$ & $0.773$ & -- & $0.413$ & $0.413$ \\
    \hline \\[-1.8ex]
    \multicolumn{11}{@{}l}{\textit{Center-concentrated design}} \\
    AdaptBonf & $0.782$ & $0.815$ & $0.813$ & $0.816$ & $0.817$ & $0.818$ & $0.821$ & $0.713$ & $0.805$ & $0.804$ \\
    AdaptCalib & $0.676$ & $0.709$ & $0.710$ & $0.715$ & $0.717$ & $0.717$ & $0.719$ & $0.602$ & $0.699$ & $0.699$ \\
    AdaptMono & $0.727$ & $0.831$ & $0.839$ & $0.858$ & $0.850$ & $0.866$ & $0.860$ & $0.400$ & $0.796$ & $0.797$ \\
    Conservative & $0.703$ & $0.703$ & $0.703$ & $0.703$ & $0.703$ & $0.703$ & $0.703$ & $0.703$ & $0.703$ & $0.703$ \\
    Oracle & $0.351$ & $0.351$ & $0.703$ & $0.703$ & $0.703$ & $0.703$ & $0.703$ & -- & $0.351$ & $0.351$ \\
    \hline \\[-1.8ex]
  \end{tabular}
  \par\smallskip
  \begin{minipage}{0.92\linewidth}\footnotesize
  $f_{1}=0$, $f_{2}=0.6x_{1}+0.5x_{2}$, $f_{3}=0.8\lVert(x)_{\mathcal{V}+}\rVert^{1/2}$, $f_{4}$ and $f_{5}$ are the least favorable functions for the ends of AdaptCalib that direct power to $\Lambda_{+,\mathcal{V}}(1,1)$, $h_{+}(x)=\lVert(x)_{\mathcal{V}+}\rVert^{1/2}$ and $h_{-}(x)=-\lVert(x)_{\mathcal{V}-}\rVert^{1/2}$. At the design points, $f_{\mathcal{G}}$ equals $\lVert s(x)\rVert^{1/2}$ if $x_{\mathcal{V}}\leq0$, $-\lVert s(x)\rVert^{1/2}$ if $x_{\mathcal{V}}\geq0$ and $0$ otherwise, where $s(x)$ sets the coordinates in $\mathcal{V}$ to $0$. The function $f_{\mathcal{G}}$ lies in $\mathcal{G}_{\mathcal{V}}(1/2,1)$ but in no H\"older class, and it puts every end of AdaptMono at its maximum bias. $f_{6}(x)=0.4x_{1}+0.11\{e^{-x_{2}^{2}/(2\cdot0.08^{2})}-1\}$ (a bump in $x_{2}$) and $f_{7}(x)=0.4x_{1}+0.07\{\sin(x_{2}/0.08-\pi/4)+\sin(\pi/4)\}$ (a sine in $x_{2}$) are nondecreasing in $x_{1}$, not monotone in $x_{2}$, and lie in $\Lambda_{+,\mathcal{V}}(1,1)$ (H\"older ratios $0.92$ and $0.96$ on the design and on an $81\times81$ grid). AdaptMono directs power to the same four parameter spaces and maintains coverage over $\mathcal{G}_{\mathcal{V}}(1/2,1)$, the functions that are nondecreasing in $x_{\mathcal{V}}$ and satisfy $\lvert f(x)-f(0)\rvert\leq\lVert x\rVert^{1/2}$ when $x_{v}=0$ for all $v\in\mathcal{V}$, and AdaptBonf, AdaptCalib and Conservative over $\Lambda_{+,\mathcal{V}}(1/2,1)$. Conservative and Oracle are the minimax affine fixed-length CIs of \citet{Donoho1994} for $\Lambda_{+,\mathcal{V}}(1/2,1)$ and for the smallest parameter space $\Lambda_{+,\mathcal{V}}(\gamma_{j},1)$ that contains the regression function. AdaptCalib uses $\tau^{\ast}=0.03700$ and AdaptMono $\tau^{\ast}=0.02150$ (uniform) and $0.02100$ (center-concentrated). The rows of AdaptBonf, AdaptCalib and the fixed-length CIs at $f_{1}$ to $f_{5}$ and $h_{\pm}$ reproduce the corresponding entries of Tables~S2 and S3. Monte Carlo standard errors of the average lengths are at most $0.0018$. Empty intersections (length zero, counted as not covering) occurred in at most 0.31\% of the draws, only for AdaptMono at $f_{\mathcal{G}}$.
  \end{minipage}
\end{table}}

\newcommand{\VOneMonoCoverageTable}{%
\begin{table}[t]
  \caption{Coverage probabilities of CIs under monotonicity in one variable, including AdaptMono ($\mathcal{V}=\{1\}$, $n=800$, $J=4$, $\mathcal{F}_{j}=\Lambda_{+,\mathcal{V}}(\gamma_{j},1)$ with $(\gamma_{1},\dots,\gamma_{4})=(1,4/5,2/3,1/2)$)}
  \label{tab:v1-mono-coverage}
  \centering
  \setlength{\tabcolsep}{3pt}
  \begin{tabular}{@{\extracolsep{5pt}} lcccccccccc} \\[-1.8ex]\hline
    \hline \\[-1.8ex] & $f_{1}$ & $f_{2}$ & $f_{3}$ & $f_{4}$ & $f_{5}$ & $h_{+}$ & $h_{-}$ & $f_{\mathcal{G}}$ & $f_{6}$ & $f_{7}$ \\
    \hline \\[-1.8ex]
    \multicolumn{11}{@{}l}{\textit{Uniform design}} \\
    AdaptBonf & $1.000$ & $1.000$ & $0.997$ & $0.992$ & $0.994$ & $0.992$ & $0.993$ & $1.000$ & $1.000$ & $1.000$ \\
    AdaptCalib & $0.998$ & $0.999$ & $0.980$ & $0.955$ & $0.952$ & $0.955$ & $0.952$ & $0.998$ & $0.999$ & $0.999$ \\
    AdaptMono & $0.998$ & $1.000$ & $0.986$ & $0.977$ & $0.974$ & $0.977$ & $0.974$ & $0.951$ & $1.000$ & $0.999$ \\
    Conservative & $1.000$ & $1.000$ & $0.987$ & $0.959$ & $0.960$ & $0.952$ & $0.950$ & $1.000$ & $1.000$ & $1.000$ \\
    Oracle & $0.999$ & $0.999$ & $0.987$ & $0.959$ & $0.960$ & $0.952$ & $0.950$ & -- & $0.997$ & $0.998$ \\
    \hline \\[-1.8ex]
    \multicolumn{11}{@{}l}{\textit{Center-concentrated design}} \\
    AdaptBonf & $1.000$ & $1.000$ & $0.999$ & $0.993$ & $0.992$ & $0.992$ & $0.992$ & $1.000$ & $1.000$ & $1.000$ \\
    AdaptCalib & $1.000$ & $1.000$ & $0.982$ & $0.956$ & $0.952$ & $0.956$ & $0.952$ & $0.999$ & $1.000$ & $1.000$ \\
    AdaptMono & $0.999$ & $1.000$ & $0.988$ & $0.975$ & $0.975$ & $0.975$ & $0.975$ & $0.950$ & $1.000$ & $1.000$ \\
    Conservative & $1.000$ & $1.000$ & $0.991$ & $0.959$ & $0.959$ & $0.952$ & $0.952$ & $1.000$ & $1.000$ & $1.000$ \\
    Oracle & $0.998$ & $0.998$ & $0.991$ & $0.959$ & $0.959$ & $0.952$ & $0.952$ & -- & $0.995$ & $0.998$ \\
    \hline \\[-1.8ex]
  \end{tabular}
  \par\smallskip
  \begin{minipage}{0.92\linewidth}\footnotesize
  $f_{1}=0$, $f_{2}=0.6x_{1}+0.5x_{2}$, $f_{3}=0.8\lVert(x)_{\mathcal{V}+}\rVert^{1/2}$, $f_{4}$ and $f_{5}$ are the least favorable functions for the ends of AdaptCalib that direct power to $\Lambda_{+,\mathcal{V}}(1,1)$, $h_{+}(x)=\lVert(x)_{\mathcal{V}+}\rVert^{1/2}$ and $h_{-}(x)=-\lVert(x)_{\mathcal{V}-}\rVert^{1/2}$. At the design points, $f_{\mathcal{G}}$ equals $\lVert s(x)\rVert^{1/2}$ if $x_{\mathcal{V}}\leq0$, $-\lVert s(x)\rVert^{1/2}$ if $x_{\mathcal{V}}\geq0$ and $0$ otherwise, where $s(x)$ sets the coordinates in $\mathcal{V}$ to $0$. The function $f_{\mathcal{G}}$ lies in $\mathcal{G}_{\mathcal{V}}(1/2,1)$ but in no H\"older class, and it puts every end of AdaptMono at its maximum bias. $f_{6}(x)=0.4x_{1}+0.11\{e^{-x_{2}^{2}/(2\cdot0.08^{2})}-1\}$ (a bump in $x_{2}$) and $f_{7}(x)=0.4x_{1}+0.07\{\sin(x_{2}/0.08-\pi/4)+\sin(\pi/4)\}$ (a sine in $x_{2}$) are nondecreasing in $x_{1}$, not monotone in $x_{2}$, and lie in $\Lambda_{+,\mathcal{V}}(1,1)$ (H\"older ratios $0.92$ and $0.96$ on the design and on an $81\times81$ grid). AdaptMono directs power to the same four parameter spaces and maintains coverage over $\mathcal{G}_{\mathcal{V}}(1/2,1)$, the functions that are nondecreasing in $x_{\mathcal{V}}$ and satisfy $\lvert f(x)-f(0)\rvert\leq\lVert x\rVert^{1/2}$ when $x_{v}=0$ for all $v\in\mathcal{V}$, and AdaptBonf, AdaptCalib and Conservative over $\Lambda_{+,\mathcal{V}}(1/2,1)$. Conservative and Oracle are the minimax affine fixed-length CIs of \citet{Donoho1994} for $\Lambda_{+,\mathcal{V}}(1/2,1)$ and for the smallest parameter space $\Lambda_{+,\mathcal{V}}(\gamma_{j},1)$ that contains the regression function. AdaptCalib uses $\tau^{\ast}=0.03700$ and AdaptMono $\tau^{\ast}=0.02150$ (uniform) and $0.02100$ (center-concentrated). The rows of AdaptBonf, AdaptCalib and the fixed-length CIs at $f_{1}$ to $f_{5}$ and $h_{\pm}$ reproduce the corresponding entries of Tables~S2 and S3. Monte Carlo standard errors of the coverage probabilities are at most $0.0022$. Empty intersections (length zero, counted as not covering) occurred in at most 0.31\% of the draws, only for AdaptMono at $f_{\mathcal{G}}$.
  \end{minipage}
\end{table}}

\newcommand{\FeasibleCoverageTable}{%
\begin{table}[t]
  \caption{Coverage probabilities of AdaptCalib and AdaptMono computed with an estimate of $\sigma$, under non-Gaussian and heteroskedastic errors ($\mathcal{V}=\{1\}$, uniform design, $n=800$, $J=4$, $\mathcal{F}_{j}=\Lambda_{+,\mathcal{V}}(\gamma_{j},1)$ with $(\gamma_{1},\dots,\gamma_{4})=(1,4/5,2/3,1/2)$)}
  \label{tab:feasible-coverage}
  \centering
  \setlength{\tabcolsep}{3pt}
  \begin{tabular}{@{\extracolsep{3pt}} llccccccccc} \\[-1.8ex]\hline
    \hline \\[-1.8ex] & & & \multicolumn{4}{c}{AdaptCalib} & \multicolumn{4}{c}{AdaptMono} \\
    \cline{4-7} \cline{8-11} \\[-1.8ex]
    Errors & $\sigma$ & Mean (s.d.) & $f_{1}$ & $h_{+}$ & $h_{-}$ & $f_{2}$ & $f_{1}$ & $h_{+}$ & $h_{-}$ & $f_{2}$ \\
    \hline \\[-1.8ex]
    $N(0,1)$ & $\hat\sigma$ & $0.999$ ($0.035$) & $0.998$ & $0.954$ & $0.951$ & $0.999$ & $0.998$ & $0.975$ & $0.974$ & $1.000$ \\
     & $\hat\sigma_{0}$ & $0.999$ ($0.068$) & $0.998$ & $0.953$ & $0.950$ & $0.998$ & $0.998$ & $0.975$ & $0.972$ & $1.000$ \\
     & $1$ & -- & $0.998$ & $0.955$ & $0.952$ & $0.999$ & $0.998$ & $0.977$ & $0.974$ & $1.000$ \\
    $t_{5}/\sqrt{5/3}$ & $\hat\sigma$ & $0.999$ ($0.055$) & $0.999$ & $0.945$ & $0.951$ & $0.999$ & $0.998$ & $0.970$ & $0.974$ & $0.999$ \\
     & $\hat\sigma_{0}$ & $0.995$ ($0.102$) & $0.999$ & $0.944$ & $0.951$ & $0.999$ & $0.998$ & $0.969$ & $0.973$ & $0.999$ \\
     & $1$ & -- & $0.999$ & $0.945$ & $0.952$ & $0.999$ & $0.998$ & $0.972$ & $0.975$ & $0.999$ \\
    $\mathrm{Exp}(1)-1$ & $\hat\sigma$ & $0.997$ ($0.058$) & $0.998$ & $0.943$ & $0.960$ & $0.998$ & $0.998$ & $0.968$ & $0.981$ & $0.999$ \\
     & $\hat\sigma_{0}$ & $0.993$ ($0.110$) & $0.998$ & $0.948$ & $0.949$ & $0.999$ & $0.998$ & $0.971$ & $0.975$ & $0.999$ \\
     & $1$ & -- & $0.998$ & $0.941$ & $0.963$ & $0.998$ & $0.998$ & $0.966$ & $0.984$ & $0.999$ \\
    $(\chi^{2}_{1}-1)/\sqrt{2}$ & $\hat\sigma$ & $0.998$ ($0.074$) & $0.998$ & $0.938$ & $0.964$ & $0.998$ & $0.998$ & $0.965$ & $0.986$ & $0.999$ \\
     & $\hat\sigma_{0}$ & $0.989$ ($0.139$) & $0.998$ & $0.944$ & $0.949$ & $0.999$ & $0.998$ & $0.970$ & $0.977$ & $1.000$ \\
     & $1$ & -- & $0.998$ & $0.936$ & $0.968$ & $0.998$ & $0.998$ & $0.963$ & $0.990$ & $0.999$ \\
    $\sigma(x)N(0,1)$ & $\hat\sigma$ & $0.978$ ($0.037$) & $0.979$ & $0.881$ & $0.874$ & $0.981$ & $0.980$ & $0.928$ & $0.920$ & $0.991$ \\
     & $\hat\sigma_{0}$ & $1.232$ ($0.085$) & $0.996$ & $0.939$ & $0.932$ & $0.997$ & $0.996$ & $0.969$ & $0.964$ & $0.999$ \\
     & $1.5$ & -- & $1.000$ & $0.974$ & $0.972$ & $1.000$ & $1.000$ & $0.990$ & $0.989$ & $1.000$ \\
    \hline \\[-1.8ex]
  \end{tabular}
  \par\smallskip
  \begin{minipage}{0.92\linewidth}\footnotesize
  Each CI is computed as if $\sigma(\cdot)$ were the constant in the column $\sigma$, including $\tau^{\ast}(\hat\sigma)$ and its slack. The global $\hat\sigma$ is the nearest-neighbor estimator $\hat\sigma^{2}=(2n)^{-1}\sum_{i}(y_{i}-y_{i^{\ast}})^{2}$; the local $\hat\sigma_{0}$ uses the same squared differences for the $m=\lceil n^{4/5}\rceil=211$ design points closest to $x_{0}$; the next column gives the mean and standard deviation of the estimate at $f_{1}$. The rows with $\sigma=1$ (the true value) and $\sigma=1.5=\sigma(0)$ use that value instead of an estimate. All errors have variance $1$, except that the heteroskedastic errors have $\sigma(x)=1.5-2\lVert x\rVert_{2}$, which lies between $0.5$ and $1.5$ on the design. $\tau^{\ast}(s)$ is computed on a grid of $s$ with ratio $e^{0.0025}$, and each draw uses the smaller of the two grid values of $\tau^{\ast}$ that bracket $\hat\sigma$. AdaptMono directs power to the same four parameter spaces and maintains coverage over $\mathcal{G}_{\mathcal{V}}(1/2,1)$, the functions that are nondecreasing in $x_{\mathcal{V}}$ and satisfy $\lvert f(x)-f(0)\rvert\leq\lVert x\rVert^{1/2}$ when $x_{v}=0$ for all $v\in\mathcal{V}$, and AdaptBonf, AdaptCalib and Conservative over $\Lambda_{+,\mathcal{V}}(1/2,1)$. Monte Carlo standard errors of the coverage probabilities are at most $0.0034$.
  \end{minipage}
\end{table}}

\newcommand{\FeasibleLengthsTable}{%
\begin{table}[t]
  \caption{Average lengths of AdaptCalib and AdaptMono computed with an estimate of $\sigma$, under non-Gaussian and heteroskedastic errors ($\mathcal{V}=\{1\}$, uniform design, $n=800$, $J=4$, $\mathcal{F}_{j}=\Lambda_{+,\mathcal{V}}(\gamma_{j},1)$ with $(\gamma_{1},\dots,\gamma_{4})=(1,4/5,2/3,1/2)$)}
  \label{tab:feasible-lengths}
  \centering
  \setlength{\tabcolsep}{3pt}
  \begin{tabular}{@{\extracolsep{3pt}} llccccccccc} \\[-1.8ex]\hline
    \hline \\[-1.8ex] & & & \multicolumn{4}{c}{AdaptCalib} & \multicolumn{4}{c}{AdaptMono} \\
    \cline{4-7} \cline{8-11} \\[-1.8ex]
    Errors & $\sigma$ & Mean (s.d.) & $f_{1}$ & $h_{+}$ & $h_{-}$ & $f_{2}$ & $f_{1}$ & $h_{+}$ & $h_{-}$ & $f_{2}$ \\
    \hline \\[-1.8ex]
    $N(0,1)$ & $\hat\sigma$ & $0.999$ ($0.035$) & $0.747$ & $0.790$ & $0.802$ & $0.786$ & $0.802$ & $0.954$ & $0.974$ & $0.927$ \\
     & $\hat\sigma_{0}$ & $0.999$ ($0.068$) & $0.746$ & $0.789$ & $0.801$ & $0.786$ & $0.801$ & $0.953$ & $0.974$ & $0.926$ \\
     & $1$ & -- & $0.747$ & $0.790$ & $0.802$ & $0.787$ & $0.802$ & $0.954$ & $0.974$ & $0.927$ \\
    $t_{5}/\sqrt{5/3}$ & $\hat\sigma$ & $0.999$ ($0.055$) & $0.747$ & $0.791$ & $0.802$ & $0.787$ & $0.802$ & $0.954$ & $0.974$ & $0.927$ \\
     & $\hat\sigma_{0}$ & $0.995$ ($0.102$) & $0.745$ & $0.788$ & $0.800$ & $0.784$ & $0.799$ & $0.952$ & $0.972$ & $0.924$ \\
     & $1$ & -- & $0.748$ & $0.791$ & $0.803$ & $0.788$ & $0.802$ & $0.955$ & $0.975$ & $0.928$ \\
    $\mathrm{Exp}(1)-1$ & $\hat\sigma$ & $0.997$ ($0.058$) & $0.746$ & $0.789$ & $0.801$ & $0.786$ & $0.799$ & $0.951$ & $0.972$ & $0.925$ \\
     & $\hat\sigma_{0}$ & $0.993$ ($0.110$) & $0.744$ & $0.787$ & $0.799$ & $0.783$ & $0.797$ & $0.949$ & $0.969$ & $0.922$ \\
     & $1$ & -- & $0.748$ & $0.791$ & $0.802$ & $0.787$ & $0.801$ & $0.953$ & $0.974$ & $0.926$ \\
    $(\chi^{2}_{1}-1)/\sqrt{2}$ & $\hat\sigma$ & $0.998$ ($0.074$) & $0.745$ & $0.789$ & $0.801$ & $0.785$ & $0.799$ & $0.951$ & $0.972$ & $0.924$ \\
     & $\hat\sigma_{0}$ & $0.989$ ($0.139$) & $0.741$ & $0.785$ & $0.796$ & $0.781$ & $0.794$ & $0.946$ & $0.966$ & $0.918$ \\
     & $1$ & -- & $0.746$ & $0.790$ & $0.802$ & $0.786$ & $0.800$ & $0.952$ & $0.973$ & $0.926$ \\
    $\sigma(x)N(0,1)$ & $\hat\sigma$ & $0.978$ ($0.037$) & $0.733$ & $0.774$ & $0.786$ & $0.770$ & $0.788$ & $0.940$ & $0.961$ & $0.911$ \\
     & $\hat\sigma_{0}$ & $1.232$ ($0.085$) & $0.814$ & $0.853$ & $0.862$ & $0.851$ & $0.879$ & $1.037$ & $1.049$ & $1.016$ \\
     & $1.5$ & -- & $0.888$ & $0.925$ & $0.931$ & $0.925$ & $0.966$ & $1.127$ & $1.132$ & $1.113$ \\
    \hline \\[-1.8ex]
  \end{tabular}
  \par\smallskip
  \begin{minipage}{0.92\linewidth}\footnotesize
  Each CI is computed as if $\sigma(\cdot)$ were the constant in the column $\sigma$, including $\tau^{\ast}(\hat\sigma)$ and its slack. The global $\hat\sigma$ is the nearest-neighbor estimator $\hat\sigma^{2}=(2n)^{-1}\sum_{i}(y_{i}-y_{i^{\ast}})^{2}$; the local $\hat\sigma_{0}$ uses the same squared differences for the $m=\lceil n^{4/5}\rceil=211$ design points closest to $x_{0}$; the next column gives the mean and standard deviation of the estimate at $f_{1}$. The rows with $\sigma=1$ (the true value) and $\sigma=1.5=\sigma(0)$ use that value instead of an estimate. All errors have variance $1$, except that the heteroskedastic errors have $\sigma(x)=1.5-2\lVert x\rVert_{2}$, which lies between $0.5$ and $1.5$ on the design. $\tau^{\ast}(s)$ is computed on a grid of $s$ with ratio $e^{0.0025}$, and each draw uses the smaller of the two grid values of $\tau^{\ast}$ that bracket $\hat\sigma$. AdaptMono directs power to the same four parameter spaces and maintains coverage over $\mathcal{G}_{\mathcal{V}}(1/2,1)$, the functions that are nondecreasing in $x_{\mathcal{V}}$ and satisfy $\lvert f(x)-f(0)\rvert\leq\lVert x\rVert^{1/2}$ when $x_{v}=0$ for all $v\in\mathcal{V}$, and AdaptBonf, AdaptCalib and Conservative over $\Lambda_{+,\mathcal{V}}(1/2,1)$. Monte Carlo standard errors of the average lengths are at most $0.0024$.
  \end{minipage}
\end{table}}

\newcommand{\BoundaryRatesTable}{%
\begin{table}[t]
  \caption{Ordered moduli at a boundary point and at an interior point ($k=2$, $\mathcal{V}=\{1\}$, $\mathcal{F}_{1}=\Lambda_{+,\mathcal{V}}(1,1)$, $\mathcal{F}_{2}=\Lambda_{+,\mathcal{V}}(1/2,1)$, $\delta=z_{0.975}$)}
  \label{tab:boundary-rates}
  \centering
  \setlength{\tabcolsep}{4pt}
  \begin{tabular}{@{\extracolsep{5pt}} lcccc} \\[-1.8ex]\hline
    \hline \\[-1.8ex] & \multicolumn{2}{c}{Boundary point} & \multicolumn{2}{c}{Interior point} \\
    \cline{2-3} \cline{4-5} \\[-1.8ex]
    & $\omega(\delta,\mathcal{F}_{1},\mathcal{F}_{2})$ & $\omega(\delta,\mathcal{F}_{2},\mathcal{F}_{1})$ & $\omega(\delta,\mathcal{F}_{1},\mathcal{F}_{2})$ & $\omega(\delta,\mathcal{F}_{2},\mathcal{F}_{1})$ \\
    \hline \\[-1.8ex]
    $n=500$ & $0.5896$ & $0.6525$ & $0.5416$ & $0.5425$ \\
    $n=2{,}000$ & $0.4420$ & $0.5079$ & $0.4113$ & $0.4107$ \\
    $n=8{,}000$ & $0.3324$ & $0.3959$ & $0.3122$ & $0.3122$ \\
    $n=32{,}000$ & $0.2505$ & $0.3104$ & $0.2372$ & $0.2374$ \\
    $n=128{,}000$ & $0.1883$ & $0.2423$ & $0.1808$ & $0.1809$ \\
    \hline \\[-1.8ex]
    Fitted rate & $0.206$ & $0.178$ & $0.198$ & $0.198$ \\
    Theory & $1/5$ & $1/6$ & $1/5$ & $1/5$ \\
    \hline \\[-1.8ex]
  \end{tabular}
  \par\smallskip
  \begin{minipage}{0.92\linewidth}\footnotesize
  The boundary design is uniform on $[0,2a]\times[-a,a]$ and the interior design uniform on $[-a,a]^{2}$, $a=1/(2\sqrt{2})$; both have diameter $1$ and density $2$, and $x_{0}=0$ lies on the lower boundary in the monotone coordinate of the first. $\omega(\delta,\mathcal{F}_{1},\mathcal{F}_{2})$ governs the upper ends and $\omega(\delta,\mathcal{F}_{2},\mathcal{F}_{1})$ the lower ends of the CIs. Entries are averages over 100 designs for each $n$ (standard errors at most $0.0022$). The fitted exponent is minus the slope of the least squares fit of $\log$ of the average on $\log n$.
  \end{minipage}
\end{table}}

%% file: tables/phase2_empirical_tables.tex
\newcommand{\EmpiricalPhaseTwoSummaryTable}{%
\begin{table}[t]
  \caption{Summary statistics: Chinese chemical industry, 2001, outliers removed from the inputs only}
  \label{tab:empirical-summary-phase2}
  \centering
  \begin{tabular}{@{\extracolsep{5pt}} lcccc}
    \\[-1.8ex]\hline \hline \\[-1.8ex] & Min & Mean & Median & Max \\
    \hline
    \\[-1.8ex]
    Log output & $0.000$ & $9.891$ & $9.925$ & $14.630$ \\
    Log fixed asset & $7.463$ & $10.807$ & $10.752$ & $14.226$ \\
    Log labor & $3.664$ & $6.343$ & $6.378$ & $9.142$ \\
    \hline \\[-1.8ex]
  \end{tabular}
  \par\smallskip
  \begin{minipage}{0.92\linewidth}\footnotesize
  $n=1{,}663$ firms: the 1,725 firms of the chemical industry less those outside Tukey's fences ($1.5$ times the interquartile range beyond the hinges) for log fixed asset (43 firms) or log labor (29 firms), 62 firms in all. The outcome is not trimmed. \citet{horowitz2017nonparametric}, who also trim the outcome, have $n=1{,}638$. The target is $(k_{0},\ell_{0})=(10.7518,6.3784)$, the medians, and the norm uses the standard deviations $s_{k}=1.2360$ and $s_{\ell}=1.0158$; the diameter of the design in this norm is $\Delta=7.388$.
  \end{minipage}
\end{table}}

\newcommand{\EmpiricalPhaseTwoTable}{%
\begin{table}[t]
  \caption{95\% confidence intervals for $f(k_{0},\ell_{0})$ under monotonicity in both inputs and in one input}
  \label{tab:empirical-phase2}
  \centering
  \setlength{\tabcolsep}{3pt}
  \begin{tabular}{@{\extracolsep{3pt}} lccccccc} \\[-1.8ex]\hline
    \hline \\[-1.8ex] & CI & Length & $\tau^{\ast}$ & $K$ & $j$ & $n_{+}$ & ESS \\
    \hline \\[-1.8ex]
    \multicolumn{8}{@{}l}{\textit{$\mathcal{V}=\{1,2\}$: monotone in capital and labor; $J=6$, $\gamma_{J}=0.001$}} \\
    AdaptMono & $\mathbf{[9.299,10.347]}$ & $\mathbf{1.047}$ & $0.00717$ & $0.00$ & $4,\ 1$ & $7,\ 61$ & $4.2,\ 35.6$ \\
    AdaptCalib & $\mathbf{[9.299,10.347]}$ & $\mathbf{1.047}$ & $0.00717$ & $0.00$ & $4,\ 1$ & $7,\ 61$ & $4.2,\ 35.6$ \\
    AdaptBonf & $\mathbf{[9.209,10.358]}$ & $\mathbf{1.149}$ & $0.00417$ & $0.00$ & $5,\ 1$ & $6,\ 62$ & $2.8,\ 39.1$ \\
    Conservative & $[2.542,17.258]$ & $14.716$ & -- & -- & -- & $1197$ & $1185.7$ \\
    Restricted & $[9.274,10.003]$ & $0.729$ & -- & -- & -- & $124$ & $76.6$ \\
    DHZ & $[9.476,9.853]$ & $0.377$ & -- & -- & -- & $108$ & -- \\
    \hline \\[-1.8ex]
    \multicolumn{8}{@{}l}{\textit{$\mathcal{V}=\{1\}$: monotone in capital only; $J=4$, $\gamma_{J}=0.4$}} \\
    AdaptCalib & $\mathbf{[8.223,11.244]}$ & $\mathbf{3.021}$ & $0.02625$ & $2.21$ & $1$ & $12,\ 16$ & $8.2,\ 10.5$ \\
    AdaptMono & $\mathbf{[8.267,11.175]}$ & $\mathbf{2.908}$ & $0.01300$ & $0.00$ & $1$ & $12,\ 18$ & $8.0,\ 11.8$ \\
    Conservative & $[8.047,11.904]$ & $3.857$ & -- & -- & -- & $10$ & $6.9$ \\
    Restricted & $[9.211,10.044]$ & $0.833$ & -- & -- & -- & $82$ & $55.2$ \\
    \hline \\[-1.8ex]
    \multicolumn{8}{@{}l}{\textit{$\mathcal{V}=\{2\}$: monotone in labor only; $J=4$, $\gamma_{J}=0.4$}} \\
    AdaptCalib & $\mathbf{[8.546,11.247]}$ & $\mathbf{2.701}$ & $0.02525$ & $1.43$ & $1$ & $10,\ 11$ & $7.7,\ 8.7$ \\
    AdaptMono & $\mathbf{[8.635,11.248]}$ & $\mathbf{2.613}$ & $0.01375$ & $0.00$ & $1$ & $12,\ 11$ & $7.3,\ 8.7$ \\
    Conservative & $[8.099,11.635]$ & $3.536$ & -- & -- & -- & $8$ & $6.3$ \\
    Restricted & $[9.246,10.080]$ & $0.834$ & -- & -- & -- & $83$ & $57.8$ \\
    \hline \\[-1.8ex]
  \end{tabular}
  \par\smallskip
  \begin{minipage}{0.92\linewidth}\footnotesize
  $n=1{,}663$; every CI is computed as if $\sigma(\cdot)=\hat\sigma=1.088$. $\mathcal{F}_{j}=\Lambda_{+,\mathcal{V}}(\gamma_{j},C_{j})$ with $(\gamma_{1},\dots,\gamma_{6})=(1,0.8,0.6,0.4,0.2,0.001)$ and $C_{j}=C_{6}\Delta^{\gamma_{6}-\gamma_{j}}$ with $C_{6}=14.612$ and $\Delta=7.388$: $(C_{1},\dots,C_{6})=(1.982,2.956,4.410,6.579,9.814,14.612)$. The CIs use $\mathcal{F}_{1},\dots,\mathcal{F}_{J}$: $J=6$ ($\gamma_{J}=0.001$, $C_{J}=14.612$) for $\mathcal{V}=\{1,2\}$; $J=4$ ($\gamma_{J}=0.4$, $C_{J}=6.579$) for $\mathcal{V}=\{1\}$ and $\mathcal{V}=\{2\}$. AdaptCalib, AdaptBonf and Conservative maintain coverage over $\mathcal{F}_{J}$, and AdaptMono over $\mathcal{G}_{\mathcal{V}}(\gamma_{J},C_{J})$, which for $\mathcal{V}=\{1,2\}$ contains all functions that are monotone in both inputs. Restricted is valid only over $\mathcal{F}_{1}$. Conservative and Restricted are the minimax affine fixed-length CIs of \citet{Donoho1994}. DHZ is the CI of \citet{deng2021confidence} with critical value $1.80$ and $\hat\sigma$. For DHZ, $n_{+}$ is the number of firms in its block. $\tau^{\ast}$: the level of each one-sided CI ($0.05/12$ for AdaptBonf); $K$: $K_{\tau}$ at $\tau^{\ast}$. $j$: the parameter spaces whose one-sided CIs give the lower and the upper end. $n_{+}$: number of observations with positive weight; ESS: $(\sum_{i}w_{i}^{2})^{-1}$, for the lower and the upper end when they differ (for the fixed-length CIs, of the estimator). The adaptive CIs are in bold.
  \end{minipage}
\end{table}}

\newcommand{\EmpiricalPhaseTwoSensitivityTable}{%
\begin{table}[t]
  \caption{Sensitivity of the 95\% confidence intervals for $f(k_{0},\ell_{0})$ to $C$, to the estimate of $\sigma$ and, with one monotone input, to $\gamma_{J}$}
  \label{tab:empirical-phase2-sensitivity}
  \centering
  \setlength{\tabcolsep}{3pt}
  \begin{tabular}{@{\extracolsep{3pt}} lcccccc} \\[-1.8ex]\hline
    \hline \\[-1.8ex] & \multicolumn{2}{c}{Primary} & \multicolumn{2}{c}{$C_{6}=2\hat C$} & \multicolumn{2}{c}{Local $\hat\sigma_{0}$} \\
    \cline{2-3} \cline{4-5} \cline{6-7} \\[-1.8ex]
    & CI & Length & CI & Length & CI & Length \\
    \hline \\[-1.8ex]
    \multicolumn{7}{@{}l}{\textit{$\mathcal{V}=\{1,2\}$, $J=6$ ($\gamma_{J}=0.001$)}} \\
    AdaptMono & $[9.299,10.347]$ & $1.047$ & $[9.354,10.478]$ & $1.124$ & $[9.345,10.344]$ & $0.999$ \\
    AdaptCalib & $[9.299,10.347]$ & $1.047$ & $[9.354,10.478]$ & $1.124$ & $[9.345,10.344]$ & $0.999$ \\
    AdaptBonf & $[9.209,10.358]$ & $1.149$ & $[9.240,10.512]$ & $1.272$ & $[9.240,10.353]$ & $1.113$ \\
    Conservative & $[2.542,17.258]$ & $14.716$ & $[-4.776,24.551]$ & $29.327$ & $[2.543,17.256]$ & $14.713$ \\
    Restricted & $[9.274,10.003]$ & $0.729$ & $[9.190,10.171]$ & $0.981$ & $[9.281,9.999]$ & $0.718$ \\
    DHZ & $[9.476,9.853]$ & $0.377$ & -- & -- & $[9.481,9.848]$ & $0.366$ \\
    \hline \\[-1.8ex]
    \multicolumn{7}{@{}l}{\textit{$\mathcal{V}=\{1\}$, $J=4$ ($\gamma_{J}=0.4$)}} \\
    AdaptCalib & $[8.223,11.244]$ & $3.021$ & $[7.161,12.362]$ & $5.201$ & $[8.238,11.230]$ & $2.992$ \\
    AdaptMono & $[8.267,11.175]$ & $2.908$ & $[7.155,12.165]$ & $5.010$ & $[8.281,11.153]$ & $2.872$ \\
    Conservative & $[8.047,11.904]$ & $3.857$ & $[6.900,13.012]$ & $6.112$ & $[8.069,11.888]$ & $3.820$ \\
    Restricted & $[9.211,10.044]$ & $0.833$ & $[9.062,10.176]$ & $1.114$ & $[9.221,10.041]$ & $0.820$ \\
    \hline \\[-1.8ex]
    \multicolumn{7}{@{}l}{\textit{$\mathcal{V}=\{1\}$, $J=5$ ($\gamma_{J}=0.2$)}} \\
    AdaptCalib & $[6.566,14.051]$ & $7.485$ & $[6.665,17.494]$ & $10.829$ & $[6.607,14.036]$ & $7.429$ \\
    AdaptMono & $[6.444,13.942]$ & $7.498$ & $[6.644,17.487]$ & $10.843$ & $[6.497,13.924]$ & $7.427$ \\
    Conservative & $[5.274,14.948]$ & $9.674$ & $[0.747,18.157]$ & $17.411$ & $[5.279,14.906]$ & $9.627$ \\
    \hline \\[-1.8ex]
    \multicolumn{7}{@{}l}{\textit{$\mathcal{V}=\{1\}$, $J=6$ ($\gamma_{J}=0.001$): degenerate$^{\dagger}$}} \\
    AdaptCalib & $[6.861,24.358]$ & $17.497$ & $[6.861,39.044]$ & $32.184$ & $[6.925,24.354]$ & $17.429$ \\
    AdaptMono & $[6.719,24.539]$ & $17.819$ & $[6.691,39.228]$ & $32.537$ & $[6.787,24.531]$ & $17.743$ \\
    Conservative & $[7.406,25.593]$ & $18.187$ & $[7.406,40.202]$ & $32.796$ & $[7.455,25.544]$ & $18.089$ \\
    \hline \\[-1.8ex]
    \multicolumn{7}{@{}l}{\textit{$\mathcal{V}=\{2\}$, $J=4$ ($\gamma_{J}=0.4$)}} \\
    AdaptCalib & $[8.546,11.247]$ & $2.701$ & $[7.702,11.764]$ & $4.062$ & $[8.567,11.228]$ & $2.661$ \\
    AdaptMono & $[8.635,11.248]$ & $2.613$ & $[7.910,11.870]$ & $3.960$ & $[8.668,11.225]$ & $2.557$ \\
    Conservative & $[8.099,11.635]$ & $3.536$ & $[7.054,12.544]$ & $5.490$ & $[8.118,11.615]$ & $3.497$ \\
    Restricted & $[9.246,10.080]$ & $0.834$ & $[9.222,10.343]$ & $1.122$ & $[9.253,10.074]$ & $0.821$ \\
    \hline \\[-1.8ex]
    \multicolumn{7}{@{}l}{\textit{$\mathcal{V}=\{2\}$, $J=5$ ($\gamma_{J}=0.2$)}} \\
    AdaptCalib & $[7.570,12.801]$ & $5.230$ & $[7.726,12.964]$ & $5.237$ & $[7.615,12.774]$ & $5.159$ \\
    AdaptMono & $[7.893,12.843]$ & $4.950$ & $[7.802,12.921]$ & $5.119$ & $[7.987,12.789]$ & $4.802$ \\
    Conservative & $[4.756,13.680]$ & $8.924$ & $[0.649,16.693]$ & $16.043$ & $[4.771,13.650]$ & $8.879$ \\
    \hline \\[-1.8ex]
    \multicolumn{7}{@{}l}{\textit{$\mathcal{V}=\{2\}$, $J=6$ ($\gamma_{J}=0.001$): degenerate$^{\dagger}$}} \\
    AdaptCalib & $[9.143,12.761]$ & $3.617$ & $[9.143,12.761]$ & $3.617$ & $[9.208,12.696]$ & $3.487$ \\
    AdaptMono & $[9.385,12.519]$ & $3.135$ & $[9.385,12.519]$ & $3.135$ & $[9.443,12.461]$ & $3.018$ \\
    Conservative & $[2.358,19.491]$ & $17.134$ & $[-5.032,26.705]$ & $31.737$ & $[2.391,19.456]$ & $17.064$ \\
    \hline \\[-1.8ex]
  \end{tabular}
  \par\smallskip
  \begin{minipage}{0.92\linewidth}\footnotesize
  Primary: $C_{6}=\hat C=14.612$ and the global $\hat\sigma=1.088$, as in Table~4. $C_{6}=2\hat C$ doubles every $C_{j}$, and DHZ does not depend on $C$. The local $\hat\sigma_{0}=1.058$ uses the squared differences of the $m=\lceil n^{4/5}\rceil=378$ firms closest to $x_{0}$. With $C_{6}=\hat C/2$, AdaptMono for $\mathcal{V}=\{1,2\}$, whose coverage does not depend on $C$, gives $[9.254,10.209]$ (length $0.955$), with $\tau^{\ast}=0.00792$. For $\mathcal{V}=\{1\}$ and $\mathcal{V}=\{2\}$, $J=5$ and $J=6$ add the next parameter spaces with the same $C_{j}$. Restricted, valid over $\mathcal{F}_{1}$, does not depend on $J$. $^{\dagger}$With $\gamma_{J}=0.001$ the coverage classes restrict $f$ along the other input almost entirely through its oscillation, so that only the firms whose other input equals its median are informative without a bias close to $C_{6}$: for $\mathcal{V}=\{1\}$, one firm with $\ell=\ell_{0}$ and $k<k_{0}$ for the lower ends and no firm with $k>k_{0}$ for the upper ends; for $\mathcal{V}=\{2\}$, one firm with $k=k_{0}$ and $\ell<\ell_{0}$ and one firm with $\ell>\ell_{0}$. The ends of the adaptive CIs use only these firms when there are any, and otherwise allow for a bias close to $C_{6}$.
  \end{minipage}
\end{table}}

\newcommand{\EmpiricalPhaseOneTable}{%
\begin{table}[t]
  \caption{95\% confidence intervals for $f(k_{0},\ell_{0})$ with the local specification of \citet{kwon2020adaptive}: outliers removed from all three variables, firms within distance $1/2$ of $(k_{0},\ell_{0})$, kernel estimate of $\sigma$, a common $C$}
  \label{tab:empirical-phase1}
  \centering
  \setlength{\tabcolsep}{3pt}
  \begin{tabular}{@{\extracolsep{3pt}} lcccc} \\[-1.8ex]\hline
    \hline \\[-1.8ex] & \multicolumn{2}{c}{$C=\hat C$} & \multicolumn{2}{c}{$C=2\hat C$} \\
    \cline{2-3} \cline{4-5} \\[-1.8ex]
    & CI & Length & CI & Length \\
    \hline \\[-1.8ex]
    Conservative ($\gamma = \gamma_6$) & $[6.791,12.641]$ & $5.850$ & $[4.000,15.436]$ & $11.436$ \\
    Restricted ($\gamma = \gamma_1$) & $[9.234,10.277]$ & $1.042$ & $[9.227,10.565]$ & $1.337$ \\
    Conservative, $J=1$ ($\gamma = \gamma_6$) & $[6.767,12.667]$ & $5.901$ & $[3.977,15.463]$ & $11.487$ \\
    Restricted, $J=1$ ($\gamma = \gamma_1$) & $[9.235,10.488]$ & $1.253$ & $[9.180,10.742]$ & $1.562$ \\
    AdaptBonf & $\mathbf{[9.187,10.625]}$ & $\mathbf{1.437}$ & $\mathbf{[9.081,10.866]}$ & $\mathbf{1.785}$ \\
    AdaptCalib & $\mathbf{[9.301,10.594]}$ & $\mathbf{1.293}$ & $\mathbf{[9.142,10.822]}$ & $\mathbf{1.680}$ \\
    CL & $[7.689,12.277]$ & $4.588$ & $[5.907,13.861]$ & $7.954$ \\
    \hline \\[-1.8ex]
  \end{tabular}
  \par\smallskip
  \begin{minipage}{0.92\linewidth}\footnotesize
  The local specification of \citet{kwon2020adaptive}: $n=1{,}636$ after removing outliers in log output, log fixed asset and log labor, of which the $n_{\rm eff}=272$ firms within distance $1/2$ of $(k_{0},\ell_{0})$ enter the CIs; $\hat\sigma=0.956$ from a kernel regression; $\mathcal{V}=\{1,2\}$ and a common $C=\hat C=5.593$ for all six parameter spaces. AdaptBonf uses $\tau=0.05/12$ and AdaptCalib $\tau^{\ast}=0.00942$ ($C=\hat C$) and $0.00892$ ($C=2\hat C$). Conservative and Restricted are the minimax affine fixed-length CIs of \citet{Donoho1994}. The rows marked $J=1$ use the construction of AdaptBonf with $J=1$, and CL is the CI of \citet{cai2004adaptation}. The adaptive CIs are in bold.
  \end{minipage}
\end{table}}

%% file: main.bbl
\begin{thebibliography}{44}

\bibitem[\protect\citeauthoryear{Armstrong}{2015}]{armstrong2015AdaptiveTestingRegression}
\begin{barticle}[author]
\bauthor{\bsnm{Armstrong},~\bfnm{Timothy~B.}\binits{T.~B.}}
(\byear{2015}).
\btitle{Adaptive Testing on a Regression Function at a Point}.
\bjournal{The Annals of Statistics}
\bvolume{43}
\bpages{2086--2101}.
\bdoi{10.1214/15-AOS1342}
\end{barticle}
\endbibitem

\bibitem[\protect\citeauthoryear{Armstrong and
  Koles{\'a}r}{2018}]{armstrong2018optimal}
\begin{barticle}[author]
\bauthor{\bsnm{Armstrong},~\bfnm{Timothy~B.}\binits{T.~B.}} \AND
  \bauthor{\bsnm{Koles{\'a}r},~\bfnm{Michal}\binits{M.}}
(\byear{2018}).
\btitle{Optimal Inference in a Class of Regression Models}.
\bjournal{Econometrica}
\bvolume{86}
\bpages{655--683}.
\bdoi{10.3982/ECTA14434}
\end{barticle}
\endbibitem

\bibitem[\protect\citeauthoryear{Armstrong and
  Koles{\'a}r}{2020}]{armstrong2020simple}
\begin{barticle}[author]
\bauthor{\bsnm{Armstrong},~\bfnm{Timothy~B.}\binits{T.~B.}} \AND
  \bauthor{\bsnm{Koles{\'a}r},~\bfnm{Michal}\binits{M.}}
(\byear{2020}).
\btitle{Simple and Honest Confidence Intervals in Nonparametric Regression}.
\bjournal{Quantitative Economics}
\bvolume{11}
\bpages{1--39}.
\bdoi{10.3982/QE1199}
\end{barticle}
\endbibitem

\bibitem[\protect\citeauthoryear{Banerjee and
  Wellner}{2001}]{banerjee2001likelihood}
\begin{barticle}[author]
\bauthor{\bsnm{Banerjee},~\bfnm{Moulinath}\binits{M.}} \AND
  \bauthor{\bsnm{Wellner},~\bfnm{Jon~A.}\binits{J.~A.}}
(\byear{2001}).
\btitle{Likelihood Ratio Tests for Monotone Functions}.
\bjournal{The Annals of Statistics}
\bvolume{29}
\bpages{1699--1731}.
\bdoi{10.1214/aos/1015345959}
\end{barticle}
\endbibitem

\bibitem[\protect\citeauthoryear{Bauer, Stoer and
  Witzgall}{1961}]{bauer1961absolute}
\begin{barticle}[author]
\bauthor{\bsnm{Bauer},~\bfnm{Friedrich~L.}\binits{F.~L.}},
  \bauthor{\bsnm{Stoer},~\bfnm{Josef}\binits{J.}} \AND
  \bauthor{\bsnm{Witzgall},~\bfnm{Christoph}\binits{C.}}
(\byear{1961}).
\btitle{Absolute and Monotonic Norms}.
\bjournal{Numerische Mathematik}
\bvolume{3}
\bpages{257--264}.
\bdoi{10.1007/BF01386026}
\end{barticle}
\endbibitem

\bibitem[\protect\citeauthoryear{Beliakov}{2005}]{beliakov2005monotonicity}
\begin{barticle}[author]
\bauthor{\bsnm{Beliakov},~\bfnm{Gleb}\binits{G.}}
(\byear{2005}).
\btitle{Monotonicity Preserving Approximation of Multivariate Scattered Data}.
\bjournal{BIT Numerical Mathematics}
\bvolume{45}
\bpages{653--677}.
\bdoi{10.1007/s10543-005-0028-x}
\end{barticle}
\endbibitem

\bibitem[\protect\citeauthoryear{Borwein, Burke and
  Lewis}{2004}]{borwein2004cone}
\begin{barticle}[author]
\bauthor{\bsnm{Borwein},~\bfnm{Jonathan~M.}\binits{J.~M.}},
  \bauthor{\bsnm{Burke},~\bfnm{James~V.}\binits{J.~V.}} \AND
  \bauthor{\bsnm{Lewis},~\bfnm{Adrian~S.}\binits{A.~S.}}
(\byear{2004}).
\btitle{Differentiability of Cone-Monotone Functions on Separable {Banach}
  Space}.
\bjournal{Proceedings of the American Mathematical Society}
\bvolume{132}
\bpages{1067--1076}.
\bdoi{10.1090/S0002-9939-03-07149-1}
\end{barticle}
\endbibitem

\bibitem[\protect\citeauthoryear{Cai}{2012}]{cai2012MinimaxAdaptiveInference}
\begin{barticle}[author]
\bauthor{\bsnm{Cai},~\bfnm{T.~Tony}\binits{T.~T.}}
(\byear{2012}).
\btitle{Minimax and Adaptive Inference in Nonparametric Function Estimation}.
\bjournal{Statistical Science}
\bvolume{27}
\bpages{31--50}.
\bdoi{10.1214/11-STS355}
\end{barticle}
\endbibitem

\bibitem[\protect\citeauthoryear{Cai and Low}{2004}]{cai2004adaptation}
\begin{barticle}[author]
\bauthor{\bsnm{Cai},~\bfnm{T.~Tony}\binits{T.~T.}} \AND
  \bauthor{\bsnm{Low},~\bfnm{Mark~G.}\binits{M.~G.}}
(\byear{2004}).
\btitle{An Adaptation Theory for Nonparametric Confidence Intervals}.
\bjournal{The Annals of Statistics}
\bvolume{32}
\bpages{1805--1840}.
\bdoi{10.1214/009053604000000049}
\end{barticle}
\endbibitem

\bibitem[\protect\citeauthoryear{Cai and
  Low}{2006}]{cai2006AdaptiveConfidenceBalls}
\begin{barticle}[author]
\bauthor{\bsnm{Cai},~\bfnm{T.~Tony}\binits{T.~T.}} \AND
  \bauthor{\bsnm{Low},~\bfnm{Mark~G.}\binits{M.~G.}}
(\byear{2006}).
\btitle{Adaptive Confidence Balls}.
\bjournal{The Annals of Statistics}
\bvolume{34}
\bpages{202--228}.
\bdoi{10.1214/009053606000000146}
\end{barticle}
\endbibitem

\bibitem[\protect\citeauthoryear{Cai, Low and
  Xia}{2013}]{cai2013AdaptiveConfidenceIntervals}
\begin{barticle}[author]
\bauthor{\bsnm{Cai},~\bfnm{T.~Tony}\binits{T.~T.}},
  \bauthor{\bsnm{Low},~\bfnm{Mark~G.}\binits{M.~G.}} \AND
  \bauthor{\bsnm{Xia},~\bfnm{Yin}\binits{Y.}}
(\byear{2013}).
\btitle{Adaptive Confidence Intervals for Regression Functions under Shape
  Constraints}.
\bjournal{The Annals of Statistics}
\bvolume{41}
\bpages{722--750}.
\bdoi{10.1214/12-AOS1068}
\end{barticle}
\endbibitem

\bibitem[\protect\citeauthoryear{Cai, Low and
  Ma}{2014}]{cai2014AdaptiveConfidenceBands}
\begin{barticle}[author]
\bauthor{\bsnm{Cai},~\bfnm{T.~Tony}\binits{T.~T.}},
  \bauthor{\bsnm{Low},~\bfnm{Mark~G.}\binits{M.~G.}} \AND
  \bauthor{\bsnm{Ma},~\bfnm{Zongming}\binits{Z.}}
(\byear{2014}).
\btitle{Adaptive Confidence Bands for Nonparametric Regression Functions}.
\bjournal{Journal of the American Statistical Association}
\bvolume{109}
\bpages{1054--1070}.
\bdoi{10.1080/01621459.2013.879260}
\end{barticle}
\endbibitem

\bibitem[\protect\citeauthoryear{Cattaneo, Jansson and
  Nagasawa}{2024}]{cattaneo2024bootstrap}
\begin{barticle}[author]
\bauthor{\bsnm{Cattaneo},~\bfnm{Matias~D.}\binits{M.~D.}},
  \bauthor{\bsnm{Jansson},~\bfnm{Michael}\binits{M.}} \AND
  \bauthor{\bsnm{Nagasawa},~\bfnm{Kenichi}\binits{K.}}
(\byear{2024}).
\btitle{Bootstrap-Assisted Inference for Generalized {Grenander}-Type
  Estimators}.
\bjournal{The Annals of Statistics}
\bvolume{52}
\bpages{1509--1533}.
\bdoi{10.1214/24-AOS2402}
\end{barticle}
\endbibitem

\bibitem[\protect\citeauthoryear{Chatterjee, Guntuboyina and
  Sen}{2018}]{chatterjee2018matrix}
\begin{barticle}[author]
\bauthor{\bsnm{Chatterjee},~\bfnm{Sabyasachi}\binits{S.}},
  \bauthor{\bsnm{Guntuboyina},~\bfnm{Adityanand}\binits{A.}} \AND
  \bauthor{\bsnm{Sen},~\bfnm{Bodhisattva}\binits{B.}}
(\byear{2018}).
\btitle{On Matrix Estimation under Monotonicity Constraints}.
\bjournal{Bernoulli}
\bvolume{24}
\bpages{1072--1100}.
\bdoi{10.3150/16-BEJ865}
\end{barticle}
\endbibitem

\bibitem[\protect\citeauthoryear{Chen and Samworth}{2016}]{chen2016shape}
\begin{barticle}[author]
\bauthor{\bsnm{Chen},~\bfnm{Yining}\binits{Y.}} \AND
  \bauthor{\bsnm{Samworth},~\bfnm{Richard~J.}\binits{R.~J.}}
(\byear{2016}).
\btitle{Generalized Additive and Index Models with Shape Constraints}.
\bjournal{Journal of the Royal Statistical Society: Series B (Statistical
  Methodology)}
\bvolume{78}
\bpages{729--754}.
\bdoi{10.1111/rssb.12137}
\end{barticle}
\endbibitem

\bibitem[\protect\citeauthoryear{Chernozhukov, Fern{\'a}ndez-Val and
  Galichon}{2009}]{chernozhukov2009improving}
\begin{barticle}[author]
\bauthor{\bsnm{Chernozhukov},~\bfnm{Victor}\binits{V.}},
  \bauthor{\bsnm{Fern{\'a}ndez-Val},~\bfnm{Iv{\'a}n}\binits{I.}} \AND
  \bauthor{\bsnm{Galichon},~\bfnm{Alfred}\binits{A.}}
(\byear{2009}).
\btitle{Improving Point and Interval Estimators of Monotone Functions by
  Rearrangement}.
\bjournal{Biometrika}
\bvolume{96}
\bpages{559--575}.
\bdoi{10.1093/biomet/asp030}
\end{barticle}
\endbibitem

\bibitem[\protect\citeauthoryear{Chetverikov, Santos and
  Shaikh}{2018}]{chetverikov2018econometrics}
\begin{barticle}[author]
\bauthor{\bsnm{Chetverikov},~\bfnm{Denis}\binits{D.}},
  \bauthor{\bsnm{Santos},~\bfnm{Andres}\binits{A.}} \AND
  \bauthor{\bsnm{Shaikh},~\bfnm{Azeem~M.}\binits{A.~M.}}
(\byear{2018}).
\btitle{The Econometrics of Shape Restrictions}.
\bjournal{Annual Review of Economics}
\bvolume{10}
\bpages{31--63}.
\bdoi{10.1146/annurev-economics-080217-053417}
\end{barticle}
\endbibitem

\bibitem[\protect\citeauthoryear{Datta, Mukherjee and
  Sen}{2026}]{datta2026optimal}
\begin{barticle}[author]
\bauthor{\bsnm{Datta},~\bfnm{Ashley~(Pratyay)}\binits{A.~P.}},
  \bauthor{\bsnm{Mukherjee},~\bfnm{Somabha}\binits{S.}} \AND
  \bauthor{\bsnm{Sen},~\bfnm{Bodhisattva}\binits{B.}}
(\byear{2026}).
\btitle{Optimal Confidence Bands for Shape-Restricted Regression in
  Multidimensions}.
\bjournal{Bernoulli}
\bvolume{32}
\bpages{2916--2943}.
\bdoi{10.3150/26-BEJ1973}
\end{barticle}
\endbibitem

\bibitem[\protect\citeauthoryear{Deng, Han and
  Zhang}{2021}]{deng2021confidence}
\begin{barticle}[author]
\bauthor{\bsnm{Deng},~\bfnm{Hang}\binits{H.}},
  \bauthor{\bsnm{Han},~\bfnm{Qiyang}\binits{Q.}} \AND
  \bauthor{\bsnm{Zhang},~\bfnm{Cun-Hui}\binits{C.-H.}}
(\byear{2021}).
\btitle{Confidence Intervals for Multiple Isotonic Regression and Other
  Monotone Models}.
\bjournal{The Annals of Statistics}
\bvolume{49}
\bpages{2021--2052}.
\bdoi{10.1214/20-AOS2025}
\end{barticle}
\endbibitem

\bibitem[\protect\citeauthoryear{Deng and Zhang}{2020}]{deng2020isotonic}
\begin{barticle}[author]
\bauthor{\bsnm{Deng},~\bfnm{Hang}\binits{H.}} \AND
  \bauthor{\bsnm{Zhang},~\bfnm{Cun-Hui}\binits{C.-H.}}
(\byear{2020}).
\btitle{Isotonic Regression in Multi-Dimensional Spaces and Graphs}.
\bjournal{The Annals of Statistics}
\bvolume{48}
\bpages{3672--3698}.
\bdoi{10.1214/20-AOS1947}
\end{barticle}
\endbibitem

\bibitem[\protect\citeauthoryear{Donoho}{1994}]{Donoho1994}
\begin{barticle}[author]
\bauthor{\bsnm{Donoho},~\bfnm{David~L.}\binits{D.~L.}}
(\byear{1994}).
\btitle{Statistical Estimation and Optimal Recovery}.
\bjournal{The Annals of Statistics}
\bvolume{22}
\bpages{238--270}.
\bdoi{10.1214/aos/1176325367}
\end{barticle}
\endbibitem

\bibitem[\protect\citeauthoryear{D{\"u}mbgen}{1998}]{dumbgen1998NewGoodnessoffitTests}
\begin{barticle}[author]
\bauthor{\bsnm{D{\"u}mbgen},~\bfnm{Lutz}\binits{L.}}
(\byear{1998}).
\btitle{New Goodness-of-Fit Tests and Their Application to Nonparametric
  Confidence Sets}.
\bjournal{The Annals of Statistics}
\bvolume{26}
\bpages{288--314}.
\bdoi{10.1214/aos/1030563987}
\end{barticle}
\endbibitem

\bibitem[\protect\citeauthoryear{D{\"u}mbgen}{2003}]{duembgen2003optimal}
\begin{barticle}[author]
\bauthor{\bsnm{D{\"u}mbgen},~\bfnm{Lutz}\binits{L.}}
(\byear{2003}).
\btitle{Optimal Confidence Bands for Shape-Restricted Curves}.
\bjournal{Bernoulli}
\bvolume{9}
\bpages{423--449}.
\bdoi{10.3150/bj/1065444812}
\end{barticle}
\endbibitem

\bibitem[\protect\citeauthoryear{D{\"u}mbgen and
  Spokoiny}{2001}]{dumbgen2001multiscale}
\begin{barticle}[author]
\bauthor{\bsnm{D{\"u}mbgen},~\bfnm{Lutz}\binits{L.}} \AND
  \bauthor{\bsnm{Spokoiny},~\bfnm{Vladimir~G.}\binits{V.~G.}}
(\byear{2001}).
\btitle{Multiscale Testing of Qualitative Hypotheses}.
\bjournal{The Annals of Statistics}
\bvolume{29}
\bpages{124--152}.
\bdoi{10.1214/aos/996986504}
\end{barticle}
\endbibitem

\bibitem[\protect\citeauthoryear{Fokianos, Leucht and
  Neumann}{2020}]{fokianos2020integrated}
\begin{barticle}[author]
\bauthor{\bsnm{Fokianos},~\bfnm{Konstantinos}\binits{K.}},
  \bauthor{\bsnm{Leucht},~\bfnm{Anne}\binits{A.}} \AND
  \bauthor{\bsnm{Neumann},~\bfnm{Michael~H.}\binits{M.~H.}}
(\byear{2020}).
\btitle{On Integrated ${L}^{1}$ Convergence Rate of an Isotonic Regression
  Estimator for Multivariate Observations}.
\bjournal{IEEE Transactions on Information Theory}
\bvolume{66}
\bpages{6389--6402}.
\bdoi{10.1109/TIT.2020.3013390}
\end{barticle}
\endbibitem

\bibitem[\protect\citeauthoryear{Genovese and
  Wasserman}{2005}]{genovese2005ConfidenceSetsNonparametric}
\begin{barticle}[author]
\bauthor{\bsnm{Genovese},~\bfnm{Christopher~R.}\binits{C.~R.}} \AND
  \bauthor{\bsnm{Wasserman},~\bfnm{Larry}\binits{L.}}
(\byear{2005}).
\btitle{Confidence Sets for Nonparametric Wavelet Regression}.
\bjournal{The Annals of Statistics}
\bvolume{33}
\bpages{698--729}.
\bdoi{10.1214/009053605000000011}
\end{barticle}
\endbibitem

\bibitem[\protect\citeauthoryear{Genovese and
  Wasserman}{2008}]{genovese2008AdaptiveConfidenceBands}
\begin{barticle}[author]
\bauthor{\bsnm{Genovese},~\bfnm{Christopher~R.}\binits{C.~R.}} \AND
  \bauthor{\bsnm{Wasserman},~\bfnm{Larry}\binits{L.}}
(\byear{2008}).
\btitle{Adaptive Confidence Bands}.
\bjournal{The Annals of Statistics}
\bvolume{36}
\bpages{875--905}.
\bdoi{10.1214/07-AOS500}
\end{barticle}
\endbibitem

\bibitem[\protect\citeauthoryear{Gin{\'e} and
  Nickl}{2010}]{gine2010ConfidenceBandsDensitya}
\begin{barticle}[author]
\bauthor{\bsnm{Gin{\'e}},~\bfnm{Evarist}\binits{E.}} \AND
  \bauthor{\bsnm{Nickl},~\bfnm{Richard}\binits{R.}}
(\byear{2010}).
\btitle{Confidence Bands in Density Estimation}.
\bjournal{The Annals of Statistics}
\bvolume{38}
\bpages{1122--1170}.
\bdoi{10.1214/09-AOS738}
\end{barticle}
\endbibitem

\bibitem[\protect\citeauthoryear{Groeneboom and
  Jongbloed}{2014}]{groeneboom2014shape}
\begin{bbook}[author]
\bauthor{\bsnm{Groeneboom},~\bfnm{Piet}\binits{P.}} \AND
  \bauthor{\bsnm{Jongbloed},~\bfnm{Geurt}\binits{G.}}
(\byear{2014}).
\btitle{Nonparametric Estimation under Shape Constraints: Estimators,
  Algorithms and Asymptotics}.
\bpublisher{Cambridge University Press}, \baddress{Cambridge}.
\bdoi{10.1017/CBO9781139020893}
\end{bbook}
\endbibitem

\bibitem[\protect\citeauthoryear{Groeneboom and
  Jongbloed}{2015}]{groeneboom2015nonparametric}
\begin{barticle}[author]
\bauthor{\bsnm{Groeneboom},~\bfnm{Piet}\binits{P.}} \AND
  \bauthor{\bsnm{Jongbloed},~\bfnm{Geurt}\binits{G.}}
(\byear{2015}).
\btitle{Nonparametric Confidence Intervals for Monotone Functions}.
\bjournal{The Annals of Statistics}
\bvolume{43}
\bpages{2019--2054}.
\bdoi{10.1214/15-AOS1335}
\end{barticle}
\endbibitem

\bibitem[\protect\citeauthoryear{Guntuboyina and
  Sen}{2018}]{guntuboyina2018shape}
\begin{barticle}[author]
\bauthor{\bsnm{Guntuboyina},~\bfnm{Adityanand}\binits{A.}} \AND
  \bauthor{\bsnm{Sen},~\bfnm{Bodhisattva}\binits{B.}}
(\byear{2018}).
\btitle{Nonparametric Shape-Restricted Regression}.
\bjournal{Statistical Science}
\bvolume{33}
\bpages{568--594}.
\bdoi{10.1214/18-STS665}
\end{barticle}
\endbibitem

\bibitem[\protect\citeauthoryear{Han and Zhang}{2020}]{han2020limit}
\begin{barticle}[author]
\bauthor{\bsnm{Han},~\bfnm{Qiyang}\binits{Q.}} \AND
  \bauthor{\bsnm{Zhang},~\bfnm{Cun-Hui}\binits{C.-H.}}
(\byear{2020}).
\btitle{Limit Distribution Theory for Block Estimators in Multiple Isotonic
  Regression}.
\bjournal{The Annals of Statistics}
\bvolume{48}
\bpages{3251--3282}.
\bdoi{10.1214/19-AOS1928}
\end{barticle}
\endbibitem

\bibitem[\protect\citeauthoryear{Han
  et~al.}{2019}]{han2019IsotonicRegressionGeneral}
\begin{barticle}[author]
\bauthor{\bsnm{Han},~\bfnm{Qiyang}\binits{Q.}},
  \bauthor{\bsnm{Wang},~\bfnm{Tengyao}\binits{T.}},
  \bauthor{\bsnm{Chatterjee},~\bfnm{Sabyasachi}\binits{S.}} \AND
  \bauthor{\bsnm{Samworth},~\bfnm{Richard~J.}\binits{R.~J.}}
(\byear{2019}).
\btitle{Isotonic Regression in General Dimensions}.
\bjournal{The Annals of Statistics}
\bvolume{47}
\bpages{2440--2471}.
\bdoi{10.1214/18-AOS1753}
\end{barticle}
\endbibitem

\bibitem[\protect\citeauthoryear{Hengartner and
  Stark}{1995}]{hengartner1995FiniteSampleConfidenceEnvelopes}
\begin{barticle}[author]
\bauthor{\bsnm{Hengartner},~\bfnm{Nicolas~W.}\binits{N.~W.}} \AND
  \bauthor{\bsnm{Stark},~\bfnm{Philip~B.}\binits{P.~B.}}
(\byear{1995}).
\btitle{Finite-Sample Confidence Envelopes for Shape-Restricted Densities}.
\bjournal{The Annals of Statistics}
\bvolume{23}
\bpages{525--550}.
\bdoi{10.1214/aos/1176324534}
\end{barticle}
\endbibitem

\bibitem[\protect\citeauthoryear{Hoffmann and
  Nickl}{2011}]{hoffmann2011AdaptiveInferenceConfidencea}
\begin{barticle}[author]
\bauthor{\bsnm{Hoffmann},~\bfnm{Marc}\binits{M.}} \AND
  \bauthor{\bsnm{Nickl},~\bfnm{Richard}\binits{R.}}
(\byear{2011}).
\btitle{On Adaptive Inference and Confidence Bands}.
\bjournal{The Annals of Statistics}
\bvolume{39}
\bpages{2383--2409}.
\bdoi{10.1214/11-AOS903}
\end{barticle}
\endbibitem

\bibitem[\protect\citeauthoryear{Horowitz and
  Lee}{2017}]{horowitz2017nonparametric}
\begin{barticle}[author]
\bauthor{\bsnm{Horowitz},~\bfnm{Joel~L.}\binits{J.~L.}} \AND
  \bauthor{\bsnm{Lee},~\bfnm{Sokbae}\binits{S.}}
(\byear{2017}).
\btitle{Nonparametric Estimation and Inference under Shape Restrictions}.
\bjournal{Journal of Econometrics}
\bvolume{201}
\bpages{108--126}.
\bdoi{10.1016/j.jeconom.2017.06.019}
\end{barticle}
\endbibitem

\bibitem[\protect\citeauthoryear{Jacho-Ch{\'a}vez, Lewbel and
  Linton}{2010}]{jacho2010identification}
\begin{barticle}[author]
\bauthor{\bsnm{Jacho-Ch{\'a}vez},~\bfnm{David}\binits{D.}},
  \bauthor{\bsnm{Lewbel},~\bfnm{Arthur}\binits{A.}} \AND
  \bauthor{\bsnm{Linton},~\bfnm{Oliver}\binits{O.}}
(\byear{2010}).
\btitle{Identification and Nonparametric Estimation of a Transformed Additively
  Separable Model}.
\bjournal{Journal of Econometrics}
\bvolume{156}
\bpages{392--407}.
\bdoi{10.1016/j.jeconom.2009.11.008}
\end{barticle}
\endbibitem

\bibitem[\protect\citeauthoryear{Kwon and
  Kwon}{2020a}]{kwon2020InferenceRegressionDiscontinuity}
\begin{bmisc}[author]
\bauthor{\bsnm{Kwon},~\bfnm{Koohyun}\binits{K.}} \AND
  \bauthor{\bsnm{Kwon},~\bfnm{Soonwoo}\binits{S.}}
(\byear{2020}a).
\btitle{Inference in Regression Discontinuity Designs under Monotonicity}.
\bhowpublished{Working paper, arXiv:2011.14216}.
\end{bmisc}
\endbibitem

\bibitem[\protect\citeauthoryear{Kwon and Kwon}{2020b}]{kwon2020adaptive}
\begin{bmisc}[author]
\bauthor{\bsnm{Kwon},~\bfnm{Koohyun}\binits{K.}} \AND
  \bauthor{\bsnm{Kwon},~\bfnm{Soonwoo}\binits{S.}}
(\byear{2020}b).
\btitle{Adaptive Inference in Multivariate Nonparametric Regression Models
  Under Monotonicity}.
\bhowpublished{Working paper, arXiv:2011.14219v1}.
\end{bmisc}
\endbibitem

\bibitem[\protect\citeauthoryear{Lepskii}{1991}]{lepskii1991ProblemAdaptiveEstimationa}
\begin{barticle}[author]
\bauthor{\bsnm{Lepskii},~\bfnm{O.~V.}\binits{O.~V.}}
(\byear{1991}).
\btitle{On a Problem of Adaptive Estimation in {G}aussian White Noise}.
\bjournal{Theory of Probability \& Its Applications}
\bvolume{35}
\bpages{454--466}.
\bdoi{10.1137/1135065}
\end{barticle}
\endbibitem

\bibitem[\protect\citeauthoryear{Low}{1997}]{low1997nonparametric}
\begin{barticle}[author]
\bauthor{\bsnm{Low},~\bfnm{Mark~G.}\binits{M.~G.}}
(\byear{1997}).
\btitle{On Nonparametric Confidence Intervals}.
\bjournal{The Annals of Statistics}
\bvolume{25}
\bpages{2547--2554}.
\bdoi{10.1214/aos/1030741084}
\end{barticle}
\endbibitem

\bibitem[\protect\citeauthoryear{Robins and
  {van~der~Vaart}}{2006}]{robins2006AdaptiveNonparametricConfidencea}
\begin{barticle}[author]
\bauthor{\bsnm{Robins},~\bfnm{James}\binits{J.}} \AND
  \bauthor{\bsnm{{van~der~Vaart}},~\bfnm{Aad}\binits{A.}}
(\byear{2006}).
\btitle{Adaptive Nonparametric Confidence Sets}.
\bjournal{The Annals of Statistics}
\bvolume{34}
\bpages{229--253}.
\bdoi{10.1214/009053605000000877}
\end{barticle}
\endbibitem

\bibitem[\protect\citeauthoryear{Wang and Ghosal}{2023}]{wang2023coverage}
\begin{barticle}[author]
\bauthor{\bsnm{Wang},~\bfnm{Kang}\binits{K.}} \AND
  \bauthor{\bsnm{Ghosal},~\bfnm{Subhashis}\binits{S.}}
(\byear{2023}).
\btitle{Coverage of Credible Intervals in {B}ayesian Multivariate Isotonic
  Regression}.
\bjournal{The Annals of Statistics}
\bvolume{51}
\bpages{1376--1400}.
\bdoi{10.1214/23-AOS2298}
\end{barticle}
\endbibitem

\bibitem[\protect\citeauthoryear{Westling, Gilbert and
  Carone}{2020}]{westling2020causal}
\begin{barticle}[author]
\bauthor{\bsnm{Westling},~\bfnm{Ted}\binits{T.}},
  \bauthor{\bsnm{Gilbert},~\bfnm{Peter}\binits{P.}} \AND
  \bauthor{\bsnm{Carone},~\bfnm{Marco}\binits{M.}}
(\byear{2020}).
\btitle{Causal Isotonic Regression}.
\bjournal{Journal of the Royal Statistical Society: Series B (Statistical
  Methodology)}
\bvolume{82}
\bpages{719--747}.
\bdoi{10.1111/rssb.12372}
\end{barticle}
\endbibitem

\end{thebibliography}


\begin{thebibliography}{13}

\bibitem[\protect\citeauthoryear{Armstrong and
  Koles{\'a}r}{2018}]{armstrong2018optimal}
\begin{barticle}[author]
\bauthor{\bsnm{Armstrong},~\bfnm{Timothy~B.}\binits{T.~B.}} \AND
  \bauthor{\bsnm{Koles{\'a}r},~\bfnm{Michal}\binits{M.}}
(\byear{2018}).
\btitle{Optimal Inference in a Class of Regression Models}.
\bjournal{Econometrica}
\bvolume{86}
\bpages{655--683}.
\bdoi{10.3982/ECTA14434}
\end{barticle}
\endbibitem

\bibitem[\protect\citeauthoryear{Armstrong and
  Koles{\'a}r}{2020}]{armstrong2020simple}
\begin{barticle}[author]
\bauthor{\bsnm{Armstrong},~\bfnm{Timothy~B.}\binits{T.~B.}} \AND
  \bauthor{\bsnm{Koles{\'a}r},~\bfnm{Michal}\binits{M.}}
(\byear{2020}).
\btitle{Simple and Honest Confidence Intervals in Nonparametric Regression}.
\bjournal{Quantitative Economics}
\bvolume{11}
\bpages{1--39}.
\bdoi{10.3982/QE1199}
\end{barticle}
\endbibitem

\bibitem[\protect\citeauthoryear{Cai and Low}{2004a}]{cai2004adaptation}
\begin{barticle}[author]
\bauthor{\bsnm{Cai},~\bfnm{T.~Tony}\binits{T.~T.}} \AND
  \bauthor{\bsnm{Low},~\bfnm{Mark~G.}\binits{M.~G.}}
(\byear{2004}a).
\btitle{An Adaptation Theory for Nonparametric Confidence Intervals}.
\bjournal{The Annals of Statistics}
\bvolume{32}
\bpages{1805--1840}.
\bdoi{10.1214/009053604000000049}
\end{barticle}
\endbibitem

\bibitem[\protect\citeauthoryear{Cai and Low}{2004b}]{cai2004minimax}
\begin{barticle}[author]
\bauthor{\bsnm{Cai},~\bfnm{T.~Tony}\binits{T.~T.}} \AND
  \bauthor{\bsnm{Low},~\bfnm{Mark~G.}\binits{M.~G.}}
(\byear{2004}b).
\btitle{Minimax Estimation of Linear Functionals over Nonconvex Parameter
  Spaces}.
\bjournal{The Annals of Statistics}
\bvolume{32}
\bpages{552--576}.
\bdoi{10.1214/009053604000000094}
\end{barticle}
\endbibitem

\bibitem[\protect\citeauthoryear{Cai, Low and
  Xia}{2013}]{cai2013AdaptiveConfidenceIntervals}
\begin{barticle}[author]
\bauthor{\bsnm{Cai},~\bfnm{T.~Tony}\binits{T.~T.}},
  \bauthor{\bsnm{Low},~\bfnm{Mark~G.}\binits{M.~G.}} \AND
  \bauthor{\bsnm{Xia},~\bfnm{Yin}\binits{Y.}}
(\byear{2013}).
\btitle{Adaptive Confidence Intervals for Regression Functions under Shape
  Constraints}.
\bjournal{The Annals of Statistics}
\bvolume{41}
\bpages{722--750}.
\bdoi{10.1214/12-AOS1068}
\end{barticle}
\endbibitem

\bibitem[\protect\citeauthoryear{Donoho}{1994}]{Donoho1994}
\begin{barticle}[author]
\bauthor{\bsnm{Donoho},~\bfnm{David~L.}\binits{D.~L.}}
(\byear{1994}).
\btitle{Statistical Estimation and Optimal Recovery}.
\bjournal{The Annals of Statistics}
\bvolume{22}
\bpages{238--270}.
\bdoi{10.1214/aos/1176325367}
\end{barticle}
\endbibitem

\bibitem[\protect\citeauthoryear{Genz}{2004}]{genz2004numerical}
\begin{barticle}[author]
\bauthor{\bsnm{Genz},~\bfnm{Alan}\binits{A.}}
(\byear{2004}).
\btitle{Numerical Computation of Rectangular Bivariate and Trivariate Normal
  and $t$ Probabilities}.
\bjournal{Statistics and Computing}
\bvolume{14}
\bpages{251--260}.
\bdoi{10.1023/B:STCO.0000035304.20635.31}
\end{barticle}
\endbibitem

\bibitem[\protect\citeauthoryear{Han and Zhang}{2020}]{han2020limit}
\begin{barticle}[author]
\bauthor{\bsnm{Han},~\bfnm{Qiyang}\binits{Q.}} \AND
  \bauthor{\bsnm{Zhang},~\bfnm{Cun-Hui}\binits{C.-H.}}
(\byear{2020}).
\btitle{Limit Distribution Theory for Block Estimators in Multiple Isotonic
  Regression}.
\bjournal{The Annals of Statistics}
\bvolume{48}
\bpages{3251--3282}.
\bdoi{10.1214/19-AOS1928}
\end{barticle}
\endbibitem

\bibitem[\protect\citeauthoryear{Hunter}{1976}]{hunter1976upper}
\begin{barticle}[author]
\bauthor{\bsnm{Hunter},~\bfnm{David}\binits{D.}}
(\byear{1976}).
\btitle{An Upper Bound for the Probability of a Union}.
\bjournal{Journal of Applied Probability}
\bvolume{13}
\bpages{597--603}.
\bdoi{10.2307/3212481}
\end{barticle}
\endbibitem

\bibitem[\protect\citeauthoryear{Kwon and
  Kwon}{2020}]{kwon2020InferenceRegressionDiscontinuity}
\begin{bmisc}[author]
\bauthor{\bsnm{Kwon},~\bfnm{Koohyun}\binits{K.}} \AND
  \bauthor{\bsnm{Kwon},~\bfnm{Soonwoo}\binits{S.}}
(\byear{2020}).
\btitle{Inference in Regression Discontinuity Designs under Monotonicity}.
\bhowpublished{Working paper, arXiv:2011.14216}.
\end{bmisc}
\endbibitem

\bibitem[\protect\citeauthoryear{Slepian}{1962}]{slepian1962onesided}
\begin{barticle}[author]
\bauthor{\bsnm{Slepian},~\bfnm{David}\binits{D.}}
(\byear{1962}).
\btitle{The One-Sided Barrier Problem for {G}aussian Noise}.
\bjournal{Bell System Technical Journal}
\bvolume{41}
\bpages{463--501}.
\bdoi{10.1002/j.1538-7305.1962.tb02419.x}
\end{barticle}
\endbibitem

\bibitem[\protect\citeauthoryear{{van~der~Vaart} and
  Wellner}{1996}]{vandervaart1996WeakConvergenceEmpirical}
\begin{bbook}[author]
\bauthor{\bsnm{{van~der~Vaart}},~\bfnm{A.~W.}\binits{A.~W.}} \AND
  \bauthor{\bsnm{Wellner},~\bfnm{Jon~A.}\binits{J.~A.}}
(\byear{1996}).
\btitle{Weak Convergence and Empirical Processes: With Applications to
  Statistics}.
\bpublisher{Springer}, \baddress{New York}.
\bdoi{10.1007/978-1-4757-2545-2}
\end{bbook}
\endbibitem

\bibitem[\protect\citeauthoryear{Vershynin}{2018}]{vershynin2018high}
\begin{bbook}[author]
\bauthor{\bsnm{Vershynin},~\bfnm{Roman}\binits{R.}}
(\byear{2018}).
\btitle{High-Dimensional Probability: An Introduction with Applications in Data
  Science}.
\bpublisher{Cambridge University Press}, \baddress{Cambridge}.
\bdoi{10.1017/9781108231596}
\end{bbook}
\endbibitem

\end{thebibliography}
